\documentclass[12pt]{amsart}
\usepackage{amsmath,amssymb,latexsym}
\usepackage{fullpage}

\theoremstyle{plain}
\newtheorem{theorem}{Theorem}
\newtheorem{proposition}{Proposition}

\newtheorem{corollary}{Corollary}

\begin{document}
\date{}

\title[Lifting of Automorphic Forms ]
{\bf On Certain Global Constructions of Automorphic Forms Related to
Small Representations of $F_4$}
\author{ David Ginzburg}

\dedicatory{ To the memory of S. Rallis}

\begin{abstract}
In this paper we consider some global constructions of liftings of
automorphic representations attached to some commuting pairs in the
exceptional group $F_4$. We consider two families of integrals. The
first uses the minimal representation on the double cover of
$F_4$, and in the second we consider examples of integrals of
descent type associated with unipotent orbits of $F_4$.

\end{abstract}

\thanks{ The author is partly supported by the Israel Science
Foundation grant number 259/14 }

\address{ School of Mathematical Sciences\\
Sackler Faculty of Exact Sciences\\ Tel-Aviv University, Israel
69978 } \maketitle \baselineskip=18pt

\section{\bf Introduction}

One of the important aspects of the Langlands conjectures is the
study of correspondence of automorphic representations between two
groups. Let $H$ and $G$ be two linear algebraic groups defined over a global field $F$. Given a
homomorphism between the $L$ groups of these two groups, the general
conjectures predict a functorial lifting between automorphic
representations of $H$ and $G$.

There are several ways to study lifting of automorphic
representations between two groups. Two powerful methods are the
converse Theorem and the Arthur trace formula. The strength of these
methods are their generality. On the other hand these methods are
not explicit, in the sense that they do not actually construct the
correspondence, but rather prove its existence.

A third method to construct these liftings is what we refer to as
the small representations method. The idea of this method is as
follows. Let $M$ be a reductive group. Suppose that we can embed the
groups $G$ and $H$ as a commuting pair inside $M$. By that we mean
that we can embed these two groups inside $M$ and under this
embedding the two groups commute one with the other. Let $\Theta$
denote an automorphic  representation of $M({\bf A})$. Here ${\bf
A}$ is the ring of adeles of a global field $F$. Let $\pi$ denote an
automorphic representation of $H({\bf A})$. Then one can construct
an automorphic function of $G({\bf A})$ by means of the integral
\begin{equation}\label{intro1}
f(g)=\int\limits_{H(F)\backslash H({\bf A})}
\int\limits_{V(F)\backslash V({\bf A})}\varphi_\pi(h)
\theta(v(h,g))\psi_V(v)dvdh
\end{equation}
Here $V$ is a certain unipotent subgroup of $M$ which is normalized
by the embedding of $G$ and $H$. Also $\psi_V$ is a character of $V$
which is preserved by these two groups. Finally, the function
$\varphi_\pi$ is a vector in the space of $\pi$, and $\theta$ is a
vector in the space of $\Theta$. Assuming that the above integral
converges (this will happen, for example, if $\pi$  is a cuspidal
representation), denote by $\sigma(\pi,\Theta)$ the automorphic
representation of $G({\bf A})$ generated by all functions $f(g)$
defined above. The above discussion can be easily extended to automorphic representations of
metaplectic covering groups of algebraic groups. Obviously, when considering integrals of the
type of \eqref{intro1} defined over metaplectic  covering groups, one should make sure that the
cover splits. Otherwise the integrals will not be well defined.

Given the above construction, there are several natural problems
regarding the relations between the representations $\pi$ and
$\sigma(\pi,\Theta)$. The first problem is the issue of the
cuspidality of $\sigma(\pi,\Theta)$. In other words, what are the
conditions on $\pi$ and $\Theta$, if any,  so that
$\sigma(\pi,\Theta)$ will be a cuspidal representation of $G({\bf
A})$. The second problem is to understand when $\sigma(\pi,\Theta)$
is nonzero. The third problem is to study the functoriality of the
lift. Assume, for example, that $\sigma(\pi,\Theta)$ is a direct sum of irreducible automorphic 
representations.  Then, one wants to check the relations between the
unramified constituents of $\pi$ and $\Theta$ with those of each irreducible summand of
$\sigma(\pi,\Theta)$. There are other problems one can study. For
example, when is the representation $\sigma(\pi,\Theta)$
irreducible. Another interesting problem is to try to characterize
the image of the lift by means of a period integral. However, the
above three problems are the basic ones, and should be studied prior
to anything else. The machinery for studying these issues is quite
routine. To verify cuspidality one needs to study the constant terms
along unipotent radicals of maximal parabolic subgroups of $G$. The
nonvanishing of the lift is usually done by showing that the
function $f(g)$ has a certain nonzero Fourier coefficient. The
unramified computation is done by the study of certain bilinear or
trilinear forms.

We consider a few examples. There are two extreme cases. The first
one is when the unipotent group is trivial. In this case integral
\eqref{intro1} is given by
\begin{equation}\label{intro2}
f(g)=\int\limits_{H(F)\backslash H({\bf A})} \varphi_\pi(h)
\theta((h,g))dh
\end{equation}
The most well known example of this type is when $M$ is the double cover of the
symplectic group. In this example  $H$ is an
orthogonal group and $G$ itself is a symplectic group or its double cover. The representation $\Theta$ is the minimal representation which is defined on the double cover of $M({\bf A})$. This case was studied by many authors. A reference
for this example can be found in \cite{R}. Other cases which involve
the minimal representation can be found in \cite{G-R-S4} where the
group $M$ is one of the exceptional groups of type $E_6, E_7$ and
$E_8$. We remark
that there are constructions given by \eqref{intro2} which do not
involve the minimal representation. Some cases in \cite{G2} are
such.

The other extreme case is when the group $H$ is the trivial group. Thus,
integral \eqref{intro1} is then given by
\begin{equation}\label{intro3}
f(g)= \int\limits_{V(F)\backslash V({\bf A})}\theta(vg)\psi_V(v)dv
\end{equation}
In this case, which is known as the descent method, there is an
automorphic representation $\pi$ which is built inside the
representation $\Theta$. An example of this type can be found in
\cite{G-R-S2}, \cite{G-R-S3} and \cite{G-R-S7}. 

Finally, there are also examples where both groups $H$ and $V$ are
nonzero. See for example \cite{G2}, \cite{G3} and \cite{G4}.

Prior to any computations it is natural to ask the question of how
to construct lifting using integral \eqref{intro1}. In other words,
one would like to look for systematic ways to construct such
examples. To give some heuristic of how to find such examples, it is
convenient to use the language of unipotent orbits. In \cite{G1},
one associates with a unipotent orbit of a reductive group, a set of
Fourier coefficients. This is done for the classical groups, however
it is done in a similar way for the exceptional groups. In fact, in
this paper, we work out this association in the case of the $F$ split
exceptional group $F_4$. Let $\sigma$ denote an automorphic
representation of a reductive group $L$. To this data we attach a
set of unipotent orbits which we denote by ${\mathcal O}_L(\sigma)$.
We say that ${\mathcal O}\in {\mathcal O}_L(\sigma)$ if $\sigma$ has
no nonzero Fourier coefficient associated with any unipotent orbit
${\mathcal O}'$ which is greater than ${\mathcal O}$. Also, the
representation $\sigma$ has a nonzero Fourier coefficient associated
with the unipotent orbit ${\mathcal O}$. For more details on this
set see \cite{G1}. It is not known if ${\mathcal O}_L(\sigma)$ can
contain more than one element. However, if it does contain only one
element, this means that $\sigma$ has no nonzero Fourier coefficient
associated with any unipotent orbit which is greater than or not
related to ${\mathcal O}_L(\sigma)$. Henceforth we shall assume that
for all representations in question, this set consists of one
element.  We can then define the dimension of the representation
$\sigma$ to be $\text{dim}\sigma=\frac{1}{2}\text{dim}{\mathcal
O}_L(\sigma)$. For basic properties of unipotent orbits and their
dimensions, see \cite{C-M}.

To explain our method, let $H$ and $G$ be two reductive groups such
that there is a homomorphism from $^LH$ to $^LG$. Let $\pi$ denote
an irreducible cuspidal representation of the group $H({\bf A})$.
Suppose that one can construct an automorphic representation
$\Theta$ on a group $M({\bf A})$, and assume that
$\sigma(\pi,\Theta)$, as defined by integral \eqref{intro1}, is a
functorial lift from $\pi$ corresponding to the above $L$ group homomorphism. Then, in all known cases, the
following dimension identity holds,
\begin{equation}\label{intro4}
\text{dim}\pi+\text{dim}\Theta=\text{dim}H+\text{dim}V+
\text{dim}\sigma(\pi,\Theta)
\end{equation}
It is important to emphasize  that we do not claim that for any setup
which satisfy equation \eqref{intro4}, then integral \eqref{intro1}
will give a functorial correspondence. In these notations we view
the descent method as a limit case when $H$ is the identity group,
and hence its dimension is zero, and hence $\text{dim}\pi=0$.

To make things clear, we consider a few examples. Let $H=SO_{2n}$ be
the split orthogonal group, and let $G=Sp_{2n}$. Then we have the
$L$ group homomorphism from $SO_{2n}({\bf C})$ to $SO_{2n+1}({\bf
C})$. Let $M=\widetilde{Sp}_{4n^2}$, the double cover of the
symplectic group. Let $\Theta$ denote the minimal representation of
$M({\bf A})$. Let $\pi$ denote a generic irreducible cuspidal
representation of $H({\bf A})$. Then, it follows from \cite{R}, that
integral \eqref{intro2} produces a functorial correspondence, and
one can show that $\sigma(\pi,\Theta)$ is a generic representation
as well. We verify identity \eqref{intro4} for this case. Indeed, in
this case we have $\text{dim}\pi=n^2-n$, $\text{dim}\Theta=2n^2$,
$\text{dim}H=2n^2-n$, and $\text{dim}\sigma(\pi,\Theta)=n^2$. The
dimension of these representations are derived from the general
formula for dimension of unipotent orbits as given in \cite{C-M}.
Thus, since $\pi$ is generic, then ${\mathcal
O}_{SO_{2n}}(\pi)=((2n-1)1)$. The  dimension of this orbit is
$2(n^2-n)$ and hence $\text{dim}\pi=n^2-n$. The representation
$\Theta$  is associated with the minimal orbit which is
$(21^{4n^2-2})$ and hence, it follows from \cite{C-M} that its
dimension is $2n^2$. It is now easy to verify identity
\eqref{intro4} in this case.

As another example of this type,  consider the case when $H=PGL_3$
and $G=G_2$. Here, $\Theta$ is the minimal representation of the
exceptional group $E_6({\bf A})$. It follows from \cite{G-R-S4} that
if $\pi$ is an irreducible cuspidal representation of $PGL_3({\bf
A})$, and hence generic, then integral \eqref{intro2} produces a
functorial lifting with $\sigma(\pi,\Theta)$ being generic. Since
$\text{dim}\pi=3,\ \text{dim}\Theta=11,\ \text{dim}PGL_3=8$ and
$\text{dim}\sigma(\pi,\Theta)=6$, it follows that identity
\eqref{intro4} holds. 

As an another example we consider an example of a construction which
is a descent construction, that is, uses the lifting as given by
\eqref{intro3}. Consider the case given in \cite{G-R-S2} and
\cite{G-R-S3}. In this case one obtains the descent from cuspidal
representations of $GL_{2n}({\bf A})$ to cuspidal generic
representations of $\widetilde{Sp}_{2n}({\bf A})$. Even though the
integral given for the descent in the above references involves also
the theta representation of $\widetilde{Sp}_{2n}({\bf A})$, it does
not alter the identity \eqref{intro4}. In the beginning of Section 4 we study in details these type
of constructions. In the construction of the
descent in this example, $\Theta$ is a certain residue of an Eisenstein series, and
one can show ( see \cite{G-R-S7}) that this residue is attached to the unipotent orbit
$((2n)^2)$ of $Sp_{4n}$. Thus $\text{dim}\Theta= 4n^2-n$. The
dimension of $V$ is $3n^2-n$, and since $\sigma(\pi,\Theta)$ is
generic it follows that $\text{dim}\sigma(\pi,\Theta)=n^2$. Thus
identity \eqref{intro4} holds. Strictly speaking this lift is not a functorial lift 
which corresponds to some $L$ groups homomorphism. However, one can view it as
an inverse map to the $L$ group homomorphism from $Sp_{2n}({\bf C})$ to
$GL_{2n}({\bf C})$.

In this paper we consider examples in the exceptional group $F_4$,
of global constructions as given by integrals \eqref{intro2} and
\eqref{intro3} which satisfy the dimension equation \eqref{intro4}. More specifically, in the
notations of integrals \eqref{intro2} and \eqref{intro3}, we will consider such integrals where
$M=F_4$.
Our main concern in this paper is to find conditions when such a
construction produces a cuspidal image, and under what conditions
the construction is nonzero. As follows from the beginning of
Section three, in almost all global integrals of the type of
integral \eqref{intro2}, which satisfy the dimension equation
\eqref{intro4}, the representation $\Theta$ needs to be a minimal
representation. In other words, we need ${\mathcal O}(\Theta)=A_1$.
Section two is mainly devoted to the
construction of such a representation on the double cover of $F_4$, and the study of its basic
properties. This representation is defined as a certain residue of
an Eisenstein series, essentially induced from the Borel subgroup.
In addition, in that Section we also collect information about the
structure of the Fourier coefficients of automorphic representations
of $F_4({\bf A})$ and its double cover.

In Section three we study integral \eqref{intro2} for five commuting
pairs inside $F_4$. The pairs are $(SL_3\times SL_3);\ (SL_2\times
SL_2, Sp_4);\ (SL_2, SL_4);\ (SO_3, G_2)$ and $(SL_2, Sp_6)$. In
each case we study when the lifting from one to the other is
cuspidal, and give a condition when it is nonzero. The computations
are straightforward and use the properties of the minimal
representation as were established in Section two.

In Section four we consider the descent map, that is integral
\eqref{intro3} for some unipotent orbits of $F_4$. At subsection 4.1
we list all possible unipotent orbits of $F_4$, and using the
dimension equation \eqref{dim2}, which is a variant of the dimension
equation \eqref{intro4}, we obtain conditions on the the dimension
of the automorphic representation involved in the construction. In
subsection 4.2 we fix notations and some preliminary results
concerning the nature of the answer we expect to get using the
descent map. Finally, in subsection 4.3 we consider some examples in
detail. That is, we study conditions for integral \eqref{intro3} to
define a cuspidal representation, and conditions for the
nonvanishing of the descent. The examples we choose to carry out are
chosen mainly by our belief that they are of some interest.

As can be seen the missing ingredient in this paper is the local
unramified theory. The main reason for this is that this issue is
different in nature from the issue of cuspidality and the
nonvanishing. Indeed, one of our goals in this paper is to show that
when studying cuspidality and nonvanishing, the answer can be
phrased in terms of the structure of the unipotent orbits of the
group in question. In other words, when studying these two
properties, the only ingredients we need to know about the
automorphic representation $\Theta$ is what Fourier coefficients it
supports. However, in subsection 3.6 we give a conjecture about the
functorail lifting of each of the above five commuting pairs. 

In Section five, we
construct two examples of automorphic representations which are
attached to specific unipotent orbits in $F_4$. As can be seen,
unramified considerations do enter the calculations.

\section{\bf  The Minimal Representation of $F_4$}

\subsection{\bf General Notations}

For $1\le i\le 4$, let $\alpha_i$ denote the four simple roots of
$F_4$. We label the roots of $F_4$ according to the diagram
$$\overset{\alpha_1}{0}----\overset{\alpha_2}{0} ==>==
\overset{\alpha_3}{0}----\overset{\alpha_4}{0}$$ Here
$\alpha_1,\alpha_2$ are the long simple roots and
$\alpha_3,\alpha_4$ are the short simple roots.

Given a root, positive or negative, we denote by $\{x_\alpha(r)\}$ the
one dimensional unipotent subgroup attached to the root $\alpha$.
For $1\le i\le 4$, let $h_i(t_i)$ denote the one dimensional torus
in $F_4$ which is associated to the $SL_2$ generated by
$<x_{\pm\alpha_i}(r)>$. Then
$h(t_1,t_2,t_3,t_4)=\prod_{i=1}^4h_i(t_i)$ is the maximal split  torus of
$F_4$. For $1\le i\le 4$, we shall denote by $w[i]$ the simple
reflection which corresponds to the simple root $\alpha_i$. We
shall write $w[i_1i_2\ldots i_m]$ for $w[i_1]w[i_2]\ldots w[i_m]$.

Let $F$ be a global field, and let ${\bf A}$ be its ring of adeles.
By $\psi$ we denote a nontrivial character of $F\backslash {\bf A}$. We shall denote by
$J_n$ the matrix of order $n$ which has ones on the other diagonal and zero elsewhere. The
matrix $e_{i,j}$ will denote a matrix of order $n$ which has one at the $(i,j)$ entry,
and zero elsewhere.

We denote by $\widetilde{F}_4$ the double cover of $F_4$. The
construction of this group and its basic properties follows from
\cite{M}.

Many of the computations done in this paper require the knowledge of
commutating relations and conjugations which involves one parameter
unipotent subgroups. We refer to \cite{G-S} from which all the
relevant data can be extracted.

Given an automorphic representation $\pi$ and a unipotent subgroup
$V$, we denote by $\varphi_\pi^V$ its constant term along $V$.
Here $\varphi_\pi$ is a vector in the space of $\pi$. In other
words, we denote
$$\varphi_\pi^V(g)=\int\limits_{V(F)\backslash V({\bf
A})}\varphi_\pi(vg)dv$$
In this paper we consider unipotent groups $U$ and characters $\psi_U$ which are defined on
the group $U(F)\backslash U({\bf A})$. Typically, these unipotent subgroups will be generated by one
dimensional unipotent subgroups $x_\gamma(r)$ where $\gamma$ is a positive root. For example, suppose that $U$ is the one dimensional subgroup associated with the root $\gamma$. In this case we shall write $U=\{x_\gamma(r) : r\in R\}$ where $R$ is a certain ring. When the ring $R$ is clear we shall write $U=\{x_\gamma(r)\}$ for short. 
Given roots $\gamma_1,\ldots, \gamma_l$, positive or negative, we shall denote by 
$<x_{\gamma_1}(r),\ldots, x_{\gamma_l}(r)>$ the group generated by all one dimensional unipotent
subgroups $x_{\gamma_i}(r)$.

A convenient way to describe   the character $\psi_U$ is as follows.
Let $\gamma_1,\ldots,\gamma_l$ denote $l$ positive roots of $F_4$, and assume that the one dimensional unipotent subgroup $x_{\gamma_i}(r_i)$ are all in $U$ but not in $[U,U]$. Given
$u\in U$, write $u=x_{\gamma_1}(r_1)\ldots x_{\gamma_l}(r_l)u'$ where $u'\in U$ is any element
which when written as a product of one dimensional unipotent subgroups associated with positive roots, then none of these roots are $\gamma_1,\ldots,\gamma_l$. Then define $\psi_U(u)=\psi_U(
x_{\gamma_1}(r_1)\ldots x_{\gamma_l}(r_l)u')=\psi(a_1r_1+\cdots +a_lr_l)$. Here $a_i\in F^*$.

\subsection{\bf Unipotent Orbits and Fourier Coefficients in
$F_4$}

In this subsection, let $G=F_4$. In this part we will describe how
to associate to a given unipotent orbit in $G$,  a set of Fourier
coefficients. In \cite{G1} it is explained how to construct this
correspondence for automorphic representations of the classical
groups. Another reference which studies unipotent orbits and Fourier coefficients for the group $F_4$ is \cite{G-H}.

According to the Bala-Carter classification, each unipotent orbit is
represented by a diagram of $G$ whose nodes are labelled by the
numbers zero, one and two. We shall denote these numbers by $\epsilon_i$ for 
all $1\le i\le 4$.  A list of the possible diagrams can be
found, for example, in \cite{C} page 401. As usual an unlabelled node
in the diagram corresponds to the number zero. Henceforth, we
identify the set of unipotent orbits with the set of all such
diagrams.

We associate to each diagram a set of Fourier coefficients. Let $P$
be a parabolic subgroup of $G$. We list the parabolic subgroups of
$G$ according to the unipotent elements of the form
$x_{\pm\alpha_i}(r)$ which are contained in the Levi part of the
parabolic subgroup. Thus for example, we denote by $P_{\alpha_1}$
the parabolic subgroup whose Levi part is generated by
$<x_{\pm\alpha_1}(r),T>$ where $T$ is the maximal split torus of
$G$. With these notations, the four maximal parabolic subgroups of
$G$ are $P_{\alpha_1,\alpha_2,\alpha_3},
P_{\alpha_1,\alpha_2,\alpha_4}, P_{\alpha_1,\alpha_3,\alpha_4}$ and
$P_{\alpha_2,\alpha_3,\alpha_4}$. A similar notation will be used
for the Levi part and the unipotent radical of a parabolic subgroup.
For example, $M_{\alpha_1}$ and $U_{\alpha_1}$ will denote the Levi
part and the unipotent radical of $P_{\alpha_1}$.

To each unipotent orbit we attach a parabolic subgroup defined as
follows. Suppose that $\Delta\subset \{\alpha_j : j\in \{1,2,3,4\}\}$ is
the set of simple roots in the diagram which are labeled zero. To
this unipotent orbit we associate the parabolic subgroup $P_\Delta$. We shall denote its Levi part by $M_\Delta$, and its unipotent radical by $U_\Delta$.
For example, to the unipotent orbit, which is denoted by $B_2$, and
whose diagram is given by
$$  \ \ \  \overset{2}{0}---- 0 ==>== {0}----\overset{1}{0}$$
we attach the parabolic subgroup $P_{\alpha_2,\alpha_3}$. Here $\epsilon_1=2,\ 
\epsilon_2=\epsilon_3=0$ and $\epsilon_4=1$.

It will be convenient to
confuse between a root $\alpha$ and the one parameter unipotent subgroup $\{x_\alpha(r)\}$. 
Thus, for example, if $\{ x_\alpha(r)\}\subset U$ for some unipotent subgroup $U$, 
we will say that $\alpha$ is a root in $U$. By abuse of notations we will sometimes denote
it by $\alpha\in U$.  Given
$\alpha=\sum_{i=1}^4n_i\alpha_i$ we shall also denote this root by
$(n_1n_2n_3n_4)$. Given a parabolic subgroup $P_\Delta$ as above,
the set of roots in $U_\Delta$ are  those roots $(n_1n_2n_3n_4)$
such that  $\sum_{\alpha_i\notin {\Delta}} n_i>0$. For example, the
roots in $U_{\alpha_1,\alpha_3}$ are the roots $(n_1n_2n_3n_4)$ such
that $n_2+n_4>0$. Once again, we emphasize that we are confusing a
root with the one dimensional unipotent subgroup attached to this
root.

Next, we determine a partition of all the roots in $U_\Delta$. For
any natural number $n$ we  denote  $U_\Delta'(n)=\{\alpha\in U_\Delta
: \sum_{i=1}^4\epsilon_i n_i=n\}$.  Let $U_\Delta(n)$ denote the unipotent
subgroup of $U_\Delta$ which is generated by all one parameter
subgroups $\{x_\alpha(r)\}$ such that $\alpha\in U_\Delta'(m)$ where $m\ge
n$. Notice that $U_\Delta=U_\Delta(1)$ and if in the corresponding diagram all $\epsilon_i\ne 1$, then $U_\Delta=U_\Delta(2)$.  We are mainly interested in the group
$U_\Delta(2)$. It is not hard to check that $M_\Delta$ acts on this group.

As an example consider the above diagram attached to the unipotent orbit $B_2$. In this case,
we have $\Delta=\{\alpha_2,\alpha_3 \}$. The
parabolic subgroup attached to it is $P_{\alpha_2,\alpha_3}$, and
we can identify $M_{\alpha_2,\alpha_3}$ with $GL_1^2\times Sp_4$.
We list the 20 roots in $U_\Delta$ according to the sets
$U_\Delta'(n)$.  We have
$$U_\Delta'(1)=\{ (0001);(0011);(0111);(0121)\}$$ 
$$U_\Delta'(2)=\{ (1000);(1100);(1110);(1120);(1220)\}\cup \{ (0122)\}$$
$$U_\Delta'(3)=\{ (1111);(1121);(1221);(1231)\}$$
$$U_\Delta'(4)=\{ (1122);(1222);(1232);(1242);(1342)\}\ \ \ \ \
U_\Delta'(6)=\{ (2342)\}$$  

In general, we are interested in the action of $M_\Delta$ on the group 
$U_\Delta(2)/[U_\Delta(2),U_\Delta(1)]$. It follows from the general theory that 
$M_\Delta$ preserves this group and acts as a finite direct sum of irreducible representations. 
For example, for the unipotent orbit $B_2$, it follows from the above that $M_\Delta$ 
acts as a direct sum of  a five 
dimensional irreducible representation and a one dimensional representation.
We mention that this action of $M_\Delta$ can be lifted trivially to the 
unipotent group $U_\Delta(2)$.

Fix a unipotent orbit ${\mathcal O}$, and attach to it a set $\Delta$ as described above. 
Then, defined over the complex numbers ${\bf C}$, (or any other algebraically closed field ), 
the group $M_\Delta({\bf C})$, has an open
orbit when acting on $U_\Delta(2)({\bf C})$. Denote a representative of this orbit by $u_{\mathcal O}$. Thus, we may identify $u_{\mathcal O}$ with a unipotent element in $U_\Delta(2)({\bf C})$. 
It follows from the general theory, see \cite{C}, that the connected component
of the stabilizer of $u_{\mathcal O}$ inside $M_\Delta({\bf C})$, is a
reductive group. We shall denote this reductive group by
$C(u_{\mathcal O})^0$. A list of these reductive groups is given in \cite{C} page 401.

We now explain how to associate a set of Fourier coefficients to a
unipotent orbit $\mathcal O$. Assume first that all nodes in the diagram associated with
$\mathcal O$ are zeros or twos. Let $\Delta$ be as above and let
$u_{\mathcal O}$ denote any unipotent element in $G(F)$ which lies in
$U_\Delta(2)(F)$, such that the stabilizer of $u_{\mathcal O}$ inside
$M_\Delta(F)$ is of the same type as $C(u_{\mathcal O})^0$. We consider a
few examples. Suppose that $\mathcal O$ is the unipotent orbit
labelled $B_3$. Its diagram is
$$  \ \ \  \overset{2}{0}----\overset{2} 0 ==>== {0}----{0}$$
Thus, $P_\Delta=P_{\alpha_3,\alpha_4}$ and
$M_{\alpha_3,\alpha_4}=GL_1\times GL_3$. We have
$$U_\Delta'(2)=\{ (0100);(0110);(0111);(0120);(0121);(0122)\}\cup
\{(1000)\}$$ Thus, the action of $M_{\alpha_3,\alpha_4}$ on the group 
$U_\Delta(2)/[U_\Delta(2),U_\Delta(1)]$, and hence on the group $U_\Delta(2)$,
is a sum of two irreducible representations. The first representation, is the six
dimensional irreducible representation, which up to the action of the torus, is the
symmetric square representation.
The second representation is  a one dimensional
representation. According to \cite{C} page 401, the group
$C(u_\Delta)^0$ is of type $A_1$. 

Thus, to define the corresponding Fourier coefficient, we look at all possible
non-conjugate elements $u_0\in U_\Delta(2)(F)$ such that the stabilizer 
inside $M_{\alpha_3,\alpha_4}(F)$, under its action on $U_\Delta(2)(F)$ as defined
above, is a group of type
$A_1$ defined over $F$.  Since the action is via the symmetric square representation,
one can choose the elements $u_0$ to be any element in the set
\begin{equation}\label{rep1}
\{x_{1000}(1)x_{0100}(\beta_1)x_{0112}(\beta_2)x_{0122}(\beta_3)
: \beta_j\in  F^*\}
\end{equation} 
It is not hard to check that the stabilizer is
an orthogonal group $SO_3$ which depend on the choice of $\beta_j$.
Let $\varphi$ be an automorphic function defined on $G({\bf A})$. To
a given element $u_0$ in the above set, we associate the
Fourier coefficient
\begin{equation}\label{four1}
\int\limits_{U_\Delta(F)\backslash U_\Delta({\bf
A})}\varphi(u)\psi_{U,u_0}(u)du
\end{equation}
where $\psi_{U,u_0}$ is defined as follows.  Write $u\in U_\Delta$ as
$u=x_{1000}(r)x_{0100}(r_1)x_{0112}(r_2)x_{0122}(r_3)u'$ and  define
$\psi_{U,u_0}(u)=\psi(r+\beta_1r_1+\beta_2r_2+\beta_3r_3)$. See subsection 2.1 for the precise notations.
As we vary $u_0$ in the  set of representatives given in \eqref{rep1},
we associate with the unipotent orbit labelled $B_3$ a set
of Fourier coefficients, given by integrals \eqref{four1}.

As an another example, consider the unipotent orbit labeled
$F_4(a_1)$. Its diagram is
$$  \ \ \  \overset{2}{0}----\overset{2} {0} ==>== {0}----\overset{2}{0}$$
Thus, the parabolic subgroup attached to this orbit is
$P_{\alpha_3}$ and $M_{\alpha_3}=GL_1^2\cdot GL_2$. In this case
we have
$$U_\Delta'(2)=\{ (0100);(0110);(0120)\}\cup \{ (0001);(0011)\}\cup
\{ (1000)\}$$ It follows from \cite{C} that the connected component
of the stabilizer is the identity group. Consider the set of
unipotent elements in $U_{\alpha_3}(F)$
$$\{x_{1000}(1)x_{0011}(1)x_{0100}(\beta_1)x_{0120}(\beta_2) :
\beta_j\in F^*\}$$ It is not hard to check that the connected
component of each such element, is the identity group. In a similar
way as in \eqref{four1}, we associate with the unipotent orbit
$F_4(a_1)$ a set of Fourier coefficients.

Next consider unipotent orbits where at least one of the nodes in the corresponding
diagram  is
labelled with the number one. First assume that there is exactly one
node which is labelled with one, and all other nodes are labeled with
zero. There are exactly four such unipotent orbits which are
associated with the four maximal parabolic subgroups. In this case
we consider $U_\Delta(2)$ and proceed in a similar way as we did in
the case where all nonzero nodes are labelled with twos. For example,
consider the unipotent orbit $A_1+\widetilde{A}_1$. The diagram
attached to this orbit is
$$  \ \ \  \overset{}{0}----\overset{1} {0} ==>== {0}----\overset{}{0}$$
Hence, the parabolic subgroup which corresponds to this orbit is
$P_{\alpha_1,\alpha_3,\alpha_4}$. Its Levi part is $GL_2\cdot
SL_3$. From \cite{C} we know that the connected component of the
stabilizer is a group of type $A_1\times A_1$. We have
$$U_\Delta'(2)=\{ (1220);(1221);(1222);(1231);(1232);(1242)\}$$ The
action of the Levi part on $U_\Delta(2)$ is as follows. The $GL_2$ part
acts as a power of the determinant, and the $SL_3$ part via the
symmetric square representation. As before, it is not hard to check
that the set
$$\{x_{1220}(\beta_1)x_{1222}(\beta_2)x_{1242}(\beta_3) : \beta_j\in F^*\}$$
contains a set of representatives for all the orbits such that the connected component of
the stabilizer inside $M_{\alpha_1,\alpha_3,\alpha_4}$ will be of
type $A_1\times A_1$.  As in \eqref{four1} we define
\begin{equation}\label{four2}
\int\limits_{U_\Delta(2)(F)\backslash U_\Delta(2)({\bf
A})}\varphi(u)\psi_{U_\Delta(2),u_0}(u)du
\end{equation}
where $\psi_{U_\Delta(2),u_0}$ is defined as follows. Given $\beta_j\in
F^*$, let
$u_0=x_{1220}(\beta_1)x_{1222}(\beta_2)x_{1242}(\beta_3)$. For
$u\in U_\Delta(2)$ write
$u=x_{1220}(r_1)x_{1222}(r_2)x_{1242}(r_3)u_1$ and
define $\psi_{U_\Delta(2),u_0}(u)=\psi(\beta_1 r_1+\beta_2 r_2+\beta_3
r_3)$.

Finally, we need to consider the unipotent orbits whose corresponding diagram  has  one
node labelled one and at least one more node which is labelled with a
nonzero number. There are four such orbits. The way of attaching the
Fourier coefficients to these orbits are similar to the way we did
in the other cases. To make things clear, in each of the four cases
we shall write the set of representatives of the various orbits under the action
of $M_\Delta(F)$.  Then, given an element $u_0$ in the
corresponding set, we define the corresponding set of Fourier coefficients
as in \eqref{four2}.

First consider the unipotent orbit
$B_2$. Its diagram, the parabolic subgroup attached to this orbit,
and the sets $U_\Delta'(n)$  were all described above. The connected
component of the stabilizer is $A_1\times A_1$. Consider the set
$$\{x_{1100}(1)x_{1120}(\beta_1)x_{0122}(\beta_2) : \beta_j\in F^*\}$$ Then it contains 
the set of all representatives of the various orbits under the action of $M_\Delta(F)$.

Next, we consider the unipotent orbit $\widetilde{A}_2 + A_1$. Its
diagram is
$$  \ \ \  \overset{}{0}----\overset{1} {0} ==>== {0}----\overset{1}{0}$$
The connected component of the stabilizer is a group of type
$A_1$. We have
$$U_\Delta'(2)=\{ (0111);(0121);(1111);(1121)\}\cup\{(1220)\}$$
The Levi part, which is $GL_2\times GL_2$ acts on this set as the
tensor product representation and as a one dimensional
representation. In this case, $M_\Delta(F)$ acts transitively, and the representative of the
open orbit is given by 
$x_{0121}(1)x_{1111}(1)x_{1220}(1)$.

The unipotent orbit labelled as $C_3(a_1)$ has the corresponding
diagram
$$  \ \ \  \overset{1}{0}----\overset{} {0} ==>== \overset{1}{0}----\overset{}{0}$$
The connected component of the stabilizer is a group of type
$A_1$. We have
$$U_\Delta'(2)=\{ (0120);(0121);(0122)\}\cup\{(1110);(1111)\}$$
Hence, $M_\Delta=GL_2\times GL_2$ acts as a sum of two irreducible representations. 
On the first representation, one copy of $GL_2(F)$ acts as the symmetric square 
representation and the other copy acts as a one dimensional representation. On the second
irreducible representation one copy of $GL_2(F)$ acts as the standard representation and the
other copy acts as a one dimensional representation. 
A set of unipotent representatives for the various orbits  is included in the set
$$\{x_{0120}(\beta_1)x_{0122}(\beta_2)x_{1111}(1) : \beta_j\in F^*\}$$

The last case is the unipotent orbit labelled $C_3$. Its diagram is
$$  \ \ \  \overset{1}{0}----\overset{} {0} ==>== \overset{1}{0}----\overset{2}{0}$$
The connected component of the stabilizer is a group of type
$A_1$. In this case the action is transitively, and as a representative of the open 
orbit, we can take the element $x_{0120}(1)x_{1110}(1)x_{0001}(1)$.

\subsubsection{\bf On the Fourier Coefficients Attached to the Orbits
$F_4(a_2)$ and $F_4(a_3)$}

For later reference we give some details concerning the Fourier
coefficients of these two unipotent orbits. We start with
$F_4(a_2)$. In this case $P_\Delta=P_{\alpha_1,\alpha_3}$. The roots
in $U_\Delta'(2)$ are
$$(0001);\ (0011);\ (0100);\ (1100);\ (0110);\
(1110);\ (0120);\ (1120)$$ The group of characters defined on the
group $U_\Delta(F)\backslash U_\Delta({\bf A})$ is defined as
follows. Write $u\in U_\Delta$ as $u=z(m_1,m_2)y(r_1,\ldots,r_6)u'$
where $u'\in [U_\Delta,U_\Delta]$,
$z(m_1,m_2)=x_{0001}(m_1)x_{0011}(m_2)$ and
$$y(r_1,\ldots,r_6)=x_{0100}(r_1)x_{0110}(r_2)x_{0120}(r_3)
x_{1100}(r_4)x_{1110}(r_5)x_{1120}(r_6)$$ Denote
$$\text{Mat}_{2\times 4}'=\left \{R\in \text{Mat}_{2\times 4} : R=\begin{pmatrix}
r_3&r_4&r_5&r_6\\ r_1&r_2&r_3&-r_4\end{pmatrix}\right \}$$ 
We mention that the motivation for dealing with this abelian group is from a certain matrix 
realization of the group $GSpin_7$. Embedding $GSpin_7$ inside $GSO_8$, the following described
action is derived from the action of $M_\Delta= GL_2\times GL_2$ on a unipotent radical of a maximal 
parabolic subgroup of $GSpin_7$.

Given a matrix $A$ in
$$\text{Mat}_{4\times 2}'=\left \{A\in \text{Mat}_{4\times 2} : A=\begin{pmatrix}
a_1&a_2\\ a_3&a_4\\ a_5&a_1\\ a_6&-a_3\end{pmatrix}\right \}$$ and
$\gamma_1,\gamma_2\in F$ define for $u=z(m_1,m_2)y(r_1,\ldots ,r_6)u'\in U_\Delta$ parameterized as
above
$$\psi_{U_\Delta,A,\gamma_1,\gamma_2}(u)=\psi\left (\text{tr}\left [A
\begin{pmatrix}
r_3&r_4&r_5&r_6\\ r_1&r_2&r_3&-r_4\end{pmatrix}\right ]\right
)\psi(\gamma_1 m_1+\gamma_2 m_2)$$ The action of the Levi part of
$M_{\alpha_1,\alpha_3}(F)$ on the group characters is given as
follows.

First, let $g$ be an element in  $SL_2(F)$ which is generated by $<x_{\pm 1000}(r)>$. The
action of this group is given by
$$\psi_{U_\Delta,A,\gamma_1,\gamma_2}\mapsto
\psi_{U_\Delta,B,\gamma_1,\gamma_2}\ \ \ \ \ \ \ \ B=\begin{pmatrix}
g&\\&g^{-1}\end{pmatrix}A$$ Next, let $h\in SL_2(F)$ generated by
$<x_{\pm 0010}(r)>$. Consider first the action of $x_{0010}(m)$. It is
given by
$$\psi_{U_\Delta,A,\gamma_1,\gamma_2}\mapsto
\psi_{U_\Delta,B,\gamma_1',\gamma_2'}\ \ \ \ \
B=\begin{pmatrix} 1&&m&\\ &1&&-m\\ &&1&\\
&&&1\end{pmatrix}A\begin{pmatrix} 1&m\\ &1
\end{pmatrix};\ \ \ \ \begin{pmatrix} \gamma_1'\\ \gamma_2'\end{pmatrix}=
\begin{pmatrix} 1&\\ m&1
\end{pmatrix}\begin{pmatrix} \gamma_1\\ \  \gamma_2\end{pmatrix}$$
The action of $x_{-0010}(m)$ is defined similarly by taking the
corresponding transpose matrices.

Finally, the action of the maximal torus $h(t_1,t_2,t_3,t_4)$ of
$F_4$ is given by
$$\psi_{U_\Delta,A,\gamma_1,\gamma_2}\mapsto
\psi_{U_\Delta,B,\gamma_1',\gamma_2'}\ \ \ \ \ \ B=T_1AT_2; \ \ \
\gamma_1'=t_3^{-1}t_4^2\gamma_1;\ \ \
\gamma_2'=t_2t_3^{-1}t_4^{-1}\gamma_2$$ where $T_1=\text{diag}\
(t_1^{-1}t_2t_3^{-1}t_4, t_1t_3^{-1}t_4, t_1^{-1}t_3,
t_1t_2^{-1}t_3)$ and $T_2=\text{diag}\ ( t_3t_4^{-2},
t_2t_3^{-1}t_4^{-1})$.

The Fourier coefficient \eqref{four1} corresponds to the unipotent
orbit $F_4(a_2)$ if and only if the connected component of the
stabilizer of the character $\psi_{U_\Delta,A,\gamma_1,\gamma_2}$,
is trivial.

The situation for $F_4(a_3)$ is similar. Here
$P_\Delta=P_{\alpha_1,\alpha_3,\alpha_4}$ and
$M_{\alpha_1,\alpha_3,\alpha_4}$ is generated by $SL_2\times SL_3$ and the maximal split torus $T$. The roots in $U_\Delta'(2)$ are all 12 positive roots of the form
$n_1\alpha_1+\alpha_2+n_3\alpha_3+n_4\alpha_4$ where $n_i\ge 0$. There are 6 roots
such that $n_1=0$ and 6 such that $n_1=1$. The six roots which have $n_1=0$
are $\{(0100);\ (0110);\ (0120);\ (0111);\ (0121);\ (0122)\}$. Write
an element $u\in U_\Delta$ as
$u=y(r_1,\ldots,r_6)z(m_1,\ldots,m_6)u'$ where
\begin{equation}\label{f4a3}
y(r_1,\ldots,r_6)=x_{0100}(r_1)x_{0110}(r_2)x_{0120}(r_3)x_{0111}(r_4)
x_{0121}(r_5)x_{0122}(r_6)
\end{equation}
and
$$z(m_1,\ldots,m_6)=x_{1100}(m_1)x_{1110}(m_2)x_{1120}(m_3)x_{1111}(m_4)
x_{1121}(m_5)x_{1122}(m_6)$$ Here, $u'\in [U_\Delta,U_\Delta]$. We
can relate these elements with the group $Mat_{3\times 3}^0=\{ x\in
Mat_{3\times 3}\ :\ J_3x=x^tJ_3\}$  where $J_3$ is the $3\times 3$
matrix defined in subsection 2.1. The relation is given by
$$y(r_1,\ldots,r_6)\mapsto \begin{pmatrix} r_4&r_5&r_6\\
r_2&r_3&r_5\\ r_1&r_2&r_4 \end{pmatrix}\ \ \ \ \ z(m_1,\ldots,m_6)
\mapsto \begin{pmatrix} m_4&m_5&m_6\\
m_2&m_3&m_5\\ m_1&m_2&m_4 \end{pmatrix}$$ To describe the characters
of the group $U_\Delta(F)\backslash U_\Delta({\bf A})$, let $A,B\in
Mat_{3\times 3}^0$. Then define, for an element $u\in U_\Delta$
parameterized as above
\begin{equation}\label{f4a31}
\psi_{U_\Delta,A,B}(u)=\psi\left (\text{tr}\left [ A\begin{pmatrix} r_4&r_5&r_6\\
r_2&r_3&r_5\\ r_1&r_2&r_4 \end{pmatrix}+B\begin{pmatrix} m_4&m_5&m_6\\
m_2&m_3&m_5\\ m_1&m_2&m_4 \end{pmatrix}\right ]\right )
\end{equation}
Thus we can identify the group characters of $U_\Delta(F)\backslash
U_\Delta({\bf A})$ by pairs $(A,B)$ as above. The action of
$M_{\alpha_1,\alpha_3,\alpha_4}(F)$ is as follows. First, given
$g\in SL_3(F)$ we have $g(A,B)=(gAJ_3g^tJ_3, gBJ_3g^tJ_3)$. Then,
for $h=\begin{pmatrix} a&b\\ c&d\end{pmatrix}\in SL_2(F)$, we have
$h(A,B)=(aA+bB,cA+dB)$. This action can be easily extended to an action of the group $GL_2\times
GL_3$. Doing so, we can describe the action of the torus $T$. We only need to describe the action
of $h(1,t,1,1)$. This is given by the above action using the diagonal matrix $g=\text{diag}(t,1,1)$,
and then $h=\text{diag}(1,t^{-1})$.   

The Fourier coefficient \eqref{four1} attached to the character
$\psi_{U_\Delta,A_0,B_0}$,  corresponds to the unipotent orbit
$F_4(a_3)$, if the connected component of the stabilizer of the pair
$(A_0,B_0)$ is trivial. This can be checked using the Lie algebras
of these groups, and extending the above action to $GL_2\times GL_3$. Thus, if $((h,g))(A_0,B_0)=(A_1,B_1)$ is an element in $GL_2\times GL_3$, then differentiating, we obtain the two equations
\begin{equation}\label{f4a32}
g_1A_0+A_0J_3g_1^tJ_3+a_1A_0+b_1B_0=0\ \ \ \ \
g_1B_0+B_0J_3g_1^tJ_3+c_1A_0+d_1B_0=0
\end{equation}
Here $g_1\in Mat_{3\times 3}$ and $h_1=\begin{pmatrix} a_1&b_1\\
c_1&d_1\end{pmatrix}$ is a $2\times 2$ matrix.
Clearly, the group of matrices $(h_1,g_1)=(-2tI_2,tI_3)$ with $t\in
F$, is a solution to these two equations. We refer to this
solution as the trivial solution. Indeed, on the group level this solution corresponds to the torus
element $(t^{-2}I_2,tI_3)\in GL_2\times GL_3$, but from the above realization of the action on
the unipotent matrices in $F_4$, this torus is not in $M_{\alpha_1,\alpha_3,\alpha_4}$.

Thus, if the solution to these two equations is only the trivial solution, then the Fourier coefficient \eqref{four1} attached to the character $\psi_{U_\Delta,A_0,B_0}$
corresponds to the unipotent orbit $F_4(a_3)$.

\subsubsection{\bf Root Exchange}

In the following Sections during the computations, we will carry out
several Fourier expansions. One type of this expansions will repeat
itself several times, and therefore it is convenient to state it in
generality. We shall refer to this process as root exchange. This process was described in generality in \cite{G-R-S7} subsection 7.1.  This process has a local analogous which uses the
notion of twisted Jacquet modules. In \cite{G-R-S2} subsection 2.2, the global process stated in   \cite{G-R-S7} is formulated and carried out using the local language. In this paper, the proofs
are global by nature, and therefore we prefer to use the global version. However, it should be
emphasized that a similar proof can be stated and carried out in the local situation.

In this paper we will perform the expansions on a root by root process. For that reason we
prefer to state the process of root exchange using a slightly different notations. We should also
emphasize that the computations involved do not contribute any cocycle. This is true in both
the global and the local version.

A typical integral that we start is an integral given by
\begin{equation}\label{exch1}
\int\limits_{(F\backslash {\bf A})^2}\int\limits_{U(F)\backslash
U({\bf A})}f(ux_\alpha(m)x_\beta(r))\psi(m)dudmdr
\end{equation}
Here $f$ is an automorphic function, and $\alpha$ and $\beta$ are
two roots, need not be positive roots.  Also, $U$ is a certain
unipotent group normalized by $x_\alpha(m)$ and $x_\beta(r)$. We assume that $[x_\beta(r),x_\alpha(m)]\in U$

Consider the following integral as a function of $g$,
\begin{equation}\label{exch2}\notag
L(g)=\int\limits_{F\backslash {\bf A}}\int\limits_{U(F)\backslash
U({\bf A})}f(ux_\alpha(m)g)\psi(m)dudm
\end{equation}
and assume that it is left invariant under $x_\gamma(\delta)$ for
all $\delta\in F$. That is $L(x_\gamma(\delta)g)=L(g)$. Here
$\gamma$ is any root, positive or negative, which satisfies the
commutation relation $[x_\beta(r),x_\gamma(t)]=x_\alpha(c\beta
t)u'$ with $u'\in U$. Here $c\in F^*$, a scalar which result from the structure constants in $F_4$.
With these assumptions we can expand integral $L(g)$ along $x_\gamma(t)$
where $t\in F\backslash {\bf A}$. We obtain
\begin{equation}\label{exch21}\notag
\sum_{\delta\in F}\int\limits_{(F\backslash {\bf
A})^2}\int\limits_{U(F)\backslash U({\bf
A})}f(ux_\alpha(m)x_\gamma(t)g)\psi(m+\delta t)dtdudmdr
\end{equation}
From this we deduce that integral \eqref{exch1} is equal to
\begin{equation}\label{exch3}\notag
\int\limits_{F\backslash {\bf
A}}\sum_{\delta\in F}\int\limits_{(F\backslash {\bf
A})^2}\int\limits_{U(F)\backslash U({\bf
A})}f(ux_\alpha(m)x_\gamma(t)x_\beta(r))\psi(m+\delta t)dtdudmdr
\end{equation}
Since $f$ is automorphic then for all $g$ and all $\delta\in F$ we
have $f(x_\beta(\delta)g)=f(g)$. Using that, and the above
commutation relations, the above integral is equal to
\begin{equation}\label{exch4}\notag
\int\limits_{F\backslash {\bf
A}}\sum_{\delta\in F}\int\limits_{(F\backslash {\bf
A})^2}\int\limits_{U(F)\backslash U({\bf A})}f(ux_\alpha(m+\delta
t)x_\gamma(t)x_\beta(r+\delta))\psi(m+\delta t)dtdudmdr
\end{equation}
Changing variables, and collapsing summation over $\delta$ with
integration over $r$, this integral is equal to
\begin{equation}\label{exch5}
\int\limits_{\bf A}\int\limits_{(F\backslash {\bf
A})^2}\int\limits_{U(F)\backslash U({\bf
A})}f(ux_\alpha(m)x_\gamma(t)x_\beta(r))\psi(m)dtdudmdr
\end{equation}
Arguing as in \cite{Ga-S} one can easily show that the above
integral is zero for all choice of data, if and only if the integral
\begin{equation}\label{exch6}
\int\limits_{(F\backslash {\bf A})^2}\int\limits_{U(F)\backslash
U({\bf A})}f(ux_\alpha(m)x_\gamma(t))\psi(m)dtdudm
\end{equation}
is zero for all choice of data. Hence, we deduce that integral
\eqref{exch1} is zero for all choice of data if and only if integral
\eqref{exch5} or integral \eqref{exch6} are zero for all choice of
data. Referring to this process we will say that we exchanged the
root $\beta$ by the root $\gamma$.

\subsection{\bf Eisenstein series and their Residues}

In this Section we consider certain Eisenstein series on $G=F_4$
and study some of their residues. The basic reference  for this type of construction is \cite{K-P}. We also follow the ideas
of the construction of a small representation of the double cover of
odd orthogonal groups. This was done in \cite{B-F-G1}, and we refer
to that paper for more details.

The Theta representation we construct will be a residue of an Eisenstein series associated with an induction from the Borel subgroup. We review
how this is constructed. Let $B$ denote the Borel subgroup of $G$, and let $T\subset B$ denote its maximal split torus. Let $\chi$ denote
a character of $T$. Let $\widetilde{T}$ denote the inverse image of $T$ inside $\widetilde{G}$. Let $Z(\widetilde{T})$ denote the center of
$\widetilde{T}$, and let $\widetilde{T}_0$ denote any maximal abelian subgroup of $\widetilde{T}$. The character $\chi$ defines a genuine character
of $Z(\widetilde{T})$ in the obvious way, and we extend it in any way to a character of $\widetilde{T}_0$. Inducing up to $\widetilde{T}$, extending it
trivially to $\widetilde{B}$, and then inducing to $\widetilde{G}$, we obtain a representation of $\widetilde{G}$ which we denote by $Ind_{\widetilde{B}}^{\widetilde{G}}\chi$.
It follows from  \cite{K-P} that this representation is uniquely determined by the character $\chi$ defined on $Z(\widetilde{T})$. These
statements are true both locally and globally. 

Let $\chi_{\bar{s}}$ denote the character of $T$ defined as follows. Given $h(t_1,t_2,t_3,t_4)\in T$ we define  $\chi_{\bar{s}}(h(t_1,t_2,t_3,t_4))=
|t_1|^{s_1}|t_2|^{s_2}|t_3|^{s_3}|t_4|^{s_4}$. Let $E_G^{(2)}(g,\bar{s})$ denote the Eisenstein series defined on $\widetilde{G}({\bf A})$ which is associated
with the induced representation $Ind_{\widetilde{B}({\bf A})}^{\widetilde{G}({\bf A})}\chi_{\bar{s}}\delta_B^{1/2}$. The poles of this Eisenstein series
are determined by the intertwining operators corresponding to elements $w$ of the Weyl group of $G$. The poles of these factors can be determined by using the
factors 
\begin{equation}\label{pole1}
c_w(\chi_{\bar{s}})=\prod_{\alpha>0,
w(\alpha)<0}\frac{(1-\chi_{\bar{s}}(a_\alpha)^{n(\alpha)})^{-1}}
{(1-q^{-1}\chi_{\bar{s}}(a_\alpha)^{n(\alpha)})^{-1}}
\end{equation}
where $n(\alpha)=1$ for the short roots and $n(\alpha)=2$ for the long roots. Consider first the contribution from the long Weyl element in $W$. A simple application of \eqref{pole1} implies that the poles of the corresponding intertwining operator are determined by 
\begin{equation}\label{pole2}
Z_S(\bar{s})=\frac{\zeta^S(2s_1)\zeta^S(2s_2)\zeta^S(s_3)\zeta^S(s_4)L^S(\bar{s})}{\zeta^S(2s_1+1)\zeta^S(2s_2+1)\zeta^S(s_3+1)\zeta^S(s_4+1)L^S(\bar{s}+1)}
\notag
\end{equation}
Here the four partial zeta factors are the terms contributed from the simple roots $\alpha$ in the product in \eqref{pole1}. The factor $L^S(\bar{s})$ is a product
of 20 partial zeta factors evaluated at points of the form $\sum_{i=1}^4n_i s_i$ with $n_i\ge 0$ and such that $n_1+n_2+n_3+n_4\ge 2$. The set $S$ is a finite set, such
that outside of $S$ all places are finite unramified places.  From this we deduce that $Z_S(\bar{s})$ has a simple multi pole at $s_1=s_2=\frac{1}{2}$ and 
$s_3=s_4=1$. Its not hard to prove that all other intertwining operators are holomorphic at this point. Hence, the Eisenstein series $E_G^{(2)}(g,\bar{s})$ has a multi-residue 
at that point. Denote this multi-residue representation by $\Theta_G^{(2)}$. If there is no confusion we shall denote it simply by $\Theta$. Thus, the representation
$\Theta$ is a sub-quotient of the representation $Ind_{\widetilde{B}({\bf A})}^{\widetilde{G}({\bf A})}\chi_{\bar{s}_0}\delta_B^{1/2}$ where $\chi_{\bar{s}_0}(h(t_1,t_2,t_3,t_4)=
|t_1t_2|^{1/2}|t_3t_4|$.

We will not need it, but we mention that the representation $\Theta$ is a subrepresentation of the
induced representation $Ind_{\widetilde{B}({\bf A})}^{\widetilde{G}({\bf A})}\chi_\Theta$ where
$\chi_\Theta(h(t_1,t_2,t_3,t_4))=|t_1t_2|^{1/2}$.

Let $P=MU$ denote a maximal parabolic subgroup of $G=F_4$, where $M$ is the Levi part of $P$, and $U$ is its unipotent radical. Let $M^0$ denote the subgroup of $M$ which is generated by all copies of $SL_2=<x_{\pm\alpha}(r)>$ where $\alpha$ is a positive root in $M$. There are four cases which we now list. First, if $P=P_{\alpha_1,\alpha_2,\alpha_3}$, then $M^0=Spin_7$. When $P=P_{\alpha_1,\alpha_2,\alpha_4}$ or $P_{\alpha_1,\alpha_3,\alpha_4}$ then $M^0=SL_2\times SL_3$, and when $P=P_{\alpha_2,\alpha_3,\alpha_4}$ then $M^0=Sp_6$. 

Using induction by
stages, we can write $Ind_{\widetilde{B}({\bf A})}^{\widetilde{G}({\bf A})}\chi_{\bar{s}_0}\delta_B^{1/2}$ as $Ind_{\widetilde{P}({\bf A})}^{\widetilde{G}
({\bf A})}\tau_P\delta_P^{1/2}$, where $\tau_P$ is an automorphic representation of $\widetilde{M}({\bf A})$. Thus, in the case when $P=P_{\alpha_1,\alpha_2,\alpha_3}$, then $\tau_P$ restricted to
$\widetilde{M}^0({\bf A})=Spin_7^{(2)}({\bf A})$, is a minimal representation of this group. Indeed, this follows by comparing the parameters between those of $\Theta$ and the parameters of the
minimal representation of $Spin_7^{(2)}({\bf A})$ as established in \cite{B-F-G2}. In the case when
$P=P_{\alpha_1,\alpha_2,\alpha_4}$ we obtain that $\tau_P$ restricted to $\widetilde{SL}_3({\bf A})\times SL_2({\bf A})$ is the representation $\Theta_{SL_3}\times 1$, and similarly when $P=P_{\alpha_1,\alpha_3,\alpha_4}$ then we obtain the representation $\Theta_{SL_2}\times 1$ of
$\widetilde{SL}_2({\bf A})\times SL_3({\bf A})$. These two cases are obtained by comparing with the construction of the Theta representations as done in \cite{K-P}. Finally, when $P=P_{\alpha_2,\alpha_3,\alpha_4}$ we obtain the right most residue representation of the Siegel Eisenstein series 
defined on $\widetilde{Sp}_6({\bf A})$. This can be verified using the result of \cite{I2}. 
Motivated by the above, let $\widetilde{M}_0$ denote the subgroup of $\widetilde{M}$ defined as follows. When $M$ is the Levi part of $P_{\alpha_1,\alpha_2,\alpha_3}$ or of $P_{\alpha_2,\alpha_3,\alpha_4}$, we define $\widetilde{M}_0=\widetilde{M}^0$. When $M$ is the Levi part of 
$P_{\alpha_1,\alpha_2,\alpha_4}$, define $\widetilde{M}_0=\widetilde{SL}_3\times SL_2$, and in the
last case, when $M$ is the Levi part of $P_{\alpha_1,\alpha_3,\alpha_4}$, we define $\widetilde{M}_0=\widetilde{SL}_2\times SL_3$. 
A representation of the group $\widetilde{M}_0({\bf A})$ will  said to be a minimal representation 
if the only nontrivial Fourier coefficients this representation has, corresponds to the minimal 
orbit specified as follows. In the case when $\widetilde{M}_0=Spin_7^{(2)}$ we refer to $\tau_P$ as a minimal representation if  the only nonzero Fourier coefficients of this representation 
corresponds to the unipotent orbit $(2^21^3)$. When 
$\widetilde{M}_0=\widetilde{SL}_3({\bf A})\times SL_2({\bf A})$ we refer to $\tau_P$ as a minimal
representation if the only nonzero Fourier coefficients of this representation 
corresponds to the unipotent orbit $(21)$ on $\widetilde{SL}_3$ and trivial on $SL_2$.
For $\widetilde{M}_0=\widetilde{SL}_2({\bf A})\times SL_3({\bf A})$ we refer to $\tau_P$ as a minimal representation if it is trivial on $SL_3$. Finally when $\widetilde{M}_0=
\widetilde{Sp}_6$ we refer to $\tau_P$ as a minimal
representation if the only nonzero Fourier coefficients of this representation 
corresponds to the unipotent orbit $(21^4)$. 
It is a consequence of the above mentioned references that the representation $\tau_P$
restricted to $\widetilde{M}_0({\bf A})$, is a minimal representation. The case where $M_0=Sp_6$
follows from the Siegel-Weil identity as established in \cite{I2}.

\begin{proposition}\label{eisen1}
Let $P=MU$ denote any one of the four maximal parabolic subgroup of $G=F_4$. With the above notations, the constant term $\Theta^{U}(g)$ when restricted to the group
$\widetilde{M}_0({\bf A})$ defines a minimal representation of this group.
More over, the residue representation $\Theta$ is square integrable.
\end{proposition}
\begin{proof}
We shall work out the details in the case $P=P_{\alpha_2,\alpha_3,\alpha_4}$. The other cases
are done in a similar way.  

Let $P=P_{\alpha_2,\alpha_3,\alpha_4}$. Then using induction by stages as above, we deduce that
the representation $\Theta$ is a residue at $s=27/32$ of the Eisenstein series $\widetilde{E}_{\tau_P}(g,s)$ associated with the induced representation $Ind_{\widetilde{P}({\bf A})}^{\widetilde{G}({\bf A})}\tau_P\delta_P^{s}$. To get this value of $s$, we start by noticing
that $\chi_{\bar{s}_0}\delta_B^{1/2}(h(t_1,t_2,t_3,t_4))=|t_1t_2|^{3/2}|t_3t_4|^2$. On $Sp_6$,
we have the identity  $|t_2|^{3/2}|t_3t_4|^2=(\delta_{B(GL_3)}\delta_{P(Sp_6)}^{7/8})(h(1,t_2,t_3,t_4))$. Here $P(GL_3)$ is the maximal parabolic subgroup of $Sp_6$ whose Levi part is $GL_3$, and $B(GL_3)$ is the Borel subgroup of $GL_3$. Extending this character to $T$, we obtain 
$$(\delta_{B(GL_3)}\delta_{P(Sp_6)}^{7/8})(h(t_1,t_2,t_3,t_4))=|t_1|^{-\frac{21}{4}}|t_2|^{3/2}|t_3t_4|^2$$
We have $\delta_P^s(h(t_1,t_2,t_3,t_4))=|t_1|^{8s}$. Hence, when matching the character 
$\delta_{B(GL_3)}\delta_{P(Sp_6)}^{7/8}\delta_P^s$ with $|t_1t_2|^{3/2}|t_3t_4|^2$ we get $s=27/32$.

We need to study the constant term of this Eisenstein series. We
use the method of \cite{K-R}. See also \cite{B-F-G2} and
\cite{G-R-S1} for similar cases.  Consider the constant term
along $U$. In other words, let
$$\widetilde{E}_{\tau_P}^U(g,s)=\int\limits_{U(F)\backslash
U({\bf A})}\widetilde{E}_{\tau_P}(ug,s)du$$  
Unfolding the Eisenstein series for $\text{Re}(s)$ large, we need to consider the space of double cosets $P(F)\backslash G(F)/P(F)$.  This space has five elements, and as representatives, we can choose the five Weyl elements $e,\ w[1],\ w[12321],\ w[12324321]$ and the long Weyl element in 
this space which we denote by $w_0$. Notice that all of these elements are of order two, and hence $M_w=M_{w^{-1}}$.

We start with the contribution of $w_0$. Since $\Theta$ is a residue of this Eisenstein
series, we deduce that at the point $s=27/32$, where the residue occurs, the intertwining $(M_{w_0}f_s)(m)$ operator has a simple pole.
Arguing as in \cite{G-R-S1} pages 78-81 we deduce that at the bad
places, after a suitable normalization by the local factors of the
normalizing factor of the Eisenstein series,  the local intertwining
operators are holomorphic at the above point. Thus
$(M_{w_0}f_s)(m)$ has a simple pole at $s=27/32$. As a function of $g\in Sp_6^{(2)}({\bf A})$, 
the function $(M_{w_0}f_s)(m)$ belongs to the space of $\tau_P$ restricted to the group $\widetilde{M}_0({\bf A})$. As we stated before the Proposition this representation is a minimal representation.

Next we consider the contribution from the other four
representatives. The term which corresponds to the identity is just
the section which is clearly holomorphic. The three other
representatives contributes each to the constant term an Eisenstein
series defined on $\widetilde{Sp_6}({\bf A})$.  This Eisenstein
series has the form $\widetilde{E}(m,M_wf_s,s')$ where $M_wf_s$ is
the corresponding intertwining operator and $s'$ is a certain linear
function in $s$. When $w=w[1]$ or when $w=w[12324321]$ we get the
Eisenstein series associated to the induced representation
$Ind_{Q({\bf A})}^{\widetilde{Sp}_6({\bf A})}\delta_Q^{s'}$ where
$Q$ is the maximal parabolic subgroup of $Sp_6$ whose Levi part is
$GL_3$. When $w=w[12321]$ we obtain an Eisenstein series associated
with induction from the parabolic subgroup whose Levi part is
$GL_1\times Sp_4$.

This procedure is fairly standard. See \cite{K-R}, or \cite{B-F-G2}
for an example in the covering group. As an example, consider the
case when $w=w[1]$. We have $w\alpha =\alpha$ when
$\alpha=\pm(0010)$ and $\alpha=\pm(0001)$. Also $w(0100)=(1100)$.
This means that $w$ conjugates the subgroup $P(GL_3)$ into $P$. Here
$P(GL_3)$ is the maximal parabolic subgroup of $Sp_6$ which
contains the group $GL_3$. Thus, the contribution to
$\widetilde{E}_{\tau_P}^U(g,s)$ from this Weyl element is
$$\sum_{\gamma\in P(GL_3)(F)\backslash Sp_6(F)}\int\limits_
{V(F)\backslash V({\bf A})}\int\limits_{U_w({\bf A})}f_s(vwu\gamma
g)dvdu$$ Here $V$ denotes the unipotent radical of the parabolic
subgroup $P(GL_3)$, and $U_w=<x_{1000}(r)>$. Thus, as a function of $m\in \widetilde{Sp}_6({\bf
A})$, this term is equal to $\widetilde{E}(m,M_wf_s,s')$, the
Eisenstein series associated to the induced representation
$Ind_{P(GL_3){\bf A}}^{\widetilde{GSp}_6({\bf
A})}\delta_{P(GL_3)}^{s'}$. From the above integral we obtain that
$s'=2s+\frac{5}{16}$. It is also easy to verify that the
intertwining operator $M_wf_s$ is holomorphic at $s=\frac{27}{32}$,
and hence,  we deduce that
$\widetilde{E}(m,M_wf_s,s')$ is holomorphic at $s=\frac{27}{32}$ which corresponds to the point $s'=2$.

The other two cases are similar, and in both we obtain  that they
are holomorphic at $s=27/32$. Hence, all other four Weyl elements contributes a function to the constant term $\widetilde{E}_{\tau_P}^U(g,s)$, which is holomorphic at the point $s=27/32$. From
this the Proposition follows for this maximal parabolic subgroup $P$. As mentioned above, the
other cases are similar and will be omitted.

Finally, to prove the square integrability we use Jacquet's
criterion \cite{J1}. This follows from the fact that $\Theta_G^{(2)}$ is a sub-representation of
$Ind_{\widetilde{B}({\bf A})}^{\widetilde{G}({\bf A})}\chi_\Theta\delta_B^{1/2}$ where
$\chi_\Theta(h(t_1,t_2,t_3,t_4))=|t_1t_2|^{-1/2}|t_3t_4|^{-1}$ is in the
negative Weyl chamber. 

\end{proof}

Proposition \ref{eisen1} has a local version. Let $\Theta'$ denote any irreducible summand of
$\Theta$. Let $\nu$ denote any finite place where the local representation $\Theta'_\nu$ is unramified. Then the representation $\Theta'_\nu$ is the unramified  subrepresentation of  $Ind_{\widetilde{B}({F_\nu})}^{\widetilde{G}({F_\nu})}\chi_\Theta$. One can characterize this subrepresentation as the space of all functions $f\in Ind_{\widetilde{B}({F_\nu})}^{\widetilde{G}({F_\nu})}\chi_\Theta$ such that $I_wf=0$ for all Weyl elements of $F_4$. Here $I_w$ is the 
intertwining operator corresponding to $w$. This claim is a consequence of the periodicity 
Theorem in \cite{K-P} adopted to the group $F_4$. It is all also simple to verify the claim
that $I_wf=0$ when $w$ corresponds to a simple reflection. It should be mentioned that this
intertwining operators need not converge at the point $\chi_\Theta$. In that case one views
the above statement in the sense of meromorphic continuation.

Let $P=MU$ denote any one of the four maximal parabolic subgroups of $G$. Construct the Jacquet module $J_U(\Theta'_\nu)$. A representation 
of $\widetilde{M}_0(F_\nu)$ is said to be minimal, if it has no nonzero local functionals which corresponds to any unipotent orbit which is greater than the one specified in the global situation.
As in the global case given in Proposition \ref{eisen1}, we obtain
\begin{corollary}\label{eisen11}
As a representation of $\widetilde{M}_0(F_\nu)$, the Jacquet module $J_U(\Theta'_\nu)$ is a minimal representation.
\end{corollary}

Returning to the global case, to prove that $\Theta$ is indeed a minimal representation of the group $\widetilde{G}({\bf A})$, we start by considering the Fourier coefficients which the
Eisenstein series $\widetilde{E}_{\tau_P}(m,s)$ does not support.
We will do it for the case when $P=P_{\alpha_2,\alpha_3,\alpha_4}$. To emphasize the relation of $\tau_P$ to the residue representation of $\widetilde{Sp}_6$, we shall write $\Theta_6$ instead of $\tau_P$. We also refer the reader to \cite{C} page 440 for the description of the
partial order of the unipotent orbits in $F_4$. We prove
\begin{proposition}\label{eisen2}
Let ${\mathcal O}$ denote a unipotent orbit which is greater than or
equal to the unipotent orbit $\widetilde{A}_2$. Then
$\widetilde{E}_{\Theta_6}(m,s)$ has no nonzero Fourier
coefficients corresponding to ${\mathcal O}$. 
\end{proposition}
\begin{proof}
The diagram which is attached to the unipotent orbit
$\widetilde{A}_2$ is given by
$$  \ \ \  \overset{}{0}----\overset{} {0} ==>== \overset{}{0}----\overset{2}{0}$$
In the notations of the previous subsection, we have
$$U_\Delta'(2)=\{(0001);(0011);(0111);(1111);(0121);(1121);(1221);(1231)\}$$
The character $\psi_{U,u_\Delta}$ can be defined as follows. Given
$u\in U_\Delta$ write $u=x_{0121}(r_1)x_{1111}(r_2)u'$ and define
$\psi_{U,u_\Delta}(u)=\psi(r_1+r_2)$. To prove the Proposition, it is enough to prove that the integral
\begin{equation}\label{eisen1.1}
\int\limits_{U_\Delta(F)\backslash U_\Delta({\bf A})}
\widetilde{E}_{\Theta_6}(um,s)\psi_{U,u_\Delta}(u)du
\end{equation}
is zero for all choice of data. It is also clear that it is enough
to show this for $Re(s)$ large. In this proof, let
$P=P_{\alpha_2,\alpha_3,\alpha_4}$ and $U=U_\Delta=U_{\alpha_1,\alpha_2,\alpha_3}$. Unfolding the
Eisenstein series, we need to analyze the set $P(F)\backslash G(F)/U(F)$.
It is clear that a set of representatives for this set
can be chosen in the form $wu_w$ where $w$ is a Weyl element and
$u_w$ is a unipotent element inside $Spin_7(F)$. However, since the
exceptional group $G_2$ is the stabilizer of the above character, it
is in fact enough to consider representatives inside the set $wu_w$
where $wu_w$ is a representative of $P(F)\backslash G(F)/G_2(F)U(F)$. From this it is not hard to deduce that a set of representatives is contained inside the set
$$W_0=\{e, w[123], w[1234], w[123243], w[123214323], w[1232143234],
w[1232143213243]\}$$ This can be seen by first considering the set $P(F)\backslash G(F)/Spin_7(F)U(F)$ and then further study relevant double cosets of the form $R(F)\backslash Spin_7(F)
/G_2(F)$ where $R$ is a suitable maximal parabolic subgroup of $Spin_7$. We omit the details.

In other words we may choose representatives to be only Weyl elements. Thus we have
\begin{equation}\label{eisen2.2}
\int\limits_{U_\Delta(F)\backslash U_\Delta({\bf A})}
\widetilde{E}_{\Theta_6}(um,s)\psi_{U,u_\Delta}(u)du= \sum_{w\in
W_0}\int\limits_{U_\Delta^w(F)\backslash U_\Delta({\bf
A})}f_s(wum)\psi_{U,u_\Delta}(u)du
\end{equation}
Here $U_\Delta^w=w^{-1}U_\Delta w\in P$. We will now show that each
summand of the right hand side is zero. If $w\in W_0$ is such that
$wx_{1111}(r)w^{-1}\in U_{\alpha_2,\alpha_3,\alpha_4}$ then we get zero contribution form that
summand, because $f_s$ is left invariant under $U_{\alpha_2,\alpha_3,\alpha_4}({\bf A})$ and
$\psi_{U,u_\Delta}$ is not trivial on $x_{1111}(r)$. Since the Weyl
elements $e, w[123],w[123243]$ and $w[123214323]$ have this
property, they contribute zero.

As for the other three Weyl elements, we will use the minimality of $\Theta_6$. See right
before Proposition \ref{eisen1}. Consider first the
Weyl element $w[1234]$. It follows by direct conjugation that we obtain the integral
\begin{equation}\label{weil0}
\int\limits_{(F\backslash {\bf A})^7}\theta_6\left (\begin{pmatrix}
I_2&X&Y\\ &I_2&X^*\\ &&I_2
\end{pmatrix}g\right )\psi(\text{tr}(X))dxdy\notag
\end{equation} 
as an inner integration. Here $X\in Mat_{2\times 2}$, and $Y$ and $X^*$ are defined so that the above matrix is in $Sp_6$. 
This Fourier coefficient corresponds to the unipotent orbit $(3^2)$ in
$Sp_6$ ( see \cite{G1}), which is greater than the minimal orbit $(21^4)$. Hence, by the minimality of $\Theta_6$, it
is zero for all choice of data. 

Next consider the two Weyl elements $w[1232143234]$ and $w[1232143213243]$. In these two cases,
we obtain the integral
\begin{equation}\label{weil1}
\int\limits_{(F\backslash {\bf A})^5}\theta_6\left (\begin{pmatrix}
1&x&y\\ &I_4&x^*\\ &&1
\end{pmatrix}g\right )\widetilde{\psi}(x)dxdy
\end{equation}
or a conjugation of it by a Weyl element of $Sp_6$, as an inner integration.
Here $x\in Mat_{1\times 4}$ and $y\in {\bf A}$. The character
$\widetilde{\psi}$ is defined as follows. If $x=(x_{1,j})\in
Mat_{1\times 4}$, define $\widetilde{\psi}(x)=\psi(x_{1,1})$. To prove that this integral is zero we use the fact that $\Theta_6$ is a minimal representation of $\widetilde{Sp}_6({\bf A})$.  Conjugate in the above integral  by the discrete element 
$$w'=\begin{pmatrix} 1&&&\\ &&J_2&\\ &-J_2&&\\ &&&1\end{pmatrix}$$ 
where $J_2$ was defined in subsection 2.1. Then, expanding along the unipotent subgroup
$x(r)=I_4+re_{2,5}$, and using suitable conjugation, we obtain the integral
$$\int\limits_{(F\backslash {\bf A})^3}\theta_6\left (\begin{pmatrix} I_2&&X\\ &I_2&\\ &&I_2 \end{pmatrix}\right )\psi'(X)dX$$  Here $X=\begin{pmatrix} y&z\\ r&y\end{pmatrix}$ and $\psi'(X)=\psi(y)$. It is not hard to check that this Fourier coefficient corresponds to the unipotent orbit $(2^21^2)$ of $Sp_6$. See \cite{G1}. By the minimality of $\Theta_6$, this integral is zero for all choice of data. Thus integral \eqref{weil1} is zero for all choice of data.

Returning to the integral \eqref{eisen2.2}, we obtain that any summand on the right hand side
is zero, and hence the integral on the left hand side of \eqref{eisen2.2} is zero for $Re(s)$
large, and hence zero for all $s$. This proves that
$\widetilde{E}_{\Theta_6}(m,s)$ has no nonzero Fourier coefficient
with respect to the unipotent orbit $\widetilde{A}_2$.

Now we have to prove that for every unipotent orbit ${\mathcal O}$
which is greater than $\widetilde{A}_2$, the Eisenstein series has
no nonzero Fourier coefficient which correspond to this orbit. This
can be done in two ways. One way is to argue in a similar way as we
did with the orbit $\widetilde{A}_2$. For example, it easy to prove
this way that $\widetilde{E}_{\Theta_6}(m,s)$ is not generic, that
is, it has no nonzero Fourier coefficient which correspond to the
unipotent orbit whose label is $F_4$. Another way is to start with
integral \eqref{eisen1.1}, use Fourier expansions and get the other
orbits. For example, consider the orbit $\widetilde{A}_2+A_1$. Its
diagram is
$$  \ \ \  \overset{}{0}----\overset{1} {0} ==>== \overset{}{0}----\overset{1}{0}$$
and the corresponding Fourier coefficient was described in the
previous subsection. Not to confuse with the group $U_\Delta$ as was
defined in \eqref{eisen1.1}, for this proof only, we shall write
$V_\Delta$ instead of $U_\Delta$.  Thus, we need to show that the integral
$$\int\limits_{V_\Delta(F)\backslash V_\Delta({\bf A})}
\widetilde{E}_{\Theta_6}(vm,s)\psi_{V,v_\Delta}(v)dv$$ is zero for
all choice of data. Here $\psi_{V,v_\Delta}$ is defined as follows.
For $v\in V_\Delta$, write
$v=x_{0121}(r_1)x_{1111}(r_2)x_{1220}(r_3)v'$. Define
$\psi_{V,v_\Delta}(v)=\psi(r_1+r_2+\beta r_3)$ where $\beta\in F^*$.

Let $V$ denote the subgroup of $V_\Delta$ which consists of all
roots in $V_\Delta$ such that the coefficient of $\alpha_4$ is
greater than zero. Thus $\text{dim}V=13$ and it is a subgroup of
$U_\Delta$ as defined before integral \eqref{eisen1.1}. Notice that
restricted to the group $V$ we have $\psi_{V,v_\Delta}=\psi_{U,u_\Delta}$,
where the right most character is defined in integral
\eqref{eisen1.1}. Clearly it is enough to prove that the integral
over $V(F)\backslash V({\bf A})$ is zero. Starting with this Fourier
coefficient, we expand along the unipotent group
$\{x_{0001}(l_1)x_{0011}(l_2)\}$ with points in $F\backslash {\bf A}$.
We have
$$\int\limits_{V(F)\backslash V({\bf A})}
\widetilde{E}_{\Theta_6}(vm,s)\psi_{V,v_\Delta}(v)dv=$$
$$\int\limits_{V(F)\backslash V({\bf A})}
\sum_{\delta_i\in F} \int\limits_{(F\backslash {\bf A})^2}
\widetilde{E}_{\Theta_6}(x_{0001}(l_1)x_{0011}(l_2)vm,s)
\psi_{V,v_\Delta}(v)\psi(\delta_1 l_1+\delta_2 l_2)dl_idv$$
Conjugating, from left to right, by the discrete elements
$x_{0110}(-\delta_2)x_{1110}(-\delta_1)$ and changing variables, we
obtain integral \eqref{eisen1.1} as inner integration, which we
proved to be zero for all choice of data. Thus, this Eisenstein
series has no nonzero Fourier coefficients which corresponds to the
unipotent orbit $\widetilde{A}_2+A_1$. Continuing similarly, we
obtain the vanishing of all Fourier coefficients which corresponds
to any unipotent orbit which is greater than $\widetilde{A}_2$.
\end{proof}

\subsection{ A Minimal Representation of $F_4$}

In this subsection we will prove that the residue of the Eisenstein
series, constructed in the previous Sections and denoted there by
$\Theta$, is indeed a minimal representation for the double cover of
$F_4$. In other words we will prove
\begin{theorem}\label{mini1}
Let ${\mathcal O}$ denote a unipotent orbit of $F_4$. Suppose that
${\mathcal O}$ is greater than the minimal orbit which is labeled by
$A_1$. Then $\Theta$ has no nonzero Fourier coefficient which is
attached to the unipotent orbit ${\mathcal O}$.
\end{theorem}

\begin{proof}

We first explain the idea of the proof. Denote by ${\mathcal
O}(\Theta)$ the set of all unipotent orbits of $F_4$ defined as
follows. We have ${\mathcal O}\in {\mathcal O}(\Theta)$ if and only
if the representation $\Theta$ has no nonzero Fourier coefficient
associated with any unipotent orbit which is greater than or not
related to the unipotent orbit ${\mathcal O}$. Also, we require that
$\Theta$ do have a nonzero Fourier coefficient associated with the
orbit ${\mathcal O}$. With these notations the statement of the
Theorem is that ${\mathcal O}(\Theta)$ consists of one unipotent orbit which is the orbit $A_1$.

First, we prove that $\Theta$ has a nonzero Fourier coefficient corresponding to the unipotent 
orbit $A_1$. The diagram corresponding to this orbit is
$$  \ \ \  \overset{1}{0}----\overset{} {0} ==>== \overset{}{0}----\overset{}{0}$$
and the corresponding set of Fourier coefficients is given by
$$\int\limits_{F\backslash {\bf A}}\theta(x_{2342}(r)m)\psi(r)dr$$
It is clear that any nontrivial automorphic representation has such a nonzero Fourier coefficient.
In particular it holds for the representation $\Theta$.

From this and from Proposition \ref{eisen2} it follows that ${\mathcal O}(\Theta)$ consists of one
unipotent orbit which is greater or equal than $A_1$, and which is less than or
equal to the unipotent orbit $B_2$. To prove the Theorem, we fix a
unipotent orbit ${\mathcal O}$ which is greater than $A_1$ and less
or equal to $B_2$. There are such five orbits. They are $B_2,\ A_2+\widetilde{A}_1,\ A_2,\
A_1+\widetilde{A}_1,$ and $\widetilde{A}_1$. We will assume that ${\mathcal
O}(\Theta)={\mathcal O}$, where ${\mathcal O}$ is any one of these five orbits, and we shall derive a contradiction. The way to derive the contradiction is as follows. We consider the
stabilizer of ${\mathcal O}$. It follows from \cite{C} p. 401 that
for all unipotent orbit ${\mathcal O}\ne A_2$, the stabilizer always contains a unipotent subgroup.
This is also true for some Fourier coefficients associated with the unipotent orbit $A_2$, but not for all of them. We shall not need much information on the various unipotent orbit representatives of the orbit $A_2$. However, this information is contained in \cite{I} Section 5.  Assume that we are given a certain Fourier coefficient associated with the unipotent orbit ${\mathcal O}$. Suppose that it is given by the integral
\begin{equation}\label{fc1}
\int\limits_{V(F)\backslash V({\bf A})}\theta(vg)\psi_V(v)dv
\end{equation}
and suppose also that the stabilizer of $\psi_V$ contains  an abelian unipotent subgroup $Z$.  We then consider the Fourier coefficient
\begin{equation}\label{fc2}
\int\limits_{Z(F)\backslash Z({\bf A})} \int\limits_{V(F)\backslash
V({\bf A})}\theta(vz)\psi_V(v) \psi_Z(z)dzdv
\end{equation}
Here $\psi_Z$ is any character defined on $Z(F)\backslash Z({\bf
A})$. If we show that the above integral is zero for all choice of
characters $\psi_Z$,  this will prove that integral \eqref{fc1} is
zero for all choice of data, and hence contradict the assumption
that ${\mathcal O}(\Theta)={\mathcal O}$. To show that the above
integral is zero for all characters we use Fourier expansions to
express the integral as a sum of two types of Fourier coefficients.
The first type are Fourier coefficients which corresponds to
unipotent orbits which are greater than or not related to ${\mathcal
O}$. These coefficients will be zero by our assumption that
${\mathcal O}(\Theta)={\mathcal O}$. The second type are Fourier
coefficients of the type
$$\int\limits_{Y(F)\backslash Y({\bf
A})}\theta^{U(R)}(y)\psi_Y(y)dy$$ Here $\theta^{U(R)}$ is the
constant term of the function $\theta$ along $U(R)$, where $U(R)$ is
the unipotent radical of a maximal parabolic subgroup $R$ of $F_4$.
The group $Y$ is a unipotent subgroup of $M(R)$, the Levi part of
$R$. We then show that the character $\psi_Y$ is a character which corresponds to a unipotent orbit
of $M(R)$ which is not the minimal orbit. Then using Proposition \ref{eisen1} we deduce that this
integral is zero.

We should mention that the proof is local by nature. Indeed, all the
above ideas can be expressed by means of twisted Jacquet modules for
a local constituent of an irreducible summand of the global
representation $\Theta$. We shall use this fact below. However,
mainly because of the Fourier expansions that we perform, it is
convenient to use a global local argument.

We start with the unipotent orbit $B_2$. In other words, we shall
assume that ${\mathcal O}(\Theta)=B_2$ and derive a contradiction.
This unipotent orbit was described in subsection 2.2. A Fourier
coefficient attached to this orbit is given by integral \eqref{fc1}
where the roots in $V$ are given in the beginning of subsection 2.2.
The roots in $V$ contains all 15 roots  of the form
$\alpha=(n_1n_2n_3n_4)$ with $n_1\ge 1$, and the root $(0122)$. 
Up to the action of $M(B_2)=T\cdot Sp_4$, a general character of the group
$V$ is defined as follows. Write
$v=x_{1100}(r_1)x_{1120}(r_2)x_{0122}(r_3)v'$ where $v\in V$ and
define $\psi_{V,\beta}(v)=\psi(r_1+\beta r_2+r_3)$. Here $\beta\in F^*$. From
\cite{C} we deduce that the stabilizer is a group of type $A_1\times
A_1$. In fact, when $\beta$ is a square, then the stabilizer is the
group $Spin_4=SL_2\times SL_2$ and when $\beta$ is not a square we
obtain the group $Spin(1,3)$ which depends on $\beta$. In both cases the stabilizer contains the
unipotent subgroup generated by $\{x_{0100}(r_1)x_{0120}(-\beta r_1)\}$ and $\{x_{0110}(r_1)\}$.
When $\beta$ is a square,  then after a suitable
conjugation, we may choose $\psi_{V,\beta}$ as follows. Write as above
$v=x_{1110}(r_1)x_{0122}(r_2)v'$ and define $\psi_{V,\beta}(v)=\psi(r_1+r_2)$. We shall omit $\beta$
from the notations and write $\psi_V$.
With this choice the stabilizer contains the unipotent group
$\{ x_{0100}(m_1)x_{0120}(m_2)\}$. For simplicity we shall carry out the details when $\beta$ is a square. The other case is similar. 

We start by enlarging the group $V$ to a group $V_1$ whose dimension is 18. To do so, consider the two roots $(0111);\ (0121)$. Define the group $V_1$ to be the group generated by $V$ and by 
$\{ x_{0111}(r_1)x_{0121}(r_2)\}$. Then it follows from \cite{G-R-S3} Lemma 1.1 that integral \eqref{fc1}
is zero for all choice of data if and only if the integral 
\begin{equation}\label{fc20}
\int\limits_{V_1(F)\backslash V_1({\bf A})}\theta(v)\psi_V(v)dv
\end{equation}
is zero for all choice of data. Here we view the character $\psi_V$ as a character of $V_1$ by extending it trivially. This is well defined from the commutation relations in $F_4$. We also mention that the unipotent group $\{ x_{0100}(m_1)x_{0120}(m_2)\}$ stabilizes the group $V_1$.

Choose $Z$ to be the unipotent subgroup $\{x_{0120}(m_2)\}$. Our goal is to prove that integral
\eqref{fc2}, with $V_1$ replacing $V$,  is zero for all characters of $Z$. In other words, we
show that the integral
\begin{equation}\label{fc21}
\int\limits_{F\backslash {\bf A}} \int\limits_{V_1(F)\backslash V_1({\bf
A})}\theta(vx_{0120}(m))\psi_V(v)\psi(am)dmdv
\end{equation}
is zero for all $a\in F$. Assume first that $a\ne 0$. In this case
the above integral is a Fourier coefficient which corresponds to the
unipotent orbit $C_3(a_1)$. Indeed, this Fourier coefficient was
described in subsection 2.2. Using the left invariant properties of
the function $\theta$, we have $\theta(g)=\theta(w[4]g)$.
Conjugating by $w[4]$ from left to right, we obtain exactly the
Fourier coefficient described in subsection 2.2. By our assumption
on ${\mathcal O}(\Theta)$ this integral is zero. Next we consider
the case when $a=0$. We further expand along the unipotent group
$\{x_{0100}(m_1)\}$. Consider first the contribution from the
nontrivial orbit. Conjugating by $w[3]$ we obtain
$w[3]x_{0100}(m_1)w[3]^{-1}=x_{0120}(m_1)$. Hence, when we consider
the nontrivial character, we obtain integral \eqref{fc21}, with a
suitable $a\in F^*$, as inner integration. Hence we get zero.

We are left with the contribution of the trivial orbit.  Therefore,
it is enough to  prove that the integral
\begin{equation}\label{fc3}
\int\limits_{(F\backslash {\bf A})^2} \int\limits_{V_1(F)\backslash
V_1({\bf A})}\theta(vx_{0100}(m_1)x_{0120}(m_2))\psi_V(v)dm_1dm_2dv
\end{equation}
is zero for all choice of data. Expand integral \eqref{fc3} along
the unipotent abelian group $\{x_{0111}(r_1)x_{0121}(r_2)\}$. Thus, integral
\eqref{fc3} is equal to
$$
\sum_{\gamma_i\in F}\int\limits_{(F\backslash {\bf A})^2}
\int\limits_{(F\backslash {\bf A})^2} \int\limits_{V_1(F)\backslash
V_1({\bf A})} $$
$$\theta(x_{0111}(r_1)x_{0121}(r_2)vx_{0100}(m_1)x_{0120}(m_2))
\psi_V(v)\psi(\gamma_1 r_1+\gamma_2 r_2)dr_1dr_2dm_1dm_2dv$$ For all
$\gamma_i\in F$ we have
$\theta(g)=\theta(x_{0001}(-\gamma_2)x_{0011}(-\gamma_1)g)$.
Plugging this into the above integral and changing variables, we
obtain
$$
\sum_{\gamma_i\in F}\int\limits_{(F\backslash {\bf A})^2}
\int\limits_{(F\backslash {\bf A})^2} \int\limits_{V_1(F)\backslash
V_1({\bf A})}
$$
$$\theta(x_{0111}(r_1)x_{0121}(r_2)vx_{0100}(m_1)x_{0120}(m_2)
x_{0001}(-\gamma_2)x_{0011}(-\gamma_1))
\psi_V(v)dr_1dr_2dm_1dm_2dv$$ Hence, to prove that integral
\eqref{fc3} is zero for all choice of data, it is enough to prove
that the integral
\begin{equation}\label{fc311}
\int\limits_{(F\backslash {\bf A})^4} \int\limits_{V_1(F)\backslash
V_1({\bf
A})}\theta(x_{0111}(r_1)x_{0121}(r_2)vx_{0100}(m_1)x_{0120}(m_2))
\psi_V(v)dr_1dr_2dm_1dm_2dv
\end{equation}
is zero for all choice of data.

Let $V_2$ denote the unipotent group generated by the group $V_1$ and
the abelian group $\{
x_{0111}(r_1)x_{0121}(r_2)x_{0100}(m_1)x_{0120}(m_2)\}$. Thus the
dimension of $V_2$ is 20. Conjugating by the Weyl element $w[2134]$,
integral \eqref{fc311} is equal to
\begin{equation}\label{fc4}
\int\theta(x_{1000}(r_1)x_{0121}(r_2)v'x_{-1100}(m_1)x_{-1000}(m_2)w[2134])
\psi(r_1+r_2)dr_1dr_2dm_1dm_2dv'
\end{equation}
Here $v'$ is a product over all other 16 one dimensional unipotent
subgroups corresponding to roots in $w[2134]V_2w[2134]^{-1}$. All
variables are integrated over $F\backslash {\bf A}$. We now apply
Fourier expansion to integral \eqref{fc4}. Expand this integral
along the unipotent subgroup $\{ x_{1221}(t)\}$. Thus, integral
\eqref{fc4} is equal to
\begin{equation}\label{fc5}
\int\sum_{\gamma\in
F}\int\theta(x_{1221}(t)x_{1000}(r_1)x_{0121}(r_2)v'x_{-1000}(m_1)\times
\end{equation}
$$x_{-1100}(m_2)w[2134])
\psi(r_1+r_2+\gamma t)dtdr_1dr_2dm_1dm_2dv'$$ We have
$$x_{-(1100)}(-\gamma)x_{1221}(t)=x_{0121}(-\gamma
t)x_{1342}(t\gamma^2)x_{1221}(t)x_{-(1100)}(-\gamma)$$ The function
$\theta$ is left invariant under $x_{-(1100)}(-\gamma)$. Performing
the above conjugation in \eqref{fc5}, changing variables and
collapsing summation with integration, we obtain
\begin{equation}\label{fc6}
\int\limits_{\bf
A}\int\theta(x_{1221}(t)x_{1000}(r_1)x_{0121}(r_2)v'x_{-1000}(m_1)\times
\end{equation}
$$x_{-1100}(m_2)w[2134])
\psi(r_1+r_2)dtdr_1dr_2dm_1dm_2dv'$$ where the adelic integration is
over the variable $m_2$. This is the process of root exchange we
refer to in subsection 2.2.2. Indeed, in the notations of that
subsection, let $\alpha=(0121);\ \beta=-(1100)$ and $\gamma=(1221)$.
Thus we exchange the root $-(1100)$ by $(1221)$. Next we repeat the
same process, and we exchange the root $-(1000)$ by $(1100)$. It
follows that integral \eqref{fc6} is zero provided we can show that
the integral
\begin{equation}\label{fc7}
\int\limits_{Y(F)\backslash Y({\bf
A})}\theta^{U(R),\psi}(y)\psi_Y(y)dy
\end{equation}
is zero. Here, $R=P_{\alpha_2,\alpha_3,\alpha_4}$ is the maximal parabolic subgroup of $F_4$ whose
Levi part is $GSp_6$, and $U(R)$ is its unipotent radical. Also,
$$\theta^{U(R),\psi}(g)=\int\limits_{U(R)(F)\backslash U(R)({\bf
A})}\theta(ug)\psi_{U(R)}(u)du$$ where $\psi_{U(R)}$ is defined as
follows. Write $u\in U(R)$ as $u=x_{1000}(r)u'$. Then
$\psi_{U(R)}(u)=\psi_{U(R)}(x_{1000}(r)u')=\psi(r)$. Finally, the
group $Y$ consists of all roots $\{ (0010); (0011); (0120);$ $ (0121);
(0122)\}$. The character $\psi_Y$ is defined by
$\psi_Y(y)=\psi_Y(x_{0121}(m_1)y')=\psi(m_1)$. We now do two more
exchange of roots. First we exchange the root $(0110)$ by $(0011)$,
and then exchange $(0111)$ by $(0010)$. Then, conjugating by the
Weyl element $w[43]$, integral \eqref{fc7} is zero for all choice of data if and only if
the integral
\begin{equation}\label{fc71}
\int\limits_{Y_1(F)\backslash Y_1({\bf
A})}\theta^{U(R),\psi}(y_1w[43])\psi_{Y_1}(y_1)dy_1
\end{equation}
is zero for all choice of data. Here $Y_1$ is the unipotent subgroup which consists of the roots
$\{(0110);(0111);(0120);(0121);(0122)\}$, and
$\psi_{Y_1}(y_1)=\psi_{Y_1}(x_{0110}(r)y'_1)=\psi(r)$.

Next, we expand along the group $x_{0100}(t)$.  Thus, integral
\eqref{fc71} is a sum of integrals of the form
\begin{equation}\label{fc8}
\int\limits_{F\backslash {\bf A}}\int\limits_{Y_1(F)\backslash
Y_1({\bf
A})}\theta^{U(R),\psi}(x_{0100}(r)y_1)\psi_{Y_1}(y_1)\psi(\gamma
r)drdy_1
\end{equation}
where $\gamma\in F$.

Conjugating by the element  $x_{0010}(-\gamma)$, and changing
variables we obtain that the integral \eqref{fc8} is zero provided
the integral
\begin{equation}\label{fc9}
\int\limits_{F\backslash {\bf A}}\int\limits_{Y_1(F)\backslash
Y_1({\bf A})}\theta^{U(R),\psi}(x_{0100}(r)y_1)\psi_{Y_1}(y_1)drdy_1
\end{equation}
is zero for all choice of data.  Thus integral \eqref{fc7} is zero
for all choice of data  if integral \eqref{fc9} is zero for all
choice of data. Expand integral \eqref{fc9} along the unipotent
group $x_{0001}(m_1)x_{0011}(m_2)$. The contribution from the
nontrivial orbit is zero. Indeed, in this case we obtain
\begin{equation}\label{fc10}
\int\limits_{Y_1(F)\backslash Y_1({\bf A})}\int\limits_{(F\backslash
{\bf A})^2}\theta^{U(R),\psi}(x_{0001}(m_1)x_{0011}(m_2)y_1)
\psi_{Y_1}'(y_1)\psi(\gamma_1m_1+\gamma_2m_2)dm_1dm_2dy_1
\end{equation}
where $\gamma_1,\gamma_2\in F$ are not both zero. As follows from
subsection 2.2 this Fourier coefficient is associated with the
unipotent orbit $F_4(a_1)$, and hence zero for all choice of data.
Thus we are left with integral \eqref{fc10} where
$\gamma_1=\gamma_2=0$. In this case we can write integral
\eqref{fc10} as
\begin{equation}\label{fc11}\notag
\int\limits_{Y_2(F)\backslash Y_2({\bf
A})}\theta^{V(L)}(y_2)\psi_{Y_2}(y_2)dy_2
\end{equation}
Here $L=P_{\alpha_1,\alpha_2,\alpha_3}$ is the maximal parabolic subgroup of $F_4$ whose Levi part
is $GSpin_7$. We denote its unipotent radical by $V(L)$, and
$\theta^{V(L)}$ is the constant term along $V(L)$. The group $Y_2$
is a unipotent subgroup of $GSpin_7$. It consists of all positive
roots in that group except $(0010)$. Thus its dimension is eight.
The character $\psi_{Y_2}$ is defined as follows
$\psi_{Y_2}(y_2)=\psi_{Y_2}(x_{1000}(t_1)x_{0110}(t_2)y_2')=\psi(t_1+t_2)$.
This Fourier coefficient is associated with the unipotent orbit
$(51^2)$ of $Spin_7$. Applying Proposition \ref{eisen1} this
integral is zero. This completes the case of the unipotent orbit
$B_2$, when $\beta$ as defined before integral \eqref{fc21} is a
square. As mentioned above, the case when $\beta$ is not a square is similar and will be omitted.

Next we assume that ${\mathcal O}(\Theta)=A_2+\widetilde{A}_1$. The corresponding Fourier coefficient was not described explicitly, and we do it now. In this case the set $U_\Delta'(2)$
consists of all nine roots of the form $\sum n_i\alpha_i$ where $n_3=2$. Thus $\text{dim}\ U_\Delta(2)=14$ and write $V=U_\Delta(2)$. Then the corresponding Fourier coefficient is given by integral \eqref{fc1} where $\psi_V$ is defined as follows. Write $v=x_{1220}(r_1)x_{0122}(r_2)
x_{1121}(r_3)v'$. Then $\psi_V(v)=\psi(r_1+r_2+r_3)$. As stated in \cite{C}, the stabilizer of this character is a group of type $A_1$, and it can be identified with the split
orthogonal group $SO_3$. Hence it contains a unipotent subgroup. This unipotent subgroup is generated by $\{x_{1000}(r)x_{0100}(-r)x_{1100}(ar^2)x_{0001}(r)\}$ where $a\in F^*$.

For the unipotent orbit $A_2$ the situation is different. In this case the group $V=U_{\alpha_2, 
\alpha_3,\alpha_4}$, and the stabilizer of this orbit is a group of type $A_2$. It follows from \cite{I} that over the rational points there is a choice of a character $\psi_V$ such that the stabilizer is the group $SL_3(F)$. But there is also a choice of characters such that the stabilizer 
is various unitary groups. The character $\psi_V$ whose stabilizer is $SL_3$, is given as follows.
Let $v=x_{1000}(r_1)x_{1342}(r_2)v'$. Then $\psi_V(v)=\psi(r_1+r_2)$. A unipotent subgroup which is contained in
the stabilizer is, for example, $\{x_{0010}(m_1)x_{0001}(m_2)x_{0011}(m_3)\}$. We shall refer to this Fourier coefficient as to the split Fourier coefficient associated with the unipotent orbit $A_2$.

To prove that ${\mathcal O}(\Theta)$ is not $A_2+\widetilde{A}_1$, or to prove that $\Theta$ has
no nozero split Fourier coefficient associated with the unipotent orbit $A_2$, we apply the same 
ideas as we did in the case of the orbit $B_2$. We omit the details.

However, we still have to
consider the Fourier coefficients associated with the other representatives of the unipotent
orbit $A_2$. Here we give a local argument. In details, let $\Theta'$ denote any irreducible 
summand of $\Theta$. Let $\nu$ be a finite unramified place.
As mentioned in the beginning of the proof, the above arguments for the unipotent orbits $B_2,\ A_2+\widetilde{A}_1$ and for the split Fourier coefficient corresponding
to the unipotent orbit $A_2$, all work in a similar way for the representation $\Theta'_\nu$. In
other words we may assume that ${\mathcal O}(\Theta'_\nu)$ is the unipotent orbit  $A_2$ for any unramified place $\nu$.
Given a Fourier coefficient of $\Theta$ associated to the unipotent orbit $A_2$, we may choose a
place $\nu$ such that the stabilizer of the corresponding Jacquet module will be the group $SL_3$.
Arguing as in the global case, using corollary \ref{eisen11}, we know that this Jacquet module
is zero. Hence, we can deduce that the corresponding Fourier coefficient is zero for all choice
of data, and for all representative associated with the unipotent orbit $A_2$. Thus we may assume that ${\mathcal O}(\Theta)$ is at most $A_1+\widetilde{A}_1$.

Assume that ${\mathcal O}(\Theta)=A_1+\widetilde{A}_1$. The set of
Fourier coefficients associated with this orbit is described in
subsection 2.2. We shall view these Fourier coefficients in an
extended way. More precisely, in the notations of subsection 2.2,
consider the set of roots $U_\Delta'(1)$. This set consists of 12
roots which are
$$U_\Delta'(1)=\{ (0100); (1100); (0110); (1110);
(0111); (0120); (0121); (1111);$$  $$(1120); (0122); (1121);
(1122)\}$$ The center of the group $U_\Delta$ is given by the group
$Y=\{ x_{1342}(m_1)x_{2342}(m_2)\}$. As can be checked, the quotient
$U_\Delta/Y$ has a structure of a generalized Heisenberg group. Let
${\mathcal H}_{13}$ denote the Heisenberg group with 13 variables.
We view this group as all 13 tuples $(r_1,\ldots,r_6,t_1,\ldots,
t_6,z)$ where the product is given as in \cite{I1}.  Recall from
subsection 2.2 that the set of Fourier coefficients associated to
the unipotent orbit $A_1+\widetilde{A}_1$ are parameterized by a subset of triples
$\beta_1,\beta_2,\beta_3\in F^*$. For fixed $\beta_i$, the Fourier
coefficient is given by integral \eqref{four2}. Define a
homomorphism $l$ from $U_\Delta/Y$ onto ${\mathcal H}_{13}$ as
follows.
$$l(x_{0100}(r_1)x_{0110}(r_2)x_{0111}(r_3)x_{0120}(r_4)
x_{0121}(r_5)x_{0122}(r_6))=(r_1,\ldots,r_6,0,\ldots,0)$$
$$l(x_{1100}(t_1)x_{1110}(t_2)x_{1111}(t_3)x_{1120}(t_4)
x_{1121}(t_5)x_{1122}(t_6))=(0,\ldots,0,t_1,\ldots,t_6,0)$$
$$l(x_{1220}(z_1)x_{1221}(z_2)x_{1222}(z_3)x_{1231}(z_4)
x_{1232}(z_5)x_{1242}(z_6))=(0,\ldots,0,\beta_1 z_1+\beta_2
z_3+\beta_3 z_6)$$ We extend $l$ trivially from $U_\Delta/Y$ to
$U_\Delta$ by $l(Y)=0$. Consider the integral
\begin{equation}\label{fc31}
\int\limits_{U_\Delta(F)\backslash U_\Delta({\bf
A})}\widetilde{\theta}^{\psi}_{\phi}(l(u)g)\theta(ug)du
\end{equation}
Here $\widetilde{\theta}^{\psi}_{\phi}$ is a vector in the theta
representation of the group ${\mathcal H}_{13}({\bf
A})\cdot\widetilde{Sp}_{12}({\bf A})$. The function $\phi$ is a
Schwartz function of ${\bf A}^6$. Arguing as in Lemma 1.1 in
\cite{G-R-S3}, we deduce that integral \eqref{four2} is zero for all
choice of data if and only if integral \eqref{fc31} is zero for all
choice of data. Consider the $SL_2$ generated by $\{x_{\pm 1000}(r)\}$.
One can check that this group is inside the stabilizer of the
character as defined in integral \eqref{four2}. Hence, if we take $g\in SL_2$, then
integral \eqref{fc31} defines an automorphic function in the
of this group. It is not hard to check
that this copy of $SL_2$ splits under the double cover when embedded
inside $\widetilde{Sp}_{12}$. Indeed, after a suitable conjugation
we can embed it inside $Sp_{12}$ as $g\to
\text{diag}(g,g,g,g^*,g^*,g^*)$. However, this copy of $SL_2$ does
not split under the double cover of $F_4$. Therefore, as a function
of $g$, integral \eqref{fc31} defines a genuine automorphic function
of $\widetilde{SL}_{2}({\bf A})$.  Our goal is to prove that integral
\eqref{fc31} is zero for all choice of data. Since the identity
function is not genuine, it follows that integral \eqref{fc31} is
zero for all choice of data if and only if, for all $a\in F^*$ the
integral
\begin{equation}\label{fc32}
\int\limits_{F\backslash {\bf A}} \int\limits_{U_\Delta(F)\backslash
U_\Delta({\bf A})}\widetilde{\theta}^{\psi}_{\phi}(l(u)x_{1000}(r))
\theta(ux_{1000}(r))\psi(ar)drdu
\end{equation}
is zero for all choice of data. Arguing as in Lemma 1.1 in
\cite{G-R-S3}, integral \eqref{fc32} is zero for all choice of data
if and only if the integral
\begin{equation}\label{fc33}
\int\limits_{F\backslash {\bf A}}
\int\limits_{U_\Delta(2)(F)\backslash U_\Delta(2)({\bf
A})}\int\limits_{V(F)\backslash V({\bf A})}
\theta(vux_{1000}(r))\psi_{U_\Delta(2),u_0}(u)\psi(ar)dvdrdu
\end{equation}
is zero for all choice of data. Here $U_\Delta(2)$ and $\psi_{U_\Delta(2),u_0}$
are as defined in integral \eqref{four2}. Also, the group $V$ is the
unipotent subgroup of $F_4$ defined by
$$V=\{x_{1100}(t_1)x_{1110}(t_2)x_{1111}(t_3)x_{1120}(t_4)
x_{1121}(t_5)x_{1122}(t_6)\}$$ Let $R=P_{\alpha_2,\alpha_3,\alpha_4}$ denote the maximal parabolic
subgroup of $F_4$ whose levi part is $GSp_6$. Denote its unipotent
radical by $U(R)$. Then integral \eqref{fc33} is equal to
\begin{equation}\label{fc34}
\int\limits_{U(R)(F)\backslash U(R)({\bf A})}
\theta(u)\psi_{U(R)}(u)du
\end{equation}
where $\psi_{U(R)}$ is defined as follows. We have
$$\psi_{U(R)}(u)=\psi_{U(R)}(x_{1000}(r_1)x_{1220}(r_2)x_{1222}(r_3)
x_{1242}(r_4)u')=\psi(ar_1+\beta_1 r_2 +\beta_2 r_3 +\beta_3 r_4)$$
It follows from the description given in \cite{I} Section 5 that the above
Fourier coefficient is associated with the unipotent orbit $A_2$.
Therefore, from the assumption ${\mathcal
O}(\Theta)=A_1+\widetilde{A}_1$, it follows that the integral
\eqref{fc34} is zero for all choice of data. Thus, integral
\eqref{fc31} is zero for all choice of data and we derived a
contradiction. Hence ${\mathcal O}(\Theta)$ is less than the orbit
$A_1+\widetilde{A}_1$.

Finally we consider the case ${\mathcal O}(\Theta)=\widetilde{A}_1$.
The set of Fourier coefficients attached to this orbits can be
described as follows. Let $U_\Delta'$ denote the unipotent group
defined by
$$U_\Delta'=\{  (0122);
(1122); 1222); (1232); (1242); (1342); (2342)\}$$ As before we confuse
between a root $\alpha$ and its corresponding one dimensional
unipotent group $x_\alpha(r)$. For $\beta\in (F^*)^2\backslash F^*$ we define a character
$\psi_{U_\Delta',\beta}$ of this group as follows. Given $u\in
U_\Delta'$ let
$\psi_{U_\Delta',\beta}(x_{1222}(r_1)x_{1242}(r_2))=\psi(r_1+\beta
r_2)$. Then, the Fourier coefficients associated with this unipotent orbit, are given by 
\begin{equation}\label{fc41}
\int\limits_{U_\Delta'(F)\backslash U_\Delta'({\bf A})}
\theta(u)\psi_{U_\Delta',\beta}(u)du
\end{equation}
The stabilizer inside $Spin_7$ of $\psi_{U_\Delta',\beta}$ contains a unipotent subgroup, for example the group generated by $\{x_{1000}(r)\}$. 
As in the case of $B_2$, it is convenient to separate into two cases. First when $\beta$ is a square, and the second case is when it is not a square. We will consider the first case, and omit
the details in the second one.

When $\beta$ is a square we can conjugate by a suitable
element, and integral \eqref{fc41} is zero for all choice of data if
and only if the integral
\begin{equation}\label{fc42}
\int\limits_{U_\Delta'(F)\backslash U_\Delta'({\bf A})}
\theta(u)\psi_{U_\Delta'}(u)du
\end{equation}
is zero for all choice of data, where now
$\psi_{U_\Delta'}(u)=\psi_{U_\Delta'}(x_{1232}(r)u')=\psi(r)$.
Arguing in a similar way as in the proof of Lemma 1.1 in
\cite{G-R-S3}, see also a similar case right before \eqref{fc311},
implies that we may consider the integral
\begin{equation}\label{fc43}\notag
\int\limits_{U'(F)\backslash U'({\bf A})}
\int\limits_{U_\Delta'(F)\backslash U_\Delta'({\bf A})}
\theta(u'u)\psi_{U_\Delta'}(u)du'du
\end{equation}
In other words, integral \eqref{fc42} is zero for all choice of
data, provided the above integral is zero for all choice of data.
Here $U'$ is the unipotent group which is defined by $U'=\{(0111);
(1111); (1221); (1231)\}$. Let $V=U'U_\Delta'$ and define $\psi_V$
to equal $\psi_{U_\Delta'}$ on $U_\Delta'$ extended trivially to
$V$. It follows from \cite{C} that the stabilizer of $\psi_V$ is a
group of type $A_3$. It is not hard to check that it is the group
$SL_4$ which contains the abelian unipotent group $Z=\{
x_{0120}(m_1)x_{1120}(m_2)x_{1220}(m_3)\}$. Consider the automorphic
function of $\widetilde{SL}_4({\bf A})$ defined by
\begin{equation}\label{fc44}
f(g)=\int\limits_{V(F)\backslash V({\bf A})} \theta(vg)\psi_{V}(u)du
\end{equation}
Since the above group $SL_4$ does not split under the double cover
of $F_4$, then $f(g)$ is a genuine function. Expand this function
along the group $Z$. The group $SL_3(F)$ embedded in $SL_4(F)$ in
the obvious way, acts on this expansion, and we obtain two orbits
under this action. Arguing as in the case when ${\mathcal
O}(\Theta)=A_1+\widetilde{A}_1$ we deduce that to prove that
integral \eqref{fc44} is zero for all choice of data, it is enough
to prove that the integral
\begin{equation}\label{fc45}
\int\limits_{(F\backslash {\bf A})^3} \int\limits_{V(F)\backslash
V({\bf A})}
\theta(vx_{0120}(m_1)x_{1120}(m_2)x_{1220}(m_3))\psi_{V}(u)\psi(m_1)dm_idu
\end{equation}
is zero for all choice of data. Indeed, if the above integral is
zero for all choice of data, then $f(g)$ is equal to its constant
term corresponding to a unipotent radical of a maximal parabolic
subgroup. This is true only if $f(g)$ is the identity function which
is not the case.  Using the left invariant property of $\theta$, we
have $\theta(h)=\theta(w[214]h)$. Conjugating $w[214]$ in integral
\eqref{fc45} from left to right, and exchanging the root $(0010)$ by
$(1221)$, we obtain that integral \eqref{fc45} defines a Fourier
coefficient associated with the unipotent orbit
$A_1+\widetilde{A}_1$ which is greater than ${\mathcal
O}(\Theta)=\widetilde{A}_1$. Hence it is zero, and hence integral
\eqref{fc42} is zero for all choice of data. Once again we derived a
contradiction.

It follows that $\Theta$ has no nonzero Fourier coefficients which
corresponds to any unipotent orbit which is greater or equal to
$\widetilde{A}_1$. This completes the proof of the Theorem.

\end{proof}

\subsection{\bf  Properties of the Minimal Representation}

In this subsection we shall derive basic properties of the
representation $\Theta$. These properties are all a consequence of
the smallness properties of this representation.

From Theorem \ref{mini1} we deduce two important properties of the
representation $\Theta$. Let $U$ denote the Heisenberg unipotent
radical of $F_4$. In other words, let $U=U_{\alpha_2,\alpha_3,\alpha_4}$.
Let $Z=\{x_{2342}(r)\}$ denote the one dimensional unipotent group
attached to the highest root of $F_4$. Thus, the group $Z$ is the
center of $U$.  Define a character $\psi_U$ of $U(F)\backslash
U({\bf A})$ as follows. For $u\in U$, write $u=x_{1000}(r)u'$. Define $\psi_U(u)=\psi(r)$.
( See subsection 2.1) For
any $g\in F_4({\bf A})$, denote
$$\theta^{U,\psi}(g)=\int\limits_{U(F)\backslash U({\bf A})}
\theta(ug)\psi_U(u)du$$ Similarly, we denote
$$\theta^{U}(g)=\int\limits_{U(F)\backslash U({\bf A})}
\theta(ug)du$$

From Theorem \ref{mini1} we deduce
\begin{proposition}\label{property2}
With the above notations, we have the following expansion
\begin{equation}\label{exp1}
\int\limits_{Z(F)\backslash Z({\bf A})}\theta(zg)dz=\theta^{U}(g)+
\sum_{\gamma\in Q(F)\backslash Sp_6(F)}\sum_{\epsilon\in F^*}
\theta^{U,\psi}(h_2(\epsilon)\gamma g)
\end{equation}
Here $Q$ is the maximal parabolic subgroup of $Sp_6$, whose Levi
part is the group $GL_3$.
\end{proposition}

\begin{proof}
The group $Z\backslash U$ is abelian. Hence, we have the following
Fourier expansion
$$\int\limits_{Z(F)\backslash Z({\bf A})}\theta(zg)dz=
\sum_{\gamma\in L(F)}\int\limits_{U(F)\backslash U({\bf
A})}\theta(ug)\psi_\gamma(u)du$$ where $L(F)$ runs over all
characters of $Z({\bf A})U(F)\backslash U({\bf A})$. We can identify
the group $L(F)$ with $F^{14}\simeq U(F)/Z(F)$. The group $Sp_6(F)$
acts on $L(F)$ as the third fundamental representation of $Sp_6$. We
have three type of orbits. First, we have the orbit corresponding to
the zero vector. Then, we have the orbit generated by the group
$\{x_{1000}\}$. The third type of orbits, are all the other ones not
included in the first two. It is not hard to show that the Fourier
coefficients which corresponds to an orbit of the third type
correspond to a unipotent orbit which is greater than the unipotent
orbit $A_1$. By Theorem \ref{mini1} they contribute zero to the
above expansion. Thus we are left with the first two type of orbits.
The trivial orbit corresponds to the constant term, and the second
one corresponds to the Fourier coefficient $\theta^{U,\psi}$. From
this expansion \eqref{exp1} follows.
\end{proof}

Another result which can be derived from Theorem \ref{mini1} is the
following. Let $U_Q$ denote the unipotent radical of $Q$ where $Q$
is the parabolic subgroup of $Sp_6$ which was defined in Proposition
\ref{property2}. Let $Q^0$ denote the subgroup of $Q$ defined by
$Q^0=SL_3\cdot U_Q$. We have
\begin{proposition}\label{property3}
For all $q\in Q^0({\bf A})$, we have
\begin{equation}\label{exp2}
\theta^{U,\psi}(qg)=\theta^{U,\psi}(g).
\end{equation}
\end{proposition}
\begin{proof}
Let $U_{Sp_6}$ denote the maximal unipotent subgroup of $Sp_6$.
The group $Q^0({\bf A})$ is generated by $U_{Sp_6}({\bf A})$ and
the two simple reflections $w[3]$ and $w[4]$. Clearly \eqref{exp2}
holds for the above two simple reflections. Thus its enough to
prove \eqref{exp2} for $q\in U_{Sp_6}({\bf A})$. The group $U_Q$
is abelian. Hence we can consider the Fourier expansion of
$\theta^{U,\psi}$ along this group. We have
$$\theta^{U,\psi}(g)=\sum_{\gamma}\int\limits_{U_Q(F)\backslash
U_Q({\bf A})}\theta^{U,\psi}(vg)\psi_\gamma(v)dv$$ where we sum over
all characters of the group $U_Q(F)\backslash U_Q({\bf A})$. We
claim that for all nontrivial characters, the Fourier coefficient
$$\int\limits_{U_Q(F)\backslash
U_Q({\bf A})}\theta^{U,\psi}(vg)\psi_\gamma(v)dv$$ is zero for all
choice of data. This follows from the same type of arguments as in
the proof of Theorem \ref{mini1}. Indeed, when considering suitable
Fourier expansions of the above integral we obtain two types of
integrals. The first type are Fourier coefficients which are
associated with unipotent orbits which are greater than $A_1$.
Hence, by Theorem \ref{mini1} they are zero. The second type is an
integral of the form
$$\int\limits_{Y(F)\backslash Y({\bf A})}\theta^{U(R)}(y)
\psi_Y(y)dy$$ Here $R$ is a certain maximal parabolic subgroup of $F_4$ and
$U(R)$ is its unipotent radical. The group $Y$ is a subgroup of $M(R)$,
the Levi part of $R$. Finally, the character $\psi_Y$ is associated
with a unipotent orbit which is greater than the minimal orbit of
$M(R)$. Thus, from Proposition \ref{eisen1} this integral is zero
for all choice of data.

Hence, only the constant term remains, and we proved \eqref{exp2}
for all $q\in U_{Q}({\bf A})$. In a similar way, using again
Proposition \ref{eisen1}, we obtain the invariance property of
$\theta^{U,\psi}$ along the adelic points of $U_{Sp_6}/U_Q$.
\end{proof}

The next  Proposition relate the minimal representation of
$\widetilde{F}_4$ to the theta representation defined on the
symplectic group $\widetilde{Sp}_{14}$. Consider the Fourier
coefficient corresponding to the unipotent orbit $A_1$. In other
words, consider the integral
$$\theta^{Z,\psi_\beta}(g)=\int\limits_{F\backslash {\bf A}}
\theta(x_{2342}(r)g)\psi(\beta r)dr$$ Here $\beta\in F^*$. This
Fourier coefficient defines an automorphic representation of
$\widetilde{Sp}_6({\bf A})$. Let ${\mathcal H}_{15}$ denote the Heisenberg group
with 15 variables. The group $U$ is isomorphic to ${\mathcal
H}_{15}$. We shall denote this isomorphism by $\iota$. We have

\begin{proposition}\label{property6}
With the above notations, the space of functions
$$\theta_{Sp_{14}}^{\phi,\psi_\beta}(\iota(u)\varpi_3(g))$$ is a dense subspace
in the space of functions $\theta^{Z,\psi_\beta}(ug)$. Here $g\in
\widetilde{Sp}_6({\bf A})$, $u\in U({\bf A})$ and
$\theta_{Sp_{14}}^{\phi,\psi_\beta}\in
\Theta_{Sp_{14}}^{\phi,\psi_\beta}$ is the theta representation of
${\mathcal H}_{15}({\bf A})\cdot \widetilde{Sp}_{14}({\bf A})$
attached to the character $\psi_\beta$. Also, we denote by
$\varpi_3$ the third fundamental representation of $Sp_6$.
\end{proposition}

\begin{proof}
It follows from \cite{I1} that the space of functions
$$\theta_{Sp_{14}}^{\phi,\psi_\beta}(\iota(u)\varpi_3(g))
\int\limits_{U(F)\backslash U({\bf A})}
\theta_{Sp_{14}}^{\phi',\psi_\beta}(\iota(v)\varpi_3(g))
\theta(vg)dv$$ is a dense subspace in the space of functions
$\theta^{Z,\psi_\beta}(ug)$. The result will follows once we prove
that as a function of $g\in Sp_6({\bf A})$, the integral
$$\int\limits_{U(F)\backslash U({\bf A})}
\theta_{Sp_{14}}^{\phi',\psi_\beta}(\iota(v)\varpi_3(g))
\theta(vg)dv$$ is the identity function. Since the embedding of
$Sp_6$ in both $Sp_{14}$ via the third fundamental representation
does not split under the double cover, we deduce that the above
integral is not a genuine function. Hence, to obtain the result, it
is enough to prove that for all $a\in F^*$ the integral
$$\int\limits_{U(F)\backslash U({\bf A})}
\int\limits_{F\backslash {\bf A}}
\theta_{Sp_{14}}^{\phi',\psi_\beta}(\iota(v)\varpi_3(x_{0122}(r)))
\theta(vx_{0122}(r))\psi(ar)drdv$$ is zero for all choice of data.
Unfolding the theta function, we obtain as an inner integration the
integral
$$\int\limits_{V(F)\backslash V({\bf A})}\theta(v)\psi_V(v)dv$$
Here, the group $V$ is the unipotent subgroup of $F_4$ which is
associated with the seven positive roots of $F_4$ of the form
$(n_1n_2n_3n_4)$ with $n_4=2$. The character $\psi_V$ is defined as
$\psi_V(v)=\psi_V(x_{0122}(r_1)x_{2342}(r_2)v')=\psi(ar_1+\beta
r_2)$. Thus, the above integral is a Fourier coefficient which is
associated with the unipotent orbit $\widetilde{A}_1$. From Theorem
\ref{mini1} it is zero for all choice of data.

\end{proof}

\subsection{ On Minimal Representations of the Group $\widetilde{Sp}_6({\bf A})$} 

Let $\Theta_{Sp_6}^{(2)}$ denote a minimal representation of $\widetilde{Sp}_6({\bf A})$.
By definition this means that given any unipotent orbit of $Sp_6$ which is greater than $(21^4)$, then all Fourier coefficients of $\Theta_{Sp_6}^{(2)}$ which are associated with this orbit (see \cite{G1}) are zero for all choice of data.
In the computations we shall perform we will need for the representation $\Theta_{Sp_6}^{(2)}$, similar properties to the ones we stated and proved in subsection 2.5. More
precisely, we will need analogous results to those which are stated in Propositions 
\ref{eisen1}, \ref{property2}, \ref{property3} and \ref{property6}.

Recall that $Sp_6$ has three maximal parabolic subgroups. Let $P(GL_3)$ denote the maximal parabolic subgroup of $Sp_6$ whose Levi part is $GL_3$. Similarly, we shall denote the other two maximal parabolic subgroups by $P(GL_2\times SL_2)$ and $P(GL_1\times Sp_4)$. We denote by $U(GL_3)$ the
unipotent radical of $P(GL_3)$, and use similar notations for the other two maximal parabolic subgroups.  We remark
that the group $GL_3$ embedded in $Sp_6$ as the Levi part of $P(GL_3)$, splits under the
double cover of $Sp_6$.
To prove the analogous 
Proposition to Proposition \ref{eisen1}, we define the group $\widetilde{M}_0$ for each maximal parabolic subgroup $P$. When $P=P(GL_3)$ we denote $\widetilde{M}_0=GL_3$.
When $P=P(GL_2\times SL_2)$ we define $\widetilde{M}_0=GL_2\times \widetilde{SL}_2$, and when $P=P(GL_1\times Sp_4)$ we denote $\widetilde{M}_0=\widetilde{Sp}_4$. When 
$\widetilde{M}_0=GL_3$, a representation  of $\widetilde{M}_0({\bf A})$ is said to be minimal if it is one dimensional. When $\widetilde{M}_0=GL_2\times \widetilde{SL}_2$, 
a representation  of $\widetilde{M}_0({\bf A})$ is said to be minimal if it is one dimensional on $GL_2$. Finally, when $\widetilde{M}_0=\widetilde{Sp}_4$, a representation  of $\widetilde{M}_0({\bf A})$ is said to be minimal if it is a minimal representation of
$\widetilde{Sp}_4$, that is its only nonzero Fourier coefficients are associated with the
unipotent orbit $(21^2)$ of $Sp_4$. We start with
\begin{proposition}\label{propertysp1}
Let $U$ denote any unipotent radical of a maximal parabolic subgroup of $Sp_6$. Then, as a
representation of $\widetilde{M}_0({\bf A})$, the constant term $\Theta_{Sp_6}^{(2),U}$ is
a minimal representation.
\end{proposition}

\begin{proof}
Consider the case when $U$ us the unipotent radical of $P(GL_3)$. In this case, consider
the one dimensional unipotent subgroup $N=\{x(r)=I_6+r(e_{1,3}-e_{4,6})\}$. Here $e_{i,j}$ is the
matrix of size six which has a one at the $(i,j)$ entry and zero otherwise. Expand the
constant term $\Theta_{Sp_6}^{(2),U}$ along the group $N(F)\backslash N({\bf A})$. We claim that for all $a\in F^*$, the integral
$$\int\limits_{F\backslash {\bf A}}\theta_{Sp_6}^{(2),U}(x(r))\psi(ar)dr$$ is zero for all
choice of data. Here $\theta_{Sp_6}^{(2)}$ is a vector in the space of 
$\Theta_{Sp_6}^{(2)}$. Indeed, in this case the above integral contains as an inner integration a Fourier coefficient which corresponds to the unipotent orbit $(2^21^2)$. Since  
$\Theta_{Sp_6}^{(2)}$ is a minimal representation, these Fourier coefficients are all zero. 
This means that as a function of $GL_3({\bf A})$, the constant term $\Theta_{Sp_6}^{(2),U}$
is invariant under a copy of $SL_2({\bf A})$. Thus, as a function of $GL_3({\bf A})$, the constant term $\Theta_{Sp_6}^{(2),U}$ is a one dimensional representation. 

The other two maximal parabolic subgroups are treated in the same way.

\end{proof}

The next Proposition is the $Sp_6$ version of Propositions \ref{property2} and \ref{property3}. Let $U$ denote the unipotent radical of the parabolic subgroup $P(GL_3)$.
In terms of matrices we can identify $U$ with all matrices of the form $\begin{pmatrix}
I&X\\ &I\end{pmatrix}$ where $I=I_3$ and $X\in Mat_3^0=\{ X\in Mat_3 : X=J_3 X^tJ_3\}$. 
Let $\psi_U$ be defied as 
$$\psi_U(u)=\psi_U\left (\begin{pmatrix} I_3&X\\ &I_3\end{pmatrix}\right )=\psi(x_{3,1})$$ and denote
$$\theta_{Sp_6}^{(2),U,\psi}(g)=\int\limits_{U(F)\backslash U({\bf A})}
\theta_{Sp_6}^{(2)}(ug)\psi_U(u)du$$
If we embed the group $GL_3$ inside $Sp_6$ as $g\mapsto \text{diag}(g,g^*)$, then the stabilizer of $\psi_U$ inside $GL_3$ is the group of all matrices of the form 
$$L_0(GL_3)=\left \{\begin{pmatrix} h&y\\ &1\end{pmatrix}\ \ \ h\in GL_2, \ \ y\in Mat_{2\times 1}\right \}$$ Let $L(GL_3)$ denote the maximal parabolic subgroup of $GL_3$ which contains $L_0(GL_3)$. Finally, let $L^0(GL_3)$ denote the subgroup of $L_0(GL_3)$
such that $h\in SL_2$. With these notations we prove
\begin{proposition}\label{propertysp2}
We have the following expansion,
\begin{equation}\label{minexp1}
\theta_{Sp_6}^{(2)}(g)=\theta_{Sp_6}^{(2),U}(g)+\sum_{\gamma\in L(GL_3)(F)\backslash GL_3(F)} \sum_{\epsilon\in \{\pm 1\}\backslash F^*}\theta_{Sp_6}^{(2),U,\psi}(h(\epsilon)\gamma g)
\end{equation}
Here $h(\epsilon)=\text{diag}(I_2,\epsilon,\epsilon^{-1},I_2)$. Moreover we have
\begin{equation}\label{minexp2}
\theta_{Sp_6}^{(2),U,\psi}(qg)=\theta_{Sp_6}^{(2),U,\psi}(g)
\end{equation}
for all $q\in L^0(GL_3)({\bf A})$.
\end{proposition}
\begin{proof}
The proof is similar to the proof of Propositions \ref{property2} and \ref{property3}.
Notice that $U$ is an abelian group. Therefore, we can expand $\theta_{Sp_6}^{(2)}(g)$
along $U(F)\backslash U({\bf A})$. The group $GL_3(F)$ acts on the character group of
$U(F)\backslash U({\bf A})$, and all characters except the trivial one and any character
that is in the same orbit of $\psi_U$, contribute zero to the expansion. This follows
from the fact that any other character produces a Fourier coefficient which is associated
with a unipotent orbit which is greater than $(21^4)$. From this, identity \eqref{minexp1}
follows.

As for identity \eqref{minexp2}, it follows from similar arguments. Indeed, let 
$N=\{ I_6+r_1(e_{1,2}-e_{5,6})+r_2(e_{1,3}-e_{4,6})\}$. Expanding $\theta_{Sp_6}^{(2),U,\psi}(g)$ along $N(F)\backslash N({\bf A})$, it follows from the fact that $\Theta_{Sp_6}^{(2)}$ is a minimal representation, that nontrivial characters of $N(F)\backslash N({\bf A})$ contributes zero to the expansion. Thus $\theta_{Sp_6}^{(2),U,\psi}(g)=
\theta_{Sp_6}^{(2),UN,\psi}(g)$. Since $L^0(GL_3)({\bf A})$ is generated by $N({\bf A})$
and the Weyl element $\text{diag}(J_2,I_2,J_2)$, identity \eqref{minexp2} follows.

\end{proof}

Finally, we prove the analogous to Proposition \ref{property6}. To do that, let $Z$ denote the unipotent subgroup  defined by $Z=\{x(r)= I_6+re_{1,6}\}$. For $\beta\in F^*$, denote
$$ \theta_{Sp_6}^{(2),Z,\psi_\beta}(g)=\int\limits_{F\backslash {\bf A}}
\theta_{Sp_6}^{(2)}(x(r)g)\psi(\beta r)dr$$
Let $U$ denote the unipotent radical of the maximal parabolic subgroup 
$P(GL_1\times Sp_4)$. Then $U$ can be identified with the Heisenberg group ${\mathcal H}_5$.
As in Proposition \ref{property6} we have
\begin{proposition}\label{propertysp3}
With the above notations, the space of functions
$$\theta_{Sp_{4}}^{\phi,\psi_\beta}(ug)$$ is a dense subspace
in the space of functions $\theta_{Sp_6}^{(2),Z,\psi_\beta}(ug)$. Here $g\in
\widetilde{Sp}_4({\bf A})$, $u\in U({\bf A})$ and
$\theta_{Sp_{4}}^{\phi,\psi_\beta}\in
\Theta_{Sp_{4}}^{\phi,\psi_\beta}$ is the theta representation of
${\mathcal H}_{5}({\bf A})\cdot \widetilde{Sp}_{4}({\bf A})$
attached to the character $\psi_\beta$. 
\end{proposition}

\section{\bf  Commuting Pairs in $F_4$}

Let $(H,G)$ be a commuting pair in the group $F_4$. By that we mean
that the two groups commute one with the other, but they need not be
a dual pair. Let $\mathcal{E}$ denote an automorphic representation
of the group $F_4({\bf A})$. Let $\pi$ denote an irreducible
cuspidal representation of $H({\bf A})$, and let
\begin{equation}\label{com1}
f(g)=\int\limits_{H(F)\backslash H({\bf A})}\varphi_\pi(h)E((h,g))dh
\end{equation}
Here $E$ is a vector in the space of $\mathcal{E}$, and
$\varphi_\pi$ is a vector in the space of $\pi$. Denote by
$\sigma(\pi,\mathcal{E})$ the automorphic representation of $G({\bf
A})$ generated by all the functions $f(g)$  defined above.

As explained in the introduction we are looking for those cases
which satisfy equation \eqref{intro4}. In this case, since $V$ is
trivial, equation \eqref{intro4} is given by
\begin{equation}\label{com2}
\text{dim}\ \pi+ \text{dim}\  \mathcal{E}=\text{dim}\ H+ \text{dim}
\ \sigma(\pi,\mathcal{E})
\end{equation}

We will consider the following commuting pairs:\\
{\bf 1)}\ $(H,G)=(SL_3,SL_3)$.\\
{\bf 2)}\ $(H,G)=(SL_2\times SL_2, Sp_4)$.\\
{\bf 3)}\ $(H,G)=(SL_2, SL_4)$.\\
{\bf 4)}\ $(H,G)=(SO_3, G_2)$.\\
{\bf 5)}\ $(H,G)=(SL_2, Sp_6)$.

The way these groups are embedded inside $F_4$ will be discussed
below. In each of the above cases we check the conditions such that
equation \eqref{com2} holds. Notice that in integral \eqref{com1},
there is a symmetry between $H$ and $G$. In other words, given an
irreducible cuspidal representation $\sigma$ of the group $G({\bf
A})$, we can consider the representation of $H({\bf A})$ generated
by the space of functions
\begin{equation}\label{com3}
\int\limits_{G(F)\backslash G({\bf A})}\varphi_\sigma(h)E((h,g))dg
\end{equation}
The corresponding equation for this case is
\begin{equation}\label{com4}
\text{dim}\ \sigma+ \text{dim}\  \mathcal{E}=\text{dim}\ G+
\text{dim} \ \pi(\sigma,\mathcal{E})
\end{equation}
Thus, in each of the above cases we should check both options. The
representation $\mathcal{E}$ is defined on $F_4$, and hence its
dimension should be a half of the dimension of some unipotent orbit
of $F_4$. For a list of the unipotent orbits, and their dimensions, 
we refer the reader to \cite{C-M} page 128.
It follows from that list that the minimal representation, the one
constructed in the previous Section, is of dimension 8. The one
above it is of dimension 11, and so on.  We have

{\bf 1)}\ $(H,G)=(SL_3,SL_3)$. Since $\pi$ is cuspidal, then it is
generic, and hence $\text{dim}\ \pi=3$. We have $\text{dim}\
SL_3=8$. Hence, equation \eqref{com2} is $\text{dim}\
\mathcal{E}-\text{dim}\ \sigma(\pi,\mathcal{E})= 5$. Since
$\sigma(\pi,\mathcal{E})$ is an automorphic representation of
$SL_3$, its dimension is at most 3, and hence the only option is
that $\text{dim}\ \mathcal{E}=8$ and $\text{dim}\
\sigma(\pi,\mathcal{E})=3$.

{\bf 2)}\ $(H,G)=(SL_2\times SL_2, Sp_4)$. Here $\text{dim}\ H=6$,
and $\text{dim}\ \pi=2$.  Hence we have $\text{dim}\
\mathcal{E}-\text{dim}\ \sigma(\pi,\mathcal{E})= 4$. The
representation $\sigma(\pi,\mathcal{E})$ is an automorphic
representation of $Sp_4$, hence its dimension is 2,3 or 4. Thus, the
only option is $\text{dim}\ \mathcal{E}=8$ and $\text{dim}\
\sigma(\pi,\mathcal{E})=4$. Thus we expect $\sigma(\pi,\mathcal{E})$
to be generic.

To consider the options for integral \eqref{com3} we notice that
$\text{dim}\ G=\text{dim}\ Sp_4=10$, and  since
$\pi(\sigma,\mathcal{E})$ is an automorphic representation on
$SL_2({\bf A})\times SL_2({\bf A})$, then $\text{dim}\
\pi(\sigma,\mathcal{E})=1,2$. Thus, we have two options, first
$12=\text{dim}\ \mathcal{E}+\text{dim}\ \sigma$ and the second is
$11=\text{dim}\ \mathcal{E}+\text{dim}\ \sigma$. The representation
$\sigma$ is a cuspidal representation on $Sp_4$, and hence its
dimension is at most 4. Thus in both cases we have $\text{dim}\
\mathcal{E}=8$. In the first case we get $\text{dim}\ \sigma=4$ and
in the second $\text{dim}\ \sigma=3$.

{\bf 3)}\ $(H,G)=(SL_2, SL_4)$. Since $\text{dim}\ H=3$ and
$\text{dim}\ \pi=1$, we obtain $\text{dim}\ \mathcal{E}-\text{dim}\
\sigma(\pi,\mathcal{E})= 2$. Thus, the only option is $\text{dim}\
\mathcal{E}=8$ and $\text{dim}\ \sigma(\pi,\mathcal{E})=6$.

In the other direction, we have $\text{dim}\ G=\text{dim}\ SL_4=15$.
Also, since $\sigma$ is cuspidal, it must be generic, and hence
$\text{dim}\ \sigma=6$. The group $H=SL_2$, and hence $\text{dim}\
\pi(\sigma,\mathcal{E})=1$. Thus we obtain $15+1= \text{dim}\
\mathcal{E}+6$, or $\text{dim}\ \mathcal{E}=10$. From \cite{C-M} it
follows that there is no unipotent orbit whose dimension is 20, and
hence we dont expect a representation of $F_4$ whose dimension is
10.

{\bf 4)}\ $(H,G)=(SO_3, G_2)$. As in the previous case we obtain
$\text{dim}\ \mathcal{E}-\text{dim}\ \sigma(\pi,\mathcal{E})= 2$.
Thus, the only option is $\text{dim}\ \mathcal{E}=8$ and
$\text{dim}\ \sigma(\pi,\mathcal{E})=6$. Hence, we expect the image
of this lift to be a generic representation of $G_2$.

In the other direction we have $\text{dim}\ G=\text{dim}\ G_2=14$.
Since $H=SO_3$, then $\text{dim}\ \pi(\sigma,\mathcal{E})=1$. Also,
$\sigma$ is a cuspidal representation of $G_2$, and hence
$\text{dim}\ \sigma=5,6$. This implies that $14+1=\text{dim}\
\mathcal{E}+ \text{dim}\ \sigma$, and hence $\text{dim}\
\mathcal{E}= 9,10$. By \cite{C-M} we dont expect a representation
which such dimensions.

{\bf 5)}\ $(H,G)=(SL_2, Sp_6)$. Here $H$ is of the same type as the
previous two cases, and hence we get the identity $\text{dim}\
\mathcal{E}-\text{dim}\ \sigma(\pi,\mathcal{E})= 2$. Since
$\sigma(\pi,\mathcal{E})$ is a representation of $Sp_6$, its
dimension is at most 9. Thus $\text{dim}\ \mathcal{E}$ is at most
11, and there are two cases. First, when $\text{dim}\
\mathcal{E}=11$ and then $\text{dim}\ \sigma(\pi,\mathcal{E})=9$. In
this case $\sigma(\pi,\mathcal{E})$ is a generic representation. The
second case is when $\text{dim}\ \mathcal{E}=8$ and $\text{dim}\
\sigma(\pi,\mathcal{E})=6$.

In the other direction, since $\text{dim}\ G=\text{dim}\ Sp_6=21$,
and $H=SL_2$ then $\text{dim}\ \pi(\sigma,\mathcal{E})=1$,  and
hence $\text{dim}\ \mathcal{E}+ \text{dim}\ \sigma= 22$. The
representation $\sigma$ is cuspidal, and hence $\text{dim}\
\sigma=6, 8,9$. From this we obtain that $\text{dim}\ \mathcal{E}=
15, 14, 13$. From \cite{C-M} we deduce that the last case is
impossible, but it is possible that $\text{dim}\ \mathcal{E}= 15,
14$.

As can be seen from the above in all cases, except case number {\bf
5)}, the only representation $\mathcal{E}$ of $F_4$ which satisfies
the dimension equations \eqref{com2} or \eqref{com4} is the minimal
representation $\Theta$. In the following subsections we shall
consider the above cases. In each case we will determine when the
image of the lift is cuspidal and when it is nonzero. We will
consider both liftings given by integrals \eqref{com1} and
\eqref{com3} even though the dimension formula may not work in both
directions. We do that since studying the other direction as well
may give us some information of how to characterize the image of the
lift. In this paper we only consider  the case when
$\mathcal{E}=\Theta$, the minimal representation of the double cover
of $F_4$. This implies that some of the representations are defined
on the double cover of $H$ or $G$.

\subsection{ \bf The Commuting  Pair $(SL_3, SL_3)$}

In this subsection we will study the lifting from the double cover
of $GL_3$ to the linear group $SL_3$, and the lift from $GL_3$ to
the double cover of $SL_3$. We shall denote by $\widetilde{SL}_3$
the double cover of $SL_3$, and similarly for $GL_3$.

\subsubsection{\bf From $\widetilde{GL}_3$ to $SL_3$}

To construct this lifting, we first embed the commuting pair $(SL_3,
SL_3)$ inside $F_4$ as follows. The first copy of $SL_3$ is generated by
$<x_{\pm(1000)}(r_1)$, $ x_{\pm(0100)}(r_2),x_{\pm(1100)}(r_3)>$ and
the other copy is generated by $<x_{\pm(0001)}(r_1),
x_{\pm(1231)}(r_2)$, $x_{\pm(1232)}(r_3)>$. Notice that the first copy
is generated by unipotent elements which corresponds to long roots,
and the second copy by unipotent elements corresponding to short
roots. This means that the first copy of $SL_3$, when embedded as
above inside $F_4$, does not split under the covering, but the
second copy does.

Let $\widetilde{\pi}$ denote a cuspidal representation of the
group $\widetilde{GL}_3({\bf A})$. We consider the integral
\begin{equation}\label{sl31}
f(h)=\int\limits_{SL_3(F)\backslash SL_3({\bf
A})}\widetilde{\varphi}(g)\theta((h,g))dg
\end{equation}
Here $\widetilde{\varphi}$ is a vector in the space of
$\widetilde{\pi}$ and $(h,g)\in (SL_3({\bf A}), SL_3({\bf A}))$
embedded in $\widetilde{F}_4({\bf A})$ as above. In other words, the
first copy of $SL_3$ is the one which is generated by short roots in
$F_4$, and the second copy is equal to $<x_{\pm 1000}(r),
x_{\pm 0100}(r)>$. The function $f(h)$ defines an automorphic function
of $SL_3({\bf A})$. As we vary the data in integral \eqref{sl31}, we
obtain an automorphic representation of $SL_3({\bf A})$ which we
shall denote by $\sigma(\widetilde{\pi})$. Our first result is
\begin{proposition}\label{sl31.1}
The representation $\sigma(\widetilde{\pi})$ is a nonzero cuspidal
representation of $SL_3({\bf A})$.
\end{proposition}
\begin{proof}
To prove cuspidality, we have to show that the integrals
$$I=\int\limits_{V(F)\backslash V({\bf A})}f(vh)dv$$ is zero for all
choice of data, where $V$ is any unipotent radical of a maximal
parabolic  subgroup of $SL_3$. Up to conjugation there are two such
unipotent radicals. They are given by
$V_1=\{x_{0001}(r_1)x_{1232}(r_2)\}$ and
$V_2=\{x_{1231}(r_1)x_{1232}(r_2)\}$. It is easy to see that the
Weyl element $w[321323]$ conjugates $V_1$ to $V_2$ and fixes the
group $SL_3=<x_{\pm(1000)}(r_1), x_{\pm(0100)}(r_2)>$. Hence, to prove the
cuspidality of $\sigma(\widetilde{\pi})$, it is enough to show that
the constant term of $f(h)$ along $V=V_2$, is zero for all choice of
data.

Let $U_1$ denote the unipotent subgroup of $F_4$ generated by all
$<x_\alpha(r)>$ where $\alpha\in\{ 0122; 1122; 1222; 1242; 1342;
2342\}$. Let $U_2=<U_1, x_{1232}(r)>$. We expand $I$ along the group
$U_1(F)\backslash U_1({\bf A})$. The group $Spin_6(F)$ generated by
$<x_{\pm(1000)}(r); x_{\pm(0100)}(r);x_{\pm(0120)}(r)>$ acts on this
expansion with three type of orbits. The first type of orbit
correspond to the set of all vectors in $F^6$ which have nonzero
length. Combining the integration over $U_1(F)\backslash U_1({\bf
A})$ with the integration over $x_{1232}(r)$ we obtain the integral
$$\int\limits_{U_2(F)\backslash U_2({\bf
A})}\theta(u_2m)\psi(\gamma\cdot u_2)du_2$$ as an inner integration to the expansion.
Here $\gamma\in F^7$ is a vector with a nonzero length. However,
this Fourier coefficient corresponds to the unipotent orbit
$\widetilde{A}_1$. By the minimality of $\Theta$ it is zero. Hence
we are left with the two orbits which corresponds to the zero vector
and to all nonzero vectors with zero length. Thus $I$ is equal to
$$\int\limits_{SL_3(F)\backslash SL_3({\bf
A})}\int\limits_{F\backslash {\bf A}}\int\limits_{U_2(F)\backslash
U_2({\bf
A})}\widetilde{\varphi}(g)\theta(u_2(x_{1231}(r_1),g))du_2dr_1dg+$$
$$\int\limits_{SL_3(F)\backslash SL_3({\bf
A})}\widetilde{\varphi}(g)\int\limits_{F\backslash {\bf
A}}\int\limits_{U_2(F)\backslash U_2({\bf A})} \sum_{\gamma\in
S(F)\backslash Spin_6(F)} \theta( u_2\gamma
(x_{1231}(r_1),g))\psi_{U_2}(u_2)du_2dr_1dg$$ where $\psi_{U_2}$ is
defined as follows. If $u_2=x_{0122}(r_1)u_2'$, then define
$\psi_{U_2}(u_2)=\psi(r_1)$. ( See subsection 2.1 for notations). Also, the group $S$ is the stabilizer of $\psi_{U_2}$ inside $Spin_6$. Thus
$$S=<x_{\pm(0100)}(r);x_{\pm(0120)}(r);x_{1000}(r);x_{1100}(r)
;x_{1120}(r);x_{1220}(r)>$$ Denote the first summand by $I'$ and the
second one by $I''$.

We start with $I''$. Let $L$ denote the maximal parabolic subgroup
of $Spin_6$ which contains the copy of $SL_3$ generated by
$<x_{\pm(1000)}(r_1), x_{\pm(0100)}(r_2)>$. The space
$S(F)\backslash Spin_6(F)/L(F)$ contains two representatives which
can be chosen to be $e$ and $w[1323]$. Thus, $I''$ is equal to
$$\int\limits_{S(2)(F)\backslash SL_3({\bf
A})}\int\limits_{F\backslash {\bf A}}\int\limits_{U_2(F)\backslash
U_2({\bf
A})}\widetilde{\varphi}(g)\theta(u_2(x_{1231}(r_1),g))\psi_{U_2}(u_2)du_2dr_1dg+$$
$$\int\limits_{S(1)(F)\backslash SL_3({\bf
A})}\widetilde{\varphi}(g)\int \sum_{\delta_i\in F} \theta(
u_2w[123]x_{0120}(\delta_1)x_{1120}(\delta_2)
(x_{1231}(r_1),g))\psi_{U_2}(u_2)du_2dr_1dg$$ where we used the left
invariant of $\theta$ under rational points to replace the Weyl
element $w[1323]$ by $w[123]$. Here, the group $S(1)$ denotes the
maximal parabolic subgroup of $SL_3$ which contains the group
$\{x_{\pm 1000}\}$. Similarly we define $S(2)$.  Also, in the second
summand, the variables $r_1$ and $u_2$ are integrated as in the
first summand. Denote the first summand by $I''_1$ and the second
one by $I''_2$. We start with $I''_1$. Expand it along the group
$U/Z$ with points in $F\backslash {\bf A}$. Here $U=U_{\alpha_2,\alpha_3,\alpha_4}$ is the
unipotent radical of the maximal parabolic subgroup of $F_4$ whose
Levi part is $GSp_6$, and $Z=\{ x_{2342}(m)\}$ is its center. Using
Proposition \ref{property2}, this expansion contains two summands.
The constant term in the expansion of $I''_1$ contributes zero to
the integral. Indeed, it is equal to
$$\int\limits_{S(2)(F)\backslash SL_3({\bf
A})}\int\limits_{F\backslash {\bf A}}\int\limits_{Z({\bf
A})U_2(F)\backslash U_2({\bf
A})}\widetilde{\varphi}(g)\theta^U(u_2(x_{1231}(r_1),g))\psi_{U_2}(u_2)du_2dr_1dg$$
The unipotent radical of $S(2)$ is the unipotent group $L=\{
x_{1000}(m_1)x_{1100}(m_2)\}$. Notice that $L$ is a subgroup of $U$.
Hence, as a function of $g$, the integral
$$\int\limits_{F\backslash {\bf A}}\int\limits_{Z({\bf
A})U_2(F)\backslash U_2({\bf
A})}\theta^U(u_2(x_{1231}(r_1),g))\psi_{U_2}(u_2)du_2dr_1$$ is left
invariant under $l\in L({\bf A})$. Hence, we get the integral
$\int\limits_{L(F)\backslash L({\bf A})}\widetilde{\varphi}(lg)dl$
as inner integration. This integral is zero by the cuspidality of
$\widetilde{\pi}$.

Thus $I''_1$ is equal to
$$\int\limits_{S(2)(F)\backslash SL_3({\bf
A})}\widetilde{\varphi}(g)\int\sum_{\gamma\in Q(F)\backslash
Sp_6(F)}\sum_{\epsilon\in F^*} \theta^{U,\psi}(h_2(\epsilon)\gamma
u_2(x_{1231}(r_1),g)) \psi_{U_2}(u_2)du_2dr_1dg$$ where $r_1$ is
integrated as before and $u_2$ is integrated over ${Z({\bf A})
U_2(F)\backslash U_2({\bf A})}$. Let $P$ denote the maximal
parabolic subgroup of $Sp_6$ whose Levi part contains $Sp_4$. The
space $Q(F)\backslash Sp_6(F)/P(F)$ consists of two elements and as
representatives we choose $e$ and $w[234]$. Hence, $I''_1$ is
equal to
$$\int\limits_{S(2)(F)\backslash SL_3({\bf
A})}\widetilde{\varphi}(g)\int\sum_{\gamma\in S(3)(F)\backslash
Sp_4(F)}\sum_{\epsilon\in F^*} \theta^{U,\psi}(h_2(\epsilon)\gamma
u_2(x_{1231}(r_1),g)) \psi_{U_2}(u_2)du_2dr_1dg+$$
$$\int\widetilde{\varphi}(g)\sum_{\gamma\in S(3)(F)\backslash
Sp_4(F)}\sum_{\delta_i\in F;\ \epsilon\in F^*}
\theta^{U,\psi}(h_2(\epsilon)w[234]y(\delta_1,\delta_2,\delta_3)\gamma
u_2(x_{1231}(r_1),g)) \psi_{U_2}(u_2)du_2dr_1dg$$ where all
variables in the second summand are integrated as in the first
summand. Also, we have
$y(\delta_1,\delta_2,\delta_3)=x_{0001}(\delta_1)x_{0011}
(\delta_2)x_{0122}(\delta_3)$. Notice that $x_{0122}(r)$ commutes
with $\gamma\in Sp_4$ and that this group actually normalizes the
group $U_2$ . Hence, in the first summand, we can conjugate this
unipotent element to the left, and using Proposition
\ref{property3}, we deduce that
$$g\mapsto \theta^{U,\psi}(h_2(\epsilon)\gamma
x_{0122}(r)u_2(x_{1231}(r_1),g))$$ is left invariant by
$x_{0122}(r)$ for all $r\in {\bf A}$. Since $\psi_{U_2}$ is
nontrivial on $x_{0122}(r)$, the first summand is zero. In the
second summand, after conjugating $u_2$ across $\gamma$, we
conjugate the unipotent element $x_{1122}(r)$ to the left. We have
$$h_2(\epsilon)w[234]y(\delta_1,\delta_2,\delta_3)x_{1122}(r)=
x_{1000}(\epsilon^{-1}r)h_2(\epsilon)w[234]y(\delta_1,\delta_2,\delta_3)$$
Changing variables, we obtain $\int\limits_{F\backslash {\bf
A}}\psi(\epsilon^{-1}r)dr$ as inner integration. This integral is
clearly zero, and hence  $I''_1=0$.

Next we consider $I''_2$. Expanding along $U/Z$, using Proposition
\ref{property2}, the nontrivial orbit contributes
$$\int\limits_{S(1)(F)\backslash SL_3({\bf
A})}\int\limits_{F\backslash {\bf A}}\int\limits_{Z({\bf A})
U_2(F)\backslash U_2({\bf A})}\widetilde{\varphi}(g)\times$$
$$ \sum_{\gamma\in
Q(F)\backslash Sp_6(F)}\sum_{\delta_i\in F,\epsilon\in F^*}
\theta^{U,\psi}(h_2(\epsilon)\gamma
u_2w[123]x_{0120}(\delta_1)x_{1120}(\delta_2)
(x_{1231}(r_1),g))\psi_{U_2}(u_2)du_2dr_1dg$$ We consider the space
$Q(F)\backslash Sp_6(F)/P(F)$. Arguing as in the computation of $I''_1$ we
obtain that this integral is zero. Thus we are left with the
contribution from the constant term
$$\int\limits_{S(1)(F)\backslash SL_3({\bf
A})}\int\limits_{(F\backslash {\bf
A})^2}\widetilde{\varphi}(g)\sum_{\delta_i\in F} \theta^{U}(
x_{0122}(r)w[123]x_{0120}(\delta_1)x_{1120}(\delta_2)
(x_{1231}(r_1),g))\psi(r)drdr_1dg$$ Conjugate $x_{1231}(r_1)$ to the
left. We obtain the integral
\begin{equation}\label{sl35}
\int\limits_{(F\backslash {\bf A})^2}\sum_{\delta_i\in F}
\theta^{U}(
x_{0121}(r_1)x_{0122}(r)w[123]x_{0120}(\delta_1)x_{1120}(\delta_2)
(1,g))\psi(r)drdr_1
\end{equation}
as inner integration. Expand this integral along the unipotent
element $x_{0120}(r_2)$. We claim that the nontrivial coefficients
contribute zero to the integral. Indeed, to show that, it is enough
to prove that the integral
$$\int\limits_{(F\backslash {\bf A})^3}\theta^{U}(
x_{0120}(r_2)x_{0121}(r_1)x_{0122}(r))\psi(\beta r_2+r)dr_1dr_2dr$$
is zero for all $\beta\in F^*$. It follows from Proposition
\ref{eisen1}, that this integral is zero if the integral
$$\int\limits_{(F\backslash {\bf A})^3}\theta_6
\left [\begin{pmatrix} 1&&&r_1&r\\ &1&&r_2&r_1\\ &&I_2&&\\ &&&1&\\
&&&&1\end{pmatrix}\right ]\psi(\beta r_2+r)dr_1dr_2dr$$ is zero for
all choice of data. Here, $\theta_6$ is a vector in the space of $\Theta_6$. This representation
was introduced right before Proposition \ref{eisen2}, and it follows from Proposition \ref{eisen1}
that it is a minimal representation for $\widetilde{Sp}_6({\bf A})$. It follows from \cite{G1}
that the above Fourier coefficient is associated with the unipotent orbit $(2^21^2)$. Hence it is
zero for all choice of data.

Thus \eqref{sl35} is equal to
\begin{equation}
\int\limits_{(F\backslash {\bf A})^3}\sum_{\delta_i\in F}
\theta^{U}(x_{0120}(r_2)
x_{0121}(r_1)x_{0122}(r)w[123]x_{0120}(\delta_1)x_{1120}(\delta_2)
(1,g))\psi(r)drdr_1dr_2\notag
\end{equation}
Using commutation relations and Proposition \ref{property3}, one can
check that as a function of $g$, this integral is left invariant
under $x_{0100}(m_1)x_{1100}(m_2)$ for all $m_i\in {\bf A}$. Thus we
get zero by the cuspidality of $\widetilde{\pi}$. From this we
deduce that $I''=0$.

Next we consider $I'$. Expand the integral along $ U(B_3)/U_2$ with
points in $F\backslash {\bf A}$. Here $U(B_3)=U_{\alpha_1,\alpha_2,\alpha_3}$ is the unipotent
radical of the maximal parabolic subgroup of $F_4$ whose Levi part
is $GSpin_7$. If $x_{\alpha}(r)\in U(B_3)$ but not in $U_2$ then
$\alpha$ is a short root. This means that if we consider a nonzero
Fourier coefficient in this expansion, we get as inner integration,
the Fourier coefficient which corresponds to the unipotent orbit
$\widetilde{A_1}$. This Fourier coefficient is zero by the
minimality of $\Theta$. Thus we are left with the constant term.
That is, $I'$ is equal to
$$\int\limits_{SL_3(F)\backslash SL_3({\bf
A})}\widetilde{\varphi}(g)\theta^{U(B_3)}((1,g))dg$$  Let $L_1$
denote the unipotent subgroup of $Spin_7$ generated by
$<x_{0120}(r);x_{1120}(r);x_{1220}(r)>$. We expand the above
integral along the group $L_1(F)\backslash L_1({\bf A})$. The group
$SL_3(F)$, embedded as above, acts on this expansion with two
orbits. Thus $I'$ is equal to
\begin{equation}\label{sl36}
\int\limits_{SL_3(F)\backslash SL_3({\bf
A})}\widetilde{\varphi}(g)\theta^{U(B_3)L_1}((1,g))dg+
\int\limits_{S(2)(F)\backslash SL_3({\bf
A})}\widetilde{\varphi}(g)\theta^{U(B_3)L_1,\psi}((1,g))dg
\end{equation}
where
$$\theta^{U(B_3)L_1,\psi}((1,g))=\int\limits_{(F\backslash {\bf A})^3}
\theta^{U(B_3)}(x_{0120}(r_1)x_{1120}(r_2)x_{1220}(r_3)(1,g))\psi(r_1)dr_i$$
Let $L_2$ denote the group generated by
$<L_1,x_{0010}(r);x_{0110}(r);x_{1110}(r)>$. In the first summand of
\eqref{sl36} we expand the integral along $L_2/ L_1$ with point in
$F\backslash {\bf A}$. The group $SL_3(F)$ acts on this expansion
with two orbit. The nontrivial orbit contributes the integral
$$\int\limits_{S(1)(F)\backslash SL_3({\bf
A})}\int\limits_{(F\backslash {\bf
A})^3}\widetilde{\varphi}(g)\theta^{U(B_3)L_1}
(x_{0010}(r_1)x_{0110}(r_2)x_{1110}(r_3)(1,g))\psi(r_1)dr_idg$$
Since $(0010)$ is a short root, then after a suitable conjugation,
we obtain as inner integration, a Fourier coefficient which
corresponds to the unipotent orbit $\widetilde{A}_1$. Thus we get
zero by the minimality of $\Theta$. The contribution of the constant
term is the integral
$$\int\limits_{SL_3(F)\backslash SL_3({\bf
A})}\widetilde{\varphi}(g)\theta^{U(B_3)L_2}((1,g))dg$$ To show that
it is zero, let $E(g,s)$ denote the Eisenstein  series of $GL_3({\bf
A})$ associated with the induced representation $Ind_{L({\bf
A})}^{GL_3({\bf A})}\delta_L^s$. Here $L$ is the maximal parabolic
subgroup of $GL_3$ whose Levi part is $GL_2\times GL_1$. Since the
identity is the residue of this Eisenstein series, then to prove
that the above integral is nonzero, it is enough to prove that the
integral
\begin{equation}\label{SL311}
\int\limits_{SL_3(F)\backslash SL_3({\bf
A})}\widetilde{\varphi}(g)\theta^{U(B_3)L_2}((1,g))E(g,s)dg
\end{equation}
is zero for $\text{Re}(s)$ large. Unfolding the Eisenstein series we
obtain
$$\int\limits_{S(1)(F)\backslash SL_3({\bf
A})}\widetilde{\varphi}(g)\theta^{U(B_3)L_2}((1,g))f(g,s)dg$$ Expand
along the unipotent group $\{x_{0100}(m_2)x_{1100}(m_3)\}$. Notice
that this group is the unipotent radical of $S(1)$. The group
$GL_2$, which is the Levi part of $S(1)$ acts on this unipotent
group with two orbits. The trivial one contributes zero by the
cuspidality of $\widetilde{\pi}$. Thus we obtain
$$\int\limits_{T(F)N(F)\backslash SL_3({\bf
A})}\int\limits_{(F\backslash {\bf
A})^2}\widetilde{\varphi}(g)\theta^{U(B_3)L_2}(x_{0100}(m_2)x_{1100}(m_3)
(1,g))\psi(m_2)f(g,s)dm_idg$$  Here $N$ is the maximal unipotent
subgroup of $SL_3$, and $T$ is a one dimensional torus. We further
expand along $\{x_{1000}(m_1)\}$. The trivial orbit contributes zero
by cuspidality of $\widetilde{\pi}$. The nontrivial orbit contributes the
integral
$$\int\limits_{(F\backslash {\bf
A})^3}\theta^{U(B_3)L_2}(x_{1000}(m_1)x_{0100}(m_2)x_{1100}(m_3)(1,g))
\psi(\gamma m_1+m_2)dm_i$$ as inner integration. Here $\gamma\in
F^*$. Applying Proposition \ref{eisen1} with
$R=P_{\alpha_1,\alpha_2,\alpha_4}$ this integral is zero.

As for the second summand of \eqref{sl36}, we expand the integral
along the unipotent group $\{x_{1000}(r)x_{1100}(r)\}$. The group
$GL_2(F)$ in $S(2)(F)$ acts on this expansion with two orbits. The
orbit which corresponds to the trivial character contributes zero by
the cuspidality of $\widetilde{\varphi}$. The nontrivial orbit
contributes zero using Proposition \ref{eisen1} with
$R=P_{\alpha_1,\alpha_2,\alpha_3}$. Thus $I'=0$. This completes the
proof of the cuspidality of the lift.

To show that the lift is always nonzero, we shall compute the
Whittaker model of the lift. In other words, we shall compute the
integral
$$W_f(h)=\int\limits_{(F\backslash {\bf A})^3}f(x_{0001}(r_1)x_{1231}(r_2)
x_{1232}(r_3)h)\psi(r_1+r_2)dr_i$$ We shall denote this unipotent
group by $V$, and the above character by $\psi_V$. Thus we need to
compute the integral
$$\int\limits_{SL_3(F)\backslash SL_3({\bf
A})}\int\limits_{V(F)\backslash V({\bf A})}\widetilde{\varphi}(g)
\theta((vh,g))\psi_V(v)dvdg$$ Following the same expansions as in
the proof of the cuspidality, we obtain that all terms contribute
zero except the integral
$$\int\limits_{S(1)(F)\backslash SL_3({\bf
A})}\int\limits_{V(F)\backslash V({\bf A})}\widetilde{\varphi}(g)
\sum_{\delta_i\in
F}\theta^{U_2,\psi}(w[123]x_{0120}(\delta_1)x_{1120}(\delta_2)(vh,g))\psi_V(v)dvdg$$
where
$$\theta^{U_2,\psi}(m)=\int\limits_{U_2(F)\backslash U_2({\bf A})}
\theta(u_2m)\psi_{U_2}(u_2)du_2$$ The group $U_2$ and the character
$\psi_{U_2}$ were defined in the beginning of the proof of the
Proposition.

The group $SL_2(F)$ generated by $<x_{\pm(1000)}(r)>$ acts on the
set $\{x_{0120}(\delta_1)x_{1120}(\delta_2) : \delta_i\in F\}$ with
two orbits. First, we claim that the contribution from the trivial
orbit is zero. Indeed, as explained in the proof of the cuspidality,
we have
$$\theta^{U_2,\psi}(m)=\int\limits_{F\backslash A}
\theta^{U_2,\psi}(x_{1111}(r)m)dr$$ This follows from the fact that
$(1111)$ is a short root, and if we expand the integral along the
unipotent group $\{x_{1111}(r)\}$, then by Theorem \ref{mini1}, all
the nontrivial Fourier coefficients will contribute zero. This means
that the function $h\mapsto \theta^{U_2,\psi}(w[123](h,g))$ is left
invariant by $x_{0001}(r)$ for all $r\in {\bf A}$. Since $\psi_V$ is
nontrivial on $x_{0001}(r)$ we get zero contribution. Thus we are
left with the nontrivial orbit. Hence, we obtain
$$W_f(h)=\int\limits_{N(F)\backslash SL_3({\bf
A})}\int\limits_{V(F)\backslash V({\bf A})}\widetilde{\varphi}(g)
\theta^{U_2,\psi}(w[123]x_{1120}(1)(vh,g))\psi_V(v)dvdg$$ Here $N$
is the maximal unipotent subgroup of $SL_3$.

Next, as in the proof of the cuspidality, we expand the above
integral along the group $U/Z$ with points in $F\backslash {\bf A}$.
As in the cuspidality part, the nontrivial orbit contributes zero.
Thus only the constant term contributes. Conjugating $v$ to the
left, $W_f(h)$ is equal to
\begin{equation}\label{whit1}
\int\limits_{N(F)\backslash SL_3({\bf A})}\widetilde{\varphi}(g)
\int\limits_{(F\backslash {\bf A})^3}
\theta^{U}(l(r_1,r_2,r)w[123]x_{1120}(1)(h,g))
\psi(r_1+r_2+r)dr_idrdvdg
\end{equation}
where $l(r_1,r_2,r)=x_{0111}(r_1)x_{0121}(r_2)x_{0122}(r)$. Denote
$$L(g)=\int\limits_{(F\backslash {\bf A})^3}
\theta^{U}(l(r_1,r_2,r)w[123]x_{1120}(1)(h,g))
\psi(r_1+r_2+r)drdr_i$$ Then, conjugating from left to right, and
changing variables, we obtain
$$L(x_{1000}(m_1)x_{0100}(m_2)x_{1100}(m_3)g)=$$
$$=\int\limits_{(F\backslash {\bf A})^3}
\theta^{U}(l(r_1,r_2,r)x_{0100}(m_1)x_{0120}(m_2)w[123]x_{1120}(1)(h,g))
\psi(r_1+r_2+r)dr_idr$$ From Proposition \ref{eisen1}, it follows
that the function $\theta^U(m)$, when restricted to
$\widetilde{Sp}_6$, is the the minimal representation $\Theta_6$. ( See before Proposition 
\ref{eisen2}). Consider the integral
$$\int\limits_{(F\backslash {\bf A})^3}
\theta^{U}(l(r_1,r_2,r)x_{0100}(m_1)x_{0120}(m_2))
\psi(r_1+r_2+r)drdv$$
Notice that $l(r_1,r_2,r)x_{0100}(m_1)x_{0120}(m_2)$ is in $Sp_6$. Therefore, we can use Proposition
\ref{propertysp2}. More precisely, we use the expansion \eqref{minexp1}, where to avoid confusion 
we shall write $U(GL_3)$ in expansion \eqref{minexp1} instead of $U$. The first summand in the 
expansion is the constant term along $U(GL_3)$. When plugging it into the above integral we get
zero because of the character $\psi(r)$. The second summand in \eqref{minexp1} contributes
$$\int\limits_{(F\backslash {\bf A})^3}
\sum_{\gamma\in L(GL_3)(F)\backslash GL_3(F)} \sum_{\epsilon\in \{\pm 1\}\backslash F^*}
\theta^{UU(GL_3),\psi}(h(\epsilon)\gamma l(r_1,r_2,r)m)
\psi(r_1+r_2+r)drdv$$ where we denoted $m=x_{0100}(m_1)x_{0120}(m_2)$. We also view the matrices $h(\epsilon)$ and $\gamma$ as elements in $F_4$ via of the embedding of $Sp_6$ inside $F_4$. The quotient  $L(GL_3)(F)\backslash GL_3(F)$ is the union of the three cells
\begin{equation}\label{cell}
 e;\ \ \begin{pmatrix} 1&&\\ &&1\\ &1&\end{pmatrix}\begin{pmatrix} 1&&\\ &1&\delta_1 \\ &&1
\end{pmatrix};\ \ \  \begin{pmatrix} &1&\\ &&1\\ 1&&\end{pmatrix}\begin{pmatrix} 1&\delta_1&
\delta_2\\ &1& \\ &&1 \end{pmatrix}
\end{equation}
Here $\delta_1,\delta_2\in F$. It is not hard to check that
the first two cells contribute zero. Indeed, this follows from the conjugation of $l(r_1,r_2,r)$
to the left across $h(\epsilon)\gamma$. As for the big cell, conjugating $l(r_1,r_2,r)$  to the left
we obtain $\int \psi((\epsilon^2-1)r)dr$, $\int \psi((\delta_2-1)r_1)dr_1$ and $\int \psi(
\delta_1-1)r_2)dr_2$ as inner integrations. Here all variables are integrated over $F\backslash
{\bf A}$. Hence, the above integral is equal to
$$\theta^{UU(GL_3),\psi}(w[34]x_{0001}(1)x_{0011}(1)x_{0100}(m_1)x_{0120}(m_2))=$$
$$=\psi(\frac{1}{2}(m_1+m_2))\theta^{UU(GL_3),\psi}(w[34]x_{0001}(1)x_{0011}(1))$$
where the last equality follows from the conjugation of the $m$ to the left, taking into an 
account the commutation relations in $F_4$.

Returning to integral \eqref{whit1}, factoring the integration over
$N$, we obtain
$$W_f(h)=\int\limits_{N({\bf A})\backslash SL_3({\bf
A})}W_{\widetilde{\varphi}}(g) 
\theta^{UU(GL_3),\psi}(w[34]x_{0001}(1)x_{0011}(1)w[123]x_{1120}(1)(h,g))dg$$ where $W_{\widetilde{\varphi}}(g)$ is the
Whittaker coefficient of the function $\widetilde{\varphi}(g)$.
Using a similar argument as in \cite{Ga-S}, we deduce that $W_f(h)$
is nonzero for some choice of data, if and only if
$W_{\widetilde{\varphi}}(g)$ is nonzero for some choice of data.
Thus the lift is always nonzero. This completes the proof of the
Proposition.
\end{proof}

\subsubsection{\bf From $GL_3$ to $\widetilde{SL}_3$}

For this lifting we consider the following embedding of
$(SL_3,SL_3)$ inside $F_4$. The first copy is generated by
$<x_{\pm(0001)}(r_1), x_{\pm(0010)}(r_2),x_{\pm(0011)}(r_3)>$ and
the second copy is  generated by $<x_{\pm(1000)}(r_1),
x_{\pm(1342)}(r_2),x_{\pm(2342)}(r_3)>$. As in the previous
subsection, the first copy of $SL_3$ splits under the cover of
$F_4$, and the second one does not.

Let $\pi$ denote a cuspidal representation of $GL_3({\bf A})$. We
consider the space of functions
\begin{equation}\label{2sl31}
\widetilde{f}(h)=\int\limits_{SL_3(F)\backslash SL_3({\bf
A})}{\varphi}(g)\theta((h,g))dg
\end{equation}
Here ${\varphi}$ is a vector in the space of ${\pi}$ and $(h,g)\in
(\widetilde{SL}_3({\bf A}), SL_3({\bf A}))$ embedded in
$\widetilde{F}_4({\bf A})$ as above. The function $\widetilde{f}(h)$
defines an automorphic function of $\widetilde{SL}_3({\bf A})$. As
we vary the data in integral \eqref{2sl31}, we obtain an automorphic
representation of $\widetilde{SL}_3({\bf A})$ which we shall denote
by $\widetilde{\sigma}({\pi})$. First we prove
\begin{proposition}\label{sl32.2}
The representation $\widetilde{\sigma}({\pi})$ is a cuspidal
representation of $\widetilde{SL}_3({\bf A})$.
\end{proposition}
\begin{proof}
To prove cuspidality, we have to show that the integrals
$$I=\int\limits_{V(F)\backslash V({\bf A})}\widetilde{f}(vh)dv$$
are zero for all choice of data, where $V$ is any maximal unipotent
subgroup of $SL_3$. Up to conjugation there are two such unipotent
radicals. They are given by $V_1=\{x_{1000}(r_1)x_{2342}(r_2)\}$ and
$V_2=\{x_{1342}(r_1)x_{2342}(r_2)\}$. The Weyl element $w[234232]$
conjugates $V_1$ to $V_2$ and fixes the group
$SL_3=<x_{\pm(0001)}(r_1), x_{\pm(0010)}(r_2),x_{\pm(0011)}(r_3)>$.
Hence, to prove the cuspidality of $\widetilde{\sigma}({\pi})$, it
is enough to show that the constant term of $\widetilde{f}(h)$ along
$V=V_2$, is zero for all choice of data.

Let $U=U_{\alpha_2,\alpha_3,\alpha_4}$. It center was denoted by $Z$. Thus $Z=\{x_{2342}(r)\}\subset V$. We have
$$I=\int\limits_{SL_3(F)\backslash SL_3({\bf
A})}\int\limits_{Z({\bf A})V(F)\backslash V({\bf
A})}{\varphi}(g)\theta^Z((h,g))dg$$ It follows from Proposition
\ref{property2}, that $I$ is equal to
$$\int\limits_{SL_3(F)\backslash SL_3({\bf
A})}\varphi(g)\theta^{U}(g)dg+$$  $$ \int\limits_{SL_3(F)\backslash
SL_3({\bf A})}\varphi(g)\int\limits_{Z({\bf A})V(F)\backslash V({\bf
A})} \sum_{\gamma\in Q(F)\backslash Sp_6(F)}\sum_{\epsilon\in F^*}
\theta^{U,\psi}(h_2(\epsilon)\gamma (v,g))dvdg$$ where $Q$ was
defined in Proposition \ref{property2}.  Denote the first summand by
$I'$ and the second summand by $I''$. From Proposition
\ref{eisen1}, it follows that $I'$ is equal to
$$\int\limits_{SL_3(F)\backslash SL_3({\bf
A})}\varphi(g)\theta_6(g)dg$$ Here $SL_3$ is
embedded in $Sp_6$ in the Levi part of the $GL_3$ parabolic
subgroup, and $\theta_6$ is a vector in the space of the representation $\Theta_6$. (See before
Proposition \ref{eisen2}). To the above integral we apply the expansion \eqref{minexp1} where
we write $U(GL_3)$ instead of $U$. The first term of the expansion contributes zero to $I'$. Indeed, 
it follows from Proposition \ref{propertysp1} that as a function of $GL_3({\bf A})$, the function
$\theta_6^{U(GL_3)}(g)$ is one dimensional. Hence, we obtain the integral $\int \varphi(g)dg$ as
inner integration. Here $g$ is integrated over $SL_3(F)\backslash SL_3({\bf A})$. By the cuspidality
of $\pi$ we get zero. The second term in \eqref{minexp1} contributes
$$\int\limits_{L'(GL_3)(F)\backslash SL_3({\bf
A})}\varphi(g)\theta_6^{U(GL_3),\psi}(g)dg$$ where $L'(GL_3)=L(GL_3)\cap GL_3$. Notice that 
$L'(GL_3)$ contains a unipotent radical of the group $SL_3$. Factoring this unipotent radical, and using \eqref{minexp2}, we obtain zero by the cuspidality of $\pi$. Thus $I'=0$.

To compute $I''$ we consider the double coset space $Q(F)\backslash
Sp_6(F)/Q(F)$. This space contains four representatives which we can
choose as $e,w[2],w[232],w[232432]$. For $1\le i\le 4$, we denote by
$I_i$ the contribution to $I''$ from each of the above four
representatives. We start with $I_1$. It is equal to
$$\int\limits_{SL_3(F)\backslash SL_3({\bf
A})}\varphi(g)\int\limits_{Z({\bf A})V(F)\backslash V({\bf
A})}\sum_{\epsilon\in F^*} \theta^{U,\psi}(h_2(\epsilon)
(v,g))dvdg$$ From Proposition \ref{property3} it follows that for
all $g\in SL_3({\bf A})$ we have
$$\theta^{U,\psi}(h_2(\epsilon)
(v,g))=\theta^{U,\psi}(h_2(\epsilon) (v,1))$$ Thus we obtain the
integral $\int\limits_{SL_3(F)\backslash SL_3({\bf
A})}\varphi(g)dg$ as inner integration. This is clearly zero, and
hence $I_1=0$. Next, the integral $I_2$ is equal to
$$\int\limits_{SL_3(F)\backslash SL_3({\bf
A})}\varphi(g)\int\limits_{Z({\bf A})V(F)\backslash V({\bf A})}
\sum_{\gamma\in S(4)(F)\backslash SL_3(F)}\sum_{\delta\in F,
\epsilon\in F^*}
\theta^{U,\psi}(h_2(\epsilon)w[2]x_{0100}(\delta)\gamma (v,g))dvdg$$
Here $S(4)$ is the maximal parabolic subgroup of $SL_3$ whose Levi
part is $GL_2$ which contains the group $SL_2=<\pm(0001)>$. This
integral is equal to
$$\int\limits_{S(4)(F)\backslash SL_3({\bf
A})}\varphi(g)\int\limits_{Z({\bf A})V(F)\backslash V({\bf A})}
\sum_{\delta\in F, \epsilon\in F^*}
\theta^{U,\psi}(h_2(\epsilon)w[2]x_{0100}(\delta)(v,g))dvdg$$ Let
$L=\{x_{0010}(l_1)x_{0011}(l_2)\}$ denote the unipotent radical of
$S(4)$. Conjugating $l\in L$ to the left, using Proposition
\ref{property3}, we have
$\theta^{U,\psi}(h_2(\epsilon)w[2]x_{0100}(\delta)(v,lg))=
\theta^{U,\psi}(h_2(\epsilon)w[2]x_{0100}(\delta)(v,g))$ for all
$l\in L({\bf A})$. Thus we obtain the integral
$\int\limits_{L(F)\backslash L({\bf A})}\varphi(lg)dl$ as inner
integration. By the cuspidality of $\pi$ this integral is zero.
Hence $I_2=0$. For $I_3$ we obtain
$$\int\limits_{S(3)(F)\backslash SL_3({\bf
A})}\varphi(g)\int\limits_{Z({\bf A})V(F)\backslash V({\bf A})}
\sum_{\delta_i\in F, \epsilon\in F^*}
\theta^{U,\psi}(h_2(\epsilon)w[232]x_{0100}(\delta_1)
x_{0110}(\delta_2)x_{0120}(\delta_3)(v,g))dvdg$$ where $S(3)$ is the
maximal parabolic subgroup of $SL_3$ whose Levi part contains the
group $SL_2=<x_{\pm(0010)}(r)>$. Denote by $L$ its unipotent
radical. Thus $L=\{x_{0001}(l_1)x_{0011}(l_2)\}$. Arguing as in the
case of $I_2$, we get zero by the cuspidality of $\pi$. Finally,
$I_4$ is equal to
\begin{equation}\label{2sl32}
\int\limits_{SL_3(F)\backslash SL_3({\bf
A})}\varphi(g)\int\limits_{Z({\bf A})V(F)\backslash V({\bf A})}
\sum_{\delta_i\in F}\sum_{ \epsilon\in F^*}
\theta^{U,\psi}(h_2(\epsilon)w[232432]y(\delta_1,\ldots,\delta_6)
(v,g))dvdg
\end{equation}
Here
$$y(\delta_1,\ldots,\delta_6)=x_{0100}(\delta_1)x_{0110}(\delta_2)
x_{0111}(\delta_3)x_{0120}(\delta_4)x_{0121}(\delta_5)x_{0122}(\delta_6)$$
We have $V/Z=\{x_{1342}(r)\}$. Hence, using commutation relations
$$h_2(\epsilon)w[232432]y(\delta_1,\ldots,\delta_6)x_{1342}(r)=
x_{1000}(\epsilon
r)h_2(\epsilon)w[232432]y(\delta_1,\ldots,\delta_6)$$ Changing
variables in $U$, we obtain $\int\limits_{F\backslash {\bf
A}}\psi(\epsilon r)dr$ as inner integration. Since $\epsilon\in
F^*$, this integral, and hence $I_4$, are both zero. This
completes the proof of the Proposition.
\end{proof}

\subsubsection {\bf On the Nonvanishing of the Lift}

It follows from Proposition \ref{sl31.1}, that the lift from
$\widetilde{GL}_3({\bf A})$ to $SL_3({\bf A})$ is always nonzero.
In this subsection we will determine a condition on a cuspidal
representation $\pi$ defined on $GL_3({\bf A})$ so that the lift
to a cuspidal representation of $\widetilde{SL}_3({\bf A})$ is
nonzero. In other words, we want to find a condition on $\pi$ such
that the representation $\widetilde{\sigma}({\pi})$ is nonzero.
This is equivalent to find a condition on $\pi$ such that integral
\eqref{2sl31} is nonzero for some choice of data. From Proposition
\ref{sl32.2} it follows that $\widetilde{\sigma}({\pi})$ is a
cuspidal representation. This means that
$\widetilde{\sigma}({\pi})$ is nonzero if and only if it is
generic. Thus we need to prove that there is a $\beta\in
(F^*)^3\backslash F^*$ such that the integral
$$W_{\widetilde{f},\beta}(h)=\int\limits_{(F\backslash {\bf
A})^3}\widetilde{f}(x_{1000}(r_1)x_{1342}(r_2)x_{2342}(r_3)h)\psi(\beta
r_1+r_2)dr_i $$ is not zero for some choice of data.

Let $\beta\in (F^*)^3\backslash F^*$. For $\mu_1,\mu_2,\mu_3\in
(F^*)^2\backslash F^*$ such that $\mu_1\mu_2\mu_3=\beta$, consider
the matrix
$$J(\mu_1,\mu_2,\mu_3)=\begin{pmatrix} &&\mu_1\\ &\mu_2&\\
\mu_3&& \end{pmatrix} $$ We shall denote by
$SO_3^{\mu_1,\mu_2,\mu_3}$ the orthogonal group which preserves
the form given by $J(\mu_1,\mu_2,\mu_3)$.

Our result is
\begin{proposition}\label{sl33.3}
Suppose that the representation $\widetilde{\sigma}({\pi})$ is
nonzero. Then there exists numbers $\mu_1,\mu_2,\mu_3$ and $\beta$
as above with $\mu_1\mu_2\mu_3=\beta$, such that the integral
\begin{equation}
\int\limits_{SO_3^{\mu_1,\mu_2,\mu_3}(F)\backslash
SO_3^{\mu_1,\mu_2,\mu_3}({\bf A})}\varphi(mg)dm
\end{equation}
is nonzero for some choice of data.
\end{proposition}
\begin{proof}
Let $L=\{x_{1000}(r_1)x_{1342}(r_2)x_{2342}(r_3)\}$ denote the
maximal unipotent subgroup of $SL_3$. Denote $\psi_{L,\beta}(l)=\psi(\beta
r_1+r_2)$. Thus
$$W_{\widetilde{f},\beta}(h)=\int\limits_{L(F)\backslash L({\bf
A})}\widetilde{f}(lh)\psi_{L,\beta}(l)dl$$

We begin the proof as in the proof of Proposition \ref{sl32.2}.
Arguing as in that proof, we can show that the contribution given by
$I'$ and by $I_1,I_2$ and $I_3$ are all zero. From this we deduce
that $W_{\widetilde{f},\beta}(h)$ is equal to
\begin{equation}
\int\limits_{SL_3(F)\backslash SL_3({\bf
A})}\varphi(g)\int\limits_{Z({\bf A})L(F)\backslash L({\bf A})}
\sum_{\delta_i\in F}\sum_{ \epsilon\in F^*}
\theta^{U,\psi}(h_2(\epsilon)w[232432]y(\delta_1,\ldots,\delta_6)
(l,g))\psi_{L,\beta}(l)dldg\notag
\end{equation}
where $y(\delta_1,\ldots,\delta_6)$ is defined as in
\eqref{2sl32}. Recall that $L$ contains the group
$\{x_{1342}(l_2)\}$. Since
$$h_2(\epsilon)w[232432]y(\delta_1,\ldots,\delta_6)x_{1342}(l_2)=
x_{1000}(\epsilon
l_2)h_2(\epsilon)w[232432]y(\delta_1,\ldots,\delta_6)$$ we obtain
the integral $\int\limits_{F\backslash {\bf A}}\psi((\epsilon
-1)l_2)dl_2$ as inner integration. Thus only the summand with
$\epsilon=1$ contribute to the above integral. The group $SL_3(F)$
acts on the set $y(\delta_1,\ldots,\delta_6)$ via the symmetric
square representation. As representatives for the various orbits,
we may choose the set
$x_{0100}(\mu_1)x_{0120}(\mu_2)x_{0122}(\mu_3)$ where $\mu_i\in
(F^*)^2\backslash F$. We have
$$w[232432]x_{0100}(\mu_1)x_{0120}(\mu_2)x_{0122}(\mu_3)x_{1000}(l_1)=$$
$$x_{1000}(\mu_1\mu_2\mu_3
l_1)u'w[232432]x_{0100}(\mu_1)x_{0120}(\mu_2)x_{0122}(\mu_3)$$ Here
$u'\in U$ such that $\psi_U(u')=1$. Changing variables in $U$ we
obtain $\int\limits_{F\backslash {\bf A}}\psi((\mu_1\mu_2\mu_3
-\beta)l_1)dl_1$ as inner integration. Thus $\mu_1\mu_2\mu_3=\beta$.
In particular, all $\mu_i\ne 0$. Given such $\mu_i$, the stabilizer
of $x_{0100}(\mu_1)x_{0120}(\mu_2)x_{0122}(\mu_3)$ inside $SL_3(F)$
is given by the orthogonal group $SO_3^{\mu_1,\mu_2,\mu_3}(F)$. Thus
we proved that $W_{\widetilde{f},\beta}(h)$ is equal to
$$\int\limits_{SO_3^{\mu_1,\mu_2,\mu_3}({\bf A})\backslash
SL_3({\bf A})}\sum_{\mu_i\in (F^*)^2\backslash
F^*,\mu_1\mu_2\mu_3=\beta}\varphi^{SO_3^{\mu_1,\mu_2,\mu_3}}(g)\times$$
$$\theta^{U,\psi}(w[232432]x_{0100}(\mu_1)x_{0120}(\mu_2)x_{0122}(\mu_3)(1,g)dg$$
where
$$\varphi^{SO_3^{\mu_1,\mu_2,\mu_3}}(g)=
\int\limits_{SO_3^{\mu_1,\mu_2,\mu_3}(F)\backslash
SO_3^{\mu_1,\mu_2,\mu_3}({\bf A})}\varphi(mg)dm$$

From this the Proposition follows.
\end{proof}

\subsection{\bf The Commuting pair $(SL_2\times SL_2,Sp_4)$}

Let $G=SL_2\times SL_2$. We embed this group inside $F_4$ as
$H=<x_{\pm(0100)}(r);x_{\pm(0120)}(r)>$. The embedding of the group $Sp_4$
inside $F_4$ is given by
$$Sp_4=<x_{\pm (1110)}(r);x_{\pm(0122)}(r); x_{\pm(1232)}(r);
x_{\pm(2342)}(r)>$$ It thus follows that both groups do not split under
the covering of $F_4$.

\subsubsection{\bf From $\widetilde{SL}_2\times \widetilde{SL}_2$ to
$\widetilde{Sp_4}$}

Let $\widetilde{\pi}=\widetilde{\pi_1}\otimes\widetilde{\pi_2}$
denote a cuspidal representation of the group $\widetilde{G}({\bf
A})$ where $\widetilde{\pi_i}$ are cuspidal representations of
$\widetilde{SL}_2({\bf A})$. Let
$\widetilde{\sigma}(\widetilde{\pi})$ denote the representation of
$\widetilde{Sp}_4({\bf A})$ generated by all automorphic functions
defined by
\begin{equation}
\widetilde{f}(h)=\int\limits_{G(F)\backslash G({\bf
A})}\widetilde{\varphi}(g)\theta((h,g))dg\notag
\end{equation}
Here $\widetilde{\varphi}$ is a function in the space of
$\widetilde{\pi}$. We start with
\begin{proposition}\label{sp41.1}
With the above notations, suppose that
$\widetilde{\pi_1}\ne\widetilde{\pi_2}$. Then the representation
$\widetilde{\sigma}(\widetilde{\pi})$ defines a cuspidal
representation of $\widetilde{Sp}_4({\bf A})$. Suppose further
that both cuspidal representations $\widetilde{\pi_i}$, have a
$\psi^{-\beta}$ Whittaker coefficient for some $\beta\in F^*$.
That is, suppose that for $i=1,2$
$$\int\limits_{F\backslash {\bf
A}}\widetilde{\varphi_i}\left [\begin{pmatrix} 1&x \\&1
\end{pmatrix}\right ]\psi(-\beta
x)dx$$ is not zero for some choice of functions
$\widetilde{\varphi_i}\in \widetilde{\pi}$. Then the
representation $\widetilde{\sigma}(\widetilde{\pi})$ is generic.
\end{proposition}
\begin{proof}
Let $V_1=\{x_{0122}(r_1)x_{1232}(r_2)x_{2342}(r_3)\}$ and
$V_2=\{x_{1110}(r_1)x_{1232}(r_2)x_{2342}(r_3)\}$ denote the two
unipotent radicals of the two maximal parabolic  subgroups of
$Sp_4$. We need to prove that for $i=1,2$ the integrals
$$I=\int\limits_{V_i(F)\backslash V_i({\bf A})}\widetilde{f}(vh)dv$$
are zero for all choice of data. Since both unipotent radicals
contain the group $Z$, we can use Proposition \ref{property2} to
deduce that $I$ is equal to
$$ \int\limits_{G(F)\backslash G({\bf
A})}\int\limits_{Z({\bf A})V_i(F)\backslash V_i({\bf
A})}\widetilde{\varphi}(g)\theta^U((v,g))dvdg+$$
$$\int\limits_{G(F)\backslash G({\bf
A})}\widetilde{\varphi}(g)\int\limits_{Z({\bf A})V_i(F)\backslash
V_i({\bf A})} \sum_{\gamma\in Q(F)\backslash
Sp_6(F)}\sum_{\epsilon\in F^*} \theta^{U,\psi}(h_2(\epsilon)\gamma
(v,g))dvdg$$ Denote the first summand by $I'$ and the second by
$I''$. Applying Proposition \ref{eisen1}, to prove that $I'$ is
zero, it is enough to prove that
\begin{equation}\label{cuspsl20}
 I_1'=\int\limits_{G(F)\backslash G({\bf
A})}\widetilde{\varphi}(g)\theta_6(g)dg
\end{equation}
is zero
for all choice of data. Here $\theta_6$ is a vector in the representation $\Theta_6$ which was
defined right before Proposition \ref{eisen2}. The embedding of $G$ inside $Sp_6$ is given 
by $(g_1,g_2)\mapsto \text{diag}(g_1,g_2,g_1^*)$. Here, for $i=1,2$ we have $g_i\in SL_2$. Expand
the above integral along the abelian unipotent subgroup
$$L=\left \{ \begin{pmatrix} I_2&&X\\ &I_2&\\ &&I_2\end{pmatrix};\ \ X=\begin{pmatrix} r&y\\ 
z&r\end{pmatrix}\right \}$$ Since $\Theta_6$ is a minimal representation, we obtain
\begin{equation}\label{cuspsl2}
I_1'=\int\limits_{G(F)\backslash G({\bf A})}\widetilde{\varphi}(g)\theta_6^L(g)dg+
\int\limits_{(N_1(F)\times SL_2(F)\backslash G({\bf
A})}\widetilde{\varphi}(g)\theta_6^{L,\psi}(g)dg
\end{equation}
In the first summand on the right hand side
\eqref{cuspsl2} we notice that $U(GL_2\times SL_2)/L$ is an abelian group. The group $U(GL_2\times SL_2)$
was defined at the beginning of subsection 2.6. Expanding along this quotient, it follows from
the fact that $\Theta_6$ is a minimal representation, that
$$\int\limits_{G(F)\backslash G({\bf A})}\widetilde{\varphi}(g)\theta_6^L(g)dg=\int\limits_{G(F)\backslash G({\bf A})}\widetilde{\varphi}(g)\theta_6^{U(GL_2\times SL_2)}(g)dg$$
From Proposition \ref{propertysp1}, and from the cuspidality of $\widetilde{\pi}$, it follows
that this last integral is zero. Next consider the second summand on the right hand side of
\eqref{cuspsl2}. In that term $N_1$ is the unipotent radical of $SL_2$ embedded in $Sp_6$ as
$n\mapsto \text{diag}(n,I_2,n^{-1})$, and 
$$\theta_6^{L,\psi}(g)=\int\limits_{L(F)\backslash L({\bf A})}\theta_6(lg)\psi_L(l)dl$$
Here $\psi_L(l)=\psi(z)$ where we use the identification of $L$ with the matrices $X$ as was 
described above. We claim that the function $\theta_6^{L,\psi}(g)$ is left invariant under 
$N_1({\bf A})$. Indeed, expanding along the group $N_1(F)\backslash N_1({\bf A})$ one can show
that all terms which corresponds to the nontrivial characters of the expansions, contribute
zero. This follows from the fact that $\Theta_6$ is a minimal representation. Hence 
$\theta_6^{L,\psi}(g)=\theta_6^{L,\psi}(ng)$ for all $n\in N_1({\bf A})$. Using that in the second 
summand on the right hand side of \eqref{cuspsl2}, it follows from the cuspidality of $\widetilde{\pi}$ that it is zero. Hence $I_1'=0$ which implies that $I'=0$.

To compute $I''$ we first consider the space of double cosets
$Q(F)\backslash Sp_6(F)/P(F)$ where $P$ is the maximal parabolic subgroup
of $Sp_6$ whose Levi part contains $Sp_4$. This space has two
representatives which we can choose to be $e$ and $w[234]$. Let
$I_1$ denote the contribution to $I''$ from $e$, and $I_2$ the
contribution from $w[234]$. Thus, $I_1$ is equal to
$$\int\limits_{G(F)\backslash G({\bf
A})}\widetilde{\varphi}(g)\int\limits_{Z({\bf A})V_i(F)\backslash
V_i({\bf A})} \sum_{\gamma\in S(3)(F)\backslash
Sp_4(F)}\sum_{\epsilon\in F^*} \theta^{U,\psi}(h_2(\epsilon)\gamma
(v,g))dvdg$$ Here $S(3)$ is the maximal parabolic subgroup of
$Sp_4$ whose Levi part contains the $SL_2$ generated by
$<x_{\pm(0010)}(r)>$. To proceed, we need to consider the space of double
cosets $S(3)(F)\backslash Sp_4(F)/G(F)$. This space contains two
representatives which we choose to be $e$ and $w[23]x_{0010}(1)$.
The first representative contributes to $I_1$ the term
$$\int\limits_{B_G(F)\backslash G({\bf
A})}\widetilde{\varphi}(g)\int\limits_{Z({\bf A})V_i(F)\backslash
V_i({\bf A})} \sum_{\epsilon\in F^*} \theta^{U,\psi}(h_2(\epsilon)
(v,g))dvdg$$ where $B_G$ is the Borel subgroup of $G$. Using
Proposition \ref{property3}, the function
$\theta^{U,\psi}(h_2(\epsilon) (v,ng))$ is invariant under $n\in
N_G({\bf A})$ where $N_G$ is the maximal unipotent subgroup of $G$.
Thus, we get zero by cuspidality.

As for the second representative, $w[23]x_{0010}(1)$, it
contributes to $I_1$ the term
$$\int\limits_{SL_2^\Delta(F)\backslash G({\bf
A})}\widetilde{\varphi}(g)\int\limits_{Z({\bf A})V_i(F)\backslash
V_i({\bf A})} \sum_{\epsilon\in F^*} \theta^{U,\psi}(h_2(\epsilon)
w[23]x_{0010}(1)(v,g))dvdg$$ Here $SL_2^\Delta$ is the group
$SL_2$ embedded diagonally inside the group $G$. Using Proposition
\ref{property3} we obtain $\int\limits_{SL_2^\Delta(F)\backslash
SL_2^\Delta({\bf A})}\widetilde{\varphi}(mg)dm$ as inner
integration. By our assumption that
$\widetilde{\pi_1}\ne\widetilde{\pi_2}$, this integral is zero.
Thus $I_1=0$.

Next, we compute $I_2$ which is equal to
$$\int\limits_{G(F)\backslash G({\bf
A})}\widetilde{\varphi}(g)\int\limits_{Z({\bf A})V_i(F)\backslash
V_i({\bf A})} \sum_{\gamma\in S(3)(F)\backslash
Sp_4(F)}\sum_{\delta_i\in F ,\epsilon\in F^*}
\theta^{U,\psi}(h_2(\epsilon)w[234]y(\delta_1,\delta_2,\delta_3)\gamma
(v,g))dvdg$$ where
$y(\delta_1,\delta_2,\delta_3)=x_{0001}(\delta_1)x_{0011}(\delta_2)
x_{0122}(\delta_3)$. As with $I_1$ we take $e$ and
$w[23]x_{0010}(1)$ for the two representatives of $S(3)(F)\backslash
Sp_4(F)/G(F)$. We denote by $I_{21}$ the contribution to $I_2$ from
the representative $e$, and by $I_{22}$ the contribution from
$w[23]x_{0010}(1)$. We start with $I_{22}$. Since the stabilizer is
$SL_2^\Delta$, then $I_{22}$ is equal to
$$\int\limits_{SL_2^\Delta(F)\backslash G({\bf
A})}\widetilde{\varphi}(g)\int\limits_{Z({\bf A})V_i(F)\backslash
V_i({\bf A})}\sum_{\delta_i\in F,\epsilon\in F^*}
\theta^{U,\psi}(h_2(\epsilon)w[23423]y_1(\delta_1,\delta_2,\delta_3)
x_{0010}(1)(v,g))dvdg$$ where
$y_1(\delta_1,\delta_2,\delta_3)=x_{0011}(\delta_1)x_{0121}(\delta_2)
x_{0122}(\delta_3)$. The unipotent element $x_{1232}(r)$ is in $V_i$
for $i=1,2$. We have
$$h_2(\epsilon)w[23423]y_1(\delta_1,\delta_2,\delta_3)
x_{0010}(1)x_{1232}(r)=x_{1000}(\epsilon
r)u'h_2(\epsilon)w[23423]y_1(\delta_1,\delta_2,\delta_3)
x_{0010}(1)$$ Here $u'\in U$ is such that $\psi_U(u')=1$. Using
the left invariant properties of $\theta^{U,\psi}$, we obtain
$\int\limits_{F\backslash {\bf A}}\psi(\epsilon r)dr$, which is
clearly zero. Thus $I_{22}=0$.

Finally, we need to consider $I_{21}$, which is equal to
$$\int\limits_{B_G(F)\backslash G({\bf
A})}\widetilde{\varphi}(g)\int\limits_{Z({\bf A})V_i(F)\backslash
V_i({\bf A})}\sum_{\delta_i\in F,\epsilon\in F^*}
\theta^{U,\psi}(h_2(\epsilon)w[234]y(\delta_1,\delta_2,\delta_3)
(v,g))dvdg$$ where $y(\delta_1,\delta_2,\delta_3)$ and $B_G$ were
defined above. We consider separately the cases for $V_1$ and
$V_2$.

Starting with $V_1$, we notice that $x_{0122}(r)$ is a unipotent
element in $V_1$. In the above integral, for $i=1$,  we collapse
summation and integration to obtain
$$\int\limits_{B_G(F)\backslash G({\bf
A})}\widetilde{\varphi}(g)\int\limits_{{\bf A}}\sum_{\delta_i\in
F,\epsilon\in F^*}
\theta^{U,\psi}(h_2(\epsilon)w[234]y(\delta_1,\delta_2,r)
(1,g))drdg$$ By commutation relations, change of variables, and
using Proposition \ref{property3}, we obtain
$$\int\limits_{{\bf A}}\sum_{\delta_i\in
F,\epsilon\in F^*}
\theta^{U,\psi}(h_2(\epsilon)w[234]y(\delta_1,\delta_2,r)
(1,x_{0100}(l)g))dr=$$
$$\int\limits_{{\bf A}}\sum_{\delta_i\in
F,\epsilon\in F^*}
\theta^{U,\psi}(h_2(\epsilon)w[234]y(\delta_1,\delta_2,r)
(1,g))dr$$ for all $l\in {\bf A}$. Thus we get zero by the
cuspidality of $\widetilde{\pi}$.

Next we consider the integral $I_{21}$ when $V=V_2$. This time,
the unipotent element $x_{1110}(r)$ is inside $V_2$. We have
$$h_2(\epsilon)w[234]y(\delta_1,\delta_2,\delta_3)x_{1110}(r)=
x_{1000}(\epsilon\delta_1\delta_2
r)u'h_2(\epsilon)w[234]y(\delta_1,\delta_2,\delta_3)$$ where
$u'\in U$ such that $\psi_U(u')=1$. Thus we obtain
$\int\limits_{F\backslash {\bf
A}}\psi(\epsilon^{-1}\delta_1\delta_2 r)dr$ as inner integration.
Hence $\delta_1\delta_2=0$. From this we deduce that $I_{21}$ is
equal to
$$\int\limits_{B_G(F)\backslash G({\bf
A})}\widetilde{\varphi}(g)\sum_{\delta_i\in
F,\delta_1\delta_2=0,\epsilon\in F^*}
\theta^{U,\psi}(h_2(\epsilon)w[234]y(\delta_1,\delta_2,\delta_3)
(1,g))drdg$$ If $\delta_1=0$, then for all $r\in {\bf A}$, using
Proposition \ref{property3},
$$\theta^{U,\psi}(h_2(\epsilon)w[234]y(0,\delta_2,\delta_3)
(1,x_{0120}(r)g))=\theta^{U,\psi}(h_2(\epsilon)w[234]y(0,\delta_2,\delta_3)
(1,g))$$ and if $\delta_2=0$, then for all $r\in {\bf A}$, using
again Proposition \ref{property3},
$$\theta^{U,\psi}(h_2(\epsilon)w[234]y(\delta_1,0,\delta_3)
(1,x_{0100}(r)g))=\theta^{U,\psi}(h_2(\epsilon)w[234]y(\delta_1,0,\delta_3)
(1,g))$$ Since $\{x_{0100}(r)\}$ and $\{x_{0120}(r)\}$ are the two
maximal unipotent radicals of the group $G$, it follows that
$I_{21}=0$ by the cuspidality of $\widetilde{\pi}$. This completes
the cuspidality part of the Proposition.

To prove that the image of the lift is generic, we need to compute
the integral
$$W_{\beta}(h)=\int\limits_{(F\backslash {\bf A})^4}\widetilde{f}(x_{1110}(r_1)
x_{0122}(r_2)x_{1232}(r_3)x_{2342}(r_4)h)\psi(r_1+\beta r_2)dr_i$$
Here $\beta\in (F^*)^2\backslash F^*$. Denoting the maximal
unipotent of $Sp_4$ by $V$, and the above character by
$\psi_{V,\beta}$, we have to prove that the integral
$$W_{\beta}(h)=\int\limits_{G(F)\backslash G({\bf A})}
\int\limits_{V(F)\backslash V({\bf A})}\widetilde{\varphi}(g)
\theta((v,g))\psi_{V,\beta}(v)dvdg$$ is not zero for some choice
of data. Performing the same expansions as in the proof of the
cuspidality part, we obtain that all integrals except the one that
corresponds to $I_{21}$ vanish. In other words, $W_{\beta}(h)$ is
equal to
$$\int\limits_{B_G(F)\backslash G({\bf
A})}\widetilde{\varphi}(g)\int\limits_{Z({\bf A})V(F)\backslash
V({\bf A})}\sum_{\delta_i\in F,\epsilon\in F^*}
\theta^{U,\psi}(h_2(\epsilon)w[234]y(\delta_1,\delta_2,\delta_3)
(v,g))\psi_{V,\beta}(v)dvdg$$ As in the computation of $I_{21}$
for the unipotent radical $V_1$, we collapse summation with
integration. As in the computation of $I_{21}$ for the unipotent
radical $V_2$, we conjugate $x_{1110}(r_1)$ from right to left and
we obtain that $\epsilon^{-1}\delta_1\delta_2=1$ . Thus,
$W_{\beta}(h)$ is equal to
$$\int\limits_{B_G(F)\backslash G({\bf
A})}\widetilde{\varphi}(g)\int\limits_{{\bf A}}\sum_{\delta_1
,\epsilon\in F^*}
\theta^{U,\psi}(h_2(\epsilon)w[234]y(\delta_1,\epsilon\delta_1^{-1},r)
(1,g))\psi(\beta r)drdg$$ The maximal torus of $G$ is given by
$T_G=\{h(1,a,b,1)\ :\ a,b\ne 0\}$. We have the identity
$$h_2(\epsilon)w[234]x_{0001}(\delta)x_{0011}(\epsilon\delta^{-1})
h(1,\epsilon^{-1},\delta^{-1},1)=h_2(\epsilon)h(1,\epsilon^{-1},\epsilon^{-1},\delta^{-1})
w[234]x_{0001}(1)x_{0011}(1)$$ Using this identity, we can
collapse summation and integration. Hence the above integral is
equal to
$$\int\limits_{N_G(F)\backslash G({\bf
A})}\widetilde{\varphi}(g)\int\limits_{{\bf A}}
\theta^{U,\psi}(w[234]y(1,1,r)(1,g))\psi(\beta r)drdg$$ where
$N_G$ is the maximal unipotent subgroup of $G$. In other words
$N_G=\{x_{0100}(r_1)x_{0120}(r_2)\}$. Factoring the integration
over $N_G$ we obtain the identity
$$W_{\beta}(h)=\int\limits_{N_G({\bf A})\backslash G({\bf
A})}W_{\widetilde{\varphi},\beta}(g)\int\limits_{{\bf A}}
\theta^{U,\psi}(w[234]y(1,1,r)(1,g))\psi(\beta r)drdg$$ Here
$$W_{\widetilde{\varphi},\beta}(g)=\int\limits_{F\backslash
{\bf A}}\widetilde{\varphi_1}\begin{pmatrix} 1&x \\&1
\end{pmatrix}g_1)\psi(-\beta
x)dx\int\limits_{F\backslash {\bf
A}}\widetilde{\varphi_2}\begin{pmatrix} 1&y \\&1
\end{pmatrix}g_1)\psi(-\beta y)dy$$
where
$\widetilde{\varphi}=\widetilde{\varphi_1}\otimes\widetilde{\varphi_2}$
and $g=(g_1,g_2)$.

From this it is clear that if the lift is non-zero then
$W_{\widetilde{\varphi},\beta}(g)$ is not zero. Using a similar
argument as in \cite{Ga-S}, it follows that the converse is also
true. Namely, if $W_{\widetilde{\varphi},\beta}(g)$ is not zero
then the lift to $\widetilde{Sp}_4$ is not zero.
\end{proof}

\subsubsection{\bf From $\widetilde{Sp}_4$ to
$\widetilde{SL}_2\times \widetilde{SL}_2$}

To study this lifting, we consider a different embedding of the
two groups. We embed the group $Sp_4$ as the Levi part of the
corresponding parabolic subgroup of $F_4$. In other words
$Sp_4=<x_{\pm(0100)}(r),x_{\pm(0010)}(r)>$. The group
$G=SL_2\times SL_2$ is generated by
$<x_{\pm(0122)}(r);x_{\pm(2342)}(r)>$.

Let $\widetilde{\pi}$ denote a cuspidal representation of
$\widetilde{Sp}_4({\bf A})$. We shall denote by
$\widetilde{\sigma}(\widetilde{\pi})$ the automorphic
representation of $G({\bf A})$ generated by all functions of the
form
\begin{equation}
\widetilde{f}(g)=\int\limits_{Sp_4(F)\backslash Sp_4({\bf
A})}\widetilde{\varphi}(h)\theta((g,h))dh\notag
\end{equation}
Here $\widetilde{\varphi}$ is a vector in the space of
$\widetilde{\pi}$. We start with
\begin{proposition}\label{sp42.2}
The representation $\widetilde{\sigma}(\widetilde{\pi})$ defines a
cuspidal representation of $G({\bf A})$.
\end{proposition}
\begin{proof}
Since the two unipotent radicals which correspond to the two
maximal parabolic subgroups of $G$, are conjugated one to the
other inside $F_4$, it is enough to prove that the integral
$$\int\limits_{Sp_4(F)\backslash Sp_4({\bf
A})}\int\limits_{F\backslash {\bf
A}}\widetilde{\varphi}(h)\theta((x_{2342}(r),h))drdh$$ is zero for
all choice of data. From Proposition \ref{property2}, this
integral is equal to
$$\int\limits_{Sp_4(F)\backslash Sp_4({\bf
A})}\widetilde{\varphi}(h)\theta^U((1,h))dh+
\int\limits_{Sp_4(F)\backslash Sp_4({\bf
A})}\widetilde{\varphi}(h)\sum_{\gamma\in Q(F)\backslash
Sp_6(F)}\sum_{\epsilon\in
F^*}\theta^{U,\psi}(h_2(\epsilon)(1,h))dh$$ Denote the first summand by $I_1$, and the second 
by $I'$. From Proposition
\ref{eisen1}, it follows that the first summand is zero. Indeed,
it is zero if the integral
$$\int\limits_{Sp_4(F)\backslash Sp_4({\bf
A})}\widetilde{\varphi}(h)\theta_6(h)dh$$ is zero
for all choice of data. To prove that we expand along $Z(F)\backslash Z({\bf A})$ where $Z$ was
defined right before Proposition \ref{propertysp3}. Thus, the above integral is equal to
\begin{equation}\label{cuspsl22}
\int\limits_{Sp_4(F)\backslash Sp_4({\bf A})}\widetilde{\varphi}(h)\theta_6^Z(h)dh+
\sum_{\beta\in F^*}\int\limits_{Sp_4(F)\backslash Sp_4({\bf A})}
\widetilde{\varphi}(h)\theta_6^{Z,\psi_\beta}(h)dh
\end{equation}
In the first term, we use the fact that $U(GL_1\times Sp_4)/Z$ is an abelian subgroup. See beginning
of subsection 2.6 for notations. From the fact that $\Theta_6$ is a minimal representation, we
deduce that 
$$\int\limits_{Sp_4(F)\backslash Sp_4({\bf A})}\widetilde{\varphi}(h)\theta_6^Z(h)dh=
\int\limits_{Sp_4(F)\backslash Sp_4({\bf A})}\widetilde{\varphi}(h)\theta_6^{U(GL_1\times Sp_4)}(h)dh$$ 
Arguing in a similar way as in integral \eqref{cuspsl2} we deduce that the above integral is zero for all choice of data. The notations of the second summand
of \eqref{cuspsl22} are as in Proposition \ref{propertysp3}, and it follows from that Proposition
that each term in the second summand of \eqref{cuspsl22} is zero. Indeed, from Proposition 
\ref{propertysp3} it follows that each term is equal to
$$\int\limits_{Sp_4(F)\backslash Sp_4({\bf A})}\widetilde{\varphi}(h)\theta_{Sp_4}^{\phi,\psi_\beta}
(h)dh$$ By cuspidality, this integral is zero. Thus $I_1=0$.

Next consider the integral $I'$. Let $P$ denote the maximal
parabolic subgroup of $Sp_6$ whose Levi part contains $Sp_4$. The
space $Q(F)\backslash Sp_6(F)/P(F)$ contains two elements which we can choose
to be $e$ and $w[234]$. The contribution to $I'$ from the identity
element is
$$\int\limits_{S(3)(F)\backslash Sp_4({\bf
A})}\widetilde{\varphi}(h)\sum_{\epsilon\in
F^*}\theta^{U,\psi}(h_2(\epsilon)(1,h))dh$$ where $S(3)$ is the
maximal parabolic subgroup of $Sp_4$ whose Levi part contains the
group generated by $<x_{\pm(0010)}(r)>$. Denote the unipotent
radical of $S(3)$ by $N(3)$. Then it follows from Proposition
\ref{property3} that the integral $\int\limits_{N(3)(F)\backslash
N(3)({\bf A})}\widetilde{\varphi}(nh)dn$ is an inner integration to
the above integral. By the cuspidality of $\widetilde{\pi}$ this
integral is zero.

The second representative contributes to $I'$ the integral
$$\int\limits_{S(3)(F)\backslash Sp_4({\bf
A})}\widetilde{\varphi}(h)\sum_{\delta_i\in F,\epsilon\in
F^*}\theta^{U,\psi}(h_2(\epsilon)w[234]y(\delta_1,
\delta_2,\delta_3)(1,h))dh$$ where $y(\delta_1,
\delta_2,\delta_3)=x_{0001}(\delta_1)x_{0011}(\delta_2)x_{0122}(\delta_3)$.
If $\delta_1=\delta_2=0$ then as in the previous representative,
we factor the subgroup $N(3)$ to get zero contribution. Otherwise, the
group $SL_2(F)$ which is generated by $<x_{\pm(0010)}(r)>$ acts on
the set $\{x_{0001}(\delta_1)x_{0011}(\delta_2):(\delta_1,
\delta_2)\ne (0,0)\}$ with one orbit. Thus the above integral is
equal to
$$\int\limits_{T(F)N(F)\backslash Sp_4({\bf
A})}\widetilde{\varphi}(h)\sum_{\delta_3\in F,\epsilon\in
F^*}\theta^{U,\psi}(h_2(\epsilon)w[234]y(0,1,\delta_3)(1,h))dh$$
where $N$ is the maximal unipotent subgroup of $Sp_4$ and $T$ is a
one dimensional torus. Let $S(2)$ denote the maximal parabolic subgroup of $Sp_4$ whose Levi part is
$GL_1\times SL_2$. Let $N(2)$ denote it's unipotent radical. Thus $N(2)=\{x_{0100}(r_1)
x_{0110}(r_2)x_{0120}(r_3)\}$.
Using commutation relations, it follows  from
Proposition \ref{property3} that the function
$$h\mapsto \theta^{U,\psi}(h_2(\epsilon)w[234]y(0,1,\delta_3)(1,h))$$
is left invariant under $N(2)({\bf A})$. Thus we get zero by the
cuspidality of $\widetilde{\pi}$. Hence $I'=0$ and the lift is
cuspidal.
\end{proof}

Next we consider the question of the nonvanishing of the lift. To do
that we need to find conditions so that the integral
$$W_{\widetilde{f},\beta}(g)=
\int\limits_{Sp_4(F)\backslash Sp_4({\bf
A})}\int\limits_{(F\backslash {\bf A})^2}
\widetilde{\varphi}(h)\theta((x_{0122}(r_1)x_{2342}(r_2)g,h))
\psi(\beta r_1+ r_2)dr_idh$$ will not be zero for some choice of
data. Here $\beta\in F^*$.

For $\delta\in F^*$, let $SO_4^\delta$ denote the stabilizer
insider $SO_5$ of a vector of length $\delta$. We have
\begin{proposition}\label{sp43.3}
The representation $\widetilde{\sigma}(\widetilde{\pi})$ is
nonzero if and only there exists $\beta\in F^*$ such that the
integral
\begin{equation}\label{sp41}
\int\limits_{SO_4^\beta(F)\backslash SO_4^\beta({\bf A})}
\widetilde{\varphi}(m)\theta_{Sp_4}^{\phi,\psi}(m)dm
\end{equation}
is not zero for some choice of data.
\end{proposition}
\begin{proof}
We compute $W_{\widetilde{f},\beta}(g)$. Using Proposition
\ref{property6}, the integral $W_{\widetilde{f},\beta}(g)$ is not
zero for some choice of data, if and only if the integral
$$\int\limits_{Sp_4(F)\backslash Sp_4({\bf
A})}\int\limits_{F\backslash {\bf A}}
\widetilde{\varphi}(h)\theta_{Sp_{14}}^{\phi',\psi}
(\varpi_3(x_{0122}(r_1)h))\psi(\beta r_1)dr_1dh$$ is not zero for
some choice of data. The group we integrate over is a subgroup of
$SL_2\times Sp_4$ embedded inside $Sp_6$ in the obvious way. Thus,
from the restriction of $\varpi_3$ to this subgroup it follows from
the well known factorization of the theta function,  that the above
integral is equal to
$$\int\limits_{Sp_4(F)\backslash Sp_4({\bf
A})}\int\limits_{F\backslash {\bf A}}
\widetilde{\varphi}(h)\theta_{Sp_{4}}^{\phi,\psi}(h)
\theta_{Sp_{10}}^{\phi_1,\psi}((\varpi_2(h),x(r_2))\psi(\beta
r_2)dr_2dh$$ Here $\varpi_2(g)$ is the degree five representation of
$Sp_4$. Also, by $(\varpi_2(g),x(r_2))$ we mean the embedding of
these groups inside the commuting pair $SO_5\times SL_2$ inside
$Sp_{10}$. Unfolding the theta function of $Sp_{10}$, we obtain only
one orbit, corresponding to vectors of length $\beta$. The
stabilizer is the group we denoted by $SO_4^\beta$. Thus,
$W_{\widetilde{f},\beta}(g)$ is equal to
$$\int\limits_{SO_4^\beta(F)\backslash Sp_4({\bf
A})}\widetilde{\varphi}(h)\theta_{Sp_{4}}^{\phi,\psi}(h)
\phi_1(l(\beta)h)dh$$ where $l(\beta)$ is a vector in $F^5$ whose
length is $\beta$. Factoring the measures, integral \eqref{sp41} appears as an inner integration. 
From this the Proposition follows.
\end{proof}

\subsection{\bf The Commuting Pair $(SL_2,SL_4)$}

In this subsection we will study the lifting from automorphic
representations defined on $SL_2({\bf A})$ to automorphic
representations defined on $\widetilde{SL}_4({\bf A})$,  and its
inverse map. We start with:

\subsubsection{\bf From $GL_2$ to $\widetilde{SL}_4$}

We consider the following embedding of $(SL_2,SL_4)$ inside the
group $F_4$. The group $SL_2$ is generated by $<
x_{\pm(0001)}(r)>$. The group  $SL_4$ is the group generated by
$$<x_{\pm(1000)}(r);x_{\pm(0100)}(r);x_{\pm(1242)}(r);x_{\pm(1100)}(r);
x_{\pm(1342)}(r);x_{\pm(2342)}(r)>$$ Since $SL_2$ is generated by
unipotent elements which correspond to short roots, this copy splits
under the double cover of $F_4$.

Let $\pi$ denote an irreducible cuspidal representation of
$GL_2({\bf A})$. We shall denote by $\widetilde{\sigma}(\pi)$ the
automorphic representation of $\widetilde{SL}_4$ spanned by all
automorphic functions
\begin{equation}
\widetilde{f}(h)=\int\limits_{SL_2(F)\backslash SL_2({\bf
A})}{\varphi}(g)\theta((h,g))dg\notag
\end{equation}
Here $h\in \widetilde{SL}_4({\bf A})$. We shall denote by
$L(\pi,s)$ the standard $L$ function associated with $\pi$. We
prove
\begin{proposition}\label{sl41.1}
Suppose that $\pi$ is an irreducible cuspidal representation of
${GL}_2({\bf A})$ such that $L(\pi,1/2)=0$. Then,
$\widetilde{\sigma}(\pi)$ defines a nonzero cuspidal
representation of $\widetilde{SL}_4({\bf A})$.
\end{proposition}
\begin{proof}
We start with the cuspidality condition. The group $SL_4$ has
three maximal parabolic subgroups. Their unipotent radicals are
given by $V_1=\{x_{1242}(r_1)x_{1342}(r_2)x_{2342}(r_3)\}$,
$V_2=\{x_{0100}(r_1)x_{1100}(r_2)x_{1342}(r_3)x_{2342}(r_4)\}$ and
$V_3=\{x_{1000}(r_1)x_{1100}(r_2)x_{2342}(r_3)\}$. The Weyl
element $w[3243423]$ conjugates $V_3$ to $V_1$ and fixes the group
$SL_2$ generated by $< x_{\pm(0001)}(r)>$. Hence it is enough to
prove that for $i=1,2$, the integral
\begin{equation}\label{sl41}
\int\limits_{SL_2(F)\backslash SL_2({\bf
A})}\int\limits_{V_i(F)\backslash V_i({\bf
A})}{\varphi}(g)\theta((v,g))dvdg
\end{equation}
is zero for all choice of data. Both unipotent subgroups $V_i$
contains the group $Z=\{x_{2342}(r)\}$. Hence, using Proposition
\ref{property2}, integral \eqref{sl41} is equal to the sum
$$\int\limits_{SL_2(F)\backslash SL_2({\bf
A})}\int\limits_{Z({\bf A})V_i(F)\backslash V_i({\bf
A})}\varphi(g)\theta^{U}((v,g))dvdg+$$
$$\int\limits_{SL_2(F)\backslash SL_2({\bf
A})}\varphi(g)\int\limits_{Z({\bf A})V_i(F)\backslash V_i({\bf
A})} \sum_{\gamma\in Q(F)\backslash Sp_6(F)}\sum_{\epsilon\in F^*}
\theta^{U,\psi}(h_2(\epsilon)\gamma (v,g))dvdg$$ Denote the first
integral by $I'$ and the second one by $I''$. From Proposition
\ref{eisen1}, it follows that the integral
$$\int\limits_{SL_2(F)\backslash SL_2({\bf
A})}\varphi(g)\theta_6(d(g))dg$$ is an inner
integration to integral $I'$. Here, for all $g\in SL_2({\bf A})$,
we set $d(g)=diag(g,I_2,g^*)$, and $\theta_6$ is a vector in the space of the representation $\Theta_6$, which was defined right before Proposition \ref{eisen2}. 
To prove this integral is zero we proceed exactly as in the proof that integral \eqref{cuspsl20} is
zero for all choice of data. Indeed, as can be seen the proof of that integral only uses one copy
of $SL_2$, the one which we embedded here as $\{d(g) : g\in SL_2\}$. Hence $I'=0$.

Next we compute $I''$. As in the proof of Proposition \ref{sl32.2},
for $1\le j\le 4$, we denote by $I_j$ the contribution to $I''$ from
each of the double coset representatives of $Q(F)\backslash Sp_6(F)/Q(F)$,
which we choose as $e,w[2],w[232]$ and $w[232432]$. The integral
$I_1$ is equal to
$$\int\limits_{SL_2(F)\backslash SL_2({\bf
A})}\int\limits_{Z({\bf A})V_i(F)\backslash V_i({\bf
A})}\varphi(g)\sum_{\epsilon\in
F^*}\theta^{U,\psi}(h_2(\epsilon)(v,g))dvdg$$ Using Proposition
\ref{property3} we obtain $\int\limits_{SL_2(F)\backslash
SL_2({\bf A})}\varphi(g)dg$ as inner integration. Thus $I_1=0$.
Next, $I_4$ is equal to
$$\int\limits_{SL_2(F)\backslash SL_2({\bf
A})}\int\limits_{Z({\bf A})V_i(F)\backslash V_i({\bf
A})}\varphi(g)\sum_{\delta_i\in F\epsilon\in
F^*}\theta^{U,\psi}(h_2(\epsilon)w[232432]m(\delta_i)(v,g))dvdg$$
Here
$$m(\delta_i)=x_{0100}(\delta_1)x_{0110}(\delta_2)x_{0111}(\delta_3)
x_{0120}(\delta_4)x_{0121}(\delta_5)x_{0122}(\delta_6)$$ Notice
that $V_i$ contains the one dimensional unipotent subgroup
$x_{1342}(r)$. From the identity $w[232432]m(\delta_i)x_{1342}(r)=
x_{1000}(r)w[232432]m(\delta_i)$, we obtain
$\int\limits_{F\backslash {\bf A}}\psi(\epsilon r)dr$ as inner
integration. Thus $I_4=0$.

Integral $I_2$ is equal to
$$\int\limits_{SL_2(F)\backslash SL_2({\bf
A})}\varphi(g)\int\limits_{Z({\bf A})V_i(F)\backslash V_i({\bf A})}
\sum_{\gamma\in S(4)(F)\backslash SL_3(F)}\sum_{\delta\in F,
\epsilon\in F^*}
\theta^{U,\psi}(h_2(\epsilon)w[2]x_{0100}(\delta)\gamma (v,g))dvdg$$
Here $S(4)$ is the maximal parabolic subgroup of $SL_3$ which
contains the group $<x_{\pm 0001}(r)>$. The space $S(4)(F)\backslash
SL_3(F)/S(4)(F)$ contains two representatives, which we can choose as $e$
and $w[3]$. The first representative contributes zero to $I_2$.
Indeed, it is equal to
$$\int\limits_{SL_2(F)\backslash SL_2({\bf
A})}\varphi(g)\int\limits_{Z({\bf A})V_i(F)\backslash V_i({\bf A})}
\sum_{\delta\in F, \epsilon\in F^*}
\theta^{U,\psi}(h_2(\epsilon)w[2]x_{0100}(\delta) (v,g))dvdg$$ It
follows from Proposition \ref{property3}, that for all $g\in
SL_2({\bf A})$ we have
$\theta^{U,\psi}(h_2(\epsilon)w[2]x_{0100}(\delta) (v,g))=
\theta^{U,\psi}(h_2(\epsilon)w[2]x_{0100}(\delta) (v,1))$. Hence, we
obtain the integral $\int\limits_{SL_2(F)\backslash SL_2({\bf
A})}\varphi(g)dg$ as inner integration. Thus, $I_2$ is equal to
$$\int\limits_{B_2(F)\backslash SL_2({\bf
A})}\varphi(g)\int\limits_{Z({\bf A})V_i(F)\backslash V_i({\bf
A})} \sum_{\delta_i\in F, \epsilon\in F^*}
\theta^{U,\psi}(h_2(\epsilon)w[23]x_{0010}(\delta_1)x_{0120}(\delta_2)
(v,g))dvdg$$ where $B_2$ is the Borel subgroup of $SL_2$. From
commutation relations in $F_4$, and using Proposition
\ref{property3}, we deduce that the function
$$g\mapsto \theta^{U,\psi}(h_2(\epsilon)w[23]x_{0010}(\delta_1)x_{0120}(\delta_2)
(v,g))$$ is left invariant under $x_{0001}(r)$ for all $r\in {\bf
A}$. Thus we obtain the integral $\int\limits_{F\backslash {\bf
A}}\varphi(\begin{pmatrix} 1&x\\ &1
\end{pmatrix})dx$ as inner integration. From the cuspidality of
$\pi$ it follows that this last integral, and hence $I_2$, is
zero.

Finally, we consider $I_3$. It is equal to
$$\int\limits_{SL_2(F)\backslash SL_2({\bf
A})}\varphi(g)\int\limits_{Z({\bf A})V_i(F)\backslash V_i({\bf
A})} \sum_{\substack{\gamma\in S(3)(F)\backslash SL_3(F)\\
\delta_i\in F,\epsilon\in F^* }}
\theta^{U,\psi}(h_2(\epsilon)w[232]y_1(\delta_1,\delta_2,\delta_3)\gamma
(v,g))dvdg$$ where
$y_1(\delta_1,\delta_2,\delta_3)=x_{0100}(\delta_1)x_{0110}(\delta_2)
x_{0120}(\delta_3)$. Also, $S(3)$ is the maximal parabolic subgroup of $SL_3$ which contains the 
group $<x_{\pm 0010}(r)>$.
The space $S(3)(F)\backslash SL_3(F)/S(4)(F)$ contains
two elements which we choose as $e$ and $w[43]$. The contribution
to $I_3$ from $e$ is equal to
$$\int\limits_{B_2(F)\backslash SL_2({\bf
A})}\varphi(g)\int\limits_{Z({\bf A})V_i(F)\backslash V_i({\bf A})}
\sum_{\delta_i\in F, \epsilon\in F^*}
\theta^{U,\psi}(h_2(\epsilon)w[232]x_{0100}(\delta_1)x_{0110}(\delta_2)
x_{0120}(\delta_3)(v,g))dvdg$$ As above, it follows from Proposition
\ref{property3} that the function
$$g\mapsto \theta^{U,\psi}(h_2(\epsilon)w[232]y_1(\delta_1,\delta_2,\delta_3)
(v,g))$$ is left invariant by $x_{0001}(r)$ for all $r\in {\bf A}$.
Hence we get zero contribution from this term. Thus $I_3$ is equal
to
\begin{equation}\label{sl42}
\int\limits_{SL_2(F)\backslash SL_2({\bf
A})}\varphi(g)\int\limits_{Z({\bf A})V_i(F)\backslash V_i({\bf
A})}\sum_{\delta_i\in F, \epsilon\in F^*}
\theta^{U,\psi}(h_2(\epsilon)w[23243]y(\delta_1,\ldots,\delta_5)
(v,g))dvdg
\end{equation}
Here
$$y(\delta_1,\ldots,\delta_5)=x_{0120}(\delta_1)x_{0121}(\delta_2)
x_{0122}(\delta_3)x_{0010}(\delta_4)x_{0011}(\delta_5)$$ Suppose
first that $i=1$. Then $x_{1242}(r)\in V_1$. We have
$$h_2(\epsilon)w[23243]y(\delta_1,\ldots,\delta_5)x_{1242}(r)=
x_{1000}(\epsilon^{-1}
r)h_2(\epsilon)w[23243]y(\delta_1,\ldots,\delta_5)$$ Changing
variables in $U$, we obtain $\int\limits_{F\backslash {\bf
A}}\psi(\epsilon^{-1}r)dr$ as inner integration. Hence $I_3$ is
zero in this case.

Next suppose that $i=2$. The group $SL_2(F)$ generated by
$<x_{\pm(0001)}(\mu)>$, acts on the group
$\{y(0,0,0,\delta_4,\delta_5) : \delta_i\in F\}$ with two orbits.

Consider first the trivial orbit. We denote the contribution to
$I_3$ from this term by $I_{31}$. Then we consider the action of the
above $SL_2(F)$ on the group $\{y(\delta_1,\delta_2,\delta_3,0,0) :
\delta_i\in F\}$. The action is given by  the symmetric square
representation. There are infinite number of orbits. First, using
the cuspidality of $\pi$, the trivial orbit and the orbits which
correspond to a nonzero vector with zero length, all contribute zero
to the integral $I_{31}$. Indeed, for the trivial orbit we obtain
$\int\limits_{SL_2(F)\backslash SL_2({\bf A})}\varphi(g)dg$ as inner
integration, and for the orbit which corresponds to nonzero vectors
with zero length, we obtain $\int\limits_{N_2(F)\backslash N_2({\bf
A})}\varphi(ng)dn$ as inner integration. Here $N_2$ is the maximal
unipotent subgroup of $SL_2$. Clearly both integrals are zero.

Thus we are left with the orbits which correspond to a vector of
nonzero length. There are infinite number of such vectors, and the
stabilizer inside $SL_2(F)$ of any such orbit, is an orthogonal
group $O_2(F)$. Factoring the measure, and using Proposition
\ref{property3} we obtain $\int\limits_{O_2(F)\backslash O_2({\bf
A})}\varphi(mg)dm$ as inner integration. The type of the orthogonal
group, depends on the representative of the orbit. From \cite{W} it
follows that the vanishing of $L(\pi,1/2)$ is equivalent to the
vanishing of all the above integrals over $O_2$. Thus $I_{31}=0$,
and $I_3$ is equal to
$$\int\limits_{N_2(F)\backslash SL_2({\bf
A})}\varphi(g)\int\limits_{Z({\bf A})V_2(F)\backslash V_2({\bf
A})} \sum_{\delta_i\in F, \epsilon\in F^*}
\theta^{U,\psi}(h_2(\epsilon)w[23243]y(\delta_1,\delta_2,\delta_3,0,1)
(v,g))dvdg$$ where $N_2$ is the maximal unipotent subgroup of
$SL_2$. The unipotent elements $x_{0100}(r_1)$ and $x_{1100}(r_2)$
are inside $V_2$. Using commutation relations we have,
$$h_2(\epsilon)w[23243]y(\delta_1,\delta_2,\delta_3,0,1)x_{1100}(r_2)=
vux_{1000}(\epsilon^{-1}\delta_1
r_2)h_2(\epsilon)w[23243]y(\delta_1,\delta_2,\delta_3,0,1)$$ where
$v$ is an element in the stabilizer of $\psi_U$ and $u\in U$ such
that $\psi_U(u)=1$. Thus, changing variables in $U$, we obtain
$\int\limits_{F\backslash {\bf A}}\psi(\epsilon^{-1}\delta_1
r_2)dr_2$ as inner integration. Hence, we may assume that
$\delta_1=0$. Next, using commutation relations we obtain
$$h_2(\epsilon)w[23243]y(0,\delta_2,\delta_3,0,1)x_{0100}(r_1)=
vuh_2(\epsilon)w[23243]y(0,\delta_2,\delta_3+r_1,0,1)$$ where $u$
and $v$ are as above. Collapsing summation with integration, $I_3$
is equal to
\begin{equation}\label{sl421}
\int\limits_{N_2(F)\backslash SL_2({\bf
A})}\varphi(g)\int\limits_{\bf A}\sum_{\delta_2\in F, \epsilon\in
F^*} \theta^{U,\psi}(h_2(\epsilon)w[23243]y(0,\delta_2,r_1,0,1)
(1,g))dr_1dg
\end{equation}
Using Proposition \ref{property3}, the function
$$g\mapsto \int\limits_{\bf A}\sum_{\delta_i\in F, \epsilon\in
F^*} \theta^{U,\psi}(h_2(\epsilon)w[23243]y(0,\delta_2,r_1,0,1)
(1,g))dr_1$$ is left invariant under $x_{0001}(r)$ for all $r\in
{\bf A}$. Thus, $I_3=0$ by the cuspidality of $\pi$. Hence integral
\eqref{sl41} is zero for all unipotent radicals $V_i$. This
completes the proof of the cuspidality of the lift.

To prove the nonvanishing of the lift, we shall compute the
Whittaker function of $\widetilde{f}$, where $\widetilde{f}$ is in
the space of  $\widetilde{\sigma}(\pi)$. Let $\beta\in
(F^*)^4\backslash F^*$. For $h\in \widetilde{SL}_4({\bf A})$,
denote by $W_{\widetilde{f},\beta}(h)$ the integral
$$\int\limits_{(F\backslash A)^6}\widetilde{f}(x_{1000}(r_1)x_{0100}(r_2)
x_{1242}(r_3)x_{1100}(r_4)x_{1342}(r_5)x_{2342}(r_6)h)\psi(\beta
r_1+r_2+r_3)dr_i$$ We shall denote this unipotent group by $V$ and
the above character by $\psi_{V,\beta}$. Thus,
$W_{\widetilde{f},\beta}(h)$ is equal to
$$\int\limits_{SL_2(F)\backslash SL_2({\bf
A})}\int\limits_{V(F)\backslash V({\bf
A})}{\varphi}(g)\theta((vh,g))\psi_{V,\beta}(v)dvdg$$ Following the
same steps as we did in the proof of the cuspidality of
$\widetilde{\sigma}(\pi)$, we obtain that all integrals, except
\eqref{sl42}, contribute zero to $W_{\widetilde{f},\beta}(h)$. Thus
$W_{\widetilde{f},\beta}(h)$ is equal to
\begin{equation}
\int\limits_{SL_2(F)\backslash SL_2({\bf
A})}\varphi(g)\int\limits_{Z({\bf A})V(F)\backslash V({\bf
A})}\sum_{\delta_i\in F, \epsilon\in F^*}
\theta^{U,\psi}(h_2(\epsilon)w[23243]y(\delta_1,\ldots,\delta_5)
(vh,g))\psi_{V,\beta}(v)dvdg\notag
\end{equation}
Using the commutation relations as after \eqref{sl42}, and arguing 
as in integral \eqref{sl421},  we deduce that $\epsilon=1$.
Continuing further as in the proof of the cuspidality,
$W_{\widetilde{f},\beta}(h)$ is equal to
$$\int\limits_{N_2(F)\backslash SL_2({\bf
A})}\varphi(g)\int\limits_{\bf A}\int\limits_{F\backslash {\bf A}}
\sum_{\delta_2\in F} \theta^{U,\psi}(w[23243]y(0,\delta_2,r_2,0,1)
(x_{1000}(r_1),g))\psi(\beta r_1+r_2)dr_1dr_2dg$$ Next we
conjugate the unipotent element $x_{1000}(r_1)$ to the left. We
have
$$w[23243]y(0,\delta_2,r_2,0,1)x_{1000}(r_1)=x_{1000}(\delta_2r_1)u'
w[23243]y(0,\delta_2,r_2,0,1)$$ Here $u'\in U$ is such that
$\psi_U(u')=1$. Thus, we obtain the integral
$\int\limits_{F\backslash {\bf A}}\psi((\delta_2-\beta)r_1)dr_1$ as
inner integration. From this we deduce that $\delta_2=\beta$. Hence,
$W_{\widetilde{f},\beta}(h)$ is equal to
$$\int\limits_{N_2(F)\backslash SL_2({\bf
A})}\varphi(g)\int\limits_{\bf
A}\theta^{U,\psi}(w[23243]y(0,\beta,r_2,0,1)
(1,g))\psi(r_2)dr_2dg$$ Using commutation relations and a change
of variables, we obtain
$$\int\limits_{\bf
A}\theta^{U,\psi}(w[23243]y(0,\beta,r_2,0,1)
(1,x_{0001}(r)g))\psi(r_2)dr_2=$$
$$\psi(\beta r)\int\limits_{\bf
A}\theta^{U,\psi}(w[23243]y(0,\beta,r_2,0,1)(1,g))\psi(r_2)dr_2$$
From this we obtain the identity
$$W_{\widetilde{f},\beta}(h)=\int\limits_{N_2({\bf A})\backslash SL_2({\bf
A})}W_{\varphi,\beta}(g)\int\limits_{\bf
A}\theta^{U,\psi}(w[23243]y(0,\beta,r_2,0,1)
(1,g))\psi(r_2)dr_2dg$$ where
$W_{\varphi,\beta}(g)=\int\limits_{F\backslash {\bf
A}}\varphi(\begin{pmatrix} 1&r\\&1 \end{pmatrix}g)\psi(\beta
r)dr$. Using similar arguments as in \cite{Ga-S} we deduce that
$W_{\widetilde{f},\beta}(h)$ is nonzero for some choice of data if
and only if $W_{\varphi,\beta}(g)$ is nonzero for some choice of
data. Since there is always a $\beta\in F^*$ such that
$W_{\varphi,\beta}(g)$ is not zero, the nonvanishing of the lift
follows.
\end{proof}

\subsubsection{\bf From $\widetilde{SL}_4$ to $SL_2$}

To study this lift we consider a different embedding of the two
groups. Viewing $SL_4$ as $Spin_6$, we embed it inside the Levi
part of the maximal parabolic subgroup of $F_4$ whose Levi part
contains $Spin_7$. Thus, the group $SL_4$ is generated by
$$<x_{\pm(1000)}(r);x_{\pm(0100)}(r);x_{\pm(1100)}(r);
x_{\pm(0120)}(r);x_{\pm(1120)}(r);x_{\pm(1220)}(r)>$$ The group
$SL_2$ is generated by $<x_{\pm(1232)}(r)>$. If we conjugate these
groups by the Weyl element $w[3213234]$ we obtain the
embedding we used in the previous subsection.

Let $\widetilde{\pi}$ denote an irreducible cuspidal
representation defined on $\widetilde{SL}_4({\bf A})$. We shall
denote by $\sigma(\widetilde{\pi})$ the automorphic representation
of $SL_2({\bf A})$ spanned by all functions
$$f(g)=\int\limits_{SL_4(F)\backslash SL_4({\bf
A})}\widetilde{\varphi}(h)\theta((h,g))dh$$ Here
$\widetilde{\varphi}$ is a vector in the space of
$\widetilde{\pi}$. We start with
\begin{proposition}\label{sl42.2}
The representation $\sigma(\widetilde{\pi})$ is nonzero if and
only if the integral
\begin{equation}\label{sl44}
\int\limits_{Sp_4(F)\backslash Sp_4({\bf
A})}\widetilde{\varphi}(m)\theta_{Sp_4}^{\phi,\psi}(m)dm
\end{equation}
is nonzero for some choice of data. Here
$\theta_{Sp_4}^{\phi,\psi}$ is the theta function defined on
$\widetilde{Sp}_4({\bf A})$.
\end{proposition}
\begin{proof}
Clearly, $\sigma(\widetilde{\pi})$ is nonzero if and only if the
integral
$$W_{f}(g)=\int\limits_{SL_4(F)\backslash SL_4({\bf
A})}\int\limits_{F\backslash {\bf A}}
\widetilde{\varphi}(h)\theta((h,x_{1232}(r)g))\psi(\beta r)drdh$$
is nonzero for some choice of data. Let $U_1$ denote the abelian
unipotent group generated by all elements of the form
$$u_1(r_1,\ldots,r_6)=x_{0122}(r_1)x_{1122}(r_2)x_{1222}(r_3)x_{1242}(r_4)
x_{1342}(r_5)x_{2342}(r_6)$$ and let $U_2=<U_1,x_{1232}(r)>$. We
expand $\theta$ along the group $U_1(F)\backslash U_1({\bf A})$.
The group $SL_4(F)=Spin_6(F)$ acts on this expansion with three
type of orbits. The first two orbits are the ones which
corresponds to the trivial orbit, and to the orbit corresponding
to nonzero vectors with zero length. Plugging these two Fourier
coefficients in $W_{f}(g)$ we obtain the integrals
$$\int\limits_{U_1(F)\backslash U_1({\bf A})} \int\limits_{F\backslash {\bf A}}
\theta(u_1(r_1,\ldots,r_6)x_{1232}(r))\psi(\epsilon r_1+\beta
r)drdr_i$$ as inner integrations. Here $\epsilon=0$ when the orbit
is the trivial one, and $\epsilon=1$ corresponds to the other orbit.
In both cases, the above Fourier coefficient corresponds to the
unipotent orbit $\widetilde{A}_1$ which is greater than the minimal
orbit. Hence, by Theorem \ref{mini1}, these Fourier coefficients are
zero.

The third type of orbits corresponds to vectors of nonzero length.
These contributes the Fourier coefficient
$$\int\limits_{U_1(F)\backslash U_1({\bf A})} \int\limits_{F\backslash {\bf A}}
\theta(u_1(r_1,\ldots,r_6)x_{1232}(r))\psi(r_3+\gamma r_4+
r)drdr_i$$ where $\gamma\in F^*$. Notice that the stabilizer of
this character inside $Spin_6=SL_4$ is $Spin_5=Sp_4$. We can
identify the group $U_2({\bf A})$ with ${\bf A}^7$. With this
identification we can write the above integral as
$$\int\limits_{U_2(F)\backslash U_2({\bf A})}
\theta(u_2)\psi(\delta\cdot u_2)du_2$$ Here we identify $u_2$ with
a column vector and $\delta=(0,0,1,1,\gamma,0,0)$. With this
identification $\delta\cdot u_2$ is the usual dot product. If
$\gamma$ is such that $\delta$ has a nonzero length, then this
Fourier coefficient corresponds to the unipotent orbit
$\widetilde{A}_1$, and as above, it is zero. There is one choice
of $\gamma$ such that the length of $\delta$ is zero. Conjugating
by a suitable discrete element, this Fourier coefficient is equal
to
$$\int\limits_{U_2(F)\backslash U_2({\bf A})}
\theta(u_2w[123]x_{0010}(1))\psi_{U_2}(u_2)du_2$$ where $\psi_{U_2}$
is defined as follows. For $u_2=x_{0122}(r)u_2'$  define
$\psi_{U_2}(u_2)=\psi(r)$. See subsection 2.1 for notations. From this we obtain that $W_f(g)$ is
equal to
$$\int\limits_{Sp_4(F)\backslash SL_4({\bf A})}
\int\limits_{U_2(F)\backslash U_2({\bf A})}\widetilde{\varphi}(h)
\theta(u_2w[123]x_{0010}(1)(h,1))\psi_{U_2}(u_2)du_2$$ Using
Proposition \ref{property2} this integral is equal to
$$\int\limits_{Sp_4(F)\backslash SL_4({\bf A})}
\int\limits_{Z({\bf A})U_2(F)\backslash U_2({\bf
A})}\widetilde{\varphi}(h)
\theta^U(u_2\mu(h,1))\psi_{U_2}(u_2)du_2dh+$$
$$\int\limits_{Sp_4(F)\backslash SL_4({\bf
A})}\varphi(g)\int\limits_{Z({\bf A})U_2(F)\backslash U_2({\bf A})}
\sum_{\gamma\in Q(F)\backslash Sp_6(F)}\sum_{\epsilon\in F^*}
\theta^{U,\psi}(h_2(\epsilon)\gamma
u_2\mu(h,1))\psi_{U_2}(u_2)du_2dh$$ where we denote
$\mu=w[123]x_{0010}(1)$. Denote the first integral by $I'$ and the
second one by $I''$. We start with $I''$. Let $P$ denote the maximal
parabolic subgroup of $Sp_6$ whose Levi part contains $Sp_4$. The
space $Q(F)\backslash Sp_6(F)/P(F)$ has two representatives which we can
choose as $e$ and $w[234]$. The first representative contributes
$$\int\limits_{S(3)(F)\backslash SL_4({\bf
A})}\varphi(g)\int\limits_{Z({\bf A})U_2(F)\backslash U_2({\bf
A})}\sum_{\epsilon\in F^*} \theta^{U,\psi}(h_2(\epsilon)
u_2\mu(h,1))\psi_{U_2}(u_2)du_2dg$$ to the integral. Here $S(3)$ is
the parabolic subgroup of $Sp_4$ whose Levi part is $GL_2$. Changing
variables in $U$ and using Proposition \ref{property3}, we obtain
that
$$\theta^{U,\psi}(h_2(\epsilon)u_2\mu(h,1))=\theta^{U,\psi}(h_2(\epsilon)
\mu(h,1))$$ for all $u_2\in U_2({\bf A})$. Thus we obtain
$\int\limits_{Z({\bf A})U_2(F)\backslash U_2({\bf
A})}\psi_{U_2}(u_2)du_2$ as inner integration. Thus the
contribution to $I''$ from this term is zero. The second
representative contributes the integral
$$\int\limits_{S(3)(F)\backslash SL_4({\bf
A})}\varphi(g)\int\limits_{Z({\bf A})U_2(F)\backslash U_2({\bf
A})}\sum_{\delta_i\in F, \epsilon\in F^*}
\theta^{U,\psi}(h_2(\epsilon)w[234]y(\delta_1,\delta_2,\delta_3)
u_2\mu(h,1))\psi_{U_2}(u_2)du_2dg$$ where
$y(\delta_1,\delta_2,\delta_3)=x_{0001}(\delta_1)x_{0011}(\delta_2)
x_{0122}(\delta_3)$. We have
$h_2(\epsilon)w[234]y(\delta_1,\delta_2,\delta_3)x_{1122}(r)=
x_{1000}(\epsilon^{-1}r)h_2(\epsilon)w[234]y(\delta_1,\delta_2,\delta_3)$.
Hence we get $\int\limits_{F\backslash {\bf
A}}\psi(\epsilon^{-1}r)dr$ as inner integration. This integral is
zero and hence $I''=0$. Thus $W_f(g)$ is equal to $I'$. Factoring
the measure, we obtain
$$\int\limits_{Sp_4({\bf A})\backslash SL_4({\bf A})}
\int\limits_{Sp_4(F)\backslash Sp_4({\bf A})} \int\limits_{Z({\bf
A})U_2(F)\backslash U_2({\bf A})}\widetilde{\varphi}(mh)
\theta^U(u_2\mu(mh,1))\psi_{U_2}(u_2)du_2dmdh$$ Arguing as in
\cite{Ga-S} we deduce that the lift is nonzero for some choice of
data if and only if the integral
$$\int\limits_{Sp_4(F)\backslash Sp_4({\bf A})}
\int\limits_{F\backslash {\bf A}}\widetilde{\varphi}(m)
\theta^U(x_{0122}(r)\mu(m,1))\psi(r)drdm$$ is nonzero for some
choice of data. The group $\mu
Sp_4\mu^{-1}=<x_{\pm(0100)}(r);x_{\pm(0010)}(r)>$. Hence, from
Proposition \ref{eisen1} it follows that the lift is nonzero
for some choice of data if and only if the integral
$$\int\limits_{Sp_4(F)\backslash Sp_4({\bf A})}
\int\limits_{F\backslash {\bf A}}\widetilde{\varphi}(m)
\theta_6(x_{0122}(r)m)\psi(r)drdm$$ is not zero
for some choice of data. It follows from Proposition \ref{propertysp3} that the above integral
is not zero for some choice of data, if and only if integral 
\eqref{sl44} is not zero for some
choice of data. This completes the proof of the Proposition.
\end{proof}

Next we address the question of cuspidality of the lift. We prove
\begin{proposition}\label{sl43.3}
The representation $\sigma(\widetilde{\pi})$ is a cuspidal
representation of $SL_2({\bf A})$.
\end{proposition}
\begin{proof}
We need to show that the integral
\begin{equation}\label{sl45}
\int\limits_{SL_4(F)\backslash SL_4({\bf
A})}\int\limits_{F\backslash {\bf A}}
\widetilde{\varphi}(h)\theta((h,x_{1232}(r)g))drdh
\end{equation}
is zero for all choice of data. We expand the theta function along
the group $U_1$ which was defined in the proof of Proposition
\ref{sl42.2}. Combining this with the integration over the group
$\{x_{1232}(r)\}$, integral \eqref{sl45} is equal to
$$\int\limits_{SL_4(F)\backslash SL_4({\bf
A})}\widetilde{\varphi}(h)\theta^{U_2}(u_2(h,1))du_2dh
+\int\limits_{S(F)\backslash SL_4({\bf
A})}\int\limits_{U_2(F)\backslash U_2({\bf A})}
\widetilde{\varphi}(h)\theta(u_2(h,1))\psi_{U_2}(u_2)du_2dh$$ Here
$S$ is the subgroup of $SL_4$ defined by $$S=<x_{\pm(0100)}(r);
x_{\pm(0120)}(r);x_{(1000)}(r);x_{(1100)}(r);x_{(1120)}(r);x_{(1220)}(r)>$$
and $\psi_{U_2}$ was defined in the proof of Proposition
\ref{sl42.2}. Denote by $I'$ the first summand and by $I''$ the
second one. We start with $I''$. Using Proposition \ref{property2}
it is equal to
$$\int\limits_{S(F)\backslash SL_4({\bf A})}
\int\limits_{Z({\bf A})U_2(F)\backslash U_2({\bf
A})}\widetilde{\varphi}(h)
\theta^U(u_2(h,1))\psi_{U_2}(u_2)du_2dh+$$
$$\int\limits_{S(F)\backslash SL_4({\bf
A})}\varphi(g)\int\limits_{Z({\bf A})U_2(F)\backslash U_2({\bf A})}
\sum_{\gamma\in Q(F)\backslash Sp_6(F)}\sum_{\epsilon\in F^*}
\theta^{U,\psi}(h_2(\epsilon)\gamma u_2(h,1))\psi_{U_2}(u_2)du_2dh$$
Arguing in a similar way as in the computation of $I''$ in the proof
of Proposition \ref{sl42.2}, we deduce that the second summand in
the above integral is zero. Indeed, using Proposition
\ref{property3} we obtain the integral
$$\int\limits_{Mat_{2\times 2}(F)\backslash Mat_{2\times 2}({\bf
A})}\widetilde{\varphi}(\begin{pmatrix} I_2&X\\&I_2
\end{pmatrix})dX$$ as inner integration to the first summand. This
is zero by the cuspidality of $\widetilde{\pi}$.

Let $U(B_3)=U_{\alpha_1,\alpha_2,\alpha_3}$. Then $U_2$ is
a subgroup of $U(B_3)$. The quotient $U(B_3)/U_2$ is an eight
dimensional abelian group and $SL_4$ acts on it as twice the standard
representation. The quotient $U(B_3)/U_2$ is generated by all unipotent groups
$\{x_\alpha(r)\}$ such that $\alpha=\sum_{i=1}^3n_i\alpha_i+\alpha_4$.
To compute $I'$ we further expend it along the group $U(B_3)/U_2$
with points in $F\backslash {\bf A}$. By the minimality of $\Theta$
only the constant term contributes. Indeed, the nontrivial Fourier
coefficients will contain, as inner integration,  a Fourier
coefficient which corresponds to the unipotent orbit
$\widetilde{A}_1$. This follows from the fact that the length of all
the above roots $\alpha$ is short. Thus
$$I'=\int\limits_{SL_4(F)\backslash SL_4({\bf
A})}\widetilde{\varphi}(h)\theta^{U(B_3)}((h,1))dh$$  To show that
this last integral is zero, let $E(h,s)$ denote the Eisenstein
series of $SL_4({\bf A})$ which is associated with the induced
representation $Ind_{R({\bf A})}^{SL_4({\bf A})}\delta_R^s$. Here
$R$ is a maximal parabolic subgroup of $SL_4$ whose Levi part is
$GL_3$. Thus, to prove that $I'$ is zero, it is enough to show that
for $Re(s)$ large, the integral
$$\int\limits_{SL_4(F)\backslash SL_4({\bf
A})}\widetilde{\varphi}(h)\theta^{U(B_3)}(h)E(h,s)dh$$ is zero for
all choice of data. Unfolding the Eisenstein series, and using from
\cite{PS1}, the well known Whittaker expansion of
$\widetilde{\varphi}$ we obtain the integral
$$\int\limits_{V(F)\backslash V({\bf A})}\theta^{U(B_3)}(vh)
\psi_V(v)dv$$ as inner integration. Here $V$ is the maximal
unipotent subgroup of $Spin_6=SL_4$ and $\psi_V$ is the Whittaker
coefficient of $V$. The Fourier coefficient given by the above
integration over $V$ is a Fourier coefficient which is associated to
a unipotent orbit of $Spin_7$ which is greater than the minimal
orbit. Hence, it follows from Proposition \ref{eisen1} that this
integral is zero.  This implies that $I'$ is zero. This completes
the proof of the cuspidality of the lift.
\end{proof}

\subsection{\bf The Commuting Pair $(SO_3,G_2)$}

In this subsection we will consider the lift of automorphic
representations from the group $SO_3({\bf A})$ to automorphic
representations of the exceptional $\widetilde{G}_2({\bf A})$. We
first consider

\subsubsection{\bf From $SO_3({\bf A})$ to $\widetilde{G}_2({\bf
A})$}

To study this lift, we consider the following embedding of the two
groups. The group $SO_3$ is generated by 
$\{x_{0010}(r)x_{0001}(-r)x_{0011}(-r^2)\}$ and by
$\{x_{-(0010)}(r)x_{-(0001)}(-r)x_{-(0011)}(-r^2)\}$. In other words,
we embed $SO_3$ inside the group $SL_3$ generated by
$<x_{\pm(0010)}(r); x_{\pm(0001)}(r)>$. With this choice, the
group $V$, the maximal unipotent subgroup of $G_2$, is generated
by
$$V=<x_{1000}(r);x_{0120}(r)x_{0111}(r);x_{1111}(r)x_{1120}(r);
x_{1231}(r)x_{1222}(r);x_{1342}(r);x_{2342}(r)>$$ The group $G_2$
is generated by $V$ and by the group generated by all unipotent
elements which corresponds to the negative roots of the above six
unipotent elements. With this choice of embedding, the group
$SO_3$ splits under the double cover, but $G_2$ does not.

Let $\pi$ denote an irreducible cuspidal representation of the
group $SO_3({\bf A})$. Let $\widetilde{\sigma}(\pi)$ denote the
automorphic representation of $\widetilde{G}_2({\bf A})$ generated
by all functions
\begin{equation}
\widetilde{f}(h)=\int\limits_{SO_3(F)\backslash SO_3({\bf
A})}{\varphi}(g)\theta((h,g))dg\notag
\end{equation}
where $h\in \widetilde{G}_2({\bf A})$ and $\varphi(g)$ is a vector
in the space of $\pi$. We shall denote by $L(\pi,s)$ the standard
$L$ function attached to $\pi$. We  prove
\begin{proposition}\label{g21.1}
Let $\pi$ be as above, and assume that $L(\pi,1/2)=0$. Then
$\widetilde{\sigma}(\pi)$ defines a generic cuspidal
representation of $\widetilde{G}_2({\bf A})$.
\end{proposition}
\begin{proof}
For $i=1,2$, we shall denote by $V_i$ the two unipotent radicals
of the maximal parabolic subgroups of $G_2$. In other words,
$$V_1=<x_{0120}(r)x_{0111}(r);x_{1111}(r)x_{1120}(r);
x_{1231}(r)x_{1222}(r);x_{1342}(r);x_{2342}(r)>$$ and
$$V_2=<x_{1000}(r);x_{1111}(r)x_{1120}(r);
x_{1231}(r)x_{1222}(r);x_{1342}(r);x_{2342}(r)>$$ To prove the
cuspidality of the lift, we need to prove that for $i=1,2$ the
integrals
\begin{equation}
\int\limits_{SO_3(F)\backslash SO_3({\bf
A})}\int\limits_{V_i(F)\backslash V_i({\bf
A})}{\varphi}(g)\theta((v,g))dvdg\notag
\end{equation}
are zero for all choice of data. The group $Z=\{x_{2342}(r)\}$ is a
subgroup of $V_i$. Hence, using Proposition \ref{property2} this
integral is equal to
$$\int\limits_{SO_3(F)\backslash SO_3({\bf
A})}\int\limits_{Z({\bf A})V_i(F)\backslash V_i({\bf
A})}{\varphi}(g)\theta^U((v,g))dvdg+ $$
$$\int\limits_{SO_3(F)\backslash SO_3({\bf
A})}\varphi(g)\int\limits_{Z({\bf A})V_i(F)\backslash V_i({\bf A})}
\sum_{\gamma\in Q(F)\backslash Sp_6(F)}\sum_{\epsilon\in F^*}
\theta^{U,\psi}(h_2(\epsilon)\gamma (v,g))dvdg$$ Denote by $I_1'$ the first summand and by $I'$ the second summand.  To show that the first summand is zero, using
Proposition \ref{eisen1}, it is enough to show that the integral
$$\int\limits_{SO_3(F)\backslash SO_3({\bf
A})}{\varphi}(g)\theta_6((1,g))dg$$ is zero for
all choice of data. Here $\theta_6$ is a vector in the space of the representation $\Theta_6$
which was defined right before Proposition \ref{eisen2}. Apply Proposition \ref{propertysp2} to this
integral, by using identity \eqref{minexp1}.  The contribution of the constant term gives us the integral 
$$\int\limits_{SO_3(F)\backslash SO_3({\bf
A})}{\varphi}(g)\theta_6^{U(GL_3)}((1,g))dg$$ 
Here $U(GL_3)$ is the unipotent radical of the parabolic group $P(GL_3)$ which was defined before
Proposition \ref{propertysp1}. In Proposition \ref{propertysp2} this unipotent group was denoted by $U$.
From Proposition \ref{propertysp1} we obtain the integral $\int\limits_{SO_3(F)\backslash SO_3({\bf
A})}{\varphi}(g)dg$ as inner integration. This integral is clearly zero. Plugging the second summand
of \eqref{minexp1} we obtain 
$$\int\limits_{SO_3(F)\backslash SO_3({\bf
A})}{\varphi}(g)\sum_{\gamma\in L_0(GL_3)(F)\backslash GL_3(F)}\theta_6^{U(GL_3),\psi}(\gamma(1,g))dg$$
where $L_0(GL_3)$ was defined right before Proposition \ref{propertysp2}. Consider the space of
double cosets $L_0(GL_3)(F)\backslash GL_3(F)/SO_3(F)$. We partition the set of representatives
$\delta$ into two sets.  The first set has the property that $\delta^{-1}L_0(GL_3)\delta
\cap SO_3$ is the maximal unipotent subgroup of $SO_3$. In this case,  from the
cuspidality of $\pi$ and from equation \eqref{minexp2} in Proposition \ref{propertysp2}, we get zero contribution. The other type of representative has the property that $\delta^{-1}L_0(GL_3)\delta\cap SO_3$ is a certain $SO_2$ which can be embedded in the split $SO_3$. Thus, applying
again equation \eqref{minexp2} in Proposition \ref{propertysp2} we get $\int\limits_{SO_2(F)\backslash SO_2({\bf A})}{\varphi}(g)dg$ as inner integration. From \cite{W}, we know that if
$L(\pi,1/2)=0$ then this integral is zero. Thus $I_1'=0$.

Next we compute $I'$. The space of the double cosets $Q(F)\backslash
Sp_6(F)/Q(F)$ contains four representatives which we can choose to be
$e,w[2],w[232]$ and $w[232432]$. For $1\le j\le 4$, we denote by
$I_j$ the contribution to $I'$ from each one of the four
representatives. First, integral $I_1$ is equal to
$$\int\limits_{SO_3(F)\backslash SO_3({\bf
A})}\varphi(g)\int\limits_{Z({\bf A})V_i(F)\backslash V_i({\bf
A})} \sum_{\epsilon\in F^*} \theta^{U,\psi}(h_2(\epsilon)
(v,g))dvdg$$ Using Proposition \ref{property3} we obtain
$\int\limits_{SO_3(F)\backslash SO_3({\bf A})}\varphi(g)dg$ as
inner integration. Thus $I_1=0$. Next, $I_4$ is equal to
$$\int\limits_{SO_3(F)\backslash SO_3({\bf
A})}\varphi(g)\int\limits_{Z({\bf A})V_i(F)\backslash V_i({\bf
A})} \sum_{\delta_i\in F,\epsilon\in F^*}
\theta^{U,\psi}(h_2(\epsilon)w[232432]y(\delta_1,\ldots,\delta_6)
(v,g))dvdg$$ where
$$y(\delta_1,\ldots,\delta_6)=x_{0100}(\delta_1)x_{0110}(\delta_2)
x_{0111}(\delta_3)x_{0120}(\delta_4)x_{0121}(\delta_5)x_{0122}(\delta_6)$$
Since
$h_2(\epsilon)w[232432]y(\delta_1,\ldots,\delta_6)x_{1342}(r)=
x_{1000}(\epsilon^{-1}r)h_2(\epsilon)w[232432]y(\delta_1,\ldots,\delta_6)$
we obtain $\int\limits_{F\backslash {\bf
A}}\psi(\epsilon^{-1}r)dr$ as inner integration. Thus $I_4=0$.

Integral $I_2$ is equal to
$$\int\limits_{SO_3(F)\backslash SO_3({\bf
A})}\varphi(g)\int\limits_{Z({\bf A})V_i(F)\backslash V_i({\bf A})}
\sum_{\gamma\in S(4)(F)\backslash SL_3(F)}\sum_{\delta_1\in F,
\epsilon\in F^*}
\theta^{U,\psi}(h_2(\epsilon)w[2]x_{0100}(\delta_1)\gamma
(v,g))dvdg$$ Here $S(4)$ is the maximal parabolic subgroup of $SL_3$
whose Levi part contains the $SL_2$ generated by
$<x_{\pm(0001)}(r)>$. The space $S(4)(F)\backslash SL_3(F)/SO_3(F)$
contains infinite number of orbits. As representative we can choose
$e, w[3]$ and $w[34]x_{0011}(\nu)$ where $\nu\in (F^*)^2\backslash
F^*$. The identity representative contributes to $I_2$ the term
$$\int\limits_{B(F)\backslash SO_3({\bf
A})}\varphi(g)\int\limits_{Z({\bf A})V_i(F)\backslash V_i({\bf
A})} \sum_{\delta_1\in F,\epsilon\in F^*}
\theta^{U,\psi}(h_2(\epsilon)w[2]x_{0100}(\delta_1) (v,g))dvdg$$
where $B$ is the Borel subgroup of $SO_3$. Let $N$ denote the
unipotent radical of $SO_3$. From Proposition \ref{property3} we
deduce that the function
$$g\mapsto \theta^{U,\psi}(h_2(\epsilon)w[2]x_{0100}(\delta_1)
(v,g))$$ is left invariant under all $n\in N({\bf A})$. Thus we get
zero by the cuspidality of $\pi$. The second representative
contributes to $I_2$ the term
$$\int\limits_{T(F)\backslash SO_3({\bf
A})}\varphi(g)\int\limits_{Z({\bf A})V_i(F)\backslash V_i({\bf A})}
\sum_{\delta_1\in F,\epsilon\in F^*}
\theta^{U,\psi}(h_2(\epsilon)w[23]x_{0120}(\delta_1) (v,g))dvdg$$
where $T$ is the maximal split torus of $SO_3$. From Proposition
\ref{property3} it follows that the function
$$g\mapsto \theta^{U,\psi}(h_2(\epsilon)w[23]x_{0120}(\delta_1)
(v,g))$$ is left invariant under $T({\bf A})$. Thus we obtain
$\int\limits_{T(F)\backslash T({\bf A})}\varphi(tg)dt$ as inner
integration. Since $L(\pi,1/2)=0$, it follows that this last
integral is zero. Thus we are left with the third family of
representatives which contributes to $I_2$ the integral
$$\sum_{\nu\in F^*}\int\limits_{S_\nu(F)\backslash SO_3({\bf
A})}\varphi(g)\int\limits_{Z({\bf A})V_i(F)\backslash V_i({\bf A})}
\sum_{\delta_1\in F,\epsilon\in F^*}
\theta^{U,\psi}(h_2(\epsilon)w[234]x_{0122}(\delta_1)x_{0011}(\nu)
(v,g))dvdg$$ where $S_\nu$ is an orthogonal group which depends on
$\nu$. We have $x_{1111}(r)x_{1120}(r)\in V_i$. Also, we have the
commutation relations
$$h_2(\epsilon)w[234]x_{0122}(\delta_1)x_{0011}(\nu)
x_{1111}(r)x_{1120}(r)=x_{1000}(\nu\epsilon^{-1}r)u'
h_2(\epsilon)w[234]x_{0122}(\delta_1)x_{0011}(\nu)$$ where $u'\in
U$ such that $\psi_U(u')=1$. Thus we obtain
$\int\limits_{F\backslash {\bf A}}\psi(\nu\epsilon^{-1}r)dr$ as
inner integration. Hence $I_2=0$.

We are left with $I_3$ which is equal to
$$\int\limits_{SO_3(F)\backslash SO_3({\bf
A})}\varphi(g)\int\limits_{Z({\bf A})V_i(F)\backslash V_i({\bf
A})} \sum_{\substack{\gamma\in S(3)(F)\backslash SL_3(F)\\
\delta_i\in F \epsilon\in F^* }}
\theta^{U,\psi}(h_2(\epsilon)w[232]y_1(\delta_1,\delta_2,\delta_3)\gamma
(v,g))dvdg$$ where
$y_1(\delta_1,\delta_2,\delta_3)=x_{0100}(\delta_1)x_{0110}(\delta_2)
x_{0120}(\delta_3)$ and $S(3)$ is the maximal parabolic subgroup of $SL_3$ which contains the $SL_2$ generated by $<x_{\pm 0010}(r)>$. As in the computations of $I_2$, the space
$S(3)(F)\backslash SL_3(F)/SO_3(F)$ contains infinite number of
orbits. As representative we can choose $e, w[4]$ and
$w[43]x_{0011}(\nu)$ where $\nu\in (F^*)^2\backslash F^*$. The
contribution from the identity element is
$$\int\limits_{B(F)\backslash SO_3({\bf
A})}\varphi(g)\int\limits_{Z({\bf A})V_i(F)\backslash V_i({\bf
A})} \sum_{\delta_i\in F,\epsilon\in F^*}
\theta^{U,\psi}(h_2(\epsilon)w[232]y_1(\delta_1,\delta_2,\delta_3)
(v,g))dvdg$$ The unipotent element $x_{1111}(r)x_{1120}(r)\in
V_i$. We have
$$x_{0100}(\delta_1)x_{1111}(r)x_{1120}(r)=x_{1111}(r)x_{1120}(r)
x_{1220}(\delta_1 r)x_{0100}(\delta_1)$$ Since
$w[232]x_{1220}(\delta_1 r)=x_{1000}(\delta_1 r)w[232]$, we obtain
$\int\limits_{F\backslash {\bf A}}\psi(\delta_1\epsilon^{-1}r)dr$
as inner integration. Thus we may assume that $\delta_1=0$. If
$\delta_2=0$, then the function
$$g\mapsto \theta^{U,\psi}(h_2(\epsilon)w[232]y_1(0,0,\delta_3)
(v,g))$$ is left invariant under $N({\bf A})$. Thus, by cuspidality
we get zero. Hence we may assume that we sum over $\delta_2\ne 0$.
The torus $T(F)$ acts transitively on the set $\{x_{0110}(\delta_2):
\delta_2\ne0\}$. Collapsing summation with integration, the above
integral is equal to
$$\int\limits_{N(F)\backslash SO_3({\bf
A})}\varphi(g)\int\limits_{Z({\bf A})V_i(F)\backslash V_i({\bf
A})} \sum_{\delta_3\in F,\epsilon\in F^*}
\theta^{U,\psi}(h_2(\epsilon)w[232]x_{0110}(1) x_{0120}(\delta_3)
(v,g))dvdg$$ Suppose first that $i=1$. Then
$x_{0111}(r)x_{0120}(r)\in V_1$. Collapsing summation with
integration, we obtain as inner integration, the integral
$$\int\limits_{\bf A}\sum_{\epsilon\in F^*}
\theta^{U,\psi}(h_2(\epsilon)w[232]x_{0110}(1)
x_{0120}(r)(v,g))dr$$ By commutation relations, one can check that
as a function of $g$ the above integral is left invariant under
$N({\bf A})$. Thus we get zero by the cuspidality of $\pi$. When
$i=2$, we have $x_{1000}(r)\in V_2$. We have
$$h_2(\epsilon)w[232]x_{0110}(1) x_{0120}(\delta_3)x_{1000}(r)=
x_{1000}(\epsilon^{-1}r)u'h_2(\epsilon)w[232]x_{0110}(1)
x_{0120}(\delta_3)$$ where $u'\in U$ is such that $\psi_U(u')=1$.
Thus we get zero in this case also. Next we consider the
contribution of $w[4]$ to $I_3$. It is equal to
$$\int\limits_{T(F)\backslash SO_3({\bf
A})}\varphi(g)\int\limits_{Z({\bf A})V_i(F)\backslash V_i({\bf
A})} \sum_{\delta_i\in F,\epsilon\in F^*}
\theta^{U,\psi}(h_2(\epsilon)w[2324]y_2(\delta_1,\delta_2,\delta_3)
(v,g))dvdg$$ where
$y_2(\delta_1,\delta_2,\delta_3)=x_{0100}(\delta_1)x_{0111}(\delta_2)
x_{0122}(\delta_3)$. We have $x_{1231}(r)x_{1222}(r)\in V_i$.
Since $w[2324]x_{1222}(r)=x_{1000}(r)w[2324]$ we get zero
contribution in this case. Finally, the last set of
representatives are
$$\sum_{\nu\in F^*}\int\limits_{S_\nu(F)\backslash SO_3({\bf
A})}\varphi(g)\int\limits_{Z({\bf A})V_i(F)\backslash V_i({\bf A})}
\sum_{\delta_1\in F,\epsilon\in F^*}
\theta^{U,\psi}(h_2(\epsilon)w_0
y_3(\delta_1,\delta_2,\delta_3)x_{0011}(\nu) (v,g))dvdg$$ where
$y_3(\delta_1,\delta_2,\delta_3)=x_{0120}(\delta_1)x_{0121}(\delta_2)
x_{0122}(\delta_3)$ and $w_0=w[23243]$. As above, we use the
unipotent matrix $x_{1231}(r)x_{1222}(r)$ to get zero. Thus $I_3$
equal to zero. This completes the proof of the cuspidality.

To prove the nonvanishing of the lift we compute its Whittaker
coefficient. In other words, we compute the integral
$$W_{\widetilde{f}}(h)=\int\limits_{V(F)\backslash V({\bf
A})}\widetilde{f}(vh)\psi_V(v)dv$$ where $\psi_V$ is defined as
follows. For $v\in V$ write
$v=x_{1000}(r_1)x_{0120}(r_2)x_{0111}(r_2)v'$.
Then $\psi_V(v)=\psi(r_1+r_2)$. See subsection 2.1 for notations. Repeating the same expansions as
in the proof of the cuspidality, we obtain zero contribution
except from the term which corresponds to the identity
representative in the computation of $I_3$. In other words
$W_{\widetilde{f}}(h)$ is equal to
$$\int\limits_{N(F)\backslash SO_3({\bf
A})}\varphi(g)\int\limits_{(F\backslash {\bf A})^2}
\sum_{\delta\in F,\epsilon\in F^*}
\theta^{U,\psi}(h_2(\epsilon)w[232]x_{0110}(1) x_{0120}(\delta)
(y(r_1,r_2)h,g))\psi(r_1+r_2)dr_idg$$ where
$y(r_1,r_2)=x_{1000}(r_1)x_{0120}(r_2)x_{0111}(r_2)$. We have
$$h_2(\epsilon)w[232]x_{0110}(1) x_{0120}(\delta)x_{1000}(r_1)=
x_{1000}(\epsilon^{-1}r_1)u'h_2(\epsilon)w[232]x_{0110}(1)
x_{0120}(\delta)$$ where $u'\in U$ is such that $\psi_U(u')=1$. Thus
we get $\int\limits_{F\backslash {\bf A}}\psi((1-\epsilon^{-1})r)dr$
as inner integration, and hence $\epsilon=1$. Next, collapsing
summation over $\delta_2$ with the integration over $r_2$,  the

above integral is equal to
$$\int\limits_{N(F)\backslash SO_3({\bf
A})}\varphi(g)\int\limits_{\bf A}
\theta^{U,\psi}(w[232]x_{0110}(1) x_{0120}(r_2)
(h,g))\psi(r_2)dr_2dg$$ Factoring the integration over $N$, we
obtain the identity
$$W_{\widetilde{f}}(h)=\int\limits_{N({\bf A})\backslash SO_3({\bf
A})}W_\varphi(g)\int\limits_{\bf A}
\theta^{U,\psi}(w[232]x_{0110}(1) x_{0120}(r_2)
(h,g))\psi(r_2)dr_2dg$$ where $W_\varphi(g)$ is the Whittaker
coefficient attached to $\varphi$. This completes the proof of the
Proposition.
\end{proof}

\subsubsection{\bf From $\widetilde{G_2}$ to $SO_3$}

To study this lift, we consider a different embedding of the
commuting pair. First, we embed the group $G_2$ inside $F_4$ as the
group generated by all unipotent elements
$<x_{\pm(1000)}(r)x_{\pm(0010)}(r);x_{\pm(0100)}(m)>$. This
embedding is the standard embedding of the group $G_2$ inside
$Spin_7$. The group $SO_3$ is the group generated by
$<x_{\pm(0001)}(r) x_{\pm(1231)}(-r)x_{\pm(1232)}(-r^2)>$. This
embedding and the one introduced in the previous subsection are
related by conjugation of the Weyl element $w[231234]$. We shall
denote by $V$ the unipotent radical subgroup of the standard Borel
subgroup of $G_2$ embedded as above.

Let $\widetilde{\pi}$ denote a cuspidal irreducible representation
of the group $\widetilde{G_2}({\bf A})$. Let
$\sigma(\widetilde{\pi})$ denote the automorphic representation of
$SO_3({\bf A})$ generated by all functions of the form
$$f(g)=\int\limits_{G_2(F)\backslash G_2({\bf
A})}\widetilde{\varphi}(h)\theta((h,g))dh$$ We start with
\begin{proposition}\label{g22.2}
The representation $\sigma(\widetilde{\pi})$ is a cuspidal
representation of $SO_3({\bf A})$.
\end{proposition}
\begin{proof}
Let $x(r)=x_{0001}(r)x_{1231}(-r)x_{1232}(-r^2)$. We need to
prove that the integral
$$\int\limits_{G_2(F)\backslash G_2({\bf
A})}\int\limits_{F\backslash {\bf
A}}\widetilde{\varphi}(h)\theta((h,x(r)g))drdh$$ is zero
for all choice of data. We expand the integral along the group
$U_2$. This group was defined in the beginning of the proof of
Proposition \ref{sl42.2}. As explained there, there are only two
orbits which contributes nonzero terms. They correspond to the
constant term and to the set of all nonzero vectors with zero
length. Thus, the above integral is equal to
$$\int\limits_{G_2(F)\backslash G_2({\bf
A})}\int\limits_{F\backslash {\bf
A}}\widetilde{\varphi}(h)\theta^{U_2}((h,x(r)g))drdh+
\int\limits_{S(2)(F)\backslash G_2({\bf A})}\int\limits_{F\backslash
{\bf
A}}\widetilde{\varphi}(h)\theta^{U_2,\psi}((h,x(r)g))drdh$$
Denote the first summand by $I'$ and the second one by $I''$. In the
above integral,  $S(2)$ is the subgroup of $G_2$ generated by
$<x_{\pm(0100)}(r),V>$. Also, we denoted
$$\theta^{U_2,\psi}(m)=\int\limits_{U_2(F)\backslash U_2({\bf
A})}\theta(u_2m)\psi_{U_2}(u_2)du_2$$ where $\psi_{U_2}$ was defined
in the proof of  Proposition \ref{sl42.2}. We mention, that in the
computation of $I''$ we used the fact that $G_2$ acts transitively
on the set of all nonzero vectors with zero length.

We start with $I''$. Let $U=U_{\alpha_2,\alpha_3,\alpha_4}$. Expand the integral along the group 
$U/Z$
with points in $F\backslash {\bf A}$. Using Proposition
\ref{property2} we obtain two terms. Thus, $I''$ is equal to
$$\int\limits_{S(2)(F)\backslash G_2({\bf
A})}\int\limits_{(F\backslash {\bf
A})^2}\widetilde{\varphi}(h)\theta^{U}(x_{0122}(r_1)
(h,x(r)g))\psi(r_1)dr_1drdh+$$
$$\int\limits_{S(2)(F)\backslash G_2({\bf
A})}\int\limits_{F\backslash {\bf
A}}\widetilde{\varphi}(h)\int\limits_{Z({\bf A})U_2(F)\backslash
U_2({\bf A})}\sum_{\substack{\gamma\in Q(F)\backslash Sp_6(F)\\
\epsilon\in F^*} } \theta^{U,\psi}(h_2(\epsilon)\gamma
u_2(h,x(r)g))\psi_{U_2}(u_2)du_2drdh$$ Arguing as in the
proof of Proposition \ref{sl31.1}, it is not hard to check that the
second summand is zero. As for the first one, after conjugation and
changing variables in $U$, we obtain
$$\int\limits_{S(2)(F)\backslash G_2({\bf
A})}\int\limits_{(F\backslash {\bf
A})^2}\widetilde{\varphi}(h)\theta^{U}(x_{0121}(r)x_{0122}(r_1)
(h,g))\psi(r_1)dr_1drdh$$ The function
$$L(h)= \int\limits_{(F\backslash {\bf
A})^2}\theta^{U}(x_{0121}(r)x_{0122}(r_1) (h,g))\psi(r_1)dr_1dr$$ is
left invariant by the unipotent radical $V(2)$ of the maximal
parabolic subgroup $S(2)$. Indeed, we have
$$V(2)=\{x_{1000}(m_1)x_{0010}(m_1)x_{1100}(m_2)x_{0110}(m_2)
x_{1110}(m_3)x_{0120}(m_3)x_{1120}(m_4)x_{1220}(m_5)\}$$ Changing
variables in $U$ the above integral is equal to
$$\int\limits_{(F\backslash {\bf
A})^2}\theta^{U}(x_{0121}(r)x_{0122}(r_1)
y(m_1,m_2,m_3))\psi(r_1)dr_1dr$$ where $y(m_1,m_2,m_3)=x_{0010}(m_1)
x_{0110}(m_2)x_{0120}(m_3)$. It follows from Proposition
\ref{eisen1} that $\theta^U$ is the  representation $\Theta_6$ defined
on $\widetilde{Sp}_6$ right before Proposition \ref{eisen2}. Thus, the above integral is equal to
$$\int\limits_{(F\backslash {\bf A})^2}\theta_6
\left [ \begin{pmatrix} 1&&&&r&r_1\\ &1&&&&r\\ &&1&&&\\ &&&1&&\\
&&&&1&\\ &&&&&1\end{pmatrix}\begin{pmatrix} 1&&&&&\\ &1&m_1&m_2&m_3&\\
&&1&&m_2&\\ &&&1&-m_1&\\ &&&&1&\\ &&&&&1\end{pmatrix}\right
]\psi(r_1)drdr_1$$ 
Denote the right most matrix by $m$. Plug in the above integral expansion \eqref{minexp1} in Proposition \ref{propertysp2}. The first 
term contributes zero since we obtain $\int\limits_{F\backslash {\bf A}}\psi(r_1)dr_1$ as inner
integration. The second summand in expansion \eqref{minexp1}, when plugged inside the above integral
can be written as a union of cells given by \eqref{cell}. It is not hard to check that the first two
cells contribute zero. The last cell contributes 
$$\sum_{\delta_i\in F} \int\limits_{(F\backslash {\bf A})^2}\theta_6^{U(GL_3),\psi}
\left [ w\begin{pmatrix} 1&\delta_1&\delta_2&&&\\ &1&&&&\\ &&1&&&\\ &&&1&&-\delta_2\\
&&&&1&-\delta_1\\ &&&&&1\end{pmatrix}\begin{pmatrix} 1&&&&r&r_1\\ &1&&&&r\\ &&1&&&\\ &&&1&&\\
&&&&1&\\ &&&&&1\end{pmatrix}m\right ]\psi(r_1)drdr_1$$ 
Here, in equation \eqref{minexp1} we wrote $U(GL_3)$ instead of $U$. Also,
$$w=\begin{pmatrix} w_0&\\ &w_0^*\end{pmatrix};\ \ \ \ w_0=\begin{pmatrix} &1&\\ &&1\\ 1&&\end{pmatrix}$$
Conjugating the matrix with
the $r$ and $r_1$ variable to the left, after changing variables in $U(GL_3)$, we may assume that  
$\delta_1=0$. Thus we obtain 
$$\sum_{\delta_2\in F}\theta_6^{U(GL_3),\psi}\left [ w\begin{pmatrix} 1&&\delta_2&&&\\ &1&&&&\\ &&1&&&\\ &&&1&&-\delta_2\\ &&&&1&\\ &&&&&1\end{pmatrix}m\right ]$$ Conjugating $m$ to the left,
changing variables in $U(GL_3)$ and using equation \eqref{minexp2} implies that the above sum
is in fact left invariant under $m\in \widetilde{Sp}_6({\bf A})$.

From this we conclude  that $L(vh)=L(v)$ for all $v\in V(2)$. Since $V(2)$ is  a unipotent radical of a maximal parabolic
subgroup of $G_2$, it follows that the integral $I''$ is zero by the
cuspidality of $\widetilde{\pi}$.

Next we consider $I'$. As in the proof of Proposition
\ref{sl42.2}, it follows that this integral is equal to
$$\int\limits_{G_2(F)\backslash G_2({\bf
A})}\widetilde{\varphi}(h)\theta^{U(B_3)}((h,g))dh$$ where $U(B_3)=U_{\alpha_1,\alpha_2,\alpha_3}$.
To prove that this integral is zero for all
choice of data, let $E(h,s)$ denote the Eisenstein series associated
with the induced representation $Ind_{L({\bf A})}^{G_2({\bf
A})}\delta_L^s$. Here $L$ is the maximal parabolic subgroup of $G_2$
which preserves a line. Consider the integral
$$\int\limits_{G_2(F)\backslash G_2({\bf
A})}\widetilde{\varphi}(h)\theta^{U(B_3)}((h,1))E(h,s)dh$$ As in
\eqref{SL311}, unfolding the Eisenstein series, using Proposition
\ref{eisen1}, we show that this integral is zero for all $Re(s)$
large. Thus its residue at $s=1$ is zero, from which it follows that
$I'=0$. Thus the lift is cuspidal.
\end{proof}

Next we shall give a criterion for the lift to be nonzero. To do
that, let
$$V_1=<x_{(1000)}(r)x_{(0010)}(r);x_{(1100)}(r)x_{(0110)}(r);
x_{(1110)}(r)x_{(0120)}(r);x_{(1120)}(r);x_{(1220)}(r)>$$ Thus,
$V_1$ is a  unipotent radical of the maximal parabolic subgroup of
$G_2$ which preserves a line. We construct a projection from $V_1$
to ${\mathcal H}_3$, the Heisenberg group with three variables,
defined as follows. Write $v\in V_1$ as
$$v=x_{(1000)}(r_1)x_{(0010)}(r_1)x_{(1100)}(r_2)x_{(0110)}(r_2)
x_{(1110)}(r_3)x_{(0120)}(r_3)x_{(1120)}(r_4)x_{(1220)}(r_5)$$ Then
we define $l:V_1\mapsto {\mathcal H}_3$ as $l(v)=(r_1,r_2,r_3)$.
Here, we identify elements in ${\mathcal H}_3$ as triples, where the
third coordinate is the center of ${\mathcal H}_3$. The group $SL_2$
generated by $<x_{\pm(0100)}(r)>$ normalizes the group $V_1$. We
have
\begin{proposition}\label{g23.3}
The representation $\sigma(\widetilde{\pi})$ is nonzero, if and
only if the integral
\begin{equation}\label{g23}
\int\limits_{SL_2(F)\backslash SL_2({\bf
A})}\int\limits_{V_1(F)\backslash V_1({\bf
A})}\widetilde{\varphi}(vm)\theta_{SL_2}^{\phi,\psi}(l(v)m)dvdm
\end{equation}
is nonzero for some choice of data. Here
$\theta_{SL_2}^{\phi,\psi}$ is a vector in the space of $\Theta_{SL_2}^\psi$, the theta representation of
${\mathcal H}_3({\bf A})\cdot\widetilde{SL}_2({\bf A})$.
\end{proposition}
\begin{proof}
Keeping the notations in the proof of Proposition \ref{g22.2}, the
lift is nonzero for some choice of data, if and only if the
integral
$$W_f(g)=\int\limits_{G_2(F)\backslash G_2({\bf
A})}\int\limits_{F\backslash {\bf
A}}\widetilde{\varphi}(h)\theta((h,x(l_1)g))\psi(l_1)dl_1dh$$
is nonzero for some choice of data. Here $x(l_1)$ was defined
in the beginning of the proof of Proposition \ref{g22.2}. Arguing as
in the proof of Proposition \ref{g22.2}, we obtain
$$W_f(g)=\int\limits_{S(2)(F)\backslash G_2({\bf
A})}\int\limits_{(F\backslash {\bf
A})^2}\widetilde{\varphi}(h)\theta^{U}(x_{0121}(l_1)x_{0122}(l_2)
(h,g))\psi(l_1+l_2)dl_idh$$ Here $S(2)=V_1\cdot SL_2$ where
the $SL_2$ is generated by $<x_{\pm(0100)}(r)>$. Factoring the
integration over $S(2)$, and plugging $g=e$, then $W_f(e)$ is
equal to
$$\int\limits_{S(2)({\bf A})\backslash G_2({\bf
A})}\int\limits_{SL_2({F})\backslash SL_2({\bf
A})}\int\limits_{V_1({F})\backslash V_1({\bf
A})}\int\limits_{(F\backslash {\bf
A})^2}\widetilde{\varphi}(v_1mh)\times $$
$$\theta^{U}(x_{0121}(l_1)x_{0122}(l_2)
v_1m(h,e))\psi(l_1+l_2)dl_idv_1dmdh$$ Arguing in a similar way as in
\cite{Ga-S}, the above integral is zero for all choice of data if
and only if the integral
$$\int\limits_{SL_2({F})\backslash SL_2({\bf
A})}\int\limits_{V_1({F})\backslash V_1({\bf
A})}\int\limits_{(F\backslash {\bf A})^2}\widetilde{\varphi}(v_1m)
\theta^{U}(x_{0121}(l_1)x_{0122}(l_2)v_1m)\psi(l_1+l_2)dl_idv_1dm$$
is zero for all choice of data. From the description of $V_1$ in
terms of roots in $F_4$, it follows after a change of variables in
$U$, that the above integral is equal to
$$\int\limits_{SL_2({F})\backslash SL_2({\bf
A})}\int\limits_{V_1({F})\backslash V_1({\bf
A})}\int\limits_{(F\backslash {\bf A})^2}\widetilde{\varphi}(v_1m)
\theta^{U}(x_{0121}(l_1)x_{0122}(l_2)y(m_1,m_2,m_3)m)\psi(l_1+l_2)dl_idv_1dm$$
where $y(m_1,m_2,m_3)=x_{0010}(m_1) x_{0110}(m_2)x_{0120}(m_3)$.
From Proposition \ref{eisen1}, this integral is zero for all
choice of data if and only if the integral
$$\int\limits_{SL_2({F})\backslash SL_2({\bf
A})}\int\limits_{(F\backslash {\bf A})^7}
\widetilde{\varphi}(k(m_1,m_2,m_3,r_1,r_2)m)\theta_6
(z(l_1,l_2) y(m_1,m_2,m_3)m) \psi(l_1+l_2)dl_idm_jdm$$ is zero for
all choice of data. Here
$z(l_1,l_2)=I_6+l_1(e_{1,5}+e_{2,6})+l_2e_{1,6}$, and
$y(m_1,m_2,m_3)=I_6+m_1(e_{2,3}-e_{4,5})+m_2(e_{2,4}+e_{3,5})
+m_3e_{2,5}$, both matrices in $Sp_6$. Also,
$k(m_1,m_2,m_3,r_1,r_2)$ is equal to
$$x_{1000}(m_1)x_{0010}(m_2)x_{1100}(m_2)x_{0110}(m_2)
x_{1110}(m_3)x_{0120}(m_3)x_{1120}(r_1)x_{1220}(r_2)$$ 
Next we use Proposition \ref{propertysp3} to obtain that the above integral is equal to
$$\int\limits_{SL_2({F})\backslash SL_2({\bf
A})}\int\limits_{(F\backslash {\bf A})^7}
\widetilde{\varphi}(k(m_1,m_2,m_3,r_1,r_2)m)\theta_{Sp_4}^{\phi,\psi}
(h(l_2) y(m_1,m_2,m_3)m) \psi(l_1)dl_1dm_jdm$$ 
Here $\theta_{Sp_4}^{\phi,\psi}$ is a vector in the
space of $\Theta_{Sp_4}^{\psi}$ which is the theta representation defined on the group  
${\mathcal H}_5({\bf A})\cdot\widetilde{Sp}_4({\bf A})$. The element $h(l_2)$ is an element in 
${\mathcal H}_5({\bf A})$ which is equal to $h(l_2)=(0,0,0,l_2,0)$. Here we view elements of 
${\mathcal H}_5({\bf A})$ as defined in \cite{I1}. Applying the
formulas of the Weil representation, see \cite{I1}, we obtain integral \eqref{g23}
as inner integration. Arguing again in a similar way as in
\cite{Ga-S}, the Proposition follows.
\end{proof}

\subsection{\bf The Commuting Pair $(SL_2,Sp_6)$}

In this subsection we study the lifting from automorphic
representations of $\widetilde{SL}_2({\bf A})$ to automorphic
representations of $\widetilde{Sp}_6({\bf A})$, and also the lifting
in the other direction.

\subsubsection{\bf From $\widetilde{SL}_2({\bf A})$ to $\widetilde{Sp}_6({\bf
A})$}

To study this lift we consider the following embedding of the groups
$SL_2$ and $Sp_6$ in $F_4$. First we embed the group $SL_2$ as the
group $<x_{\pm 0100}(r)>$. The embedding of $Sp_6$ is as the group
generated by $<x_{\pm 0120}(r); x_{\pm 0001}(r); x_{\pm 1110}(r)>$. These two
embedding do not split under the double cover. The group $Sp_6$ has
three maximal parabolic subgroups, and we denote their unipotent
radical by $V_i$ for $1\le i\le 3$. The roots inside these three
unipotent groups are $\{ (1110; (1111); (1231); (1232);
(2342)\}$ in $V_1$, $\{ (0001); (1111); (0121); (1231); (0122); (1232);
(2342)\}$ in $V_2$, and $\{ (0120); (0121); (0122);$
$(1231); (1232); (2342)\}$ in $V_3$.

Let $\widetilde{\pi}$ denote an irreducible  cuspidal representation
of $\widetilde{SL}_2({\bf A})$. The lift we consider is given by
\begin{equation}\label{sl2sp6.1}
f(g)=\int\limits_{SL_2(F)\backslash SL_2({\bf
A})}\widetilde{\varphi}(h) \theta(((h,g))dh\notag
\end{equation}
We denote by $\sigma(\widetilde{\pi})$ the automorphic
representation of $\widetilde{Sp}_6({\bf A})$ generated by the above
functions. The result we prove is
\begin{proposition}\label{propsl2sp6.1}

Let $\widetilde{\pi}$ denote an irreducible cuspidal representation
of $\widetilde{SL}_2({\bf A})$ which lift to a cuspidal
representation of $GL_2({\bf A})$. Then the representation
$\sigma(\widetilde{\pi})$ is nonzero. Assume also that integral 
\eqref{weird} is zero for all choice of data. Then, the constant terms of this
representation along the unipotent groups $V_1$ and $V_3$ are zero.

\end{proposition}

\begin{proof}

We start with the computation of the constant terms along the groups
$V_1$ and $V_3$. Thus, for $i=1,3$ we need to prove that the
integral
\begin{equation}\label{sl2sp6.2}
\int\limits_{SL_2(F)\backslash SL_2({\bf A})}
\int\limits_{V_i(F)\backslash V_i({\bf A})}\widetilde{\varphi}(h)
\theta((h,v))dvdh
\end{equation}
is zero for all choice of data. Let $Z=\{x_{2342}\}$. Then $Z\subset
V_i$ and using Proposition \ref{property2} integral \eqref{sl2sp6.2}
is equal to
\begin{equation}\label{sl2sp6.3}
\int\limits_{SL_2(F)\backslash SL_2({\bf A})}\int\limits_{Z({\bf
A})V_i(F)\backslash V_i({\bf A})}\widetilde{\varphi}(h)
\theta^U(((h,v))dvdh  +
\end{equation}
$$\int\limits_{SL_2(F)\backslash SL_2({\bf
A})}\widetilde{\varphi}(h)\int\limits_{Z({\bf A})V_i(F)\backslash
V_i({\bf A})} \sum_{\gamma\in Q(F)\backslash
Sp_6(F)}\sum_{\epsilon\in F^*}\theta^{U,\psi}(h_2(\epsilon)\gamma
(v,g))dvdh$$ The first summand is zero. Indeed, it follows from
Proposition \ref{eisen1} that we obtain the integral
\begin{equation}\label{sl2sp6.4}
\int\limits_{SL_2(F)\backslash SL_2({\bf
A})}\widetilde{\varphi}(h)\theta_6(h)dh
\end{equation}
as inner integration. Here, the embedding of $SL_2$ inside $Sp_6$ is given by
$h\to \text{diag}(1,1,h,1,1)$. Also, $\theta_6$ is a vector in the space of the representation $\Theta_6$ which was defined right before Proposition \ref{eisen2}. Let $Z$ denote the subgroup
of $Sp_6$ which was defined right before Proposition \ref{propertysp3}. Expanding the above integral 
along $Z(F)\backslash Z({\bf A})$, we consider first the contribution from the nontrivial characters. To that we use Proposition \ref{propertysp3} to obtain the integral
$$\int\limits_{SL_2(F)\backslash SL_2({\bf
A})}\widetilde{\varphi}(h)\theta_{Sp_4}^{\phi',\psi_\beta}(h)dh$$ as inner integration. Applying the
Theta representation properties, see \cite{I1} and \cite{G-R-S6}, we obtain the integral
$$\int\limits_{SL_2(F)\backslash SL_2({\bf
A})}\widetilde{\varphi}(h)\theta_{SL_2}^{\phi',\psi}(h)dh$$ 
It follows from the assumption on $\widetilde{\pi}$ that this integral
is zero for all choice of data. Thus, integral \eqref{sl2sp6.4} is equal to 
$$\int\limits_{SL_2(F)\backslash SL_2({\bf
A})}\widetilde{\varphi}(h)\theta_6^Z(h)dh$$
The quotient $U(GL_1\times Sp_4)/Z$ is abelian. Here $U(GL_1\times Sp_4)$ was defined right before Proposition \ref{propertysp1}. Expanding along this quotient, and using the fact that $\Theta_6$ is
a minimal representation, the above integral is equal to
$$\int\limits_{SL_2(F)\backslash SL_2({\bf
A})}\widetilde{\varphi}(h)\theta_6^{U(GL_1\times Sp_4)}(h)dh$$
We proceed with these Fourier expansions, and using the minimality of $\Theta_6$, we deduce that
the above integral is equal to 
\begin{equation}\label{weird}
\int\limits_{SL_2(F)\backslash SL_2({\bf
A})}\widetilde{\varphi}(h)\theta_6^V(h)dh
\end{equation}
Here $V$ is the unipotent radical of the parabolic subgroup of $Sp_6$ whose Levi part is $GL_1^2
\times SL_2$. This integral is zero by assumption. Hence the first summand in \eqref{sl2sp6.3}
is zero for all choice of data.

Next we compute the second summand in \eqref{sl2sp6.3}. Let $P$
denote the maximal parabolic subgroup of $Sp_6$ whose Levi part is
$GL_1\times Sp_4$. The space of double cosets $Q(F)\backslash Sp_6(F)/P(F)$
contains two elements, and we take $e$ and $w[234]$ as
representatives. Denote by $I_1$ the contribution from $e$ and by
$I_2$ the contribution from $w[234]$. Then
$$I_1=\int\limits_{SL_2(F)\backslash SL_2({\bf
A})}\widetilde{\varphi}(h)\int\limits_{Z({\bf A})V_i(F)\backslash
V_i({\bf A})}\sum_{\gamma\in S(3)(F)\backslash
Sp_4(F)}\sum_{\epsilon\in F^*}\theta^{U,\psi}(h_2(\epsilon)\gamma
(v,g))dvdh$$ Here $S(3)$ is the maximal parabolic subgroup of $Sp_4$
which contains the group $< x_{\pm 0010}>$. The space of double
cosets $S(3)(F)\backslash Sp_4(F)/S(2)(F)$ contains two elements with
representatives $e$ and $w[23]$. Here $S(2)$ is the maximal
parabolic subgroup of $Sp_4$ whose Levi part contains the group
$<x_{\pm 0100}>$. The contribution to $I_1$ from $e$ is given by
$$\int\limits_{B(F)\backslash SL_2({\bf
A})}\widetilde{\varphi}(h)\int\limits_{Z({\bf A})V_i(F)\backslash
V_i({\bf A})}\sum_{\epsilon\in F^*}\theta^{U,\psi}(h_2(\epsilon)
(v,g))dvdh$$ Here $B$ is the Borel subgroup of $SL_2$. Using
Proposition \ref{property3} this integral is zero by cuspidality of
$\widetilde{\pi}$. This follows from the fact that the function
$g\mapsto \theta^{U,\psi}(h_2(\epsilon) (v,g))$ is left invariant
under $\{x_{0100}(r)\}$ with $r\in {\bf A}$. The contribution to $I_1$
from $w[23]$ is given by
$$\int\limits_{B(F)\backslash SL_2({\bf
A})}\widetilde{\varphi}(h)\int\limits_{Z({\bf A})V_i(F)\backslash
V_i({\bf A})}\sum_{\delta_i\in F}\sum_{\epsilon\in
F^*}\theta^{U,\psi}(h_2(\epsilon)w[23]x_{0010}(\delta_1)x_{0120}(\delta_2)
(v,g))dvdh$$ If $\delta_1$ is zero, arguing as in the previous
integral, using Proposition \ref{property3}, we get zero
contribution by the cuspidality of $\widetilde{\pi}$. Assume
$\delta_1\ne 0$. If $i=1$, then $V_1$ contains the root $(1110)$.
Conjugating $x_{1110}(r)$ from right to left, using commutation
relations and Proposition \ref{property3}, we obtain
$\int\limits_{F\backslash {\bf A}}\psi(\epsilon^{-1}\delta_1r)dr$ as
inner integration. Since $\delta_1$ and $\epsilon$ are nonzero this
integral is zero. When $i=3$, the group $V_3$, contains
$\{x_{0120}(r)\}$. Collapse the summation over $\delta_2$ with the
corresponding integration, we then get that the function $h\to
\int\limits_{\bf A}\theta^{U,\psi}(
h_2(\epsilon)w[23]x_{0010}(\delta_1)x_{0120}(r)(1,h))dr$ is left
invariant under $\{x_{0100}(m)\}$ for all $m\in {\bf A}$. Thus we get
zero by cuspidality. Thus $I_1=0$.

We are left with $I_2$ which is equal to
$$\int\limits_{SL_2(F)\backslash SL_2({\bf
A})}\widetilde{\varphi}(h)\int\limits_{Z({\bf A})V_i(F)\backslash
V_i({\bf A})}\sum_{\gamma\in S(3)(F)\backslash Sp_4(F)}$$
$$\sum_{\delta_i\in F,\ \epsilon\in
F^*}\theta^{U,\psi}(h_2(\epsilon)w[234]x_{0001}(\delta_1)
x_{0011}(\delta_2)x_{0122}(\delta_3)\gamma (v,g))dvdh$$ As above, we
take $e$ and $w[23]$ as representatives for $S(3)(F)\backslash
Sp_4(F)/S(2)(F)$, and so $I_2$ is a sum of two integrals which we denote
by $I_{21}$ and $I_{22}$. The integral $I_{21}$ is equal to
$$\int\limits_{B(F)\backslash SL_2({\bf
A})}\widetilde{\varphi}(h)\int\limits_{Z({\bf A})V_i(F)\backslash
V_i({\bf A})}\sum_{\delta_i\in F}\sum_{\epsilon\in
F^*}\theta^{U,\psi}(h_2(\epsilon)w[234]x_{0001}(\delta_1)
x_{0011}(\delta_2)x_{0122}(\delta_3)(v,g))dvdh$$ If $\delta_2$ equal
zero, then we get zero by cuspidality. This follows from the fact
that the function $g\mapsto
\theta^{U,\psi}(h_2(\epsilon)w[234]x_{0001}(\delta_1)
x_{0122}(\delta_3)(v,g))$ is left invariant by $\{x_{0100}(r)\}$ with
$r\in {\bf A}$. Assume $\delta_2\ne 0$. The group $V_1$ contains the
root $(1111)$. Conjugating by $x_{1111}(r)$, we obtain
$\int\limits_{F\backslash {\bf A}}\psi(\epsilon^{-1}\delta_2r)dr$ as
inner integration, and hence we get zero. The group $V_3$, contains
the root $(0122)$. As in the case of $I_1$, we collapse summation
and integration , and then get zero by cuspidality. We are left with
$I_{22}$ which is equal to
$$\int\limits_{B(F)\backslash SL_2({\bf
A})}\widetilde{\varphi}(h)\int\limits_{Z({\bf A})V_i(F)\backslash
V_i({\bf A})}\sum_{\delta_i\in F}\sum_{\epsilon\in
F^*}\theta^{U,\psi}(h_2(\epsilon)w[23423]y(\delta_1,\delta_2,
\delta_3,\delta_4,\delta_5)(v,g))dvdh$$ where
$$y(\delta_1,\delta_2,
\delta_3,\delta_4,\delta_5)=x_{0010}(\delta_1)x_{0011}(\delta_2)
x_{0120}(\delta_3)x_{0121}(\delta_4)x_{0122}(\delta_5)$$ Both
unipotent subgroups contains the two roots $(1231)$ and $(1232)$. We
have
$$x_{0011}(\delta_2)x_{1231}(r)=x_{1242}(r\delta_2)x_{1231}(r)x_{0011}(\delta_2)$$
Since $w[23423]$ conjugates the root $(1242)$ to $(1000)$, it
follows that if $\delta_2\ne 0$, then the contribution to the above
integral is zero. Indeed, using these commutation relations, and a
change of variables in $U$, we obtain the integral
$\int\psi(\delta_2\epsilon r)dr$ as inner integration.  Similarly,
using the root $(1232)$ we deduce that the contribution from
$\delta_1\ne 0$ is zero. When $\delta_1=\delta_2=0$ we once again
use the left invariance of $\theta^{U,\psi}$ by $x_{0100}(r)$ with
$r\in {\bf A}$, to get zero by cuspidality. This completes the proof
that the constant terms along the unipotent groups $V_1$ and $V_3$
are zero.

Next we consider the question of the non vanishing of the lift. 
We will prove that there is a choice of data such that the integral
\begin{equation}\label{nonvanish1}
\int\limits_{ (F\backslash {\bf A})^6}f(x_{0120}(r_1)x_{0121}(r_2)
x_{0122}(r_3)x_{1231}(r_4)x_{1232}(r_5)x_{2342}(r_6))\psi(\beta
r_1+\gamma r_3+r_6)dr_i
\end{equation}
is not zero for some $\beta,\gamma\in F^*$. Assume not. Then, for
all $\beta$ and $\gamma$ and all choice of data, this integral is
zero. Plugging this into the definition of the lift, we deduce that
for all choice of data, the integral
$$\int\limits_{SL_2(F)\backslash SL_2({\bf
A})}\widetilde{\varphi}(h)\int\limits_{V(F)\backslash V({\bf A})}
\theta((h,v))\psi_{V,\beta,\gamma}(v)dvdh$$ is zero. Here we wrote
$V$ and $\psi_{V,\beta,\gamma}$  for the group generated by the 6 roots in \eqref{nonvanish1}
and for the character of this group.
The group $V$ is abelian. Let $U_1=\{x_{1231}(r), x_{1232}(r)\}$ and
$U_2=\{ x_{0120}(r), x_{0121}(r), x_{0122}(r)\}$.  Using Proposition
\ref{property6} we deduce that for all choice of data, the integral
$$\int\limits_{SL_2(F)\backslash SL_2({\bf
A})}\widetilde{\varphi}(h)\int\limits_{U_1(F)\backslash U_1({\bf
A})}\int\limits_{U_2(F)\backslash U_2({\bf
A})}\theta_{Sp_{14}}^{\phi,\psi}(\iota(u_1)\varpi_3(u_2h))
\psi_{U_2,\beta,\gamma}(u_2)du_1du_2dh$$ is zero. We describe the
embedding of the various groups inside the Heisenberg group
${\mathcal H}_{15}$ and in $Sp_{14}$. We use the parametrization as
described in integral \eqref{nonvanish1}. First, inside the
Heisenberg group we have
$$x_{1231}(r_4)x_{1232}(r_5)=(0,\ldots,0,r_4,r_5,0,0,0)$$ where the
last coordinate is the center of the Heisenberg group. Here we
identify  the group ${\mathcal H}_{15}$ with a 15 tuple. See \cite{I1}. Next
$$x_{0120}(r_1)=I_{14}+r_1e'_{1,5}+r_1e'_{2,6}+r_1e_{4,11},\ \ \
x_{0122}(r_3)=I_{14}+r_3e'_{1,9}+r_3e'_{2,10}+r_3e_{3,12}$$ Here
$e_{i,j}$ denotes the matrix of size 14 which has one at the $(i,j)$
entry and zero elsewhere, and $e'_{i,j}=e_{i,j}-e_{15-j,15-i}$. The
above two matrices are in $Sp_{14}$.  Also, we have
$$x_{0121}(r_2)=I_{14}+r_2e'_{1,7}+r_2e'_{2,8}+r_2e'_{3,11}$$
Finally, the group $SL_2$ is embedded in $Sp_{14}$ as $h\to
\text{diag}(h,I_2,h,h,h^*,I_2,h^*)$.

In the above integral we unfold the theta function. Using the action
of the Weil representation under the Heisenberg group, see
\cite{G-R-S6}, we deduce that $\xi_3=\xi_4=0$. Thus the integral
$$\int\limits_{SL_2(F)\backslash SL_2({\bf
A})}\widetilde{\varphi}(h)\int\limits_{(F\backslash {\bf A})^3}$$
$$\sum_{\xi_i\in
F}\omega_\psi(x_{0120}(r_1)x_{0121}(r_2)x_{0122}(r_3)h) \phi(
\xi_1,\xi_2,0,0,\xi_5,\xi_6,\xi_7)\psi(\beta r_1+\gamma r_3)dr_idh$$
is zero for all choice of data. From the embedding of the group
$SL_2$ in $Sp_{14}$, and from the action of the Weil
representation, we obtain that the group $SL_2(F)$ acts on the first two
coordinates $\xi_1$ and $\xi_2$ with two orbits. The trivial orbit
contributes zero. Indeed, from the embedding of the unipotent group
$\{x_{0120}(r_1)\}$ inside $Sp_{14}$, we obtain the integral
$\int\psi(\beta r_1)dr_1$ as inner integration.  Thus the above
integral is equal to
$$\int\limits_{N(F)\backslash SL_2({\bf
A})}\widetilde{\varphi}(h)\int\limits_{(F\backslash {\bf A})^3}$$
$$\sum_{\xi_i\in
F}\omega_\psi(x_{0120}(r_1)x_{0121}(r_2)x_{0122}(r_3)h) \phi(
0,1,0,0,\xi_5,\xi_6,\xi_7)\psi(\beta r_1+\gamma r_3)dr_idh$$ which
is zero for all choice of data. Here $N$ is the unipotent radical of
the Borel subgroup of $SL_2$. Applying the integration over $r_2$
and then over $r_3$, and arguing as above, we deduce  that $\xi_7=0$
and $\xi_5=\gamma$. Collapsing the summation over $\xi_6$ with the
integration over $r_1$ we obtain that the integral
$$\int\limits_{N(F)\backslash SL_2({\bf
A})}\widetilde{\varphi}(h)\int\limits_{{\bf A}}\omega_\psi(h)
\phi(0,1,0,0,\gamma,r_1,0)\psi(\beta r_1)dr_1dh$$ is zero for all
choice of data. Finally, factoring the integration over $N$ we
obtain that the integral
$$\int\limits_{N({\bf A})\backslash SL_2({\bf
A})}W^{\psi,\beta,\gamma}_{\widetilde{\varphi}}(h) \int\limits_{{\bf
A}}\omega_\psi(h)\phi(0,1,0,0,\gamma,r_1,0)\psi(\beta r_1)dr_1dh$$
is zero for all choice of data. Here
$$W^{\psi,\beta,\gamma}_{\widetilde{\varphi}}(h)=
\int\limits_{F\backslash {\bf A}}\widetilde{\varphi}\left (\begin{pmatrix}
1&y\\ &1\end{pmatrix}h\right )\psi(-\beta\gamma y)dy$$ Using a similar
argument as in \cite{Ga-S} we deduce that
$W^{\psi,\beta,\gamma}_{\widetilde{\varphi}}(h)$ is zero for all
$\beta$ and $\gamma$. This is clearly a contradiction.

\end{proof}

\subsubsection{\bf  From $\widetilde{Sp}_6$ to $\widetilde{SL}_2$}

In this case we choose the following embedding of the two groups.
First, the $Sp_6$ is generated by $<x_{\pm 0100}(r); x_{\pm 0010}(r);
x_{\pm 0001}(r)>$, and the $SL_2$ is generated by $<x_{\pm 2342}(r)>$.
This embedding is conjugated by the Weyl element $w[3124321]$ to the
embedding of the two groups as was described in the previous
subsection.

Let $\widetilde{\pi}$ denote an irreducible cuspidal representation
of $\widetilde{Sp}_6({\bf A})$. The lift we consider is
\begin{equation}\label{sp6sl2.1}
f(g)=\int\limits_{Sp_6(F)\backslash Sp_6({\bf
A})}\widetilde{\varphi}(h) \theta((h,g))dh\notag
\end{equation}
We prove the following

\begin{proposition}

Let $\sigma(\widetilde{\pi})$ denote the automorphic representation
of $\widetilde{SL}_2$ generated by the above functions. Then
$\sigma(\widetilde{\pi})$ is a cuspidal representation. It is
nonzero if and only if the integral
$$\int\limits_{Sp_6(F)\backslash
Sp_6({\bf A})}\widetilde{\varphi}(h)\theta_{Sp_{14}}^{\phi,\beta}
(\varpi_3(h))dh$$ is nonzero for some choice of data. Here $\beta\in
F^*$.

\end{proposition}

\begin{proof}

The representation $\sigma(\widetilde{\pi})$ is nonzero if and only
if the integral
$$\int\limits_{Sp_6(F)\backslash Sp_6({\bf
A})}\int\limits_{F\backslash {\bf A}}\widetilde{\varphi}(h)
\theta((h,x_{2342}(r)))\psi(\beta r)drdh$$ is nonzero for some
choice of data. Thus the claim about the nonvanishing follows from
Proposition \ref{property6}.

As for the cuspidality, we use Proposition \ref{property2} to write
the constant term of the $SL_2$ as
$$\int\limits_{Sp_6(F)\backslash Sp_6({\bf
A})}\widetilde{\varphi}(h)\theta^U((h,1))dh +
\int\limits_{Q^0(F)(F)\backslash Sp_6({\bf
A})}\widetilde{\varphi}(h)\theta^{U,\psi}((h,1))dh$$ where $Q^0$ is
the subgroup of $Q$, which is the semidirect product of $SL_3$ and
the unipotent radical of $Q$, the group  $U(Q)$. (See Proposition \ref{propertysp1}). The first summand is zero because of Proposition \ref{eisen1}. As for the second
summand, it follows from Proposition \ref{property3} that the
function $h\mapsto \theta^{U,\psi}((h,1))$ is left invariant by the
group $U(Q)({\bf A})$. Thus, we obtain zero by the cuspidality of
$\widetilde{\pi}$.

\end{proof}

\subsection{\bf The Liftings as a Functorial Lifting}

The next step, and an important one, is to determine which of the
above constructions defines a functorial liftings. It is also an
interesting problem to see if each of these pairs satisfy the
unramified Howe duality property.

More precisely, let $(H,G)$ be one of the above commuting pair.
Thus, if $\pi$ is a cuspidal irreducible representation of $G({\bf
A})$, or its double cover, and if $\sigma$ is a cuspidal irreducible
representation of $H({\bf A})$, or its double cover, we are
interested in the cases when the global integral
\begin{equation}\label{unra1}
\int\limits_{H(F)\backslash H({\bf A})} \int\limits_{H(F)\backslash
H({\bf A})} \varphi_\sigma(h)\varphi_\pi(g)\theta((h,g))dhdg
\end{equation}
is not zero for some choice of data. Here $\varphi_\sigma$ is a
vector in the space of $\sigma$, $\varphi_\pi$ is a vector in the
space of $\pi$, and $\theta$ is a vector in the space of $\Theta$.
Following \cite{G-R-S4} pages 606-608, then the nonvanishing of the
integral \eqref{unra1} implies that at any local place there is a
nonzero such a trilinear form. In other words, let $\nu$ be a place
where all representations are unramified. Let
$\sigma_\nu=Ind_{B(H)}^H\chi$ denote an unramified representation of
$H$ or its double cover, at the place $\nu$. When there is no
confusion, we shall omit $\nu$ from the notations. Similarly, let
$\pi_\nu=Ind_{B(G)}^G\mu$ denote an unramified representation of $G$
or its double cover at the place $\nu$. Then we assume that the
space
$$Hom_{G\times H}(Ind_{B(G)}^G\mu\times Ind_{B(H)}^H\chi,\theta)$$
is not zero. Here $\theta$ is the local unramified constituent of
$\Theta$ at the place $\nu$. The unramified Howe duality property
states that given $\chi$ and $\mu$ as above, then each one of these
characters determine uniquely the other.

{\bf Conjecture:} {\sl All the five commuting pairs, which were
described in the beginning of this Section, satisfy the local
unramified Howe duality property.}

In each of the five cases we studied we will now give a conjectural description of the lift.

{\bf 1)}\ $(SL_3,SL_3)$. Here the construction is from the space  of cuspidal representation defined on $\widetilde{GL}_3({\bf A})$ to the space of automorphic representations of $SL_3({\bf A})$. 
The conjectural functorial lift is the well known Shimura lift. Some information at the role of
the orthogonal period which we obtained  can be found in \cite{J2}.

{\bf 2)}\ $(SL_2\times SL_2, Sp_4)$. Here the conjectural lift is the endoscopic lift. In more details, the corresponding $L$ groups of $\widetilde{SL}_2({\bf A})$ and $\widetilde{Sp}_4({\bf A})$
are $SL_2({\bf C})$ and $Sp_4({\bf C})$. Hence, the conjecture lift in this case is corresponding to the homomorphism from  $SL_2({\bf C})\times SL_2({\bf C})$ into $Sp_4({\bf C})$. This lift is a
special case of the more general construction as studied in \cite{G-R-S7}.

{\bf 3)}\ $(SL_2, SL_4)$. The conjectural lift in this case is a special case of the conjecture stated in \cite{S}. We state it for our case. Let $\pi$ denote an irreducible cuspidal representation of $GL_2({\bf A})$. Suppose that $\pi$ is a functorial lift from $\widetilde{GL}_2({\bf A})$ which is given by the Shimura lift. Then it is conjectured in \cite{S} that $\pi$ has
a nontrivial lift to $\widetilde{GL}_4({\bf A})$ and that the image is a certain residue representation. If $\pi$ is not in the image of the Shimura correspondence, then the conjecture states that $\pi$ has a nontrivial functorial lift to a cuspidal representation of $\widetilde{GL}_4({\bf A})$. Thus, we conjecture that this commuting pair will yield this lift. We remark that
the conjecture stated in \cite{S} is for all cuspidal representations of $GL_n({\bf A})$.

{\bf 4)}\ $(SO_3, G_2)$. This is an extension of the Classical Theta lift in the symplectic groups. In details, let $\pi$ denote a cuspidal irreducible representation of $SO_3({\bf A})$. As follows from \cite{R}, if the lift to $\widetilde{SL}_2({\bf A})$ is zero, then the lift to $\widetilde{Sp}_4({\bf A})$ is a generic cuspidal representation. Here the lift is obtained using the minimal representation of $\widetilde{Sp}_{2n}({\bf A})$ where $n=3, 6$.
The conjecture in this case is that the same phenomena occurs with the exceptional group $G_2$ replacing the group $Sp_4$. In other words, if the lift of $\pi$ to $\widetilde{SL}_2({\bf A})$ is zero, then the lift to $\widetilde{G}_2({\bf A})$ is a generic cuspidal representation.
 
{\bf 5)}\ $(SL_2, Sp_6)$. In this case we showed that the image is not cuspidal. We conjecture that we obtain a residual representation of  $\widetilde{Sp}_6({\bf A})$ which we now describe. Let $\pi$ denote an irreducible cuspidal representation of $\widetilde{SL}_2({\bf A})$. Suppose that $\pi$ 
has a functorial lift to a cuspidal representation $\tau$ of $GL_2({\bf A})$. Then the partial tensor product $L$ function $L_\psi(\pi\times\tau,s')$ has a simple pole at $s'=1$. Form the Eisenstein series $E_{\tau,\pi}(g,s)$ defined on $\widetilde{Sp}_6({\bf A})$, which is associated with the induced representation $Ind_{Q'({\bf A})}^{\widetilde{Sp}_6({\bf A})}(\tau\times\pi)\delta_Q^s$. Here $Q'$ is the subgroup of $\widetilde{Sp}_6$ defined as follows. Let $Q$ denote the maximal parabolic subgroup of $Sp_6$ whose Levi part is $GL_2\times SL_2$. Let $U(Q)$ denote its unipotent radical. Then $Q'=(GL_2\times \widetilde{SL}_2)U(Q)$. Its not hard to check that this Eisenstein series has a simple pole at $s_0$ corresponding to the point $s'=1$. If we denote the
residue representation by ${\mathcal E}_{\tau,\pi}$, then the conjecture is that this is the representation obtained in this case.

\section{\bf  Global Split Descent Constructions}

In this Section we consider some global descent constructions. We
briefly recall the setup for this construction in the context of the
group $F_4$ ( for classical groups see \cite{G-R-S7}). Let ${\mathcal O}$ denote a
unipotent orbit of $F_4$. It follows from \cite{C}, that the
stabilizer of each such orbit inside a suitable Levi subgroup, is a
reductive group. As explained in Section 2, to each such orbit we
can associate a set of Fourier coefficients. Thus, to each such
orbit, we attach a unipotent group $U_\Delta$, and a set of
characters $\psi_{U,u_\Delta}$. Let $H$ denote the connected component of the reductive part of
the stabilizer of the character $\psi_{U,u_\Delta}$.  In this paper
we will only consider those characters $\psi_{U,u_\Delta}$ such that
the group $H$ is split. In some cases, one can also consider
characters such that the stabilizer is an anisotropic group.
However, for the analysis of when the lift is cuspidal and the study
of the Fourier coefficients of the lift, the split case is the
hardest case and the more interesting one.  Hence, we refer to these
constructions as split descent constructions.

Let $\mathcal{E}$ denote an automorphic representation of the group
$F_4$. In principle there is no reason not to consider also automorphic representation defined on metaplectic covering of the group $F_4$. To avoid issues related to cocyles, we shall restrict to representations of $F_4$ only. There are two cases to consider. The first, is when the
Dynkin diagram attached to the unipotent orbit ${\mathcal O}$ is a
diagram whose all nodes are labeled with zeros and twos. 
In the notations of subsection 2.2 we have in this case $U_\Delta=U_\Delta(2)$. In this
case we consider the space of functions
\begin{equation}\label{desc1}
f(h)=\int\limits_{U_\Delta(F)\backslash U_\Delta({\bf
A})}E(uh)\psi_{U,u_\Delta}(u)du
\end{equation}
Here, $E$ is a vector in the space of $\mathcal{E}$.  Thus, $f(h)$
defines an automorphic function on the group $H({\bf A})$. We denote
by $\sigma$ the representation of $H({\bf A})$ generated by all
functions $f(h)$. We refer to the representation $\sigma$ as the
descent representation of $\mathcal{E}$. If the representation
$\mathcal{E}$ depends on an automorphic representation $\tau$ of
another group, we sometimes refer to $\sigma$ as the descent
representation from $\tau$.

The second case is when the diagram attached to the unipotent orbit
contains also ones. In this case $U_\Delta=U_\Delta(1)\ne U_\Delta(2)$. In other words, the set $U_\Delta'(1)$ is not empty.  Therefore, there is a projection from the group $U_\Delta$
onto a suitable Heisenberg group. In particular the stabilizer $H$
has an embedding into a suitable symplectic group. In \cite{G-R-S3}
there is a detailed discussion of this situation for unipotent
orbits of the symplectic groups. However, these ideas hold for any
algebraic group. In this case we consider the integral
\begin{equation}\label{desc100}
f(h)=\int\limits_{U_\Delta(F)\backslash U_\Delta({\bf
A})}\widetilde{\theta}_{Sp}^{\psi,\phi}(l(u)h)E(uh)\widetilde{\psi}_{U,u_\Delta}(u)du
\end{equation}
Here  $l$ denotes the projection from $U_\Delta$ onto the Heisenberg
group. The function $\widetilde{\theta}_{Sp}^{\psi,\phi}$ is a vector in 
$\widetilde{\Theta}_{Sp}^\psi$,
the minimal representation of the double cover of the suitable
symplectic group. The character $\widetilde{\psi}_{U,u_\Delta}$ is
defined such that when combined with the character of the theta
function it produces the character $\psi_{U,u_\Delta}$. For more
details see \cite{G-R-S3} page 4 formula (1.3). The function $f(h)$
defined in \eqref{desc100} is left invariant under the rational
points of $H$. However, depending on the embedding of $H$ inside the
symplectic group, it may be a genuine function on
$\widetilde{H}({\bf A})$, the double cover of $H({\bf A})$.   

By unfolding the theta function in integral \eqref{desc100} we may
associate with this integral two more integrals which are related to
the unipotent orbit ${\mathcal O}$. The relation, as explained in
details in \cite{G-R-S3} Lemma 1.1, is that one integral is zero for
all choice of data if and only if the other is zero for all choice
of data. We briefly explain the relation. Denote $U_\Delta''
=U_\Delta(2)$. The second integral which is related to \eqref{desc100} is
\begin{equation}\label{desc101}
\int\limits_{U_\Delta''(F)\backslash U_\Delta''({\bf
A})}E(u''h)\psi_{U,u_\Delta}(u'')du''
\end{equation}
where the character  $\psi_{U,u_\Delta}$ was defined in subection 2.2.
The third related integral is defined as follows. Consider the set of roots
$U_\Delta'(1)$. Then there is a choice, in fact more than one choice, to
extend the group $U_\Delta''$ to a unipotent group $U_\Delta'$ such
that $U_\Delta''\subset U_\Delta'\subset U_\Delta$ and which
satisfies the following. The extension of $U_\Delta'$ is obtained by
adding half of the roots in $U_\Delta'(1)$ to $U_\Delta''$ in such a
way that the character $\psi_{U,u_\Delta}$ is extended trivially to
$U_\Delta'$. The integral we then consider is
\begin{equation}\label{desc102}
\int\limits_{U_\Delta'(F)\backslash U_\Delta'({\bf
A})}E(u'h)\psi_{U,u_\Delta}(u')du'
\end{equation}
These two last integrals were denoted in \cite{G-R-S3} by (1.1) and
(1.2). Lemma 1.1 in that reference states that if one of these three
integrals is zero for all choice of data, then the other two also
vanish for all choice of data. The proof is formal and applies to
all algebraic groups.

We illustrate this by an example. Consider the unipotent orbit
$A_1$. Its diagram is
$$  \ \ \  \overset{1}{0}----\overset{} {0} ==>== \overset{}{0}----\overset{}{0}$$
In this case $U_\Delta=U_{\alpha_2,\alpha_3,\alpha_4}$ is the unipotent radical of the maximal
parabolic subgroup of $F_4$ whose Levi part is $GSp_6$. Also,
$U_\Delta$ is isomorphic to ${\mathcal H}_{15}$, the Heisenberg
group consisting of 15 variables and we denote by $l$ this
isomorphism. Hence, integral \eqref{desc100} is given by
\begin{equation}\label{desc103}
f(h)=\int\limits_{U_\Delta(F)\backslash U_\Delta({\bf
A})}{\widetilde{\theta}}_{Sp_{14}}^{\psi,\phi}(l(u)h)E(uh)du
\end{equation}
where $\widetilde{\theta}_{Sp_{14}}^{\psi,\phi}$ is a vector in the minimal
representation of $\widetilde{Sp}_{14}({\bf A})$.  The connected component of the stabilizer of this unipotent orbit is the group $Sp_6$. In this case
the automorphic function $f(h)$, and the representation $\sigma$
defines a genuine automorphic function and an automorphic
representation on the group $\widetilde{Sp}_6({\bf A})$. Since $U_\Delta(2)=\{x_{2342}(r)\}$ ( see subsection 2.1 for notations), then in this example, integral \eqref{desc101} is
\begin{equation}\label{desc104}
\int\limits_{F\backslash {\bf A}}E(x_{2342}(r)h)\psi(r)dr
\end{equation}
To describe integral \eqref{desc102} we need to choose half
of the roots in $U_\Delta'(1)$, in such a way that we can extend the
character from $\{x_{2342}(r)\}$ trivially. The choice of these roots is
not unique. For example, one can choose the following roots
$$A=\{ (1122);\ (1221);\ (1222);\ (1231);\ (1232);\ (1242);\ (1342)\}$$
Thus, the group $U_\Delta'$ is the unipotent group generated by
$\{x_\alpha(r)\}$ where $\alpha\in A$ together with the root $(2342)$. The character
$\psi_{U,u_\Delta}$ is defined as follows. For $u'=x_{2342}(r)u_1'$  set $\psi_{U,u_\Delta}(u')=\psi(r)$ ( see subsection 2.1).
Thus, $\psi_{U,u_\Delta}$ is the trivial extension of the character
given in \eqref{desc104} from $U_\Delta''$ to $U_\Delta'$.

As mentioned in the introduction, our goal is to look for those
unipotent orbits, such that the integrals which define the descent
satisfies the dimension identity \eqref{intro4}. There are two
cases. When the nodes of the  diagram attached to the unipotent
orbit consists of zeros and twos, then the descent is given by
integral \eqref{desc1}. In this case, since  the
representation $\pi$, as defined in the introduction, is trivial,
the dimension identity we consider is
\begin{equation}\label{dim}
\text{dim}\ \mathcal{E}= \text{dim}\ U_\Delta + \text{dim}\ \sigma
\end{equation}
If the diagram contains also ones, then the descent is given by
integral \eqref{desc100}. In this case we also need to take into
account the theta representation on the symplectic group. Thus we
obtain
$$\text{dim}\ \mathcal{E}+\text{dim}\ \widetilde{\Theta}_{Sp}^{\psi}=
\text{dim}\ U_\Delta + \text{dim}\ \sigma$$

We have
$\text{dim}\ \widetilde{\Theta}_{Sp}^{\psi}=\frac{1}{2}(\text{dim}\ {\mathcal
H}-1)$ where ${\mathcal H}$ is the corresponding Heisenberg group.
Since this number is equal to a half of the roots in $U_\Delta'(1)$,
we obtain that the dimension formula for this case is given by
\begin{equation}\label{dim1}
\text{dim}\ \mathcal{E}= \text{dim}\ U_\Delta' + \text{dim}\ \sigma
\end{equation}
where the group $U_\Delta'$ was defined above.

We remark that in both cases one can show that the dimension of
$U_\Delta$ in the first case, and the dimension of $U_\Delta'$ in
the second case is equal to half of the dimension of the unipotent
orbit in question as listed in \cite{C-M} page 128. Thus, if we denote this unipotent orbit by
${\mathcal O}$, then equations \eqref{dim} and \eqref{dim1} are
given by
\begin{equation}\label{dim2}
\text{dim}\ \mathcal{E}= \frac{1}{2}\text{dim}\ {\mathcal O} +
\text{dim}\ \sigma
\end{equation}

\subsection{ The dimensions for $F_4$ Descents }\label{dimdesc}

In this subsection we consider all possible unipotent orbits such
that either integral \eqref{desc1} or integral \eqref{desc100}
satisfies the dimension identity \eqref{dim2}. The list of the
unipotent orbits and their stabilizers can be found in \cite{C}. We
only consider those orbits whose stabilizer contains a nontrivial
reductive group.  The dimension of $U_\Delta$ or $U_\Delta'$, which
is half of the dimension of the corresponding unipotent orbit, can
be found in \cite{C-M} page 128.
\subsubsection{\bf  The Unipotent Orbits $C_3$ and $B_3$} For these
orbits the stabilizer is a group of type $A_1$. Thus the
representation $\sigma$ is defined on that group, and hence
$\text{dim}\ \sigma=1$. Since the dimension of the two orbits is 42,
then $\frac{1}{2}\text{dim}\ {\mathcal O}=21$. Hence we look for
representations $\mathcal{E}$ such that
$\text{dim}\ \mathcal{E}=21+1=22$.
\subsubsection{\bf The Unipotent Orbit $C_3(a_1)$} The stabilizer is a
group of type $A_1$. The dimension of this unipotent orbit is 38,
and hence $\frac{1}{2}\text{dim}\ {\mathcal O}=19$. Thus
$\text{dim}\ \mathcal{E}=19+1=20$. 
\subsubsection{\bf The Unipotent Orbit $\widetilde{A}_2+A_1$} The
stabilizer is a group of type $A_1$, and the dimension of
$\frac{1}{2}\text{dim}\ {\mathcal O}$ is 18.  We have $\text{dim}\ \mathcal{E}=18+1=19$.
\subsubsection{\bf  The Unipotent Orbit $B_2$} The stabilizer is a group
of type $A_1\times A_1$. The dimension of
$\frac{1}{2}\text{dim}\ {\mathcal O}$ is 18, and hence
$\text{dim}\ \mathcal{E}=18+2=20$.
\subsubsection{\bf The Unipotent Orbit $A_2+\widetilde{A}_1$} The
stabilizer is a group of type $A_1$. The dimension of $\frac{1}{2}\text{dim}\ {\mathcal O}$ is 17.
Hence $\text{dim}\ \mathcal{E}=18$.
\subsubsection{\bf The Unipotent Orbit $\widetilde{A}_2$} Here the
stabilizer is the exceptional group $G_2$. The dimension of
$\frac{1}{2}\text{dim}\ {\mathcal O}$ is 15. Cuspidal representations $\sigma$
on $G_2({\bf A})$ can be generic, and in this case
$\text{dim}\ \sigma=6$, or, if not generic, they are associated to the
unipotent orbit $G_2(a_1)$. In this case $\text{dim}\ \sigma=5$. Thus,
there are two cases to consider. The first is
$\text{dim}\ \mathcal{E}=21$, and the second
$\text{dim}\ \mathcal{E}=20$.
\subsubsection{\bf The Unipotent Orbit $A_2$} Here the stabilizer is
a group of type $A_2$. The dimension of
$\frac{1}{2}\text{dim}\ {\mathcal O}$ is 15, and hence
$\text{dim}\ \mathcal{E}=18$.
\subsubsection{\bf The Unipotent Orbit $A_1+\widetilde{A}_1$} The
stabilizer is a group of type $A_1\times A_1$. As
$\frac{1}{2}\text{dim}\ {\mathcal O}=14$ in this case, then
$\text{dim}\ \mathcal{E}=16$. As it follows from \cite{C-M} there is
no unipotent orbit whose dimension is 32. Hence we do not expect
that a suitable $\mathcal{E}$ will exist in this case.
\subsubsection{\bf The Unipotent Orbit $\widetilde{A}_1$} The stabilizer
is a group of type $A_3$, the dimension of
$\frac{1}{2}\text{dim}\ {\mathcal O}$ is 11, and hence
$\text{dim}\ \mathcal{E}=11+6=17$.
\subsubsection{\bf The Unipotent Orbit $A_1$} Here the stabilizer is
the group $Sp_6$. Cuspidal representations on $Sp_6$ can be attached
to one of the unipotent orbits, $(6), (42)$ or $(2^3)$. Their
dimensions are $9,8$ and 6. The dimension of
$\frac{1}{2}\text{dim}\ {\mathcal O}$ is 8, and hence we expect
$\text{dim}\ \mathcal{E}=17,16$ or $14$. As mentioned above we do not
expect that a representation of dimension 16 exists for the group
$F_4$.

\subsection{\bf How to Compute Descents}

In this subsection we give some general remarks on how to compute a
descent integral. More precisely, a typical computation of a descent
construction consists of two type of computations. The first is the
computation of all constant terms of the representation $\sigma$ corresponding 
to unipotent radicals of maximal parabolic subgroups of $H$, and the second is a computation
of a certain Fourier coefficient. The first computation is done to
determine conditions when the descent is cuspidal, and the second is
done to determine when the descent is nonzero. Usually, the
computation of the constant term is harder since it involve many
unipotent orbits. In this Section we will only consider the
computation of a certain Fourier coefficient of the descent.
However, we will say a few words on the computation of the constant
terms at the end of the next subsection.

Let $\mathcal{O}$ be a unipotent orbit, and let $\mathcal{E}$ be an
automorphic representation defined on the group $F_4({\bf A})$. The group $U_\Delta'$ was defined for unipotent orbits whose diagram contains nodes labelled with the number one. It is convenient to extend the definition of the group $U_\Delta'$ to unipotent orbits whose diagrams contain nodes labelled with zeros and twos only. In this case we denote $U_\Delta'=U_\Delta$. In this way we defined the 
group $U_\Delta'$ for all unipotent orbits.
From
the discussion in the previous subsections, we are led to consider
integrals of the type
\begin{equation}\label{comp1}
\int\limits_{V(F)\backslash V({\bf
A})}\int\limits_{U_\Delta'(F)\backslash U_\Delta'({\bf A})}
E(uvh)\psi_{U,u_\Delta}(u)\psi_V(v)dvdu
\end{equation}
The group $V$ is a
certain unipotent subgroup of the stabilizer of the character
$\psi_{U,u_\Delta}$. The character $\psi_V$ is a character, possibly
the trivial one,  of the group $V(F)\backslash V({\bf A})$.

\subsubsection{\bf Unipotent Orbits and Torus Elements}

It is convenient to express things in more generality. Let $G$ be an
algebraic reductive group, and let $\mathcal{O}_G$ denote a unipotent orbit
for $G$. As explained in Section 2 for the group $G=F_4$, and in \cite{G1}
for an arbitrary classical group, 
to this orbit we associate a unipotent subgroup
$U(\mathcal{O}_G)$ of $G$, and a set of characters
$\psi_{U(\mathcal{O}_G),u_0}$ of this group. Here $u_0$ is an
element in the unipotent orbit $\mathcal{O}_G$ which defines the
character. Given an automorphic
representation $\mathcal{E}$ of $G$, we shall denote by
$\mathcal{O}_{G,u_0}(\mathcal{E})$ the Fourier coefficient given by
\begin{equation}\label{comp2}
f(h)=\int\limits_{U(\mathcal{O}_G)(F)\backslash
U(\mathcal{O}_G)({\bf A})} E(uh)\psi_{U(\mathcal{O}_G),u_0}(u)du
\end{equation}
If $H$ is a reductive group contained in the stabilizer of 
this unipotent orbit, then the function $f(h)$ is an automorphic
function of $H({\bf A})$. Let $\sigma(\mathcal{E})$ denote the
automorphic representation of $H({\bf A})$ generated by all the
functions $f(h)$ in \eqref{comp2}. If $\sigma$ is an arbitrary
automorphic representation of $H$, then given a unipotent orbit
$\mathcal{O}_H$, then as for the group $G$, we shall denote by
$\mathcal{O}_{H,v_0}(\sigma)$ the Fourier coefficient
\begin{equation}\label{comp3}\notag
\int\limits_{V(\mathcal{O}_H)(F)\backslash V(\mathcal{O}_H)({\bf
A})} \varphi_\sigma(v)\psi_{V(\mathcal{O}_H),v_0}(v)dv
\end{equation}
Here $V(\mathcal{O}_H)$ is the  unipotent subgroup of $H$ which
correspond to the unipotent orbit $\mathcal{O}_H$. Similarly,
$\psi_{V(\mathcal{O}_H),v_0}$ is the character attached to a
representative $v_0$ of this orbit.

One of the goals of the descent method is to compute the integral
\begin{equation}\label{comp4}
\int\limits_{V(\mathcal{O}_H)(F)\backslash V(\mathcal{O}_H)({\bf
A})} f(v)\psi_{V(\mathcal{O}_H),v_0}(v)dv=
\end{equation}
$$\int\limits_{V(\mathcal{O}_H)(F)\backslash V(\mathcal{O}_H)({\bf
A})} \int\limits_{U(\mathcal{O}_G)(F)\backslash
U(\mathcal{O}_G)({\bf A})} E(uv)\psi_{U(\mathcal{O}_G),u_0}(u)
\psi_{V(\mathcal{O}_H),v_0}(v)dvdu$$ This is a certain Fourier
coefficient defined on an automorphic function $E$ which lies in the 
space of a representation ${\mathcal E}$ of the group $G({\bf A})$. We
shall denote it by $\mathcal{O}_{G,u_0}(\mathcal{E})\circ
\mathcal{O}_{H,v_0}(\sigma(\mathcal{E}))$. Thus, the goal is to
express this Fourier coefficient in term of Fourier coefficients
attached to unipotent orbits of $G$. However, it is possible that we
will also obtain some constant terms in the course of this
computation. Let $P=MU$ denote a parabolic subgroup of $G$. The
constant term
$$E^U(m)=\int\limits_{U(F)\backslash U({\bf A})}E(um)du$$ defines an
automorphic representation of $M({\bf A})$. We shall denote this
representation by $\mathcal{E}^U$. If $\mathcal{O}_M$ is a unipotent
orbit of $M$, we shall denote by $\mathcal{CT}_{G,P}
[\mathcal{O}_{M,l_0}(\mathcal{E}^U)]$ the Fourier coefficient
$$\int\limits_{L(\mathcal{O}_M)(F)\backslash L(\mathcal{O}_M)({\bf
A})}E^U(l) \psi_{L(\mathcal{O}_M),l_0}dl$$

Here $L(\mathcal{O}_M)$ is the  unipotent subgroup of $M$ which
correspond to the unipotent orbit $\mathcal{O}_M$. 
Thus, to express integral \eqref{comp4} in term of Fourier
coefficients attached to unipotent orbits of $G$, and to Fourier
coefficients associated with constant terms along certain unipotent
radicals of some parabolic subgroups of $G$, is to determine an
identity of the type
\begin{equation}\label{comp5}
\mathcal{O}_{G,u_0}(\mathcal{E})\circ
\mathcal{O}_{H,v_0}(\sigma(\mathcal{E}))=
\sum_i(\mathcal{O}_i)_{G,u_i}(\mathcal{E}) + \sum_j
\mathcal{CT}_{G,P_j} [(\mathcal{O}_j)_{M,l_j}(\mathcal{E}^U)]
\end{equation}
In words, the goal is to express the Fourier coefficient defined by integral \eqref{comp4}, as a sum of two type of integrals. The first term is a sum of Fourier coefficients which corresponds to unipotent orbits of the group $G$. Thus, we want to determine the precise unipotent orbits ${\mathcal O}_i$ appearing in the first sum. The second term
on the right hand side of equation \eqref{comp5} is a sum
of constant terms corresponding to unipotent radicals of certain parabolic subgroups of $G$. In this case we need to determine which parabolic subgroup are involved, and also
what are the unipotent orbits of the group $M$ which are involved. 

The identity \eqref{comp5} is completely formal in the sense that it does not require any information on the representation ${\mathcal E}$. From this point of view, this identity can be viewed as a formal identity defined on the level of unipotent orbits only.

To make things clear we consider an example in the group $Sp_4$. We
refer the reader to \cite{G-R-S6} for notations. In \cite{G-R-S2}
the descent map from $GL_{2n}$ to $\widetilde{Sp}_{2n}$ was
introduced. To do that one uses an automorphic representation of
$Sp_{4n}$ which is defined as a residue of a certain Eisenstein
series. Consider the more general case of the descent when $n=1$.
Thus we consider the following integral
\begin{equation}\label{sp4sl21}
f(g)=\int\limits_{(F\backslash {\bf
A})^2}\widetilde{\theta}_{\psi^{\beta},\phi}((r,x,y)g)E\left (
\begin{pmatrix} 1&r&x&y\\ &1&&x\\ &&1&-r\\ &&&1\end{pmatrix}
\begin{pmatrix} 1&&\\ &g&\\&&1\end{pmatrix}\right )dxdydrdg
\end{equation}
Here $\beta\in F^*$ and $E$ is a vector in some 
automorphic representation ${\mathcal E}$ of $Sp_4({\bf A})$. Thus $f(g)$ is an automorphic representation of
$\widetilde{SL}_2({\bf A})$. We  denote by $\sigma(\mathcal{E})$ the
representation of this group generated by all functions $f(g)$. To
study when it is nonzero, we compute the integral
$$\int\limits_{F\backslash {\bf A}}f\begin{pmatrix} 1&z\\
&1\end{pmatrix}\psi(\gamma z)dz$$ where $\gamma\in F^*$. As stated above, integral \eqref{desc100} is zero for all choice of data if and only if integral \eqref{desc102} is zero for all choice of data. For the group $Sp_4$, this was proved in \cite{G-R-S3}.

For the group   $Sp_4$ the analogues to integral \eqref{desc102} is  the integral
\begin{equation}\label{sp4sl22}
\int\limits_{F\backslash {\bf A}}\int\limits_{(F\backslash {\bf
A})^2} E\left ( \begin{pmatrix} 1&&x&y\\ &1&&x\\ &&1&\\
&&&1\end{pmatrix}
\begin{pmatrix} 1&&&\\ &1&z&\\ &&1&\\ &&&1\end{pmatrix}\right
)\psi(\beta y+\gamma z)dxdydz
\end{equation}
Here  $\gamma\in (F^*)^2\backslash F^*$. The $x,y$ integration,
which in the notation of integral \eqref{comp1} correspond to the
group $U_\Delta'$, is a Fourier coefficient corresponding to the
unipotent orbit of $Sp_4$ associated with the partition $(21^2)$,
and the $z$ integration is the Whittaker coefficient of
$\sigma(\mathcal{E})$, defined on $\widetilde{SL}_2({\bf A})$, and
hence is associated with the unipotent orbit $(2)$. Thus, in the
above notations, the left hand side  of \eqref{comp5} is
$(21^2)_{Sp_4,\beta}(\mathcal{E})\circ
(2)_{\widetilde{Sl}_2,\gamma}(\sigma(\mathcal{E}))$. As explained in
\cite{G1}, the above integral corresponds to the unipotent orbit of
$Sp_4$ associated with the partition $(2^2)$. Thus, equation
\eqref{comp5} is given by
\begin{equation}\label{comp51}
(21^2)_{Sp_4,\beta}(\mathcal{E})\circ
(2)_{\widetilde{Sl}_2,\gamma}(\sigma(\mathcal{E}))=
(2^2)_{Sp_4,\beta,\gamma}(\mathcal{E})
\end{equation}
It follows from \cite{G1} that the unipotent orbit $(2^2)$ is
associated with elements $\beta,\gamma\in (F^*)^2\backslash F^*$.
The above identity is all we can say when $\beta$ and $\gamma$ are
in general position. However, it is of interest to notice that when
$\beta\gamma=-\epsilon^2$ for some $\epsilon\in F^*$ then this
identity can be written in a different form. Indeed, when
$\beta\gamma=-\epsilon^2$, one can find an element in $Sp_4(F)$,
which depends on $\beta$ and $\gamma$, such that the above integral
is equal to
\begin{equation}\label{comp6}\notag
\int\limits_{(F\backslash {\bf
A})^3} E_{\beta,\gamma} \begin{pmatrix} 1&&x&y\\ &1&z&x\\ &&1&\\
&&&1\end{pmatrix}\psi(x)dxdydz
\end{equation}
Here $E_{\beta,\gamma}$ is the right translation of the vector $E$
by this discrete element. See \cite{G-R-S2} page 880 for some
details. Let $w$ denote the Weyl element of $Sp_4$ defined by
$w=e_{1,1}+e_{2,3}-e_{3,2}+e_{4,4}$. Here $e_{i,j}$ is the matrix
of size four which has a one at the $(i,j)$ entry, and zero
otherwise. Since $w\in Sp_4(F)$, then $E(g)=E(wg)$. Hence, the above
integral is equal to
\begin{equation}\label{comp7}\notag
\int\limits_{(F\backslash {\bf
A})^3} E_{\beta,\gamma} \left (\begin{pmatrix} 1&x&&y\\ &1&&\\ &z&1&-x\\
&&&1\end{pmatrix}w\right )\psi(x)dxdydz
\end{equation}
Performing some Fourier expansions, one can show that the above
integral is equal to
\begin{equation}\label{comp8}\notag
\int\limits_{\bf A}\sum_{\delta\in F^*}\int\limits_{R(F)\backslash
R({\bf A})} E_{\beta,\gamma} (rt(z)w)\psi_{R,\delta}(r)drdz+
\int\limits_{\bf A}\int\limits_{F\backslash {\bf
A}}E_{\beta,\gamma}^{U(P)}(m(x)t(z)w)\psi(x)dxdz
\end{equation}
Here $t(z)=I_4+ze_{3,2}$ and $m(x)=I_4+x(e_{1,2}-e_{3,4})$. Also,
the group $R$ is the maximal unipotent subgroup of $Sp_4$, and
$\psi_{R,\delta}$ is the Whittaker character of $R$ defined as
follows. Write $r\in R$ as $r=x_1(r_1)x_2(r_2)r'$. Here
$x_1(r_1)=I_4+r_1(e_{1,2}-e_{3,4})$ and $x_1(r_2)=I_4+r_2e_{2,3}$.
Then we define $\psi_{R,\delta}(r)=\psi(r_1+\delta r_2)$. Finally,
the group $P$ is the maximal parabolic subgroup of $Sp_4$ whose Levi
part is $GL_2$, and we denote by $U(P)$ its unipotent radical.

Ignoring the integration over the $z$ variable, then in the notation of
\eqref{comp5}, when $\beta\gamma=-\epsilon^2$, this integral
identity is given by
\begin{equation}\label{comp9}
(21^2)_{Sp_4,\beta}(\mathcal{E})\circ
(2)_{\widetilde{Sl}_2,\gamma}(\sigma(\mathcal{E}))= \sum_{\delta\in
F^*}(4)_{Sp_4,\delta}(\mathcal{E})+\mathcal{CT}_{Sp_4,P}[(2)_{GL_2}
(\mathcal{E}^{U(P)})]
\end{equation}
We conclude that for some choice of unipotent elements $u_0$ and $v_0$, there is more than one way to write the identity
\eqref{comp5}. Experience indicate the following. There is a general
expression for identity \eqref{comp5} which holds for all values of
$u_0$ and $v_0$, and all representations $\mathcal{E}$. However, in
some cases, there is a closed condition on $u_0$ and $v_0$ which
will yield another identity. This is important once we specify the
representation $\mathcal{E}$. 

As an example to this phenomena,
consider the group $Sp_4$, and the above two identities
\eqref{comp51} and \eqref{comp9}. Let $\tau$ denote an irreducible
cuspidal representation of $GL_2({\bf A})$ with a trivial central
character, and such that $L(\tau,1/2)\ne 0$. Let $E_\tau(g,s)$
denote the Eisenstein series of $Sp_4({\bf A})$ associated with the
induced representation $Ind_{P({\bf A})}^{Sp_4({\bf
A})}\tau\delta_P^s$. The group  $P$ was defined right before identity  \eqref{comp9}. From
the assumptions on $\tau$ it follows that this Eisenstein series has
a simple pole at $s=2/3$, and we denote by $\mathcal{E}_\tau$ the
residue representation at that point. In \cite{G-R-S2} integral 
\eqref{sp4sl21} was used to construct  the descent map from $\tau$ to an
irreducible cuspidal representation of $\widetilde{SL}_2({\bf A})$. The proof of the nonvanishing of the descent used identity \eqref{comp9}. Indeed, it is proved in \cite{G-R-S2}
that the first summand on the right hand side of
\eqref{comp9} is zero and the second term is not.  From this it was  proved in \cite{G-R-S2} that the descent given by integral \eqref{sp4sl21} is not zero.

We may also consider the descent construction given by
\eqref{sp4sl21} where we take $\mathcal{E}$ to be a non-generic
cuspidal representation of $Sp_4({\bf A})$. In this case all
constant terms are zero. Since $\mathcal{E}$ is not generic, we
obtain for such representations that the right hand side of
\eqref{comp9} is zero for all choice of data. Thus, equation
\eqref{comp9} cannot be used in this case. Nevertheless, we can use equation
\eqref{comp51} to deduce the nonvanishing of integral \eqref{sp4sl21}. Indeed, it is
not hard to show that given any automorphic representation of
$Sp_4({\bf A})$, there exist $\beta$ and $\gamma$ as above, such
that integral \eqref{sp4sl22} is not zero for some choice of data.

Going back to the general case, one looks for a way to produce
expansions of the form \eqref{comp5}. To do that we will use the
following approach. As in \cite{C-M},  to any
unipotent orbit, one attaches a one dimensional torus in the group
$G$ in question. ( The notations we use are as in \cite{G1}). For example, the group $Sp_4$ has three nontrivial
unipotent orbits. They are $(4), (2^2)$ and $(21^2)$. The
corresponding one dimensional tori are $h_{(4)}(t)=\text{diag}\
(t^3,t,t^{-1},t^{-3}); h_{(2^2)}(t)=\text{diag}\
(t,t,t^{-1},t^{-1})$ and $h_{(21^2)}(t)=\text{diag}\
(t,1,1,t^{-1})$.

Suppose that we start with a unipotent orbit $\mathcal{O}_G$, and let
$\psi_{U(\mathcal{O}_G),u_0}$ be a character of the unipotent group
$U(\mathcal{O}_G)$. Let $H$ be as defined right before equation \eqref{comp4}, and
suppose that $\mathcal{O}_H$ is a unipotent orbit of $H$. See
\eqref{comp4} for notations. Let $h_{\mathcal{O}_G}(t)$ denote the
one dimensional torus of $G$ attached to $\mathcal{O}_G$, and let
$h_{\mathcal{O}_H}(t)$ denote the one dimensional torus of $H$
attached to $\mathcal{O}_H$. We view $h_{\mathcal{O}_H}(t)$ as a sub
torus of $G$ via the embedding of $H$ in $G$. Thus, the product
$h(t)=h_{\mathcal{O}_G}(t)h_{\mathcal{O}_H}(t)$ is a well defined
one dimensional torus of $G$. Assume that there is a unipotent orbit
$\mathcal{O}_G'$ of $G$ such that $h(t)$ is conjugated by a certain
Weyl element to the torus $h_{\mathcal{O}_G'}(t)$. Conjugating in
\eqref{comp4} the argument of the function $E$ by this Weyl element,
will transform the integral \eqref{comp4} into an integral over a
unipotent subgroup of $U(\mathcal{O}_G')$. Then, using some Fourier
expansions together with possible other conjugations, one can
produce a formula of the type \eqref{comp5}. At this point, we dont
know of a general method that will predict the unipotent orbits and
the constant terms  which appear in equation \eqref{comp5}. As can
be seen from \eqref{comp51} and \eqref{comp9}, the decomposition can
be different if we vary the elements $u_0$ and $v_0$.

As an example to the above argument, consider the above composition
$(21^2)\circ (2)$ in $Sp_4$. Here, to simplify the notations, we
omitted several of them. We have $h_{(21^2)}(t)=\text{diag}\
(t,1,1,t^{-1})$. The torus element which corresponds to the
partition $(2)$ in $SL_2$ is $\text{diag}\ (t,t^{-1})$. When
embedded into $Sp_4$, this torus corresponds to $h(t)=\text{diag}\
(1,t,t^{-1},1)$. Thus we obtain $h_{(21^2)}(t)h(t)= \text{diag}\
(t,t,t^{-1},t^{-1})$ which is equal to $h_{(2^2)}(t)$. Hence we dont
need any conjugation here. The equation that we get is then
$(21^2)\circ (2)=(2^2)$. But one has to remember that in certain
closed conditions on the characters, one can derive another identity
for this composition.

As an another example, consider the product $(2^3)\circ (3)$ in
$Sp_6$. It follows from \cite{C-M} that the reductive group in the stabilizer of the
unipotent orbit $(2^3)$ in $Sp_6$ is the group $SO_3$. Since we
consider only the split stabilizer, the group $SO_3$ contains a one
dimensional unipotent subgroup, and if we compute its Whittaker
coefficient, we are considering the unipotent orbit $(3)$. Thus, the
composition $(2^3)\circ (3)$ corresponds to the integral
\begin{equation}\label{comp91}
\int\limits_{F\backslash {\bf A}}\int\limits_{Mat_{3\times 3}^0(F)
\backslash Mat_{3\times 3}^0({\bf A})}E\left [ \begin{pmatrix}
I_3&X\\ &I_3\end{pmatrix}\begin{pmatrix} m(y)&\\
&m(y)^*\end{pmatrix}\right ]\psi_1(X)\psi_2(y)dXdy
\end{equation}
Here $Mat_{3\times 3}^0=\{r\in Mat_{3\times 3} : J_3r+r^tJ_3=0\}$
and
$$J_3=\begin{pmatrix} &&1\\ &1&\\ 1&&\end{pmatrix}\ \ \ \ \ \ \
m(y)=\begin{pmatrix} 1&y&*\\ &1&-y\\ &&1\end{pmatrix}$$ The star
indicates that the matrix is in $SO_3$. Also, we define
$\psi_1(X)=\psi(x_{1,1}+x_{2,2})$, and $\psi(m_2(y))=\psi(y)$. We
remark that this is not the general character which is associated to
this unipotent orbit, such that the stabilizer is the split $SO_3$.
The general one is given by $X\mapsto \psi(x_{1,1}+\beta x_{2,2})$
where $\beta\in (F^*)^2\backslash F^*$. However, the stabilizer in
each case is the same up to an outer conjugation, and hence the
formulas are the same.

Before conjugation, it will be convenient to transfer integral
\eqref{comp91} to another integral using the process of exchanging
roots. See subsection 2.2.2. In the
above integral we replace the one dimensional unipotent group
$I_6+x_{3,1}e_{3,4}$ in the $X$ variable by
$I_6+y_3(e_{1,3}-e_{4,6})$ and then $I_6+x_{2,1}(e_{2,4}+e_{3,5})$
in the $X$ variable by $I_6+y_2(e_{2,3}-e_{4,5})$. More precisely,
we expand integral \eqref{comp91} along the unipotent group
$I_6+y_2(e_{2,3}-e_{4,5})+y_3(e_{1,3}-e_{4,6})$. Then we conjugate
by a suitable discrete element in $Sp_6(F)$ and then perform a
collapsing of summation with integration. Thus, integral
\eqref{comp91} is equal to
\begin{equation}\label{comp92}\notag
\int\limits_{{\bf A}^2}\int\limits_{(F\backslash {\bf A})^7} E\left
[ \begin{pmatrix} 1&y_1&y_2&&&&\\ &1&y_3&&&\\ &&1&&&\\ &&&1&*&*
\\ &&&&1&*\\ &&&&&1\end{pmatrix}\begin{pmatrix} 1&&&x_1&x_2&x_3\\
&1&&&x_4&x_2\\ &&1&&&x_1\\ &&&1&&
\\ &&&&1&\\ &&&&&1\end{pmatrix}l(z_1,z_2)\right
]\psi(y_1+x_1+x_4)d(...)
\end{equation}
Here $l(z_1,z_2)=I_6+z_1(e_{2,4}+e_{3,5})+z_2e_{3,4}$. 

It follows
from \cite{C-M}  that
$h_{(2^3)}(t)=\text{diag}\ (t,t,t,t^{-1},t^{-1},t^{-1})$. We also have
$h_{(3)}(t)=\text{diag}\ (t^2,1,t^{-2},t^2,1,t^{-2})$, where the
last torus element is the corresponding torus element in $SO_3$ as
embedded in $Sp_6$. Thus, the product of these two tori is given by
$h(t)=\text{diag}\ (t^3,t,t^{-1},t,t^{-1},t^{-3})$. Consider the
Weyl element $w$ of $Sp_6$ given by
$w_{1,1}=w_{2,4}=w_{3,2}=w_{4,5}=w_{6,6}=1$ and $w_{5,3}=-1$. Then
$wh(t)w^{-1}= \text{diag}\ (t^3,t,t,t^{-1},t^{-1},t^{-3})$, and this
torus is equal to $h_{(42)}(t)$.

Since $w\in Sp_6(F)$, the above integral is equal to
\begin{equation}\label{comp93}\notag
\int\limits_{{\bf A}^2}\int\limits_{(F\backslash {\bf A})^7}
E(m(y_i,x_j)wl(z_1,z_2))\psi(y_1+y_2+x_5)d(...)
\end{equation}
where
$$m(y_i,x_j)=\begin{pmatrix} 1&y_1&y_2&&&&\\ &1&&&&\\ &&1&&&\\
&&&1&&-y_2
\\ &&&&1&-y_1\\ &&&&&1\end{pmatrix}\begin{pmatrix} 1&&&x_1&x_2&x_3\\
&1&&x_4&&x_2\\ &&1&x_5&x_4&x_1\\ &&&1&&
\\ &&&&1&\\ &&&&&1\end{pmatrix}$$
Next we expand the above integral along the unipotent group
$l(x_6)=I_6+x_6e_{2,5}$. We obtain
\begin{equation}\label{comp94}\notag
\int\limits_{{\bf A}^2}\sum_{\beta\in F}\int\limits_{(F\backslash
{\bf A})^8} E(m(y_i,x_j)l(x_6)wl(z_1,z_2))\psi(y_1+y_2+x_5+\beta
x_6)d(...)
\end{equation}
Partition the sum in the above integral into two summands. First,
consider the case when $\beta\in F^*$. In this case, it follows from
\cite{G1} that for each $\beta$, the corresponding Fourier
coefficient is associated with the unipotent orbit $(42)$. When
$\beta=0$ we can further manipulate the integral. Indeed,
conjugation by $s=I_6-e_{2,3}+e_{4,5}\in Sp_6(F)$ we obtain the
integral
\begin{equation}\label{comp95}\notag
\int\limits_{{\bf A}^2}\int\limits_{(F\backslash {\bf A})^8}
E(m(y_i,x_j)l(x_6)swl(z_1,z_2))\psi(y_2+x_5)d(...)
\end{equation}
Conjugating by a certain Weyl element, and using further Fourier
expansions, we can show that this integral is a sum of two terms.
The first corresponds to the Whittaker coefficient of the function
$E$, and the second to a certain constant term. We omit the details.
Thus we obtain the formula
\begin{equation}\label{comp96}
(2^3)_{Sp_6}\circ (3)_{SO_3}=\sum_{\beta\in F^*}(42)_{Sp_6,\beta}+
\sum_{\alpha\in
F^*}(6)_{Sp_6,\alpha}+\mathcal{CT}_{Sp_6,P}[(4)_{Sp_4}]
\end{equation}
We close this subsection with two remarks. The first one is
concerning the dimensions of the orbits and the representations in
question. Recall that in considering  possible descent constructions
we required a certain dimension formula to hold. This was equation
\eqref{dim2}, given by
\begin{equation}\label{comp97}
\text{dim}\ \mathcal{E}= \frac{1}{2}\text{dim}\ {\mathcal O} +
\text{dim}\ \sigma\notag
\end{equation}
For the descent to be nonzero, the representation $\mathcal{E}$
should support a nontrivial Fourier coefficient with respect to the
unipotent orbit occurring in the left hand side of \eqref{comp5}.
However,  the dimension of the
unipotent integration which occurs in integral \eqref{comp4} is $\text{dim}\ U({\mathcal O}_G) +\text{dim}\ V({\mathcal O}
_H)$. By definition this number is equal to $\frac{1}{2}\text{dim}\ {\mathcal O_G} +\text{dim} \ \sigma$.
This motivates to look for those
representations $\mathcal{E}$ of $G({\bf A})$ which satisfies the
following. First, that $\mathcal{O}_G(\mathcal{E})$ is equal to a
unipotent orbit corresponding to one of the summands occurring on
the right hand side of \eqref{comp5}. Second, we require that the
representation does not support any Fourier coefficient which
corresponds to any other term which occurs on the right hand side of
\eqref{comp5}.

To illustrate this consider the above two examples in the symplectic group. First, the $Sp_4$ case. Notice that
$\text{dim}\ (21^2)_{Sp_4}=4; \text{dim}\ (2^2)_{Sp_4}=6$ and
$\text{dim}\ (2)_{SL_2}=2$. Hence $\text{dim}\ \mathcal{E}=
\frac{1}{2}\text{dim}\ (21^2)_{Sp_4} + \text{dim}\ \sigma=2+1=3$.
Hence we look for those representations such that  
$\mathcal{O}_{Sp_4}(\mathcal{E})=(2^2)$. This can work if we use
equation \eqref{comp51}. However, if we want to use equation
\eqref{comp9}, then we need to assume also that the representation
$\mathcal{E}$ is not generic. In the $Sp_6$ case the situation is as
follows. The sum of the half of the dimensions  of the unipotent
orbits which occur in the left hand side of \eqref{comp96} is
$6+1=7$. However, half of the dimension of $(42)_{Sp_6}$ is eight
and of $(6)_{Sp_6}$ is nine. Hence the only way to get a term on the
right hand side of \eqref{comp96} whose half of the  dimension is
seven is to look for a representation $\mathcal{E}$ of $Sp_6({\bf
A})$ such that it has no nonzero Fourier coefficient associated with
any representative of the orbits $(42)$ and $(6)$, such that the
integral associated with $\mathcal{CT}_{Sp_6,P}[(4)_{Sp_4}]$ is not
zero. It is not clear if such a representation exists.

The second remark concerns the cuspidality of the descent. The
goal is to compute integral \eqref{comp4} where the group
$V({\mathcal O}_H)$ is a constant term, and the character
$\psi_{V({\mathcal O}_H),v_0}$ is the trivial character. Then,
instead of the left hand side of \eqref{comp5}, one should compute
${\mathcal O}_{G,u_0}({\mathcal E})\circ
\mathcal{CT}(\sigma({\mathcal E}))$. By that we mean that one should
express this convolution as a sum of unipotent orbits of $G$ and
certain constant term of the representation involved. Experience
indicates that at least one of the terms will involve a constant
term of the group $G$, but so far we cannot indicate which one, and
we also cannot predict in general the other terms.

\subsubsection{\bf Unipotent Orbits and Torus Elements for $F_4$}

In this subsection we determine the one dimensional torus element
attached to a given unipotent orbit of $F_4$.

Recall from  subsection 2.2 that we can partition the set of roots
in the group $U_\Delta$ as follows. As in that subsection, we will
say that a root $\alpha$ is in the unipotent group $U_\Delta$, if
the one parameter unipotent subgroup $\{x_\alpha(r)\}$ is a subgroup of $U_\Delta$. For all
$n\ge 1$, we defined $U_\Delta'(n)=\{ \alpha=\sum_{i=1}^4 n_i\alpha_i\in U_\Delta : \sum_{i=1}^4
\epsilon_in_i=n\}$. We can extend this notation and write $U_\Delta(0)$ for
all positive roots in the Levi part of the parabolic group
$P_\Delta$. Let $h_{\mathcal{O}}(t)$ denote the one dimensional
torus of $F_4$ with the property that for all $\alpha\in
U_\Delta'(n)$ we have
\begin{equation}\label{tor}
h_{\mathcal{O}}(t)x_\alpha(r)h_{\mathcal{O}}(t)^{-1}=x_\alpha(t^nr)
\end{equation}
It follows from the Bala-Carter theory that such a torus exists. For
details in the classical groups see \cite{C-M}. To compute this
torus in $F_4$, let
$h_{\mathcal{O}}(t)=h(t^{r_1},t^{r_2},t^{r_3},t^{r_4})$. Then, given
a root $\alpha\in U_\Delta'(n)$,  equation \eqref{tor} reduces to the
equation $\sum_{i=1}^4r_i<\alpha,\alpha_i>=n$. Here
$<\alpha,\alpha_i>$ is the inner product between the root $\alpha$
and the simple root $\alpha_i$. It is easy to solve these equations
in general, and the solution can be derived form the following 4
identities
\begin{equation}\label{comp10}\notag
r_1=\mathcal{G}_{\mathcal{O}}(2342);\ \ \
r_2=r_1+2r_4-\mathcal{G}_{\mathcal{O}}(1122);\ \ \
r_3=\frac{1}{2}(r_2+\mathcal{G}_{\mathcal{O}}(1242));\ \ \
r_4=\mathcal{G}_{\mathcal{O}}(1232)
\end{equation}
Here, for a positive root $\alpha\in U_\Delta'(n)$, we define
$\mathcal{G}_{\mathcal{O}}(\alpha)$ as follows. Let $\alpha=\sum
n_i\alpha_i$ and suppose that the diagram of $\mathcal{O}$ is given
by
$$\overset{\epsilon_1}{\alpha_1}----\overset{\epsilon_2}{\alpha_2} ==>==
\overset{\epsilon_3}{\alpha_3}----\overset{\epsilon_4}{\alpha_4}$$ Then we define
$\mathcal{G}_{\mathcal{O}}(\alpha)=\sum \epsilon_in_i$.

As an example consider the unipotent orbit $\mathcal{O}=B_2$. Its
diagram is
$$\overset{2}{0}----\overset{}{0} ==>==
\overset{}{0}----\overset{1}{0}$$ Hence
$\mathcal{G}_{\mathcal{O}}(2342)=2\cdot 2+1\cdot 2=6$ and
$\mathcal{G}_{\mathcal{O}}(1122)=\mathcal{G}_{\mathcal{O}}(1242)
=\mathcal{G}_{\mathcal{O}}(1232)=4$. Thus, $r_1=6; r_2=10;\ r_3=7;\
r_4=4$ and $h_{B_2}(t)=h(t^6,t^{10},t^7,t^4)$.

We list the set of all 15 tori elements in $F_4$:\\
{\bf 1)}\ $h_{A_1}(t)=h(t^2,t^3,t^2,t)$.\\
{\bf 2)}\ $h_{\widetilde{A}_1}(t)=h(t^2,t^4,t^3,t^2)$.\\
{\bf 3)}\ $h_{A_1+\widetilde{A}_1}(t)=h(t^3,t^6,t^4,t^2)$.\\
{\bf 4)}\ $h_{A_2}(t)=h(t^4,t^6,t^4,t^2)$.\\
{\bf 5)}\ $h_{\widetilde{A}_2}(t)=h(t^4,t^8,t^6,t^4)$.\\
{\bf 6)}\ $h_{A_2+\widetilde{A}_1}(t)=h(t^4,t^8,t^6,t^3)$.\\
{\bf 7)}\ $h_{B_2}(t)=h(t^6,t^{10},t^7,t^4)$.\\
{\bf 8)}\ $h_{\widetilde{A}_2+A_1}(t)=h(t^5,t^{10},t^7,t^4)$.\\
{\bf 9)}\ $h_{C_3(a_1)}(t)=h(t^6,t^{11},t^8,t^4)$.\\
{\bf 10)}\ $h_{F_4(a_3)}(t)=h(t^6,t^{12},t^8,t^4)$.\\
{\bf 11)}\ $h_{B_3}(t)=h(t^{10},t^{18},t^{12},t^6)$.\\
{\bf 12)}\ $h_{C_3}(t)=h(t^{10},t^{19},t^{14},t^8)$.\\
{\bf 13)}\ $h_{F_4(a_2)}(t)=h(t^{10},t^{20},t^{14},t^8)$.\\
{\bf 14)}\ $h_{F_4(a_1)}(t)=h(t^{14},t^{26},t^{18},t^{10})$.\\
{\bf 15)}\ $h_{F_4}(t)=h(t^{22},t^{42},t^{30},t^{16})$.

\subsection{\bf Conditions for Cuspidality and Nonvanishing of the Descents}

In this subsection we shall work out the global setup in some of the
cases mentioned in subsection \ref{dimdesc}. The choice is partly
random and partly motivated by considering examples we think to be
of some interest. More precisely, our concern is to give in each
case conditions when the descent is cuspidal and when it is not
zero. To do that we compute integral \eqref{comp4} in the case when
it is a Fourier coefficient corresponding to the relevant unipotent
orbit, or when the integration over $V$ represents a constant term
along a certain unipotent radical of a maximal parabolic subgroup of $H$.
Therefore, the precise starting integral, whether it is integral
\eqref{desc1} or \eqref{desc100} will not be important to us, hence
we ignore it. For our goal, it is enough to indicate in each case
the group $U'_\Delta$ and the character $\psi_{U,u_\Delta}$. See
integral \eqref{comp1} for notations. We will express the answer in
terms of the notations used in equation \eqref{comp5}.

Since the question of cuspidality and of the nonvanishing is a
statement of certain integral being zero for all choice of data or
not, it will be convenient in many case to ignore adelic integration
which occurs during the computations. Indeed, when performing root
exchange, as explained in subsection 2.2.2, we relate a certain Fourier
coefficient with a certain integral which involves adelic
integration. However, in all cases one can easily prove that one
integral is zero for all choice of data if and only if the other one
is zero for all choice of data. For our purposes that is enough. In
some cases we will still write the equation \eqref{comp5}, but we
mean that the left hand side is zero for all choice of data if and
only if each term on the right hand side is zero for all choice of
data.

\subsubsection{\bf The Unipotent Orbit $C_3$ } The
construction of the unipotent group $U'_\Delta$ and the characters
$\psi_{U,u_\Delta}$ were described in Section 2. In this case the
group $U'_\Delta$ is as follows. Let $U$ denote the unipotent
radical of the parabolic subgroup of $F_4$ whose Levi part contains
the $SL_2$ generated by $\{x_{\pm 0100}\}$. Thus, $U=U_{\alpha_2}$ and
$\text{dim}U=23$. Let $U'_\Delta$ denote the subgroup of $U$ which
consists of all one dimensional unipotent subgroup $\{x_\alpha(r)\}$ where
$\alpha$ is a root in $U$ which does not include the roots $(0010)$
and $(0110)$. Thus $\text{dim}U'_\Delta=21$. We define the character
$\psi_{U,u_\Delta}$ as follows. For
$u=x_{0001}(r_1)x_{1110}(r_2)x_{0120}(r_3)u'$  define
$\psi_{U,u_\Delta}(u)=\psi(r_1+r_2+r_3)$. ( See subsection 2.1 for notations). Thus, the group
$SL_2=<x_{\pm 0100}(m)>$ is in the stabilizer of $\psi_{U,u_\Delta}$.
To simplify notations, we shall denote the group $U'_\Delta$ by $V$,
and the character $\psi_{U,u_\Delta}$ by $\psi_V$.

This diagram associated with this unipotent orbit  contains nodes
which are labelled with ones. Hence the construction we use is
\eqref{desc100}. Let $\sigma$ denote the representation defined on
$SL_2({\bf A})$ which is obtained by integral \eqref{desc100}.  This copy of $SL_2$ splits under the double
cover of the relevant symplectic group. We look for conditions when
$\sigma$ is cuspidal and when it is not zero. We start with the
nonvanishing. Thus, we compute
\begin{equation}
\int\limits_{F\backslash {\bf A}}\int\limits_{V(F)\backslash V({\bf
A})}E(vx_{0100}(r))\psi_V(v)\psi(ar)drdv\notag
\end{equation}
where $a\in F^*$. In the notations of integral \eqref{comp4}, we have $V({\mathcal O}_H)=\{x_{0100}(r)\}$ and 
$V({\mathcal O}_G)=V$.

Notice that for this orbit we have
$h_{C_3}(t)=h(t^{10},t^{19},t^{14},t^8)$. Also, from the embedding
in $F_4$ of the group of type $A_1$, which is inside the stabilizer
of this orbit inside $F_4$, we deduce that its maximal torus is
$h(1,t,1,1)$. Thus the product of the two tori gives
$h_{C_3}(t)h(1,t,1,1)=h(t^{10},t^{20},t^{14},t^8)=h_{F_4(a_2)}(t)$.
Thus we expect to obtain the orbit $F_4(a_2)$ in the expansion, and
we dont need any conjugation by some Weyl elements.

We expand along the unipotent group $\{x_{0110}(m)\}$. The above
integral is equal to
\begin{equation}\label{desc2}
\sum_{\gamma\in F}\int\limits_{(F\backslash {\bf
A})^2}\int\limits_{V(F)\backslash V({\bf
A})}E(vx_{0110}(m)x_{0100}(r))\psi_V(v)\psi(ar+\gamma m)drdmdv
\end{equation}
Since the function $E$ is automorphic, then for all $\gamma\in F$ we
have $E(h)=E(x_{1000}(\gamma)h)$. Using that, and conjugating
$x_{1000}(\gamma)$ to the right, integral \eqref{desc2} is equal to
\begin{equation}\label{desc3}
\int\limits_{{\bf A}}\int\limits_{V_1(F)\backslash V_1({\bf
A})}E(v_1x_{1000}(r))\psi_{V_1,a}(v_1)drdv_1
\end{equation}
In the derivation of the above integral we also collapsed the
summation over $\gamma$ with the suitable integration. Here $V_1$ is
the unipotent radical of the parabolic subgroup of $F_4$ whose Levi
part contains $SL_2\times SL_2$ which is generated by $<x_{\pm 1000}(r), x_{\pm 0010}(r)>$. In other words, $V_1=U_{\alpha_1,\alpha_3}$, and
hence  $\text{dim} V_1=22$. The character $\psi_{V_1,a}$ is defined
as follows. For
$v_1=x_{0001}(r_1)x_{1110}(r_2)x_{0120}(r_3)x_{0100}(r_4)v_1'$ let
$\psi_{V_1,a}(v_1)=\psi(r_1+r_2+r_3+ar_4)$. It follows from Section
2 that the Fourier coefficient along $V_1$ corresponds to the
unipotent orbit $F_4(a_2)$. Arguing as in \cite{Ga-S}, integral
\eqref{desc3} is nonzero for some choice of data if and only if the
Fourier coefficient along $V_1$ is not zero for some choice of data.
From this we conclude that the representation $\sigma$ is not zero
if and only if the representation $\mathcal{E}$ has a nonzero
Fourier coefficient corresponding to the unipotent orbit $F_4(a_2)$,
which corresponds to the character $\psi_{V_1,a}$.

To study when $\sigma$ is cuspidal, we consider the constant term along the unipotent radical of the Borel subgroup of the group $SL_2$. Thus, we need to compute the integral
\begin{equation}
\int\limits_{F\backslash {\bf A}}\int\limits_{V(F)\backslash V({\bf
A})}E(vx_{0100}(r))\psi_V(v)drdv\notag
\end{equation}
Let $V_1$ denote the unipotent group generated by $V$ and
$\{x_{0100}(r)\}$. Since $E$ is automorphic, we obtain
$E(h)=E(w[3124321]h)$. Conjugating  the above integral by this Weyl
element, the above integral is equal to
\begin{equation}\label{desc4}
\int\limits_{L(F)\backslash L({\bf A})}
\int\limits_{U_2(F)\backslash U_2({\bf
A})}\int\limits_{U_1(F)\backslash U_1({\bf
A})}E(u_1u_2lw_0)\psi_{U_1}(u_1)du_1du_2dl
\end{equation}
Here $U_1$ is the maximal unipotent subgroup of $Sp_6$ as embedded
inside a Levi part of a maximal parabolic subgroup of $F_4$. The
character $\psi_{U_1}$ is the Whittaker character of $U_1$. In other
words,
$\psi_{U_1}(u_1)=\psi(x_{0100}(r_1)x_{0010}(r_2)x_{0001}(r_3)u'_1)
=\psi(r_1+r_2+r_3)$. The group $U_2$ is the unipotent subgroup of
$F_4$ generated by all $\{x_\alpha(r)\}$ where $\alpha$ is a root in the set
$\{ (1122); (1221); (1231); (1222); (1232);
(1242); (1342); (2342)\}$. Thus $\text{dim}U_2=8$. The unipotent
group $L$ is generated by all  one parameter unipotent subgroups $\{ x_\alpha(r)\}$ where $\alpha$ is a root in the set  $\{ -(1000); -(1100); -(1110);
-(1120); -(1111)\}$. The dimension of $L$ is 5. Finally, we denoted
$w_0=w[3124321]$.

Next we consider a series of root exchange in integral
\eqref{desc4}. ( See Section 2).  We first expand along $\{x_{1220}(r)\}$. Thus integral
\eqref{desc4} is equal to
\begin{equation}\label{desc5}
\int\limits_{L(F)\backslash L({\bf A})}\sum_{\gamma\in F}
\int\limits_{F\backslash {\bf A}} \int\limits_{U_2(F)\backslash
U_2({\bf A})}\int\limits_{U_1(F)\backslash U_1({\bf
A})}E(x_{1220}(r)u_1u_2lw_0)\psi_{U_1}(u_1)\psi(\gamma
r)du_1du_2drdl
\end{equation}
Using the fact that $E$ is automorphic we have
$E(h)=E(x_{-(1120)}(\gamma)h)$. Conjugating this element to the
right, changing variables, and collapsing summation with
integration, integral \eqref{desc5} is equal to
\begin{equation}\label{desc6}
\int\limits_{ {\bf A}}\int\limits_{L_1(F)\backslash L_1({\bf A})}
\int\limits_{U_3(F)\backslash U_3({\bf
A})}\int\limits_{U_1(F)\backslash U_1({\bf
A})}E(u_1u_3l_1x_{-(1120)}(m)w_0)\psi_{U_1}(u_1)du_1du_3dl_1dm
\end{equation}
Here $U_3$ corresponds to the unipotent group generated
by all $\{x_\alpha(r)\}$ where $\alpha$ is a root in the set
$$\{ (1220; (1122); (1221); (1222); (1231); (1232); (1242); (1342);
(2342)\}$$ The group $L_1$ is generated by all one dimensional unipotent subgroups $\{x_\alpha(r)\}$ where $\alpha$ is a root in the set
 $\{ -(1000); -(1100);
-(1110); -(1111)\}$. We repeat this process two more times. First we
expand along the unipotent group $\{x_{1121}(r)\}$ and use for that the
unipotent group $\{x_{-(1111)}(m)\}$. Then we expand along the group
$\{x_{1120}(r)\}$ and use the group $\{x_{-(1110)}(m)\}$. Thus, integral
\eqref{desc6} is equal to
\begin{equation}\label{desc7}
\int\limits_{Z({\bf A})}\int\limits_{L_2(F)\backslash L_2({\bf A})}
\int\limits_{U_4(F)\backslash U_4({\bf
A})}\int\limits_{U_1(F)\backslash U_1({\bf A})}E(u_1u_4l_2zw_0)
\psi_{U_1}(u_1)du_1du_4dl_2dz
\end{equation}
Here, the group $U_4$ is generated by all $\{x_\alpha(r)\}$ such that $\alpha$ is a root in the set
$$\{ (1120); (1121); (1220); (1122); (1221); (1222); (1231);
(1232); (1242); (1342); (2342)\}$$ The group $L_2$ is
generated by all $\{x_\alpha(r)\}$ such that $\alpha$ is a root in the set
$\{ -(1000); -(1100)\}$, and $Z$ is generated by all $\{x_\alpha(r)\}$ such that $\alpha$ is a root in the set $\{ -(1110); -(1111); -(1120)\}$. Arguing as in \cite{Ga-S},
integral \eqref{desc7} is zero for all choice of data, if and only
if the integral
\begin{equation}\label{desc8}
\int\limits_{L_2(F)\backslash L_2({\bf A})}
\int\limits_{U_4(F)\backslash U_4({\bf
A})}\int\limits_{U_1(F)\backslash U_1({\bf A})}E(u_1u_4l_2)
\psi_{U_1}(u_1)du_1du_4dl_2
\end{equation}
is zero for all choice of data. Next we expand integral
\eqref{desc8} along the unipotent subgroup $\{x_{1111}(r)\}$. Thus,
integral \eqref{desc8} is a sum of two integral. The first is the
contribution to \eqref{desc8} from the nontrivial orbit. In this
case, after conjugation by the Weyl element $w[21]$, it follows from
the description of the unipotent orbits given in Section 2, that the
expansion obtained is a  Fourier coefficient which corresponds to
the unipotent orbit $F_4(a_1)$. We denote this integral by $I_1$.
The second integral, denoted by $I_2$,  is the contribution from the
constant term. In this case we proceed as above. We expand along the
unipotent group $\{x_{1110}(r)\}$ and use the unipotent group
$\{x_{-(1100)}(m)\}$, and then expand along $\{x_{1100}(r)\}$ and use the
group $\{x_{-(1000)}(m)\}$. Thus, integral $I_2$ is zero for all choice
of data if and only if the integral
\begin{equation}\label{desc9}
\int\limits_{U_5(F)\backslash U_5({\bf
A})}\int\limits_{U_1(F)\backslash U_1({\bf A})}E(u_1u_4l_2)
\psi_{U_1}(u_1)du_1du_5
\end{equation}
is zero for all choice of data. Here $U_5$ is the unipotent group
which is generated by $U_4$ and $\{x_{1100}(r_1)x_{1110}(r_2)\}$.
Finally, we expand \eqref{desc9} along the unipotent group
$\{x_{1000}(r)\}$. There are two cases. The first, corresponds to the
nontrivial orbit, produce a Fourier coefficient which is associated
with the unipotent orbit $F_4$. The other case, which corresponds to
the constant term, contributes the integral
\begin{equation}\label{desc10}
\int\limits_{U_1(F)\backslash U_1({\bf A})}E^{U(C_3)}(u_1)
\psi_{U_1}(u_1)du_1
\end{equation}
Here $U(C_3)=U_{\alpha_2,\alpha_3,\alpha_4}$ , and
$E^{U(C_3)}$ denotes the constant term of $E$ along $U(C_3)$.
We summarize,

\begin{proposition}\label{b3c31}

Let $\mathcal{E}$ denote an automorphic representation of $F_4({\bf A})$ such that:\\
{\bf 1)}\ The representation $\mathcal{E}$ has no nonzero Fourier
coefficients associated with the unipotent orbits $F_4(a_1)$ and
$F_4$. Also, integral \eqref{desc10} is zero for all choice of
data.\\
{\bf 2)}\ There exists an $a\in F^*$, such that the representation
$\mathcal{E}$ has a nonzero Fourier coefficient associated with the
unipotent orbit $F_4(a_2)$ which is given by integral \eqref{desc3}.

Then the representation $\sigma$ is a nonzero cuspidal
representation of the group $SL_2({\bf A})$.

\end{proposition}

It follows from the above that identity \eqref{comp5} can be
described in this case by
$$C_3(\mathcal{E})\circ (2)_a=F_4(a_2)_a$$ Here $a\in
(F^*)^2\backslash F^*$.

\subsubsection{\bf The Unipotent Orbit $B_3$ }

We consider the descent construction which is obtained from the
unipotent orbit $B_3$. In this case, the group $U'_\Delta$, and the
character $\psi_{U,u_\Delta}$ given in integral \eqref{desc1} are as
follows. The group $U'_\Delta$ is the unipotent radical of the
parabolic subgroup of $F_4$ whose Levi part contains the group
$SL_3$ generated by $<x_{\pm (0010)}(r), x_{\pm (0001)}(r)>$. Thus,
$U'_\Delta=U_{\alpha_3,\alpha_4}$. To define the character
$\psi_{U,u_\Delta}$, write
$u=x_{0111}(r_1)x_{0120}(r_2)x_{1000}(r_3)u'$. Then $\psi_{U,u_\Delta}(u)=\psi(r_1+r_2+r_3)$. The stabilizer of
$\psi_{U,u_\Delta}$ in the above copy of $SL_3$ is the group $SO_3$.
For short we write $V$ for $U'_\Delta$ and $\psi_V$ for
$\psi_{U,u_\Delta}$. Thus, integral \eqref{desc1} produces an
automorphic representation $\sigma$ on $SO_3({\bf A})$. We have
$h_{B_3}(t)=h(t^{10},t^{18},t^{12},t^6)$. The maximal torus of
$SO_3$ is given by $h(1,1,m,m)$ where $m\in F^*$. Hence, the maximal
torus of $SL_2$ as embedded in $SO_3$ is given by $h(1,1,t^2,t^2)$.
We have $h_{B_3}(t)h(1,1,t^2,t^2)= h(t^{10},t^{18},t^{14},t^8)$.
Conjugating this torus by $w_2$ we obtain
$h_{F_4(a_2)}(t)=h(t^{10},t^{20},t^{14},t^8)$. Thus we expect to get
the unipotent orbit $F_4(a_2)$, after a suitable conjugation by a
Weyl element. The maximal unipotent subgroup of $SO_3$ is embedded
in $F_4$ as $j(r)=x_{0010}(r)x_{0001}(\eta_1
r)x_{0011}(\eta_2 r)$ where $\eta_i$ are some fixed elements
in $F^*$ determined so that $j(r)$ stabilizes the character
$\psi_V$.

The integral we need to compute is given by
\begin{equation}\label{b35}
\int\limits_{F\backslash {\bf A}}\int\limits_{V(F)\backslash V({\bf
A})}E(vj(r))\psi(ar)\psi_V(v)drdv
\end{equation}
where $a=0,1$. Thus, if $a=0$ we compute the constant term of
$\sigma$, whereas if $a=1$ we compute the Whittaker coefficient of
$\sigma$.

In both cases we start with two root exchanges as explained in
subsection 2.2.2. First we perform a Fourier expansion along the
unipotent group $\{x_{0011}(m)\}$ and exchange it by $\{x_{0100}(l)\}$. Then
we repeat this process with the roots $(0001)$ and $(0110)$. In the
case when $a=1$ we also exchange $(1100)$ by $-(0100)$. Assume that
$a=1$. Then when we conjugate the above integral by $w[32]$, we
obtain as inner integration the integral
\begin{equation}\label{b355}
\int\limits_{U_{\alpha_1,\alpha_3}(F)\backslash
U_{\alpha_1,\alpha_3}({\bf A})}E(u)\psi_U(u)du
\end{equation}
where $\psi_U$ is defined as follows. Write $u\in
U_{\alpha_1,\alpha_3}$ as
$u=x_{0001}(r_1)x_{0100}(r_2)x_{0110}(r_3)x_{1120}(r_4)u'$. Then
$\psi_U(u)=\psi(r_1+r_2+r_3+r_4)$. It follows from subsection 2.2.1,
that this Fourier coefficient corresponds to the unipotent orbit
$F_4(a_2)$. Also, it is not hard to check that this integral is zero
for all choice of data, if and only if integral \eqref{b35} is zero
for all choice of data.

The case when $a=0$ is done in a similar way as in  the case of the
unipotent orbit $C_3$. After performing the above two root exchange,
we conjugate the integral by $w_0=w[432132]$, and we obtain that
integral \eqref{b35}, with $a=0$, is zero for all choice of data, if
and only if the integral
\begin{equation}\label{b36}
\int\limits_{L(F)\backslash L({\bf A})}
\int\limits_{U_2(F)\backslash U_2({\bf
A})}\int\limits_{U_1(F)\backslash U_1({\bf
A})}E(u_1u_2lw_0)\psi_{U_1}(u_1)du_1du_2dl\notag
\end{equation}
is zero for all choice of data. Here $U_1$ is the maximal unipotent
subgroup of $Spin_7$ which is embedded in $F_4$ as a Levi part of a
maximal parabolic subgroup. The character $\psi_{U_1}$ is the
Whittaker character defined on $U_1$. Thus, if $u\in U_1$ is written
as $u=x_{1000}(r_1)x_{0100}(r_2)x_{0010}(r_3)u'$, then
$\psi_{U_1}=\psi(r_1+r_2+r_3)$. The group $U_2$ is generated by all
$\{x_\alpha(r)\}$ where
$$\alpha\in \{(1111);\ (0121);\ (1121);\ (1221);\ (1231);\ (1232);\
(1242);\ (1342);\ (2342)\}$$ Finally, the group $L$ is generated by
all $\{x_\alpha(r)\}$ such that
$$\alpha\in \{-(1122);\ -(0122);\ -(0011);\ -(0001)\}$$ This
integral is similar to the integral \eqref{desc4}. We proceed in a
similar way as in the case of the unipotent orbit $C_3$. Since the
computations are  similar, we shall omit them. To state the
conditions we obtain, we consider the integral
\begin{equation}\label{desc11}
\int\limits_{U_1(F)\backslash U_1({\bf
A})}E^{U(B_3)}(u_1)\psi_{U_1}(u_1)du_1
\end{equation}
Here $U(B_3)=U_{\alpha_1,\alpha_2,\alpha_3}$. We have,

\begin{proposition}\label{b3c32}

Let $\mathcal{E}$ denote an automorphic representation of $F_4({\bf A})$ such that:\\
{\bf 1)}\ The representation $\mathcal{E}$ has no nonzero Fourier
coefficients associated with the unipotent orbits $F_4(a_1)$ and
$F_4$. Also, integral \eqref{desc11} is zero for all choice of
data.\\
{\bf 2)}\ The representation $\mathcal{E}$ has a nonzero Fourier
coefficient associated with the unipotent orbit $F_4(a_2)$ which is
given by integral \eqref{b355}.

Then the representation $\sigma$ is a nonzero cuspidal
representation of the group $SO_3({\bf A})$.

\end{proposition}

As in the previous case, we can rephrase the nonvanishing
computation in terms of identity \eqref{comp5}. In this case we have
$$B_3(\mathcal{E})\circ (3)=F_4(a_2)$$

\subsubsection{\bf The Unipotent Orbit $C_3(a_1)$}

The diagram corresponding to this unipotent orbit contains nodes which are labelled one, and hence we use the integral \eqref{desc100} where the theta representation is defined on the
double cover of $Sp_6({\bf A})$. 
It follows from \cite{C} that the
stabilizer of this unipotent orbit is a group of type $A_1$. In this
case it is the group $SL_2({\bf A})$, and according to the choice of
character $\psi_{U,\Delta}$ which will be specified below, we have
$SL_2=<x_{\pm (0100)}>$. From the embedding of this copy of $SL_2$ inside $Sp_6$, we deduce that the representation $\sigma$ is defined over the double cover of $SL_2$.

Since our goal is to study the vanishing or
nonvanishing of certain Fourier coefficients, it is enough to study
integral \eqref{desc102}. Thus we need to describe the group
$U_\Delta'$ and the character $\psi_{U,u_\Delta}$ that we choose.
From the description of the diagram associated with this unipotent
orbit, it follows that $U_\Delta=U_{\alpha_1,\alpha_3}$. Let
$U_\Delta'$ denote the subgroup of $U_\Delta$ which consists of all
roots in $U_\Delta$ deleting the three roots $(0010);\ (0011);
(1000)$. Thus $\text{dim} U_\Delta'=19$. The character
$\psi_{U,u_\Delta}$ is defined as follows. For $u'\in U_\Delta'$
write $u'=x_{0121}(r_1)x_{1110}(r_2)x_{1111}(r_3)u''$. Then
$\psi_{U,u_\Delta}(v')=\psi(r_1+r_2+r_3)$. Denote $V=U_\Delta$ and
$V'=U_\Delta'$. The one dimensional torus associated with this orbit
is $h_{C_3(a_1)}(t)=h(t^6,t^{11},t^8,t^4)$. The one dimensional
torus of the copy of $SL_2$ which is the stabilizer of
$\psi_{U,u_\Delta}$ is $h(1,t,1,1)$. If we multiply these two tori
elements, we obtain $h_{F_4(a_3)}(t)$.

Thus, the integral we need to consider is given by
\begin{equation}\label{c3a11}
\int\limits_{F\backslash {\bf A}}\int\limits_{V'(F)\backslash
V'({\bf A})}E(v'x_{0100}(r))\psi_{U,u_\Delta}(v')\psi(\beta r)dv'dr
\end{equation}
Here $\beta\in F$. We start with the case when $\beta\ne 0$. In this
case, the above integral is equal to
\begin{equation}\label{c3a12}
\int\limits_{U_{\alpha_2}(F)\backslash U_{\alpha_2}({\bf
A})}E(u)\psi_{U,u_\Delta}(u)du\notag
\end{equation}
where now the character $\psi_{U,u_\Delta}$ is a character of the
group $U_{\alpha_2}$, and is given as follows. Write $u\in
U_{\alpha_2}$ as
$u=x_{0121}(r_1)x_{1110}(r_2)x_{1111}(r_3)x_{0100}(r_4)u'$. Then
$\psi_{U,u_\Delta}(u)=\psi(r_1+r_2+r_3+\beta r_4)$. It follows from
subsection 2.2.1 that this Fourier coefficient corresponds to the
unipotent orbit $F_4(a_3)$. Indeed, in the notations of equation
\eqref{f4a31}, the character $\psi_{U,u_\Delta}$ corresponds to the
character $\psi_{U_\Delta,A,B}$ with
$$A=\begin{pmatrix} &1&\\ &&1\\ \beta&&\end{pmatrix}\ \ \ \ \
B=\begin{pmatrix} 1&&\\ 1&&\\ &1&1\end{pmatrix}$$ Solving the
equations given in \eqref{f4a32}, we obtain that the only solution
is the trivial solution. Hence, the above character
$\psi_{U,u_\Delta}$ is associated with the unipotent orbit
$F_4(a_3)$. In the notations of equation \eqref{comp5} we proved
$$C_3(a_1)(\mathcal E)\circ (2)_\beta=F_4(a_3)_\beta$$

Next we consider the cuspidality of the lift. Thus, we need to
compute integral \eqref{c3a11} with $\beta=0$. Let $w_0=w[1234213]$.
Conjugating by this element, integral \eqref{c3a11} is equal to
\begin{equation}\label{c3a13}
\int\limits_{L(F)\backslash L({\bf A})}\int\limits_{U_2(F)\backslash
U_2({\bf A})}\int\limits_{U_1(F)\backslash U_1({\bf
A})}E(u_1u_2lw_0)\psi_{U_1}(u_1)du_1du_2dl
\end{equation}
Here $U_1$ is the unipotent subgroup of $Sp_6$ which is generated by
all $\{x_\alpha(r)\}$ where $\alpha=\sum n_i\alpha_i$ such that $n_1=0$ and
deleting the simple root $\alpha_2$. Thus $\text{dim}\ U_1=8$. The
character $\psi_{U_1}$ is defined as follows. Write
$u_1=x_{0001}(r_1)x_{0110}(r_2)x_{0011}(r_3)u_1'$. Then
$\psi_{U_1}(u_1)=\psi(r_1+r_2+r_3)$. We mention that this Fourier
coefficient of the group corresponds to the unipotent orbit $(42)$
in the group $Sp_6$. The group $U_2$ is generated by all $\{x_\alpha(r)\}$
such that
$$\alpha\in \{ (1122);\ (1221);\ (1222);\ (1231);\ (1232);\ (1242);\
(1342);\ (2342)\}$$ Finally, the group $L$ is generated by all one
dimensional unipotent elements  $\{x_\alpha(r)\}$ such that $\alpha\in \{
-(1000);\ -(1100);\ -(1110);\ -(1120)\}$.

We perform 4 root exchange as explained in subsection 2.2.2. First, we
exchange $-(1110)$ with $(1220)$, then $-(1120)$ with $(1121)$,
$-(1100)$ with $(1111)$ and $-(1000)$ with $(1110)$. Then we expand
the integral we obtain along the unipotent group
$\{x_{1000}(m_1)x_{1100}(m_2)x_{1120}(m_3)\}$. Thus, integral
\eqref{c3a13} is equal to
\begin{equation}\label{c3a14}
\int\limits_{L({\bf A})}\sum_{a,b,c\in
F}\int\limits_{U_3(F)\backslash U_3({\bf
A})}\int\limits_{U_1(F)\backslash U_1({\bf
A})}E(u_3u_2lw_0)\psi_{U_1}(u_1)\psi_{a,b,c}(u_3)du_1du_3dl
\end{equation}
Here $U_3=U_{\alpha_2,\alpha_3,\alpha_4}$.
The character $\psi_{a,b,c}(u_3)$ is defined as follows.
Write an element $u_3= x_{1000}(m_1)x_{1100}(m_2)x_{1120}(m_3)u_3'$.
Then $\psi_{a,b,c}(u_3)=\psi(am_1+bm_2+cm_3)$, where $a,b,c\in F$.

There are several cases to consider. First assume that $a=b=c=0$.
Then, in integral \eqref{c3a14}, the integration over $U_3$ is the
constant term of the function $E$ along the unipotent group $U_3$.
If $a=b=0$ and $c\ne 0$, then the combined integration over $U_1$
and $U_3$ contains as inner integration the Fourier coefficient
corresponding to the unipotent orbit $F_4(a_1)$. Finally, if
$(a,b)\ne (0,0)$ then we obtain, as inner integration, the Fourier
coefficient corresponding to the unipotent orbit $F_4(a_2)$.

We summarize
\begin{proposition}\label{c3a15}

Let $\mathcal{E}$ denote an automorphic representation of $F_4({\bf A})$ such that:\\
{\bf 1)}\ The representation $\mathcal{E}$ has no nonzero Fourier
coefficients associated with the unipotent orbits $F_4(a_1)$ and
$F_4(a_2)$ as given above. Also, the representation $\mathcal{E}$
does not support the constant term along the group $U_{\alpha_2,\alpha_3,\alpha_4}$\\
{\bf 2)}\ The representation $\mathcal{E}$ has a nonzero Fourier
coefficient associated with the unipotent orbit $F_4(a_3)$ which is
given by integral \eqref{c3a11}.

Then the representation $\sigma$ is a nonzero cuspidal
representation of the group $\widetilde{SL_2}({\bf A})$.

\end{proposition}

\subsubsection{\bf The Unipotent Orbit $B_2$}

The diagram of this unipotent orbit contains nodes which are labeled
with a one. Thus the descent in this case is given in terms of the
integral \eqref{desc100}. Here the theta representation is defined
on the group $Sp_4$.  Also the
stabilizer is a group of type $A_1\times A_1$. Since the embedding of $SL_2\times SL_2$ in $Sp_4$ does not split, the representation $\sigma$ is an automorphic representation of $\widetilde{SL}_2({\bf A})\times \widetilde{SL}_2({\bf A})$.

In the notations of equation \eqref{desc100}, let
$U_\Delta=U_{\alpha_2,\alpha_3}$. To determine the conditions for the
non vanishing and for the cuspidality, we may instead consider integral
\eqref{desc102}. Thus, in the notations of that integral, let
$U_\Delta'$ denote the subgroup of $U_{\alpha_2,\alpha_3}$ where we
omit the roots $(0001)$ and $(0011)$. The character
$\psi_{U,u_\Delta}$ is defined as follows. For
$u'=x_{1110}(r_1)x_{0122}(r_2)u_1'$ let $\psi_{U,u_\Delta}(u')=\psi(r_1+r_2)$. With this choice of character
the maximal unipotent subgroup of the stabilizer, which is
$SL_2\times SL_2$, is given by $\{x_{0100}(m_1)x_{0120}(m_2)\}$. Denote
$V=U_\Delta$ and $V'=U_\Delta'$. We have
$h_{B_2}(t)=h(t^6,t^{10},t^7,t^4)$. The corresponding torus element
of the above copy of $SL_2\times SL_2$ is $h(1,t^2,t,1)$. Hence the
product is $h_{F_4(a_3)}(t)$. Therefore, to study the nonvanishing
of this construction we consider the integral
\begin{equation}\label{b21}
\int\limits_{(F\backslash {\bf A})^2}\int\limits_{V'(F)\backslash
V'({\bf
A})}E(v'x_{0100}(m_1)x_{0120}(m_2))\psi_{U,u_\Delta}(v')\psi(am_1+m_2
)dv'dm_1dm_2\notag
\end{equation}
Here $a\in F^*$. Exchanging the root $(1000)$ by $(0110)$ as
explained in subsection 2.2.2 we obtain the integral
\begin{equation}\label{b22}
\int\limits_{\bf A}\int\limits_{U(F)\backslash U({\bf
A})}E(ux_{1000}(r))\psi_{U,a}(u)dudr
\end{equation}
Here $U=U_{\alpha_1,\alpha_3,\alpha_4}$ and the character
$\psi_{U,a}$ is defined as follows. For an element
$u=x_{0100}(r_1)x_{1110}(r_2)x_{0120}(r_3)x_{0122}(r_4)u'$ define
$\psi_{U,a}(u)=\psi(ar_1+r_2+r_3+r_4)$. In the notations of subsection
2.2.1 this corresponds to the character $\psi_{U_\Delta,A,B}$ with
$$A=\begin{pmatrix} &&1\\ &1&\\ \alpha&&\end{pmatrix}\ \ \ \ \
B=\begin{pmatrix} &&\\ 1&&\\ &1&\end{pmatrix}$$ Solving the
equations given in \eqref{f4a32}, we obtain that the only solution
is the trivial solution. Hence, the above character $\psi_{U,a}$ is
associated with the unipotent orbit $F_4(a_3)$. In the notations of
equation \eqref{comp5} we proved
$$B_2(\mathcal E)\circ (2|2)_a=F_4(a_3)_a$$ Here by $(2|2)$
we denote the unipotent orbit of $SL_2\times SL_2$ which corresponds
to the Whittaker coefficient of this group.

To study the cuspidality of the lift, we need to consider two
constant terms. One along the unipotent group $\{x_{0120}(r_1)\}$ and
the other along $\{x_{0100}(r_2)\}$. However, these two matrices are
conjugate under the Weyl element $w_3$. Moreover, this Weyl element
normalizes the group $V'$ and the character $\psi_{U,u_\Delta}$.
Hence it is enough to consider the integral
\begin{equation}\label{b23}
\int\limits_{F\backslash {\bf A}}\int\limits_{V'(F)\backslash
V'({\bf A})}E(v'x_{0120}(m))\psi_{U,u_\Delta}(v') dv'dm\notag
\end{equation}
Let $w_0=w[123421]$ and conjugate in the above integral by this Weyl element. Then the
above integral is equal to
\begin{equation}\label{b24}
\int\limits_{L(F)\backslash L({\bf A})}\int\limits_{U_2(F)\backslash
U_2({\bf A})}\int\limits_{U_1(F)\backslash U_1({\bf
A})}E(u_1u_2lw_0)\psi_{U_1}(u_1)du_1du_2dl\notag
\end{equation}
Here $U_1$ is the unipotent subgroup of $Sp_6$ which is generated by
$\{x_\alpha(r)\}$ where
$$\alpha\in\{ (0001);\ (0110);\ (0011);\ (0111);\
(0120);\ (0121);\ (0122)\}$$ The character $\psi_{U_1}$ is defined
as follows. Write $u_1=x_{0001}(r_1)x_{0120}(r_2)u_1'$. Then
$\psi_{U_1}(u_1)=\psi(r_1+r_2)$. The group $U_2$ is generated by all
$\{x_\alpha(r)\}$ such that
$$\alpha\in\{ (1122);\ (1222);\ (1231);\ (1232);\ (1242);\ (1342);\
(2342)\}$$ Finally, the group $L$ is generated by all $\{x_\alpha(r)\}$
such that
$$\alpha\in\{ -(1000);\ -(1100);\ -(1110);\ -(1120);\ -(1220)\}$$
First we exchange roots as follows. Exchange $-(1220)$ with
$(1221)$, $-(1120)$ with $(1121)$, $-(1100)$ with $(1220)$,
$-(1110)$ with $(1111)$ and $-(1000)$ with $(1120)$. Then we perform
a Fourier expansion along the roots $(0100);\ (1100)$ and $(1110)$.
Then, the above integral is equal to
\begin{equation}\label{b241}
\int\limits_{ L({\bf A})}\sum_{a,b,c}\int\limits_{U_4(F)\backslash
U_4({\bf A})}\int\limits_{U_3(F)\backslash U_3({\bf
A})}E(u_4u_3lw_0)\psi_{U_3,a}(u_3)\psi_{U_4,b,c}(u_4)du_4du_3dl
\end{equation}
Here $U_3$ is the unipotent subgroup of $Sp_6$ generated by 
$U_1$ and the group $\{x_{0100}(r)\}$. The character $\psi_{U_3,a}$
is defined as follows. Write
$u_3=x_{0001}(r_1)x_{0120}(r_2)x_{0100}(r_3)u_3'$. Then
$\psi_{U_3,a}(u_3)=\psi(r_1+r_2+ar_3)$. The group $U_4$ is generated
by  $U_2$ and the unipotent group $\{x_{1100}(r_1) x_{1110}(r_2)\}$. To
define the character $\psi_{U_4,b,c}$ write
$u_4=x_{1100}(r_1)x_{1110}(r_2)u_4'$. Then
$\psi_{U_4,b,c}(u_4)=\psi(br_1+cr_2)$. In the notations of subsection
2.2.1 which describes the characters of $F_4(a_2)$, the character
corresponding to the integral \eqref{b24}, corresponds to the
character $\psi_{U_\Delta, A,\gamma_1,\gamma_2}$ where
$$A=\begin{pmatrix} 1&a\\ &b\\ c&\\ &\end{pmatrix}\ \ \ \ \
(\gamma_1,\gamma_2)=(1,0)$$ There are several cases. Assume first
that $b,c\ne 0$. Then, we obtain as inner integration, a Fourier
coefficient which correspond to the unipotent orbit $F_4(a_2)$.
Indeed, from the description of the action on the set of characters,
as described in subsection 2.2.1 it follows that the stabilizer is a
finite group. The same happens if $a,c\ne 0$ but $b=0$ and similarly
if $a,b\ne 0$ but $c=0$. In all other cases we either get a Fourier
coefficient which corresponds to a unipotent orbit which is either
$F_4(a_1)$ or $F_4$, or we get the constant term along the unipotent
radical of the maximal parabolic subgroup whose Levi part is
$GSp_6$.

We proved,

\begin{proposition}\label{b25}

Let $\mathcal{E}$ denote an automorphic representation of $F_4({\bf A})$ such that:\\
{\bf 1)}\ The representation $\mathcal{E}$ has no nonzero Fourier
coefficients associated with the unipotent orbits $F_4,\ F_4(a_1)$
and $F_4(a_2)$ as given above. Also, the representation
$\mathcal{E}$
does not support the constant term along the group $U_{\alpha_2,\alpha_3,\alpha_4}$\\
{\bf 2)}\ The representation $\mathcal{E}$ has a nonzero Fourier
coefficient associated with the unipotent orbit $F_4(a_3)$ which is
given by integral \eqref{b22}.

Then the representation $\sigma$ is a nonzero cuspidal
representation of the group $\widetilde{SL_2}({\bf A})\times \widetilde{SL_2}({\bf A})$.
\end{proposition}

\subsubsection{\bf The Unipotent Orbit $A_2+\widetilde{A}_1$} The diagram corresponding to this orbits has nodes labelled with ones, and hence we use integral \eqref{desc100} with a
suitable theta representation. In this case we denote
$U_\Delta=U_{\alpha_1,\alpha_2,\alpha_4}$ and let $U_\Delta'$ denote
the subgroup of $U_\Delta$ generated by all roots in $U_\Delta$
omitting $\alpha\in \{ (0010);\ (0110);\ (0011)\}$. Thus
$\text{dim}\ U_\Delta'=17$. We define the character
$\psi_{U,u_\Delta}$ as follows. Write $u'\in U_\Delta'$ as
$u'=x_{0122}(r_1)x_{1121}(r_2)x_{1220}(r_3)u''$.  Then define
$\psi_{U,u_\Delta}(u')=\psi(r_1+r_2+r_3)$. The stabilizer of this
orbit is the  group  $SL_2$, and we can choose the embedding inside
$F_4$, such that its standard unipotent subgroup is the group
$x(r)=x_{1000}(r)x_{0100}(\eta_1 r)x_{1100}(\eta_2
r)x_{0001}(r)$. Here $\eta_i\in F^*$.

The torus corresponding to this orbit is given by
$h_{A_2+\widetilde{A}_1}(t)= h(t^4,t^8,t^6,t^3)$. The torus of the
above  $SL_2$ is given by $h(t^2,t^2,1,t)$, and hence their product
is $h(t^6,t^{10},t^6,t^4)$. Conjugating by the Weyl element $w[23]$,
we obtain the torus attached to the orbit $F_4(a_3)$.

Given, $z\in F$, we consider the integral
\begin{equation}\label{a2a11}
\int\limits_{F\backslash {\bf A}}\int\limits_{V'(F)\backslash
V'({\bf A})}E(v'x(r))\psi_{U,u_\Delta}(v')\psi(zr)dv'dr
\end{equation}
where we denoted $V'=U'_\Delta$. Since the computations in this
example  are quite involved, we will only sketch part of them. In
other words, we will show that when $z\ne 0$ then we do obtain the
Fourier coefficient associated with the unipotent orbit $F_4(a_3)$.
However, we also get other terms which corresponds to unipotent
orbits which are greater than $F_4(a_3)$, and also some constant
terms.

We start with the root exchange, $(0120)$ with $(1100)$, then
$(0121)$ with $(0001)$ and then $(1120)$ with $(0100)$. Conjugating
by $w[23]$ we obtain the integral
\begin{equation}\label{a2a12}
\int\limits_{L({\bf A})}\int\limits_{U_1(F)\backslash U_1({\bf
A})}E(u_1w[23]l)\psi_{U_1,z}(u_1)du_1dl
\end{equation}
Here $L$ is the unipotent group generated by all $\{x_\alpha(r)\}$ where
$\alpha\in\{ (0120);\ (0121);\ (1120)\}$. The group $U_1$ is the
subgroup of $U_{\alpha_1,\alpha_3,\alpha_4}$ omitting the two roots
$(0100)$ and $(0110)$. Thus $\text{dim} U_1=18$. To define the
character $\psi_{U_1,z}$, write
$u_1=x_{1100}(r_1)x_{1111}(r_2x_{1120}(r_3)x_{0122}(r_4)u_1'$. Then
$\psi_{U_1,z}(u_1)=\psi(zr_1+r_2+r_3+r_4)$.

Next we expand the above integral along
$\{x_{0100}(l_1)x_{0110}(l_2)\}$, and we obtain that integral
\eqref{a2a12} is zero for all choice of data if and only if the
integral
\begin{equation}\label{a2a13}
\sum_{m,n\in F}\int\limits_{U(F)\backslash U({\bf
A})}E(u)\psi_{U,z,m,n}(u)du\notag
\end{equation}
is zero for all choice of data. Here,
$U=U_{\alpha_1,\alpha_3,\alpha_4}$, and the character
$\psi_{U,z,m,n}$ is defined as follows. Write
$$u=x_{1100}(r_1)x_{1111}(r_2)x_{1120}(r_3)x_{0122}(r_4)x_{0100}(r_5)
x_{0110}(r_6)u_1'$$ Then
$\psi_{U,z,m,n}=\psi(zr_1+r_2+r_3+r_4+mr_5+nr_6)$. In general
position, this Fourier coefficient corresponds to the unipotent
orbit $F_4(a_3)$. Indeed, in the notations of subsection 2.2.1 the
above character corresponds to the character $\psi_{U_\Delta,A,B}$
with
$$A=\begin{pmatrix} &m&n\\ &&m\\ 1&&\end{pmatrix}\ \ \ \ \ \
B=\begin{pmatrix} 1&&z\\ &1&\\ &&1\end{pmatrix}$$ Solving the
equations \eqref{f4a32} with
$$g_1=\begin{pmatrix} a_1&a_2&a_3\\ b_1&b_2&b_3\\
c_1&c_2&c_3\end{pmatrix}\ \ \ \ \ h_1=\begin{pmatrix} a&b\\
c&d\end{pmatrix}$$ we obtain four variables $b_2,c_1,c_2,c_3$ which
satisfy the system
$$zb_2-(2n+z^2)c_1-3mc_2-zc_3=0$$
$$3mb_2-mzc_1-nc_2-3mc_3=0$$
$$2nb_2+(3m^2-nz)c_1+2mzc_2-2nc_3=0$$
All other  variables which appear in $(g_1,h_1)$ are determined by these
four variables. Since we always have the trivial solution, we set
$c_3=b_2+c_4$, where $c_4$ is a new variable. We then obtain the system
$$-(2n+z^2)c_1-3mc_2-zc_4=0$$
$$-mzc_1-nc_2-3mc_4=0$$
$$(3m^2-nz)c_1+2mzc_2-2nc_4=0$$ This system has a nontrivial solution if
and only if the determinant of the matrix corresponding to this
system is zero. In this case we obtain the determinant
$$f(m,n,z)= -27m^4+18nm^2z+4m^2z^3+4n^3+n^2z^2$$ Thus, for those
values of $m,n$ and $z$ such that $f(m,n,z)\ne 0$, the above Fourier
coefficient corresponds to the unipotent orbit $F_4(a_3)$. To
analyze the other orbits we need to solve the equation $f(m,n,z)=0$.
We claim that in this case we obtain Fourier coefficients which are associated to all unipotent orbits which are greater  than $F_4(a_3)$. We demonstrate this claim in the case when $m=n=0$.
In other words, we consider  the Fourier coefficient
\begin{equation}\label{a2a14}
\int\limits_{U(F)\backslash U({\bf A})}E(u)\psi_{U,z,0,0}(u)du
\end{equation}
For fixed $s_1,s_2\in F$, this integral is zero for all choice of
data if and only if the integral
\begin{equation}\label{a2a15}
\int\limits_{U(F)\backslash U({\bf
A})}E(ux_{-0011}(s_1)x_{-1000}(s_2))\psi_{U,z,0,0}(u)du\notag
\end{equation}
is zero for all choice of data. Conjugate these two elements to
the left. Recall that $U=U_{\alpha_1,\alpha_3,\alpha_4}$.  Since
these elements are inside the Levi subgroup of
$P_{\alpha_1,\alpha_3,\alpha_4}$, this conjugation preserves the
group $U$. We do however need to determine how this conjugation effects the character $\psi_{U,z,0,0}$.
To do that we consider the
conjugation
$$x_{-0011}(-s_1)x_{-1000}(-s_2)x_{1100}(r_1)x_{1111}(r_2)
x_{0122}(r_4)x_{1122}(m)x_{-0011}(s_1)x_{-1000}(s_2)$$ Conjugate
$x_{-0011}(s_1)$ across $x_{1122}(m)$. We obtain
$$x_{-0011}(-s_1)x_{-1000}(-s_2)x_{1100}(r_1-ms_1^2)x_{1111}(r_2+ms_1)
x_{0122}(r_4)x_{-0011}(s_1) x_{1122}(m)x_{-1000}(s_2)u_1$$ Here
$u_1\in U$ is a product of one dimensional unipotent subgroups
$\{x_\alpha(r)\}$ such that $\psi_{U,z,0,0}(u_1)=1$ and $\alpha$ is not any
of the above roots. Changing variables $r_1\to r_1+ms_1^2$ and then
$r_2\to r_2-ms_1$, we obtain the character
$$\psi_{U,z,s_1,s_2}(u)=\psi(zr_1+r_2+r_3+r_4+zms_1^2-ms_1)$$ We
further conjugate $x_{-0011}(s_1)$ to the left, and we obtain
$$x_{-1000}(-s_2)x_{1100}(r_1+2r_2s_1)x_{1111}(r_2)
x_{0122}(r_4)x_{1122}(m)x_{-1000}(s_2)u_2$$ where $u_2$ is defined in a similar way as $u_1$. Changing variables $r_1\to r_1-2r_2s_1$ we obtain the
character
$$\psi'_{U,z,s_1,s_2}(u)=\psi(zr_1+r_2(1-2zs_1)+r_3+r_4+m(zs_1^2-s_1))$$
Finally, conjugating by $x_{-1000}(s_2)$ we obtain
$$x_{1100}(r_1)x_{1111}(r_2)x_{0122}(r_4-ms_2)x_{1122}(m)u_3$$
Changing variables $r_4\to r_4+ms_2$ we obtain the character
$$\psi''_{U,z,s_1,s_2}(u)=\psi(zr_1+r_2(1-2zs_1)+r_3+r_4+m(zs_1^2-s_1+s_2))$$
Choosing $s_1$ and $s_2$ such that $1-2zs_1=0$ and
$zs_1^2-s_1+s_2=0$, we deduce that integral \eqref{a2a14} is zero
for all choice of data if and only if the integral
\begin{equation}\label{a2a16}
\int\limits_{U(F)\backslash U({\bf A})}E(u)\psi_{U,z}(u)du\notag
\end{equation}
is zero for all choice of data. Here $\psi_{U,z}$ is defined as
follows. For $u\in U$, write
$u=x_{1100}(r_1)x_{1120}(r_2)x_{0122}(r_3)u'$. Then
$\psi_{U,z}(u)=\psi(zr_1+r_2+r_3)$. To proceed, we conjugate by the
Weyl element $w_0=w[432341]$. Thus, the above integral is equal to
\begin{equation}\label{a2a17}
\int\limits_{L(F)\backslash L({\bf A})}\int\limits_{U_2(F)\backslash
U_2({\bf A})}\int\limits_{U_1(F)\backslash U_1({\bf
A})}E(u_2u_1lw_0)\psi_{U_1,z}(u_1)du_1du_2dl\notag
\end{equation}
Here $U_1$ is the unipotent  subgroup of $GSpin_7$ generated by all
$\{x_\alpha(r)\}$ where
$$\alpha\in\{ (0100);\ (0110);\ (0120);\ (1000);\ (1100);\ (1110);\
(1120);\ (1220)\}$$ The character $\psi_{U_1,z}$ is defined as
follows. Write $u_1=x_{0100}(r_1)x_{0120}(r_1)x_{1000}(r_3)u_1'$.
Then $\psi_{U_1,z}=\psi(zr_1+r_2+r_3)$. The group $U_2$ is generated
by all $\{x_\alpha(r)\}$ such that
$$\alpha\in\{ (1111);\ (1121);\ (1221);\ (1231);\ (1222);\ (1232);\
(1242);\ (1342);\ (2341)\}$$ and the group $L$ is generated by
$\{x_\alpha(r)\}$ where $\alpha\in\{ -(0011);\ -(0001);\ -(0122)\}$.

As explained in subsection 2.2.2 we exchange the root $-(0122)$ with
$(1122)$, then $-(0001)$ with $(0121)$ and $-(0011)$ with $(0111)$.
Thus, the above integral is equal to
\begin{equation}\label{a2a18}
\int\limits_{ L({\bf A})}\int\limits_{U_3(F)\backslash U_3({\bf
A})}\int\limits_{U_1(F)\backslash U_1({\bf
A})}E(u_2u_1lw_0)\psi_{U_1,z}(u_1)du_1du_2dl\notag
\end{equation}
Here $U_3$ is the unipotent group generated by $U_2$ and $\{x_\alpha(r)\}$
where $\alpha$ is a root in the set $\{ (0111);\ (0121);\ (1122)\}$. The next step is to
expand the above integral along $\{x_{1122}(r)\}$. We obtain
\begin{equation}\label{a2a19}
\int\limits_{ L({\bf A})}\sum_{a\in F}\int\limits_{F\backslash {\bf
A}}\int\limits_{U_3(F)\backslash U_3({\bf
A})}\int\limits_{U_1(F)\backslash U_1({\bf
A})}E(u_2x_{0122}(m)u_1lw_0)\psi_{U_1,z}(u_1)\psi(am)du_1dmdu_2dl\notag
\end{equation}
If $a\ne 0$ then the inner integration is a Fourier coefficient
which corresponds to the unipotent orbit $F_4(a_2)$. When $a=0$ we
further expand along the unipotent group
$\{x_{0001}(m_1)x_{0011}(m_2)\}$. Any nontrivial character corresponding
to this expansion yields a Fourier coefficient attached to the
unipotent orbit $F_4(a_1)$. The trivial character contributes the
integral
\begin{equation}\label{a2a110}
\int\limits_{ L({\bf A})}\int\limits_{U_1(F)\backslash U_1({\bf
A})}E^{U_{\alpha_1,\alpha_2,\alpha_3}}(u_1lw_0)\psi_{U_1,z}(u_1)du_1dl
\end{equation}

To summarize this case, we deduce that the Fourier coefficient given
by integral \eqref{a2a11}, when expressed in terms of Fourier
coefficients associated with unipotent orbits of $F_4$, 
has a contribution from all unipotent orbits which are greater than
the orbit $F_4(a_3)$. We also get the constant term \eqref{a2a110}
as a summand.

\subsubsection{\bf The Unipotent Orbit $\widetilde{A}_2$} Let $V=U_{\alpha_1,\alpha_2,\alpha_3}$.  Thus
$\text{dim}V=15$. We define a character $\psi_V$ as follows. Write
$v=x_{0121}(r_1)x_{1111}(r_2)v'$. Then define
$\psi_V(v)+\psi(r_1+r_2)$. As follows from \cite{C}, the stabilizer
inside $Spin_7$ of this character is the exceptional group $G_2$.
The embedding of the standard unipotent subgroup of $G_2$ is given
as follows
$$\{ x_{1000}(m)x_{0010}(-m);\  x_{0100}(m);\
x_{1100}(m)x_{0110}(-m);\ x_{1110}(m)x_{0120}(-m);\  x_{1120}(m);\
x_{1220}(m)\}$$ The unipotent subgroups which corresponds to the
simple roots are $\{x_{1000}(m)x_{0010}(-m)\}$ and $\{x_{0100}(m)\}$. The
group $G_2$ has two maximal parabolic subgroups, and we will denote
by $U_1$ and by $U_2$ their unipotent radicals. More precisely, we
let
$$U_1=\{ x_{0100}(m);\ x_{1100}(m)x_{0110}(-m);\
x_{1110}(m)x_{0120}(-m);\  x_{1120}(m);\ x_{1220}(m)\}$$ and
$$U_2=\{ x_{1000}(m)x_{0010}(-m);\
x_{1100}(m)x_{0110}(-m);\ x_{1110}(m)x_{0120}(-m);\  x_{1120}(m);\
x_{1220}(m)\}$$ We start by computing the unipotent radical along
$U_1$. We expand the constant term along $\{x_{1000}(m)x_{0010}(-m)\}$,
and we obtain the integral
\begin{equation}\label{desc20}
\int\limits_{U_1(F)\backslash U_1({\bf A})}\sum_{\gamma\in F}\int
\limits_{F\backslash {\bf A}}\int\limits_{V(F)\backslash V({\bf
A})}E(vx_{1000}(m)x_{0010}(-m)u_1)\psi_V(v)\psi(\gamma m)dmdu_1dv
\end{equation}
Write integral \eqref{desc20} as
\begin{equation}\label{desc21}
\int\limits_{(F\backslash {\bf A})^5}\sum_{\gamma\in F}\int
\limits_{F\backslash {\bf A}}\int\limits_{V(F)\backslash V({\bf
A})}E(vz(m_1,m_2,m_3)y(l_1,l_2,l_3))\psi_V(v)\psi(\gamma
m_1)dm_idl_jdv
\end{equation}
Here $z(m_1,m_2,m_3)=x_{1000}(m_1)x_{0010}(-m_1)
x_{1100}(m_2)x_{0110}(-m_2)x_{1110}(m_3)x_{0120}(-m_3)$, and
$y(l_1,l_2,l_3)=x_{0100}(l_1)x_{1120}(l_2)x_{1220}(l_3)$. Next we
consider a certain Fourier expansion, and we apply the root exchange
process as explained in subsection 2.2.2.

We start by expanding the above integral along the unipotent group
$\{x_{1110}(r_3)\}$. We then apply the root exchange process with the
unipotent group $\{x_{0111}(p_3)\}$. Thus, in the notions introduced
right after \eqref{fc6}, we exchange the root $(1110)$ by the root
$(0111)$. We repeat this process two more times. First we exchange
$(1100)$ by $(0011)$, and then $(1000)$ by  $(0001)$.  After that,
we  conjugate by the Weyl element $w_0=w[13234]$. Then integral
\eqref{desc21} is zero for all choice of data if and only if for
each $\gamma\in F$, the integral
\begin{equation}\label{desc2200}\notag
\int\limits_{U_3(F)\backslash U_3({\bf A})}
\int\limits_{V_3(F)\backslash V_3({\bf A})}
E(v_3u_3)\psi_{V_3,\gamma}(v_3)dv_3du_3
\end{equation}
is zero for all choice of data. Here, $\gamma\in F$, and  $V_3$ is the unipotent subgroup
generated by $\{x_\alpha(r)\}$, where $\alpha$ is in the set 
of roots 
$$\{(0100);\
(0001);\ (0011);\ (0110);\ (0120);\ (0111);\ (0121);\ (0122)\}$$
Thus $V_3$ is a subgroup of $Sp_6$  embedded in$F_4$ as the Levi part of $P_{\alpha_2,\alpha_3,\alpha_4}$. 
Denote $U(C_3)=U_{\alpha_2,\alpha_3,\alpha_4}$.
The group $U_3$ is the subgroup of $U(C_3)$
generated by all roots in $U(C_3)$ accept for the roots $(1120)$
and $(1000)$. Thus $\text{dim}U_3=13$. The character
$\psi_{V_3,\gamma}$ is defined as follows. Write
$v_3=x_{0001}(r_1)x_{0110}(r_2)x_{0120}(r_3)v_3'$. Then
$\psi_{V_3,\gamma}(v_3)=\psi(r_1+r_2+\gamma r_3)$. Next we expand
along the unipotent group $\{x_{1120}(r)\}$. Thus, we obtain the integral
\begin{equation}\label{desc221}
\sum_{\beta\in F}\int\limits_{U_3(F)\backslash U_3({\bf A})}
\int\limits_{F\backslash {\bf A}} \int\limits_{V_3(F)\backslash
V_3({\bf A})} E(x_{1120}(r)v_3u_3)\psi_{V_3,\gamma}(v_3)\psi(\beta
r)dv_3drdu_3
\end{equation}
There are two cases. First, the contribution of each summand when
$\beta\ne 0$ to the integral \eqref{desc221}, produces  a Fourier
coefficient which corresponds to the unipotent orbit $F_4(a_2)$. In
the summand, where $\beta=0$, we further expand along $\{x_{1000}(r)\}$.
Depending on $\gamma$, the nontrivial orbit contributes Fourier
coefficients which corresponds to unipotent orbits $F_4(a_1)$ and
$F_4$. The trivial orbit produces an integral of the type
\begin{equation}\label{desc222}\notag
\int\limits_{V_3(F)\backslash V_3({\bf A})}
E^{U(C_3)}(v_3)\psi_{V_3,\gamma}(v_3)dv_3
\end{equation}

The computation of the constant term along the unipotent group $U_2$
is similar and gives the same result. We record this as

\begin{proposition}\label{g21}

Suppose that the representation $\mathcal{E}$ has no nonzero Fourier
coefficients which corresponds to the unipotent orbits $F_4$,
$F_4(a_1)$ and $F_4(a_2)$. Suppose also that $E^{U(C_3)}$ is zero
for all functions $E\in \mathcal{E}$. Then the automorphic
representation $\sigma$ is a cuspidal representation.

\end{proposition}

Next we consider the nonvanishing of the descent. Here we have two
cases to consider. The first, is when the lift is generic. The
integral we consider is
\begin{equation}\label{desc201}\notag
\int\limits_{(F\backslash {\bf A})^6}\int\limits_{V(F)\backslash
V({\bf
A})}E(vz(m_1,m_2,m_3)y(l_1,l_2,l_3))\psi_V(v)\psi(l_1+m_1)dm_idl_jdv
\end{equation}
where the notations are defined in \eqref{desc21}. As in the part of
the cuspidality, we start with some roots exchange ( See subsection
2.2.2). First, we exchange $(0001)$ by $(1110)$, then $(0011)$ by
$(1100)$ and $(0111)$ by $(1000)$. Thus, the above integral is equal
to
\begin{equation}\label{desc202}
\int\limits_{{\bf A}^3}\int\limits_{Y(F)\backslash Y({\bf
A})}\int\limits_{V_1(F)\backslash V_1({\bf A})}E(v_1yl(r_1,r_2,r_3))
\psi_{V_1}(v_1)\psi_Y(y)dydv_1dr_k
\end{equation}
Here $V_1$ is the subgroup of $V$ consisting of all roots in $V$
omitting the roots $(0001); (0011)$ and $(0111)$. Thus $\text{dim}\
V_1=12$. Next, $Y$ is the maximal unipotent subgroup of $Spin_7$ as
embedded in $F_4$ as the Levi part of $P_{\alpha_1,\alpha_2,\alpha_3}$. Thus, the roots in $Y$ are all nine roots in
$F_4$ of the form $n_1\alpha_1+n_2\alpha_2+n_3\alpha_3$. The character $\psi_Y$ is
defined as
$\psi_Y(y)=\psi_Y(x_{1000}(r_1)x_{0100}(r_2)y')=\psi(r_1+r_2)$. 
Finally, we have
$l(r_1,r_2,r_3)=x_{0001}(r_1)x_{0011}(r_1)x_{0111}(r_3)$.

We have $h_{\widetilde{A}_2}(t)=h(t^4,t^8,t^6,t^4)$. We are
computing the Whittaker coefficient of the lift, which corresponds
to the unipotent orbit of $G_2$ whose label is $G_2$. The
corresponding torus, as embedded in $F_4$, is $h(t^6,t^{10},t^6,1)$.
Thus the product of these two tori is $h(t^{10},t^{18},t^{12},t^4)$.
Conjugating by $w[234]$ we get $h(t^{10},t^{20},t^{14},t^8)$ which
is equal to $h_{F_4(a_2)}(t)$. It is convenient to conjugate by
$w[3234]$ and thus, integral \eqref{desc202} is equal to
\begin{equation}\label{desc203}\notag
\int\limits_{{\bf A}^3}\int\limits_{(F\backslash {\bf
A})^2}\int\limits_{V_2(F)\backslash V_2({\bf
A})}E(v_2x_{1000}(m_1)x_{-0120}(m_2)w[3234]l(r_1,r_2,r_3))
\psi_{V_2}(v_2)dv_2dm_idr_k
\end{equation}
Here, $V_2$ is the unipotent subgroup of $F_4$ whose dimension is 19
and consists of all positive roots in $F_4$ omitting the roots
$(1000);\ (0010);\ (0110);\ (0120)$ and $(0121)$. The character
$\psi_{V_2}$ is defined by $\psi_{V_2}(v_2)=
\psi_{V_2}(x_{0001}(r_1)x_{0100}(r_2)x_{1110}(r_3)x_{1120}(r_4)v_2')
=\psi(r_1+r_2+r_3+r_4)$.

Next we exchange the root $-(0120)$ by $(0121)$ and $(0110)$ by
$(1000)$. Then we expand the integral along the unipotent subgroup
$\{x_{0120}(r)\}$. Thus, the above integral is equal to
\begin{equation}\label{desc204}\notag
\sum_{\beta\in F}\int\limits_{{\bf
A}^5}\int\limits_{V_3(F)\backslash V_3({\bf
A})}E(v_3x_{1000}(m_1)x_{-0120}(m_2)w[3234]l(r_1,r_2,r_3))
\psi_{V_3,\beta}(v_3)dv_3dm_idr_k
\end{equation}
Here $V_3$ is the unipotent subgroup of $F_4$ which consists of all
positive roots omitting the two roots $(1000)$ and $(0010)$. Thus,
$\text{dim}\ V_3=22$. Also
$$\psi_{V_3,\beta}(v_3)=\psi_{V_3,\beta}
(x_{0001}(r_1)x_{0100}(r_2)x_{1110}(r_3)x_{0120}(r_4)x_{1120}(r_5)v_3')=
\psi(r_1+r_2+r_3+\beta r_4+r_5)$$ 
Arguing as in \cite{Ga-S}, the above integral is nonzero
for some choice of data if and only if the integral
\begin{equation}\label{desc205}
\sum_{\beta\in F}\int\limits_{V_3(F)\backslash V_3({\bf A})}E(v_3)
\psi_{V_3,\beta}(v_3)dv_3
\end{equation}
is not zero for some choice of data. In the notations of subsection
2.2.1 the group $V_3=U_{\alpha_1,\alpha_3}$, and the
character $\psi_{V_3,\beta}$ is defined by
$$\psi_{V_3,\beta}(v_3)=\psi(z(m_1,m_2)y(r_1,\ldots,r_6)v_3')=\psi(
m_1+r_1+r_4+\beta r_5+r_6)$$

For $\gamma\in F$, write
$E(v_3)=E(v_3x_{0010}(\gamma)x_{0010}(-\gamma))$ and conjugate the
element $x_{0010}(\gamma)$ to the left across $v_3$. Changing
variables will change the character $\psi_{V_3,\beta}$. We write
down the commutation relations needed for the above conjugation
$[x_{1110}(r),x_{0010}(s)]=x_{1120}(2rs)$;
$[x_{0110}(r),x_{0010}(s)]=x_{0120}(2rs)$;
$[x_{1100}(r),x_{0010}(s)]=x_{1110}(rs)x_{1120}(rs^2)$ and the
relation  $[x_{0100}(r),x_{0010}(s)]=x_{0110}(rs)x_{0120}(rs^2)$.
The conjugation
$x_{0010}(-\gamma)v_3x_{0010}(\gamma)$ transforms the character
$\psi_{V_3,\beta}$  to the character
$$\psi(m_1+(1+\beta\gamma^2)r_1+(\gamma^2-\gamma)r_2-2\beta\gamma r_3+(1-2\gamma)r_4+
\beta r_5+r_6)$$ Notice that only when $\gamma=1$ and $\beta=-1$,
then the coefficients of $r_1$ and $r_2$ are zero. Choose $\gamma=1$. We separate the
sum in \eqref{desc205} into two summands. First, consider the
contribution when $\beta=-1$. Performing the above conjugation, we
obtain
\begin{equation}\label{desc206}\notag
\int\limits_{V_3(F)\backslash V_3({\bf A})}E(v_3x_{0010}(1))
\psi_1(v_3)dv_3
\end{equation}
where
$$\psi_1(v_3)=\psi(m_1+2r_3-r_4-r_5+r_6)=\psi(m_1+ \text{tr}\
\begin{pmatrix} 2&-1\\ -1& 1\end{pmatrix}\begin{pmatrix} r_3&r_4\\
 r_5& r_6\end{pmatrix})$$ The group $GL_2(F)$ which contains the
group $SL_2(F)=<x_{\pm 1000}(r)>$ acts on the group
$V_3$.
Since the matrix $\begin{pmatrix} 2&-1\\ -1& 1\end{pmatrix}$
is invertible, we can find a suitable matrix in $\delta\in GL_2(F)$,
such that the above integral is equal to
\begin{equation}\label{desc207}\notag
\int\limits_{V_3(F)\backslash V_3({\bf A})}E(v_3\delta x_{0010}(1))
\psi_2(v_3)dv_3
\end{equation}
Here $\psi_2(v_3)=\psi(m_1+r_4+r_5)$.

Consider the Weyl element $w_0=w[1234213]$. Using the fact that
$E(g)=E(w_0g)$, we conjugate this Weyl element to the right in the
above integral, and we obtain
\begin{equation}\label{desc208}\notag
\int\limits_{(F\backslash {\bf A})^5} \int\limits_{L(F)\backslash
L({\bf A})}\int\limits_{V_4(F)\backslash V_4({\bf A})}E(v_4l
a(m_1,\dots,m_5)\mu) \psi_L(l)dv_4dldm_i
\end{equation}
Here $\mu=w_0\delta x_{0010}(1)$, and $L$ is the maximal unipotent
subgroup of $Sp_6$  embedded in $F_4$ as  the Levi part of $P_{\alpha_2,\alpha_3,\alpha_4}$. The character $\psi_L$ is the Whittaker character of $L$. In other words,
$$\psi_L(l)=\psi_L(x_{0100}(l_1)x_{0010}(l_2)x_{0001}(l_3)l')=\psi(l_1+l_2+l_3)$$
The group $V_4$ is the unipotent group generated
by all $\{x_\alpha(r)\}$ where $\alpha$ is a root in
$$\{(1122);\ (1221);\ (1222);\ (1231);\ (1232);\ (1241);\ (1342);\
(2342)\}$$ Finally, we have
$$a(m_1,\dots,m_5)=x_{-1000}(m_1)x_{-1100}(m_2)x_{-1110}(m_3)
x_{-1111}(m_4)x_{-1120}(m_5)$$ Next we consider five root exchanges.
First, we exchange $-1120$ by $1220$. Then, $-1111$ by $1121$,
$-1110$ by $1120$, $-1100$ by $1110$ and $-1000$ by $1100$. After
these roots exchange, we expand the integral along
$\{x_{1000}(r_1)x_{1111}(r_2)\}$. Thus, the above integral is equal to
\begin{equation}\label{desc209}\notag
\int\limits_{ {\bf A}^5}\sum_{\beta,\gamma\in F}
\int\limits_{L(F)\backslash L({\bf A})}\int\limits_{V_5(F)\backslash
V_5({\bf A})}E(v_5l a(m_1,\dots,m_5)\mu)
\psi_L(l)\psi_{V_5,\beta,\gamma}(v_5)dv_5dldm_i
\end{equation}
Here $V_5=U(C_3)$ where $U(C_3)$ was defined right before equation \eqref{desc221}.  Also, we define the character
$\psi_{V_5,\beta,\gamma}(v_5)=\psi_L(x_{1000}(r_1)x_{1111}(r_2)v_5')=
\psi(\beta r_1+\gamma r_2)$.  There are several cases to
consider. First, if $(\beta,\gamma)=(0,0)$ then we obtain the
integral
\begin{equation}\label{desc210}
\int\limits_{ {\bf A}^5} \int\limits_{L(F)\backslash L({\bf
A})}E^{U(C_3)}(v_5l a(m_1,\dots,m_5)\mu) \psi_L(l)dldm_i
\end{equation}
When $\gamma\ne 0$, then after conjugating by the Weyl element
$w[21]$ we obtain a Fourier coefficient corresponding to the
unipotent orbit $F_4(a_1)$. When $\gamma=0$ and $\beta\ne 0$, we
obtain a Fourier coefficient which corresponds to the unipotent
orbit $F_4$.

Returning to integral \eqref{desc205}, so far we analyzed the
contribution from the term $\beta=-1$. We still need to consider the
integral
\begin{equation}\label{desc2100}\notag
\sum_{-1\ne\beta\in F}\int\limits_{V_3(F)\backslash V_3({\bf
A})}E(v_3) \psi_{V_3,\beta}(v_3)dv_3
\end{equation}
It follows from the description of the action on the group
characters of $V_3(F)\backslash V_3({\bf A})$, as given in
subsection 2.2.1, that each summand in the above integral is a
Fourier coefficient associated with the unipotent orbit $F_4(a_2)$. This completes the computations of the Whittaker coefficient of the descent.

The next case to consider is when the descent has no Whittaker
coefficient. In other words, the Fourier coefficient corresponding
to the unipotent orbit whose label is $G_2$, is zero for all choice
of data. In this case, since $\sigma$ is a cuspidal representation,  it has a nonzero Fourier coefficient associated with the unipotent orbit  $G_2(a_1)$. These Fourier coefficients are described in \cite{J-R}.
Consider the unipotent  group  $U_1$ introduced at the beginning of
this subsection. We introduce coordinates on this group as follows.
Let
$$m(r_1,\ldots,r_5)=x_{0100}(r_1)x_{1100}(-r_2)x_{0110}(r_2)x_{1110}(-r_3)
x_{0120}(r_3)x_{1120}(r_4)x_{1220}(r_5)$$ Following \cite{J-R}, we
defined three characters on this group. For $u\in U$ define
$\psi_1(u)=\psi(r_2+r_3);\ \psi_{2,a}(u)=\psi(ar_1+r_3)$ and
$\psi_{3,b,c}(u)=\psi(cr_1+br_2+r_4)$. Here $a,b,c\in F^*$.

As above, the one dimensional torus corresponding to the unipotent
orbit $\widetilde{A}_2$ is
$h_{\widetilde{A}_2}(t)=h(t^4,t^8,t^6,t^4)$ and
$h_{G_2(a_1)}(t)=h(t^2,t^4,t^2,1)$. Hence the product of these two
tori elements is $h_{F_4(a_3)}(t)=h(t^6,t^{12},t^8,t^4)$. The
Fourier coefficient we need to calculate is given by
\begin{equation}\label{desc22}\notag
\int\limits_{U_1(F)\backslash U_1({\bf A})}
\int\limits_{V(F)\backslash V({\bf
A})}E(vu_1)\psi_V(v)\psi_{U_1}(u_1)dvdu_1
\end{equation}
where $\psi_{U_1}$ is any one of the three type of characters
introduced above. As in the above computations, we first perform two
root exchange as explained in subsection 2.2.2. First, we exchange
the root $(0001)$ with the root $(1110)$, and then exchange the root
$(0011)$ with the root $(1100)$. Thus, the above integral is equal
to
\begin{equation}\label{desc23}\notag
\int\limits_{{\bf A}^2} \int\limits_{U_\Delta(F)\backslash
U_\Delta({\bf
A})}E(ux_{0001}(r_1)x_{0011}(r_2))\psi_{U_\Delta}(u)dr_idu
\end{equation}
Here $\Delta=\{\alpha_1,\alpha_2,\alpha_4\}$ and $\psi_{U_\Delta}$
is a character of $U_\Delta(F)\backslash U_\Delta({\bf A})$ which is
determined by the character $\psi_{U_1}$ as follows. Write an
element $u\in U_\Delta$ as $u=y(r_1,\ldots,r_6)z(m_1,\ldots,m_6)u'$
as right before \eqref{f4a3} in subsection 2.2.1. If
$\psi_{U_1}=\psi_1$, then
$$\psi_{U_\Delta}(u)=\psi_{U_\Delta}(
y(r_1,\ldots,r_6)z(m_1,\ldots,m_6)u')=\psi(r_5+m_1+m_2+m_4)$$ If
$\psi_{U_1}=\psi_{2,a}$ then
$\psi_{U_\Delta}(u)=\psi(ar_1+r_5+m_2+m_4)$. Finally, if
$\psi_{U_1}=\psi_{3,b,c}$ then
$\psi_{U_\Delta}(u)=\psi(cr_1+r_5+bm_1+m_3+m_4)$.

We summarize

\begin{proposition}
Let $\mathcal{E}$ denote an automorphic representation of $F_4({\bf
A})$, and
consider its descent to the exceptional group $G_2({\bf A})$.\\
{\bf a)}\ Then, the Whittaker coefficient of the descent is a sum of
Fourier coefficients corresponding to the unipotent orbits
$F_4(a_2),\ F_4(a_1),\ F_4$ and the constant term integral
\eqref{desc210}. In other words we have
$$\widetilde{A}_2(\mathcal{E})\circ
G_2=F_4(a_2)+F_4(a_1)+F_4+
\mathcal{CT}_{F_4,P_{\alpha_2,\alpha_3,\alpha_4}}[(6)_{Sp_6}]$$ {\bf
b)}\ The Fourier coefficient of the decent which corresponds to the
Fourier coefficient of $G_2$ whose label is $G_2(a_1)$ corresponds
to the Fourier coefficient $F_4(a_3)$. In other words
$$\widetilde{A}_2(\mathcal{E})\circ G_2(a_1)=F_4(a_3)$$

\end{proposition}

\section{\bf Construction of Small Representations in $F_4$}

In this Section we construct a few examples of small
representations $\mathcal{E}$ defined on the group $F_4({\bf A})$. 
By definition, we define a representation to be a small representation 
if it is not generic. 
We will consider two examples which
are constructed by means of residue representations of Eisenstein series. Let $\tau$
denote a generic irreducible cuspidal representation of $GSp_6({\bf
A})$. Denote by $L^S(\tau, Spin_7, s)$ the eight dimensional partial
Spin $L$ function attached to $\tau$. It follows from \cite{B-G},
\cite{V} and \cite{G-J} that if this $L$ function has a simple pole
at $s=1$, then the representation $\tau$ is a lift from a generic
cuspidal representation $\pi$ of the exceptional group $G_2({\bf
A})$. Let $E_\tau(g,s)$ denote the Eisenstein series defined on
$F_4$ which is associated with the induce representation
$Ind_{Q({\bf A})}^{F_4({\bf A})}\tau\delta_Q^s$. Here $Q=P_{\alpha_2,\alpha_3,\alpha_4}$ is the maximal parabolic subgroup of $F_4$ whose Levi part is $GSp_6$. The
poles of this Eisenstein series are determined by
$L^S(\tau,Spin_7,8s-4)L^S(\tau,St,16s-8)$. It follows from the
assumption of $\tau$, that the Eisenstein series has a simple pole
at $s=5/8$. Let $\mathcal{E}_\tau$ denote the residue representation
at that point.

To construct a second example, let $\tau$  denote an irreducible
cuspidal representation of $GL_2({\bf A})$, and let $\pi$ denote an
irreducible cuspidal representation of $GL_3({\bf A})$. Let
$E_{\tau,\pi}(g,s)$ denote the Eisenstein series of $F_4$ associated
with the induced representation $Ind_{R({\bf A})}^{F_4({\bf
A})}(\tau\times\pi)\delta_R^s$. Here $R$ is the maximal parabolic
subgroup of $F_4$ whose Levi part contains the group $SL_2\times
SL_3$ generated by $\{x_{\pm(1000)}(r); x_{\pm (0010)}(r); x_{\pm
(0001)}(r)\}$. The poles of this Eisenstein series are determined by
$$L^S(\tau\times\pi,5(s-1/2))L^S(\text{Sym}^2\tau\times\pi,10s-5)
L^S(\tau,15(s-1/2))L^S(\pi,20s-10)$$ Assume that $\pi$ is the
symmetric square lift of $\tau$. Then the degree nine  partial $L$
function $L^S(\text{Sym}^2\tau\times\pi,10s-5)$ has a simple pole at
$s=3/5$. If also $L^S(\tau\times\pi,1/2)$ is not zero, then the
Eisenstein series $E_{\tau,\pi}(g,s)$ has a simple pole at $s=3/5$,
and we shall denote by $\mathcal{E}_{\tau,\pi}$ the residual
representation at that point. We prove

\begin{proposition}\label{g22}

With the above notations, we have ${\mathcal
O}(\mathcal{E}_{\tau})=C_3$, and ${\mathcal
O}(\mathcal{E}_{\tau,\pi})=F_4(a_3)$.

\end{proposition}

\begin{proof}

We start with the representation $\mathcal{E}_\tau$. We need to
prove two things. First we need to prove that $\mathcal{E}_\tau$,
has no nonzero Fourier coefficients which corresponds to the
unipotent orbits which are greater than the unipotent orbit $C_3$ or
not related to it. It follows from \cite{C} that we need to prove
that $\mathcal{E}_\tau$, has no nonzero Fourier coefficients which
corresponds to the unipotent orbits $B_3, F_4(a_2), F_4(a_1)$ and
$F_4$. This we prove by a local argument. Indeed, let $\nu$ be a
finite place such that the local constituent of $\mathcal{E}_\tau$, which
we denote by $(\mathcal{E}_\tau)_\nu$, is unramified. Thus
$(\mathcal{E}_\tau)_\nu=Ind_{B}^{F_4}\chi\delta_P^{1/8}\delta_B^{1/2}$.
Here $B$ is the standard Borel subgroup of $F_4$, and $\chi$ is an
unramified character of $B$. We  omit the reference  to $\nu$ in the
notations. Let $T$ be the maximal torus of $F_4$, and we
parameterize it as $h(t_1,t_2,t_3,t_4)$. Assume that
$\chi(h(t_1,t_2,t_3,t_4)=\prod\chi_i(t_i)$ where $\chi_i$ are
unramified characters. We assume that $\tau$ is a lift from the
exceptional group $G_2$. Thus, the eight parameters of the Spin
representation are $\chi_2\chi_3(p),
\chi_2(p),\chi_3(p),1,1,\chi_3^{-1}(p),\chi_2^{-1}(p),\chi_2^{-1}(p)\chi_3^{-1}(p)$
where $p$ is a generator of the maximal ideal in the ring of
integers of $F_\nu$. From this we obtain the two relations
$\chi_1\chi_2\chi_3=\chi_1\chi_2\chi_3\chi_4=1$. Let
$w_0=w[1213423]$. Then
$$(\chi\delta_P^{1/8})^{w_0}(h(t_1,t_2,t_3,t_4))=(\chi\delta_P^{1/8})
(h(t_1t_2t_3^{-2},t_1t_2^2t_3^{-4}t_4^2,t_1t_2t_3^{-2}t_4,t_2t_3^{-1}))=$$
$$=\chi_1^{-2}\chi_2^{-4}\chi_3^{-2}\chi_4^{-1}(t_3)\chi_2^2\chi_3(t_4)
|t_1t_2t_3^{-2}|=(\mu_\chi\delta_{B_3}^{1/2})(h(t_1,t_2,t_3,t_4))$$
Here $\mu_\chi(h(t_1,t_2,t_3,t_4))=
\chi_1^{-2}\chi_2^{-4}\chi_3^{-2}\chi_4^{-1}(t_3)\chi_2^2\chi_3(t_4)$
and $B_3$ is the Borel subgroup of $GL_3$ which contains the copy of
$SL_3$ generated by $\{x_{\pm (1000)}(r); x_{\pm (0100)}(r)\}$. Hence,
$Ind_{B}^{F_4}\chi\delta_P^{1/8}\delta_B^{1/2}$ which is isomorphic
to $Ind_{B}^{F_4}(\chi\delta_P^{1/8})^{w_0}\delta_B^{1/2}=
Ind_{B}^{F_4}\mu_\chi\delta_{B_3}\delta_L^{1/2}$ where $L$ is the
parabolic subgroup of $F_4$ whose Levi part is generated by $T$ and
$SL_3=<x_{\pm (1000)}(r),x_{\pm (0100)}(r)>$. From this we conclude that
$Ind_{L}^{F_4}\mu_\chi\delta_L^{1/2}$ is a constituent of
$Ind_{B}^{F_4}\chi\delta_P^{1/8}\delta_B^{1/2}$  where now we view $\mu_\chi$
as a character of $L$.

We now proceed as in \cite{G-R-S5}. To prove that $\mathcal{E}_\tau$
has no nonzero Fourier coefficient with respect to a certain
unipotent orbit, it is enough to show that $(\mathcal{E}_\tau)_\nu$
has no nonzero local functional which share the same invariant
properties as the Fourier coefficient. From the above discussion,
this corresponds to showing that
$Ind_{L}^{F_4}\mu_\chi\delta_L^{1/2}$ has no embedding inside
$Ind_V^{F_4}\psi_V$, where $V$ is the unipotent group, and $\psi_V$
is the character, which are associated with the unipotent orbit in
question. For example, if ${\mathcal O}=F_4$, this corresponds to
the case where $V$ is the maximal unipotent subgroup of $F_4$, and
$\psi_V$ is the Whittaker character. Since
$Ind_{L}^{F_4}\mu_\chi\delta_L^{1/2}$ has no nonzero Whittaker
character, it follows that $(\mathcal{E}_\tau)_\nu$ has no nonzero
corresponding functional, and hence $\mathcal{E}_\tau$ has no nonzero
Fourier coefficient with respect to the unipotent orbit $F_4$. Next
we consider the unipotent orbit $B_3$. The Fourier coefficients
corresponding to this orbit are described right after Proposition
\ref{b3c31}. Thus, to prove the corresponding local result, it
follows from Mackey theory that it is enough to prove the following.
Given an element $g$ in the space $L\backslash F_4/V$, there is an
unipotent subgroup  $\{x_\alpha(r)\}$ contained in $V$ such that $\psi_V(x_\alpha(r))\ne 1$ and $gx_\alpha(r) g^{-1}\in L$. It follows from the definition  of
$\psi_V$ as given before \eqref{desc11}, that it is not trivial on
$\{x_\alpha(r)\}$ where $\alpha\in \{ (1000); (0100); (0120); (0122)\}$.
Let $w$ be an element in $L\backslash F_4/SL_3V$ where $SL_3=<x_{\pm
(0010)(r)}; x_{\pm (0001)}(r)>$. Then $w$ can be chosen as a Weyl element.
Thus, every representative of $L\backslash F_4/V$ can be written as
$wh$ where $w\in L\backslash F_4/V$, and $h\in SL_3$. If
$w(1000)>0$, then choosing $\alpha=(1000)$ we obtain
$whx_{1000}(r)(wh)^{-1}\in L$. This follows from the fact that
$\{x_{1000}(r)\}$ commutes with the above copy of $SL_3$. This eliminates
most representatives in $L\backslash F_4/SL_3V$, and we are left
with the following nine Weyl elements: $$w[321];\  w[4321];\
w[324321];\  w[3214321];\  w[321324321];\  w[4321324321];$$
$$w[324321324321];\  w[3214321324321];\  w[321324321324321]$$ Thus
we need to consider elements of the form $wh$ where $w$ is one of
the above nine Weyl elements, and $h\in SL_3$. We have $wbw^{-1}\in
L$ for $w$ as above and $B$ is the Borel subgroup of $SL_3$. Also,
as follows from the description of the orbit $B_3$ right after
Proposition \ref{b3c31}, the group $SO_3$ embedded in $SL_3$
stabilizes the character $\psi_V$. Thus we may take $h\in
B\backslash SL_3/SO_3$. Representatives of this space of double
cosets are
$$A=\{ e;\  w[3];\  w[4];\  w[34]x_{0011}(r);\  w[43]x_{0011}(r);\
w[434]x_{0001}(r_1)x_{0011}(r_2)\}$$ Going over all above nine Weyl
elements $w$ and all possible elements in the set $A$ we can find a
root $\alpha$ such that $\psi_V(x_\alpha(r))\ne 1$ and that
$(wa)x_\alpha(r)(wa)^{-1}\in L$ for all $a\in A$. For example, for the
Weyl element $w[321324321]$, the root $(0122)$ is suitable for all
$a\in A$. Thus we deduce that $\mathcal{E}_\tau$ has no nonzero
Fourier coefficient with respect to the unipotent orbit $B_3$. The
other two orbits left are $F_4(a_1)$ and $F_4(a_2)$ are done in a
similar way, and we shall omit the details.

Next we prove that $\mathcal{E}_\tau$ has a nonzero Fourier
coefficient which is associated to the unipotent orbit $C_3$. In
Section 2 this Fourier coefficient was described. We recall it now.
Let $V$ denote the unipotent subgroup of $F_4$ generated by all
$\{x_\alpha(r)\}$ where we exclude the roots $(1000);\ (0100)$ and
$(0010)$. Then the Fourier coefficient associated with the unipotent
orbit $C_3$ is given by integral \eqref{desc1} where $\psi_V$ is as
follows. Write $v\in V$ as
$v=x_{1110}(r_1)x_{0120}(r_2)x_{0001}(r_3)v'$. Then $\psi_V(v)=\psi(r_1+r_2+r_3)$.  We
shall assume that integral \eqref{desc1} is zero for all choice of
data, and derive a contradiction. This assumption implies that the
integral
$$\int\limits_{F\backslash {\bf A}}\int\limits_{V(F)\backslash
V({\bf A})}E(x_{0100}(m)v)\psi_V(v)dvdm$$ is zero for all choice of
data. Let $w_0=w[1234231]$. Then $E(w_0h)=E(h)$ for all $E\in
\mathcal{E}_\tau$. Thus, we obtain that the integral
\begin{equation}\label{desc30}
\int\limits_{L_1(F)\backslash L_1({\bf A})}
\int\limits_{V_1(F)\backslash V_1({\bf A})}
\int\limits_{U_1(F)\backslash U_1({\bf A})}
E(u_1v_1l_1w_0)\psi_{U_1}(u_1)du_1dv_1dl_1
\end{equation}
is zero for all choice of data. Here $U_1$ is the maximal unipotent
subgroup of $Sp_6$ embedded inside $F_4$. The character
$\psi_{U_1}$ is the Whittaker character of $U_1$. The unipotent
group $V_1$ is generated by all $\{x_\alpha(r)\}$ where $\alpha$ is in the
set
$$\{ (1122);\ (1221);\ (1222);\ (1231);\ (1232);\ (1242);\ (1342);\
(2342)\}$$ The unipotent group $L_1$ is generated by all
$\{x_{-\alpha}(r)\}$ where $\alpha$ is in the set
$$\{ (1000);\ (1100);\ (1110);\ (1111);\ (1120)\}$$ 
In the following computations we will use the process of roots exchange. See
subsection 2.2.2 for details. Expand integral
\eqref{desc30} along the unipotent group $x_{1220}(m)$. For all
$\gamma\in F$ we have by the left invariant property of $E$, that
$E(x_{-1120}(\gamma)h)=E(h)$. Arguing as in \eqref{desc2} and
\eqref{desc3}, we collapse summation with integration, and deduce
that integral \eqref{desc30} is equal to
\begin{equation}\label{desc31}
\int\limits_{\bf A} \int\limits_{L_2(F)\backslash L_2({\bf A})}
\int\limits_{V_2(F)\backslash V_2({\bf A})}
\int\limits_{U_1(F)\backslash U_1({\bf A})}
E(u_1v_2l_2x_{-1120}(m)w_0)\psi_{U_1}(u_1)du_1dv_2dl_2dm
\end{equation}
Here $V_2$ is the unipotent group generated by $V_1$ and $\{x_{1220}(r)\}$,
and $L_2$ is the subgroup of $L_1$ generated by all roots in $V$
excluding the root $-(1120)$. Next we expand integral \eqref{desc31}
along the unipotent group $\{x_{1121}(m_1)x_{1120}(m_2)\}$. Using the
group $\{x_{-(1111)}(r_1)x_{-(1110)}(r_2)\}$, integral \eqref{desc31} is
equal to
\begin{equation}\label{desc32}
\int\limits_{{\bf A}^3} \int\limits_{L_3(F)\backslash L_3({\bf A})}
\int\limits_{V_3(F)\backslash V_3({\bf A})}
\int\limits_{U_1(F)\backslash U_1({\bf A})} E(u_1v_3l_3
z(m_1,m_2,m_3)w_0)\psi_{U_1}(u_1)du_1dv_3dl_3dm_j
\end{equation}
Here
$z(m_1,m_2,m_3)=x_{-(1120)}(m_1)x_{-(1111)}(m_2)x_{-(1110)}(m_3)$
and  the group $V_3$ is generated by $V_2$ and $\{x_{1121}(r), x_{1120}(r)\}$.
The group $L_3$ is generated by all $\{x_{-\alpha}(r)\}$ where $\alpha$ is
in the set of roots $\{ (1000);\ (1100);\ (1110)\}$. Arguing as in
\cite{Ga-S} we deduce that integral \eqref{desc32} is zero for all
choice of data if and only if the inner integration over the group
$U_1, V_3$ and $L_3$ is zero for all choice of data. Next we expand
the inner integration along the unipotent group $\{x_{1111}(r)\}$. The
contribution from the nontrivial orbit is zero. Indeed, this
contribution produces a Fourier coefficient which is associated to
the unipotent orbit $F_4(a_2)$. By the first part of the
Proposition, the representation $\mathcal{E}_\tau$ do not have a nonzero
Fourier coefficients corresponding to this unipotent orbit. Hence we are left with the contribution of the constant term. As in the expansions in integrals \eqref{desc31}
and \eqref{desc32} we expand along $\{x_{1110}(r)\}$ and use for it the
group $\{x_{-(1100)}(r)\}$. Then we repeat the same process with $\{x_{1100}(r)\}$
and $\{x_{-(1000)}(r)\}$. Hence, the integral
\begin{equation}\label{desc33}
\int\limits_{V_4(F)\backslash V_4({\bf A})}
\int\limits_{U_1(F)\backslash U_1({\bf A})}
E(u_1v_4)\psi_{U_1}(u_1)du_1dv_4
\end{equation}
is zero for all choice of data. Here $V_4$ is the unipotent subgroup
of $U(C_3)$ generated by all $x_\alpha(r)\in U(C_3)$ excluding the root
$(1000)$. Here $U(C_3)=P_{\alpha_2,\alpha_3,\alpha_4}$ .
Finally, we expand integral \eqref{desc33} along the unipotent group
$x_{1000}$. The nontrivial orbit contributes zero, since the Fourier
coefficient obtained is associated with the unipotent orbit $F_4$.
Thus we are left only with the constant term. From this we deduce
that the integral
\begin{equation}\label{desc34}
\int\limits_{U_1(F)\backslash U_1({\bf A})}
E^{U(C_3)}(u_1)\psi_{U_1}(u_1)du_1
\end{equation}
is zero for all choice of data. However, from the definition of
$\mathcal{E}_\tau$ and from the fact that $\tau$ is generic this is
a contradiction. This concludes the proof of the Proposition for the
representation $\mathcal{E}_\tau$.

Next we consider the representation $\mathcal{E}_{\tau,\pi}$. For
simplicity we shall assume that $\tau$ has a trivial central
character. Since we assume that
$L^S(\text{Sym}^2\tau\times\pi,10s-5)$ has a simple pole at $s=3/5$,
this means that $\pi$ is the symmetric square lift of $\tau$. Thus,
if $(\mathcal{E}_{\tau,\pi})_\nu$ is the unramified constituent of
$\mathcal{E}_{\tau,\pi}$ at a finite place $\nu$, then it is
isomorphic to $Ind_B^{F_4}\bar{\chi}\delta_R^{1/10}\delta_B^{1/2}$.
Here $\bar{\chi}$ is the character of $T$ given by
$\bar{\chi}(h(t_1,t_2,t_3,t_4))=\chi^2(t_1t_3t_4)\chi^{-3}(t_2)$
which is extended trivially to $B$. Let $w_0=w[2132134324]$. Then
$$(\bar{\chi}\delta_R^{1/10})^{w_0}(h(t_1,t_2,t_3,t_4))=
(\bar{\chi}\delta_R^{1/10})(h(t_1t_2^{-1}t_4^2,t_1^2t_2^{-3}t_3^2t_4^2,
t_1t_2^{-2}t_3^2t_4,t_1t_2^{-1}t_3))=$$
$$\chi(t_2)|t_1^2t_2^{-3}t_3^2t_4^2|^{1/2}=\mu_\chi\delta^{1/2}_{B_2\times
B_3}(h(t_1,t_2,t_3,t_4)$$ Here
$\mu_\chi(h(t_1,t_2,t_3,t_4))=\chi(t_2)$ and $B_2\times B_3$ is the
Borel subgroup of the Levi part of the maximal parabolic subgroup
$R$. Arguing as in the previous case, we deduce that
$(\mathcal{E}_{\tau,\pi})_\nu$ is the unramified constituent of
$Ind_R^{F_4}\mu_\chi\delta_R^{1/2}$.

To prove that ${\mathcal O}(\mathcal{E}_{\tau,\pi})=F_4(a_3)$ we
first need to prove that $\mathcal{E}_{\tau,\pi}$ has no nonzero
Fourier coefficient associated with any unipotent orbit which is
greater than $F_4(a_3)$. This is done by showing that the local
constituent $(\mathcal{E}_{\tau,\pi})_\nu$ at an unramified finite
place cannot support a suitable functionals. This is done by a
double coset argument in the same way as for the representation
$\mathcal{E}_\tau$, and hence will be omitted.

To complete the proof we need to show that $\mathcal{E}_{\tau,\pi}$
has a nonzero Fourier coefficient associated with the unipotent
orbit $F_4(a_3)$. We first show that it has a nonzero Fourier
coefficient associated with the unipotent orbit $\widetilde{A}_2+
A_1$. To prove that we need to show that integral \eqref{desc1} is
not zero for some choice of data. Here $V$ is the unipotent group
defined as follows. Let $V'=U_{\alpha_1,\alpha_3}$. Its
dimension is 22. Let $V$ be the subgroup of $V'$ generated by all
$x_\alpha(r)\in V'$ excluding the roots
$$\{ (0100);\ (1100);\ (0110);\ (1110);\ (0120);\ (1120);\ (0001);\
(0011)\}$$ The character $\psi_V$ is defined as follows. Write $v\in
V$ as $v=x_{0121}(r_1)x_{1111}(r_2)x_{1220}(r_3)v'$. Then $\psi_V(v)=\psi(r_1+r_2+r_3)$. We
shall assume that integral \eqref{desc1} is zero for all choice of
data, and derive a contradiction. Let $w_0=w[213213432]$. Using the
left invariance property of $E$, we deduce that the integral
\begin{equation}\label{desc35}
\int\limits_{L_1(F)\backslash L_1({\bf A})}
\int\limits_{V_1(F)\backslash V_1({\bf A})}
\int\limits_{U_1(F)\backslash U_1({\bf A})}
E(u_1v_1l_1w_0)\psi_{U_1}(u_1)du_1dv_1dl_1
\end{equation}
is zero for all choice of data. Here $U_1$ is the maximal unipotent
subgroup of $SL_2\times SL_3$ which is contained in the Levi part of
$R$. The character $\psi_{U_1}$ is the Whittaker character of this
group. The group $V_1$ is generated by all $\{x_\alpha(r)\}$ where $\alpha$
is a root in the set $\{ (1242);\ (1232);\ (1122);\ (1121);\
(0122)\}$. The group $L_1$ is generated by all $\{x_{-\alpha}(r)\}$ where
$\alpha$ is a root in the set $\{ (1221);\ (1220);\ (1100);\
(0110);\ (0100)\}$. Since integral \eqref{desc35} is zero for all
choice of data, then any of its Fourier coefficients is zero. Thus,
we deduce that the integral
\begin{equation}\label{desc36}
\int\limits_{(F\backslash A)^2} \int\limits_{L_1(F)\backslash
L_1({\bf A})} \int\limits_{V_1(F)\backslash V_1({\bf A})}
\int\limits_{U_1(F)\backslash U_1({\bf A})}
E(u_1x_{1342}(r_1)x_{2342}(r_2)v_1l_1w_0)\psi_{U_1}(u_1)du_1dv_1dl_1dr_1dr_2
\end{equation}
is zero for all choice of data. Next we expand integral
\eqref{desc36} along the unipotent group $\{x_{1231}(r)\}$. Using the
unipotent group $\{x_{-(1221)}(r)\}$, and arguing in a similar way as in
the  integrals \eqref{desc2} and \eqref{desc3}, we deduce that the
integral
\begin{equation}\label{desc37}\notag
\int\limits_{\bf A} \int\limits_{L_2(F)\backslash L_2({\bf A})}
\int\limits_{V_2(F)\backslash V_2({\bf A})}
\int\limits_{U_1(F)\backslash U_1({\bf A})}
E(u_1v_2l_2x_{-(1221)}(m)w_0)\psi_{U_1}(u_1)du_1dv_2dl_2dm
\end{equation}
is zero for all choice of data. Here $V_2$ is the group generated by
$V_1$ and $\{x_\alpha(r)\}$ where $\alpha$ is in the set $\{(1231);\
(1342);\ (2342)\}$. The group $L_2$ is the subgroup of $L_1$
excluding $\{x_{-(1221)}(r)\}$. We can continue this process. The vanishing
assumption implies either that any Fourier coefficient of the
integral is zero, or we can perform, as above, Fourier expansions
and use collapsing of summation with integration as in a similar way
as in \eqref{desc2} and \eqref{desc3}. Eventually, we deduce that
the integral
\begin{equation}\label{desc38}
\int\limits_{L_1({\bf A})}\int\limits_{U_1(F)\backslash U_1({\bf
A})} E^{U(R)}(u_1l_1)\psi_{U_1}(u_1)du_1dl_1
\end{equation}
is zero for all choice of data. Here $U(R)$ is the unipotent radical
of $R$. Arguing as in \cite{Ga-S} we may deduce that the inner
integration of integral \eqref{desc38} is zero for all choice of
data. However, from the definition of $\mathcal{E}_{\tau,\pi}$ this
is not so. Hence we derived a contradiction.

From this we deduce that ${\mathcal O}(\mathcal{E}_{\tau,\pi})$ is
at least $\widetilde{A}_2 + A_1$. In fact, we claim that ${\mathcal
O}(\mathcal{E}_{\tau,\pi})$ cannot be equal to $\widetilde{A}_2 + A_1$. Indeed, suppose
that there is an equality. The stabilizer of the unipotent orbit
$\widetilde{A}_2 + A_1$ is a group of type $A_1$.  If we
consider integral \eqref{desc100} which corresponds to the unipotent
orbit $\widetilde{A}_2 + A_1$, it follows that the function $f(g)$
defines an automorphic function of $\widetilde{SL}_2({\bf A})$.
Hence, for some $\beta\in F^*$, the integral
$$\int\limits_{F\backslash {\bf A}}f\left ( \begin{pmatrix} 1&x\\
&1\end{pmatrix}g\right )\psi(\beta x)dx$$ is not zero for some
choice of data. This nonzero integral is a Fourier coefficient which corresponds
to a unipotent orbit which is greater than $\widetilde{A}_2 + A_1$.

Hence ${\mathcal
O}(\mathcal{E}_{\tau,\pi})>\widetilde{A}_2 + A_1$ and ${\mathcal
O}(\mathcal{E}_{\tau,\pi})\ge C_3(a_1)$. The stabilizer of the orbit $C_3(a_1)$  contains a split group of type $A_1$. Arguing in a similar way as above, we deduce that ${\mathcal
O}(\mathcal{E}_{\tau,\pi})> C_3(a_1)$, or that ${\mathcal
O}(\mathcal{E}_{\tau,\pi})\ge F_4(a_3)$. But from the local argument
introduced at the beginning of the proof, we know that ${\mathcal
O}(\mathcal{E}_{\tau,\pi})\le F_4(a_3)$. Hence we get ${\mathcal
O}(\mathcal{E}_{\tau,\pi})=F_4(a_3)$.

\end{proof}

\end{document}